\documentclass[a4paper,11pt,reqno]{amsart}
\usepackage[T1]{fontenc}
\usepackage[utf8]{inputenc}
\usepackage[margin=1.1in]{geometry}
\usepackage{lmodern}
\usepackage[english]{babel}
\usepackage{graphicx,mathrsfs}
\usepackage{cases}
\usepackage{upgreek}
\usepackage{amsmath}
\usepackage{amsfonts}
\usepackage{amssymb} 
\usepackage{amsthm}
\usepackage{bbm}
\usepackage{bm} 
\usepackage{lipsum}
\usepackage{xcolor}
\usepackage{url}
\usepackage{hyperref}
\usepackage{csquotes}
\usepackage{stmaryrd}
\allowdisplaybreaks
\numberwithin{equation}{section}

\theoremstyle{plain}
\newtheorem{thm}{Theorem}[section]
\newtheorem{lem}[thm]{Lemma}
\newtheorem{prop}[thm]{Proposition}
\newtheorem{cor}[thm]{Corollary}
\newtheorem{defn}[thm]{Definition}

\def \be {\begin{equation}}
\def \ee {\end{equation}}

\def \E {\mathbb{E}}

\def \R {\mathbb{R}}

\def \fD {\mathfrak{D}}
\def \fd {\mathfrak{d}}
\def \Int {\mathrm{Int}(\widehat{\mathcal S}_d)}
\def \H {\widehat{\mathcal{H}}}

\def\balpha{\boldsymbol \alpha}

\def\E{{\mathbb E}}
\def\ES{\llbracket d \rrbracket}
\def\ESd{\llbracket d-1 \rrbracket}

\def\R{{\mathbb R}}
\def\RR{{\mathbb R}}

\def\PP{{\mathbb P}}

\def\H{{\mathcal H}}

\renewcommand{\phi}{\varphi}
\renewcommand{\epsilon}{\varepsilon}
\renewcommand{\tilde}{\widetilde}
\renewcommand{\hat}{\widehat}

\newtheorem{innercustomgeneric}{\customgenericname}
\providecommand{\customgenericname}{}
\newcommand{\newcustomtheorem}[2]{%
  \newenvironment{#1}[1]
  {%
   \renewcommand\customgenericname{#2}%
   \renewcommand\theinnercustomgeneric{##1}%
   \innercustomgeneric
  }
  {\endinnercustomgeneric}
}

\newcustomtheorem{customthm}{Theorem}

\begin{document}

\title[Selection for potential randomized finite state MFG]{Selection by vanishing common noise for potential finite state mean field games}

\author{Alekos Cecchin}

\author{Fran\c cois Delarue}

\address[A. Cecchin and F. Delarue]
{\newline \indent Universit\'e C\^ote d'Azur, CNRS, 
Laboratoire de Math\'ematiques  J.A. Dieudonn\'e,
\newline 
\indent 28 Avenue Valrose, 06108 Nice Cedex 2, France
\newline }
\email{alekos.cecchin@univ-cotedazur.fr, francois.delarue@univ-cotedazur.fr}
\thanks{A. Cecchin and F. Delarue acknowledge the financial support of French ANR project ANR-16-CE40-0015-01
	on ``Mean Field Games''. 
}

\date{\today}



\date{\today}

\begin{abstract}
	The goal of this paper is to provide a selection principle for potential mean field games on a finite state space and, in this respect, to show that equilibria that do not minimize the corresponding mean field control problem should be ruled out. Our strategy is a tailored-made version of the vanishing viscosity method for partial differential equations. Here, the viscosity has to be understood as the square intensity of a common noise that is inserted in the mean field game or, equivalently, as the diffusivity parameter in the related parabolic version of the master equation. 
	As established in the recent contribution \cite{mfggenetic}, 
	the randomly forced mean field game becomes indeed
	uniquely solvable 
	for a relevant choice of a Wright-Fisher  common noise, the counterpart of which in the master equation is a Kimura operator on the simplex. We here elaborate on  
	\cite{mfggenetic} to make the mean field game with common noise both uniquely solvable and potential, meaning that its unique solution is in fact equal to the unique minimizer of a suitable stochastic mean field control problem. Taking the limit as the intensity of the common noise vanishes, we obtain a rigorous proof of the aforementioned selection principle. 
	As a byproduct, we get that the classical solution to the viscous master equation associated with the mean field game with common noise converges to the gradient of the value function of the mean field control problem without common noise; we hence select a particular weak solution of the master equation of the original mean field game. Lastly, we establish 
	{an intrinsic}
	uniqueness {criterion for} this solution within a suitable class.
\end{abstract}
\maketitle

\section{Introduction}

The theory of Mean Field Games (MFG) addresses Nash equilibria within infinite population of rational players subjected to mean field interactions. 
It has received a lot of attention since the pioneering works of Lasry and Lions \cite{Lasry2006,LasryLions2,LasryLions} and 
of Huang, Caines and Malham\'e \cite{HuangCainesMalhame1,Huang2007,Huang2006}.
Earlier works in the field were mostly dedicated to proving the existence of such equilibria
in a various types of settings, including 
deterministic or stochastic dynamics, 
stationary or 
time-inhomogeneous models, 
continuous or finite state spaces, local or nonlocal couplings... 
Many of the proofs in this direction go through the analysis of the so-called MFG system, which is a system of two forward and backward Partial Differential Equations (PDEs) --PDEs reducing to mere ODEs for finite state spaces-- describing both the dynamics of an equilibrium and the evolution of the cost to a typical player along this equilibrium, see for instance \cite{cardaliaguet,Gomes2013,LasryLions}
and  
\cite[Chapter 3]{CardaliaguetDelarueLasryLions} for a tiny example, together with the notes and complements in 
\cite[Chapter 3]{CarmonaDelarue_book_I} for more references. Another and slightly more recent object in the field is the master equation, which is the analogue of the Nash system for games with finitely many players and which hence describes the evolution of the value of the game in the form of a PDE set on the space of probability measures. Informally, the connection between the MFG system and the master equation is pretty simple: The MFG system is nothing but the system of characteristics of the master equation. This picture may be made rigorous when the MFG has a unique equilibrium. Provided that the coefficients of the game are smooth enough, the master equation is then expected to have itself a classical solution. In all the instances where this guess can be indeed demonstrated
(see for instance \cite{CardaliaguetDelarueLasryLions,CarmonaDelarue_book_II,cha-cri-del_AMS}
in the continuous setting and \cite{bay-coh2019,cec-pel2019} in the discrete case),
the standard assumption that is used --and in fact it is, up to some slight extensions, more or less the only one that exists-- for ensuring uniqueness is the so-called Lasry-Lions monotonocity condition, see \cite{Lasry2006,LasryLions2,LasryLions}
and \cite[Chapter 3]{CarmonaDelarue_book_I} for monotonicity on continuous state spaces and \cite{Gomes2013} on finite state spaces. In fact, monotonicity has the great advantage of being very robust (meaning that it not only forces uniqueness but also stability of the equilibria) but, at the same time, it has the drawback of being rather restrictive from a practical point of view. Unfortunately,  the master equation becomes poorly understood beyond the monotonous case. In particular, the connection between the 
MFG system and the master equation takes a dramatic turn whenever equilibria are no longer unique: In the latter case, 
there may be several possible values for the game; 
accordingly, 
classical solutions to the master equation cease to exist
and almost nothing is then known on the master equation, except maybe in few examples in which the master equation can be reduced to a one-dimensional PDE. 

This is precisely the goal of our paper to make one new step forward and to address, in a more systematic way, 
the following two questions for a suitable class of MFGs without uniqueness:
\begin{enumerate}
	\item Is it possible to select some of the equilibria of the MFG?
	\item Is it possible to select one specific solution of the master equation?
\end{enumerate}
For sure, those two questions are very challenging in full generality. Subsequently, we cannot hope for a class of MFGs that is too big. In fact, the typical examples for which those two questions have been addressed rigorously in the literature are cases where equilibria can be described through a one-dimensional parameter only, say their mean if the state space is embedded in ${\mathbb R}$, see for instance \cite{delfog2019} that addresses 
a linear-quadratic 1d MFG with Gaussian equilibria and 
\cite{bayzhang2019,cecdaifispel} that address two examples of  MFGs 
on $\{0,1\}$ and $\{-1,1\}$ respectively. We here intend to study a generalization of \cite{bayzhang2019,cecdaifispel} and to consider MFGs on a finite state space of any cardinality. However, even the latter would remain too much: We thus restrict ourselves to so-called potential games, namely to games whose cost coefficients derive from potentials; As explained in Subsection \ref{subse:example} below, this still covers the framework of \cite{bayzhang2019,cecdaifispel} and, interestingly, this provides an example where equilibria cannot be described by a single parameter.  

The great interest of potential games is that 
they are intrinsically associated with a variational problem, usually referred to as a Mean Field Control Problem (MFCP): 
The MFG {indeed} reads as a first order condition for the MFCP, meaning that 
any minimal trajectory of the corresponding MFCP solves the MFG, see for instance 
\cite{cardaliaguet,LasryLions}
and 
\cite{Gomes2013,Gueant_tree} for earlier refences on the continuous and discrete settings respectively. In short, the MFCP is here a deterministic control problem with trajectories taking values within the space of probability measures (over the state space supporting the MFG) and the cost functional of which is driven by the potentials of the cost coefficients of the original MFG. Noticeably, 
this variational interpretation of MFG has been widely used in the analysis of the MFG system, 
see for example \cite{Cardaliaguet_local_coupling,CardaliaguetGraber,CardaliaguetGraberPorrettaTonon}. Here, we want to use it as a way to rule out some of the equilibria of the MFG, namely those that are not minimizers of the MFCP: We provide examples of such equilibria in Subsection \ref{subse:example}. 
For sure, we could decide to impose this selection principle arbitrarily but, in the end, this would make little sense. The main purpose of the paper is thus to   justify rigorously such a procedure. 

Before we say more on the mathematical approaches to this selection principle, it might be worth recalling that, intuitively, MFG are to be thought of as asymptotic versions
of games with finitely many players, see for instance \cite{LasryLions} for an earlier discussion on this question together with \cite{Carmona2013,Cecchin2017,Huang2006} for a generic manner to reconstruct {approximate equilibria} to the finite game from solutions of the MFG. In this respect, the most convincing strategy for justifying the selection principle would certainly consist in proving that the equilibria 
of the corresponding finite player version of the game converge (in some way) to minimizers of the corresponding MFCP. Actually, this is precisely what is done in \cite{bayzhang2019,cecdaifispel} in a specific case where the state space has exactly two elements. However, this turns out to be a difficult approach since the passage from games with finitely many players to MFGs remains, in general and regardless any question of selection, rather subtle, see for instance \cite{CardaliaguetDelarueLasryLions,Fischer2017,Lacker_2016_convergence,Lacker2}
and 
\cite{bay-coh2019,cec-pel2019}
for several contributions on this matter in continuous and discrete settings respectively. Another strategy, already used in \cite{delfog2019} in a linear quadratic case, is to 
pass to the limit in a randomly forced version of the MFG, the limit being taken as the intensity of the random forcing tends to $0$. In the MFG folklore, this random forcing is usually referred to as a \textit{common} or \textit{systemic} noise, 
since it must be understood, in the finite version of the MFG, as a noise that is common to all the players --in contrast to idiosyncratic noises that are independent and specific to each given player--. 
We refer for instance to 
\cite{CardaliaguetDelarueLasryLions,CarmonaDelarueLacker} for 
two distinct approaches to continuous state MFGs 
with a common noise; as for the finite state case, 
we refer, among others,
to \cite{BertucciLasryLions}, 
the key idea of which is to  
force the finite-player system to have many simultaneous jumps at some random times prescribed by the common noise.
The key fact for us is that, provided that it is rich enough, the common noise may restore uniqueness, see for instance \cite{DelarueSPDE,fog2018} for continuous state MFGs and the recent article \cite{mfggenetic} 
for finite state MFGs (noticeably, the
latter also involves an additional repulsive forcing at the boundary); in brief, the cornerstone in 
\cite{mfggenetic} is to design a form a common noise, which we call Wright-Fisher, so that the corresponding master equation becomes a system of nonlinear PDEs driven by a so-called Kimura operator and hence enjoys the related Schauder like smoothing estimates established in \cite{epsteinmazzeo}. This paves the way for the following sketch: If we succeed to associate a variational structure to the MFG with a common noise --meaning that the unique equilibrium of the MFG with common noise is also 
the unique minimizer of some MFCP with common noise-- and if we then manage to show that the minimizer of the MFCP with common noise converges --in some {suitable} sense-- to solutions of the original MFCP without common noise, then we are done! Although it is quite clear, this idea is not so simple to implement: In short, 
the procedure used in 
\cite{mfggenetic} to restore uniqueness in finite state MFGs does not preserve the potential structure. Part of our job here is thus to 
elaborate on \cite{mfggenetic} in order to cook up a randomly forced version of the MFG that is uniquely solvable and that derives from a potential; equivalently, the corresponding master equation is required to coincide with the derivative of a suitable parabolic Hamilton-Jacobi-Bellman equation on the simplex, the analysis of which is here carried out explicitly by means of the properties of the Kimura operator associated with the common noise. Another task is then to take the limit as the intensity of the common noise tends to zero and 
to show that the solutions that are selected in this way are indeed minimizers of the MFCP without common noise, hence justifying the selection principle that we figured out. The last step in our program is to make the connection between 
the selection principle and the master equation: As for the potential MFG with a common noise, we show that the master equation has indeed a unique classical solution and that the latter converges almost everywhere to the gradient of the 
value function of the MFCP without common noise; following an earlier work of Kru\v zkov \cite{kruzkov}, 
we are then able to prove that 
this limit is in fact a weak solution to a conservative form of the master equation
and that it is the unique one that satisfies in addition a weak semiconcavity property. 
We provide a lengthy review of all these results in Section 
\ref{sec:main} below. 
The MFCP with common noise is introduced and studied in Section 
\ref{sec:3}. The related MFG with common noise is investigated in Section \ref{sec:4}. 
Selection of the equilibria is addressed in Subsection 
\ref{subse:selection}
and selection of a solution to the master equation is discussed in 
Subsection \ref{subse:distance:MFCP} by letting the common noise vanish, and in Section \ref{sec:uniqueness:master:equation} in an intrinsic manner.

\vspace{5pt}
\noindent\textbf{Notation.}
Throughout the text, the state space is taken as
$\ES:=\{1,\cdots,d\}$, for an integer $d \geq 2$.  
We use the generic notation $p = (p_i)_{i \in \ES}$ (with   $i$ in subscript) for elements of $\R^d$, while processes are usually denoted by ${\boldsymbol p}=((p^i_t)_{i=1,\dots,d })_{0 \leq t \leq T}$ (with $i$ in superscript).
Also, we let ${\mathcal S}_{d}:= \{ (p_{1},\cdots,p_{d}) \in (\RR_{+})^d : \sum_{i\in\llbracket d \rrbracket}p_{i}=1\}$ be the {$(d\!-\!1)$}-dimensional simplex. The Euclidean norm of $p\in\R^d$ is denoted by $|p|$.   We can identify 
${\mathcal S}_{d}$
with the convex polyhedron of {${\mathbb R}^{d-1}$ $\hat{\mathcal S}_{d}:=\{ (x_{1},\cdots,x_{d-1}) \in (\RR_{+})^{d-1} : \sum_{i \in \llbracket  d\!-\!1 \rrbracket} x_{i} \leq 1\}$}. In particular, we sometimes write 
``the interior'' of ${\mathcal S}_{d}$; in such a case, we implicitly {define} the interior 
of ${\mathcal S}_{d}$ as the {$(d\!-\!1)$}-dimensional interior of $\hat{\mathcal S}_{d}$. 
To make it clear, 
for some $p \in {\mathcal S}_{d}$, 
we write $p \in 
\textrm{Int}(\mathcal S_{d})$ 
to say that  
$p_{i}>0$ for any $i \in \ES$.
We also write $x\in \textrm{Int}(\hat{\mathcal S}_{d})$ to say that {$x \in \hat{\mathcal S}_{d}$}, $x_i >0$ for each $i\in \llbracket {d\!-\!1} \rrbracket$ and $\sum_{i \in \llbracket {d\!-\!1} \rrbracket} x_{i} < 1$.

We use the same convention when speaking about the boundary of ${\mathcal S}_{d}$:
For some $p \in {\mathcal S}_{d}$, 
we may write $p \in 
\partial \mathcal S_{d}$ to say that 
$p_{i}=0$ for some $i \in \ES$. 
Finally, 
$\delta_{i,j}$ is the Kronecker symbol, $r_+$ denotes the positive part of $r\in\mathbb{R}$
and, for two elements $(v_{i})_{i \in \ES}$ and $(w_{i})_{i \in \ES}$ of $\RR^d$, we sometimes denote the
inner product $\sum_{i \in \ES} v_{i} w_{i}$ by $\langle v_{\bullet},w_{\bullet}\rangle$.

\section{Main results}
\label{sec:main}
In order to state our main results, we first introduce step by step the several forms of MFGs that we handle in the paper. 
We start with the game without common noise, which is assumed to be 
potential. Borrowing from the PDE literature --which is justified here since we make a rather systematic use of the master equation--, this game could be called \textit{inviscid}. As we already explained in introduction, it might not be uniquely solvable, which fact is the basic rationale for inserting next a common noise in the dynamics. Following 
\cite{mfggenetic}, we may indeed cook up a form of noise --together with a repulsive forcing at the boundary-- that preserves the 
structure of the 
simplex 
and that forces the MFG to become uniquely solvable. Accordingly, the game with common noise should be called 
\textit{viscous}. Unfortunately, a striking point in our study is that the common noise, at least in the form postulated in 
\cite{mfggenetic}, destroys the potential structure of the game.
This prompts us to address in the end a new and tailored-made form of MFG that is
driven by both a common noise and 
a potential structure.



\subsection{A first form of MFG}
The general form of inviscid MFGs that we here consider is given by
the following fixed point problem: For some time horizon $T>0$, 
find an ${\mathcal S}_{d}$-valued continuous trajectory ${\boldsymbol p}=(p_{t})_{0 \le t \le T}$ 
that is an
optimal trajectory to the ${\boldsymbol p}$-dependent control problem 
\begin{equation}
	\label{eq:cost:functional}
	\inf_{\balpha=(\alpha_{t})_{0 \le t \le T}}
	J(\balpha;{\boldsymbol p}), \quad 
	J(\balpha;{\boldsymbol p}) = 
	\int_0^T \sum_{i \in \ES} q^i_t \left({\mathfrak L}^i(\alpha_t) + f^i(p_t)\right)dt +  \sum_{i \in \ES} q^i_T g^i(p_T),
\end{equation}
where ${\boldsymbol q}=(q_{t})_{0 \le t \le T}$ solves the Fokker-Planck (FP) equation
\begin{equation}
	\label{eq:fp}
	\dot{q}_{t}^i = \sum_{j \in \ES} q_{t}^j \alpha_{t}^{j,i}, \quad t \in [0,T], \quad i \in \ES,
\end{equation} 
subjected to
the initial condition $q_{0}=p_{0}$
and to 
the control $\balpha=((\alpha_{t}^{i,j})_{i,j \in \ES})_{0 \le t \le T}$ satisfying the  constraint
\begin{equation}
	\label{eq:fp:control}
	\alpha_{t}^{i,j} \geq 0, \quad i,j \in \ES, \ i \not = j \, ; \quad 
	\alpha_{t}^{i,i} = - \sum_{j \neq i} \alpha_{t}^{i,j}, \quad i \in \ES \, ; \quad t \in [0,T]. 
\end{equation}
Obviously, the latter constraint says that the 
trajectory 
${\boldsymbol q}$ may be interpreted as the collection of marginal distributions of 
a Markov process with rates $((\alpha_{t}^{i,j})_{i,j \in \ES})_{0 \le t \le T}$. 
In the definition of the cost functional \eqref{eq:cost:functional}, $(f^i)_{i \in \ES}$ and $(g^i)_{i \in \ES}$
are tuples of real valued enough smooth functions, the form of which is specified in the next subsection. As for the cost $({\mathfrak L}^i)_{i \in \ES}$, we take for convenience
\begin{equation}
	\label{eq:lagrangian:main}
	\mathfrak{L}^i(\alpha)= \tfrac12 \sum_{j\neq i} |\alpha^{i,j}|^2. 
\end{equation}
The MFG associated with 
\eqref{eq:cost:functional} and 
\eqref{eq:fp} has been widely studied. In this respect, it is worth recalling that uniqueness is known to hold true in a few settings only 
and may actually fail in many cases. 
The typical condition that is used in practice to ensure uniqueness is a form of monotonicity of the cost coefficients 
$f$ and $g$, but 
as recalled in Introduction
and as shown in the recent paper
\cite{mfggenetic}, 
uniqueness can be also restored --without any further need of monotonicity-- by adding 
to the dynamics of ${\boldsymbol q}$
a convenient kind of common noise together 
with a repulsive forcing at the boundary. 
In the presence of common noise, equilibria become random:
In \cite{mfggenetic},
candidates ${\boldsymbol p}$ for solving the 
equilibria are then sought as ${\mathcal S}_{d}$-valued continuous stochastic processes 
(on $[0,T]$) that are adapted to the (complete) filtration ${\mathbb F}$ generated by a collection of Brownian motions $((B_{t}^{i,j})_{0 \le t \le T})_{i,j \in \ES : i \not =j}$ --this collection forming the common noise--, 
{constructed on a given (complete) probability space $(\Omega,{\mathcal A},{\mathbb P})$}.
Accordingly, the FP equation 
\eqref{eq:fp} 
for ${\boldsymbol q}=(q_{t})_{0 \le t \le T}$
becomes a Stochastic Differential Equation (SDE) driven by both the common noise 
$((B_{t}^{i,j})_{0 \le t \le T})_{i,j \in \ES : i \not =j}$
and the environment ${\boldsymbol p}$, the general form of which is 
\begin{equation}
	\label{dynmfg}
	dq_{t}^i = \sum_{j\neq i} \left(q_{t}^j (\varphi(p^i_t) +\alpha_{t}^{j,i}) 
	- q_t^i(\varphi(p^j_t)+\alpha_t^{i,j})\right) dt + \frac{\varepsilon}{\sqrt{2}} \sum_{j\neq i} 
	\frac{q_{t}^i}{p_{t}^i} 
	\sqrt{p_{t}^i p_{t}^j} \bigl( dB_{t}^{i,j} - dB_{t}^{j,i} \bigr),
\end{equation}
for $i \in \ES$ and $t \in [0,T]$, 
with $q_{0}=p_{0}$ as initial condition. 
Forgetting for a while the presence of $\varphi$ --we comment more on it in the sequel, but, in our claim here, 
the reader may take it as zero--, a peculiar 
point with 
\eqref{dynmfg} is that, generally speaking,  
the components $(q_{t}^i)_{0 \le t \le T}$
are positive but 
the mass process 
$(\sum_{i=1}^d q_{t}^i)_{0 \le t \le T}$ is just equal to $1$ 
in the mean under the expectation 
${\mathbb E}$ carrying the common noise. We refer to   
\cite[Prop 2.3]{mfggenetic} for more details on this subtlety, but also on the solvability of 
\eqref{dynmfg}: Basically, 
\eqref{dynmfg} is uniquely solvable if $\balpha$ is a bounded process and 
$\varepsilon^2 \int_{0}^T (1/p_{t}^i) dt$ has finite exponential moments of sufficiently high order for any $i \in \ES$. 
Consistently with the fact that both 
${\boldsymbol p}$
and 
${\boldsymbol q}$
are random, the control process $\balpha$ is also assumed to be progressively-measurable with respect to 
${\mathbb F}$
and, in the resulting MFG with a common noise, the
cost \eqref{eq:cost:functional} is averaged out with respect to the expectation ${\mathbb E}$, namely
the cost functional becomes  
\begin{equation}
	\label{eq:cost:functional:noise}
	\inf_{\balpha=(\alpha_{t})_{0 \le t \le T}}
	J^{\varepsilon,\varphi}(\balpha;{\boldsymbol p}), 
	\quad 
	J^{\varepsilon,\varphi}(\balpha;{\boldsymbol p}) =
	{\mathbb E}
	\biggl[
	\int_0^T \sum_{i \in \ES} q^i_t \left({\mathfrak L}^i(\alpha_t) + f^i(p_t)\right)dt +  \sum_{i \in \ES} q^i_T g^i(p_T) \biggr].
\end{equation}

The reader must pay attention to the superscript $\varphi$ right above. Indeed, 
in addition to the common noise, 
the intensity of which is denoted by the positive parameter 
$\varepsilon$ in \eqref{dynmfg} ({which we take in $(0,1]$ in the sequel}),
the other main feature of \eqref{dynmfg}
is the additional $\varphi$ therein: From now on, we may no longer assume 
it to be zero.
As we alluded to, $\varphi$ is actually intended to induce a repulsive
drift that forces equilibria (dynamics of equilibria are obtained by taking ${\boldsymbol p}={\boldsymbol q}$ in \eqref{dynmfg})
to stay away from the boundary of the simplex, whenever $p_0 \in \mathrm{Int}(\mathcal{S}_d)$ --we return to this point in the next subsection--, which explains in the end why we 
are allowed to divide by $p_{t}^i$ in 
the right-hand side of 
\eqref{dynmfg}. 
To achieve this goal and apply the results of \cite{mfggenetic}, it suffices for the moment to assume that $\phi$ is  a non-increasing smooth function such that 
\begin{equation}
	\label{eq:varphi:2}
	\varphi(r) = \left\{
	\begin{array}{ll}
		\kappa, \quad &\text{if} \quad r \in [0, \theta],
		\\
		\geq 0, \quad &\text{if} \quad r > \theta .
	\end{array}
	\right.
\end{equation}
Here, $\kappa$ and $\theta$ are two additional positive parameters that permit to tune 
the intensity of the drift induced by $\varphi$. In this framework, assuming that $f$ and $g$ have suitable H\"older regularity (we return to this point next), the main result of 
\cite{mfggenetic} may be summarized as follows: 
there exists a constant $\kappa_{1}>0$ only depending on $\|f \|_{\infty}$, $\|g\|_{\infty}$, $T$ and $d$, such that for any ${\varepsilon \in (0,1]}$, $\theta >0$
and $\kappa \geq \kappa_{1}/\varepsilon^2$,  
for any initial condition $p_{0}$ such that 
$p_{0}^i >0$ for any $i \in \ES$, 
the MFG associated with
the dynamics
\eqref{dynmfg}
and with the cost functional 
\eqref{eq:cost:functional:noise} is uniquely solvable.

\subsection{Potential structure} 
\label{subse:potential}
As announced in  Introduction, our main objective in this paper is to provide a selection criterion for the original MFG \eqref{eq:cost:functional}--\eqref{eq:fp} --so without common noise and outside any monotonicity condition-- by letting the intensity 
$\varepsilon$ tend to $0$ 
and the support of $\varphi$ shrink to the boundary of the simplex 
in 
\eqref{dynmfg}
and 
\eqref{eq:cost:functional:noise}.
This is however a very ambitious program that goes far beyond the single scope of this paper. 
{Indeed, due to the numerous singularities that may emerge when taking the limit 
$\varepsilon \rightarrow 0$ (the fact that  
\eqref{eq:cost:functional}--\eqref{eq:fp} is not uniquely solvable makes a strong evidence for the existence of such singularities), there are indeed no good stability properties for the solutions to  
\eqref{dynmfg}--\eqref{eq:cost:functional:noise} when $\varepsilon$ is small.} 
To reduce much of the complexity, we here restrict ourselves to the so-called potential case. Following \cite{cardaliaguet,Gueant_tree,LasryLions}, we 
hence assume that the coefficients $f$ and $g$ derive from smooth potentials $F$ and $G$. Roughly speaking, 
this means that
\begin{equation}
	\label{potnabla}
	f^{i}(p) = \frac{\partial F}{\partial p_{i}}(p), \quad 
	g^i(p) = \frac{\partial G}{\partial p_{i}}(p), \quad p \in {\mathcal S}_{d},
\end{equation} 
but this writing is not completely satisfactory: In order to give a meaning to the two derivatives in the right-hand side above, both $F$ and $G$ must be in fact defined on an open subset of $\RR^d$ containing 
${\mathcal S}_{d}$ --recall that the latter is a $(d\!-\!1)$-dimensional manifold--. 
In case when $F$ and $G$ are just defined on the simplex, we may use instead the intrinsic derivative on the simplex, which identifies with a $(d\!-\!1)$-dimensional instead of $d$-dimensional vector. 
We refer to \cite[Subsection 3.2.1]{mfggenetic} for the definition of intrinsic derivatives, but say to clarify that, whenever 
$F$ is differentiable on a neighborhood of the simplex in $\R^d$, the intrinsic gradient ${\mathfrak D} F=(\fd_1 F,\dots, \fd_d F)\in \R^d$  of $F$ is simply given by 
the orthogonal {projection} of the $d$-dimensional gradient $\nabla F$  onto the orthogonal space to the $d$-dimensional vector 
$\bm{1}=(1,\cdots,1)$, which is the tangent space to the simplex. Hence we define ${\mathfrak D} F = \nabla F - \frac1d \langle \nabla F , \bm{1} \rangle \bm{1}$, and, when $F$ is just defined on the simplex, the intrinsic derivative is defined by the same formula, but rewritten as 
\[
\fd_i F(p) = {\partial_{\varepsilon} \bigl[ F \bigl( p + \varepsilon (e_{i}- \bar e) \bigr) \bigr]_{\vert \varepsilon = 0}},
\qquad {p \in \mathrm{Int}({\mathcal S}_d), \,  i\in\ES}.
\]    
{In the above definition, $e_{i}$ is defined as the $i$th vector of the canonical basis of ${\mathbb R}^d$ and 
$\bar e$ as $\bar e:= (e_{1}+\cdots+e_{d})/d$; in particular, $e_{i} - \bar e$ is a tangent vector to the simplex.}
From the construction, we have 
$\sum_{i\in\ES} \fd_i F =0$.
Therefore, from now on we assume that, for any $i\in\ES$ and $p\in {\mathrm{Int}({\mathcal S}_d)}$, 
\be 
\label{potdef}
\fd_i F(p)=  f^i(p)- \frac1d \sum_{j \in \ES} f^j(p), \qquad \fd_i G(p)=  g^i(p)- \frac1d \sum_{j \in \ES} g^j(p).
\ee
Note that this is satisfied in case $F$ is differentiable in a neighbourhood of $\mathcal{S}_d$ in $\R^d$ and \eqref{potnabla} holds, but \eqref{potdef}  is slightly more general than \eqref{potnabla} because, roughly speaking, it involves $d\!-\!1$ entries instead of $d$. In particular,
{we will see in Subsection 
\ref{subse:example}
that 
any two state mean field game is potential, in the sense that we can always find 
$F$ and $G$ satisfying 
\eqref{potdef}.}



A very appealing fact with potential games --without common noise-- is that they are naturally associated with a control problem. 
Actually, this connection is a general fact in game theory and it goes far beyond the single scope of MFGs. 
In the specific framework of MFGs, the underlying control problem is an MFCP, as we pointed out in introduction. 
In our setting --and once again without common noise--, the MFCP takes the form: 
\begin{equation}
	\label{eq:cost:functional:mfc}
	\inf_{\balpha=(\alpha_{t})_{0 \le t \le T}}
	{\mathcal J}(\balpha), \quad 
	{\mathcal J}(\balpha)=
	\int_0^T \biggl( \sum_{i \in \ES}  q^i_t  {\mathfrak L}^i(\alpha_t) + F(q_t)\biggr)dt +  G(q_T),
\end{equation}
where, as in 
\eqref{eq:cost:functional:mfc}, 
${\boldsymbol q}=(q_{t})_{0 \le t \le T}$ is a deterministic trajectory solving
\eqref{eq:fp}
subjected to
the initial condition $q_{0}=p_{0}$, for some 
given $p_{0} \in {\mathcal S}_{d}$, 
and to 
the deterministic 
control $\balpha=((\alpha_{t}^{i,j})_{i,j \in \ES})_{t \geq 0}$ satisfying the constraint
\eqref{eq:fp:control}. {For conveniency, we also assume that 
admissible controls are bounded, meaning that 
$\alpha^{i,j}\in L^\infty(0,T)$ (see footnote \ref{foo:linfty} for more details).}
The connection between the MFCP 
\eqref{eq:cost:functional:mfc}--\eqref{eq:fp}
and the MFG \eqref{eq:cost:functional}--\eqref{eq:fp}
has been widely addressed in the literature, see for instance 
\cite{cardaliaguet,LasryLions}
for continuous state MFGs
and \cite{Gomes2013,Gueant_tree} for finite state MFGs. Generally speaking, it says that any 
optimal trajectory 
${\boldsymbol p}=(p_{t})_{0 \le t \le T}$ to 
\eqref{eq:cost:functional:mfc}--\eqref{eq:fp}
{that stays away from the boundary of the simplex} 
solves 
the MFG associated with 
\eqref{eq:cost:functional}--\eqref{eq:fp}. However --and this is the starting point of our paper--, there are known instances of MFG equilibria that are not minimizers of 
the corresponding MFCP, see Subsection \ref{subse:example} below for a benchmark example. 
In this regard, 
our main result here is precisely to construct a selection procedure that rules 
out these non-minimal equilibria, meaning that rules out solutions ${\boldsymbol p}$ to the MFG 
\eqref{eq:cost:functional}--\eqref{eq:fp} that are not optimal trajectories of 
\eqref{eq:cost:functional:mfc}--\eqref{eq:fp}. Using the same terminology as in the previous subsection, our strategy is to associate with the inviscid MFG, 
which is hence an inviscid potential game, a viscous potential game with the following four features:
\begin{enumerate}
	\item The viscous potential game is associated with a viscous MFCP -- that is a MFCP with a 
	common noise of intensity $\varepsilon$--, in the sense that any 
	minimizer --we prove that they do exist-- of the viscous MFCP is an equilibrium of the viscous potential game;
	\item The viscous potential game is uniquely solvable, hence implying that its unique solution, say ${\boldsymbol p}^{\varepsilon,\varphi}$, is also the unique optimal trajectory
	of the viscous MFCP;
	\item The optimal trajectory ${\boldsymbol p}^{\varepsilon,\varphi}$ converges in the weak sense, as the viscosity $\varepsilon^2$  tends to 
	$0$, to a probability distribution ${\mathbb M}$ on ${\mathcal C}([0,T];{\mathcal S}_{d})$ that is supported by the set of 
	optimal trajectories of the inviscid MFCP;
	\item The cost functional of the viscous potential game, which is in the end a variant of $J^{\varepsilon,\varphi}$ in 
	\eqref{eq:cost:functional:noise}, converges in a suitable sense to the cost functional $J$ in 
	\eqref{eq:cost:functional}. In particular, the equilibrium cost of the viscous potential game converges to the mean of the equilibrium costs of the inviscid potential game under the limiting distribution ${\mathbb M}$.
\end{enumerate}

The combination of the first three items reads as a selection principle since it rules out equilibria of 
\eqref{eq:cost:functional}--\eqref{eq:fp} that are not optimizers of 
\eqref{eq:cost:functional:mfc}--\eqref{eq:fp}, whilst the last item guarantees 
some consistency in our approach as it says that the cost functional underpinning the approximating viscous potential game
is itself a good approximation of the original cost function $J$ in 
\eqref{eq:cost:functional}.
Although this strategy looks quite natural, it is in fact rather subtle. 
{The major obstacle is that, as we already said, the pair
\eqref{dynmfg}--\eqref{eq:cost:functional:noise}
is not a potential game, hence advocating for the search of a version that derives  from a potential.} 

Before we elucidate the form of the viscous potential game, we stress the fact that, at the end of the day, we are not able to address the limit of \eqref{dynmfg}--\eqref{eq:cost:functional:noise} --in its primary non-potential version--. 
This might seem rather disappointing for the reader, but, once again, this should not come as a suprise: Even though the 
viscous potential game has, as we clarify below, 
a structure that is very close to  
\eqref{dynmfg}--\eqref{eq:cost:functional:noise}, the lack of any good stability estimate 
on 
\eqref{dynmfg}--\eqref{eq:cost:functional:noise} for $\varepsilon$ small, makes really challenging the analysis of the distance between the solutions to 
\eqref{dynmfg}--\eqref{eq:cost:functional:noise} and the solutions to the viscous potential game defined below.

Our first step in the construction of a suitable viscous potential game is the construction of the viscous MFCP itself. 
To do so, we elaborate on \cite{mfggenetic}. 
Following \eqref{eq:cost:functional:mfc}, we can indeed associate with 
the dynamics 
\eqref{dynmfg} a stochastic control problem, which we precisely call viscous MFCP.
It has the following form:
\begin{equation}
	\label{eq:Cost:mfc}
	\inf_{\balpha=(\alpha_{t})_{0 \le t \le T}}
	{\mathcal J}^{\varepsilon,\varphi}(\balpha), \quad 
	{\mathcal J}^{\varepsilon,\varphi}(\balpha)=
	{\mathbb E } \biggl[ \int_0^T \biggl( \sum_{i \in \ES} p^i_t {\mathfrak L}^i(\alpha_t) + F(p_t)\biggr)dt +  G(p_T) \biggr],
\end{equation}
where ${\boldsymbol p}=(p_{t})_{0 \le t \le T}$ solves the $\balpha$-driven SDE
\begin{equation}
	\label{dynpot}
	dp_{t}^i = \sum_{j\neq i} \left(p_{t}^j (\varphi(p^i_t) +\alpha_{t}^{j,i}) 
	- p_t^i(\varphi(p^j_t)+\alpha_t^{i,j})\right) dt + \frac{\varepsilon}{\sqrt{2}} \sum_{j\neq i} 
	\sqrt{p_{t}^i p_{t}^j} \bigl( dB_{t}^{i,j} - dB_{t}^{j,i} \bigr),
\end{equation}
for $i \in \ES$ and $t \in [0,T]$, 
with $q_{0}=p_{0}$ as initial condition, and, as before, 
$\balpha$ is an ${\mathbb F}$-progressively measurable process satisfying 
\eqref{eq:fp:control} except for the fact that, for purely technical reasons, we will rescrit ourselves to processes whose off-diagonal coordinates are 
bounded by a constant $M$ that is explicitly given in terms of $f$, $g$ and $T$ (even though
\eqref{eq:fp:control}
just implies that 
the diagonal coordinates are bounded by $(d\!-\!1)M$, we will say abusively that such processes are bounded by $M$). 
The function $\varphi$ is chosen as in \eqref{eq:varphi:2} and the initial condition $p_0$ belongs to the interior of the simplex. Such  {a} condition on $p_0$ will be always assumed in the rest of the paper, the main reason being that it permits to apply results from \cite{mfggenetic}:
By Proposition 2.1 therein, the SDE 
\eqref{dynpot} 
({which is usually called a Wright-Fisher SDE})
is uniquely solvable in the strong sense if $\kappa \geq \varepsilon^2/2$ and the solution remains in the interior of the simplex, and further, by Proposition 2.2 {(also in \cite{mfggenetic})}, $\int_{0}^T (1/p_{t}^i) dt$ has exponential moments of sufficiently high order, if $\kappa$ is large enough. We recall that this latter integrability condition is necessary for the well-posedness of \eqref{dynmfg}.

In this framework, our first main result has some interest in its own, independently of the 
aforementioned selection principle. 
The functional spaces to which $F$ and $G$ are assumed to belong, and to which the value function is proved to belong, are defined in details in the Appendix, by means of local charts. These are called Wright-Fisher, as introduced in \cite{epsteinmazzeo}, and are used in \cite{mfggenetic} to prove well posedness of the MFG master equation. We just say here that: 
{\begin{enumerate}
\item 
$\mathcal{C}^{0,\gamma}_{\rm WF}(\mathcal{S}_d)$ consists of continuous functions  on ${\mathcal S}_{d}$ that are $\gamma$-H\"older continuous up to the boundary with respect to the metric associated with the 
Wright-Fisher noise in \eqref{dynpot};
accordingly, 
$\mathcal{C}^{1,\gamma}_{\rm WF}(\mathcal{S}_d)$ consists of continuous functions  on ${\mathcal S}_{d}$ that are continuously differentiable in $\textrm{\rm Int}({\mathcal S}_{d})$, with H\"older continuous derivatives up to the boundary;
both spaces are equipped with norms $\| \cdot \|_{{\rm WF},0,\gamma}$
and 
$\| \cdot \|_{{\rm WF},1,\gamma}$;
\item 
for $k=0,1$, 
$\mathcal{C}^{k,2+\gamma}_{\rm WF}(\mathcal{S}_d)$ consists of continuous functions  on ${\mathcal S}_{d}$ that are $2\!+\! k$ times continuously differentiable in  $\textrm{\rm Int}({\mathcal S}_{d})$, 
with derivatives satisfying a suitable behaviour at the boundary and a suitable H\"older regularity that depend on the order of the derivative; in particular, 
the derivatives
of order 1 (if $k=0$) and 
 of order 1 and 2 (if $k=1$) are H\"older continuous up to the boundary, but the derivative of order $2+k$ (i.e. 2 if $k=0$ and $3$ if $k=1$) may blow up at the boundary and be only locally H\"older continuous in the interior;
 both spaces are equipped with norms $\| \cdot \|_{{\rm WF},0,2+\gamma}$
and 
$\| \cdot \|_{{\rm WF},1,2+\gamma}$;
\item $\mathcal{C}^{0,\gamma}_{\rm WF}([0,T]\times\mathcal{S}_d)$ 
and $\mathcal{C}^{k,2+\gamma}_{\rm WF}([0,T]\times\mathcal{S}_d)$
are the parabolic versions of 
$\mathcal{C}^{0,\gamma}_{\rm WF}(\mathcal{S}_d)$,
and 
$\mathcal{C}^{k,2+\gamma}_{\rm WF}(\mathcal{S}_d)$; 
while the former consists of functions on $[0,T] \times {\mathcal S}_{d}$ that are H\"older continuous for a suitable metric, 
the latter consists of continuous functions  on $[0,T] \times {\mathcal S}_{d}$ that are continuously differentiable in 
time $t \in [0,T]$ and that are $2\!+\!k$ times continuously differentiable in space in $
\textrm{\rm Int}({\mathcal S}_{d})$, with derivatives satisfying a suitable behaviour at the boundary and a suitable H\"older regularity; in particular, the time derivative and the space derivatives up to order $1\!+\!k$ are H\"older continuous up to the boundary but the derivative of order $2\!+\!k$ may blow up at the boundary; 
the norms are also denoted by  $\| \cdot \|_{{\rm WF},0,\gamma}$
and $\| \cdot \|_{{\rm WF},k,2+\gamma}$ (below, the norm is understood as being for the parabolic space if the function in argument of the norm is time-space dependent). 
\end{enumerate}
Throughout the sequel}, the parameter $\gamma\in(0,1)$ is fixed. 
\begin{thm}
	\label{main:thm}
	{Recall 
$(\theta,\kappa)$ from \eqref{eq:varphi:2}.}	
	If $F \in \mathcal{C}^{1,\gamma}_{\rm WF}(\mathcal{S}_d)$ and $G\in \mathcal{C}^{1,2+\gamma}_{\rm WF}(\mathcal{S}_d)$, then there exists a constant $\kappa_{1}>0$ only depending on $\|f \|_{\infty}$, $\|g\|_{\infty}$, $T$ and $d$, such that for any ${\varepsilon \in (0,1]}$, $\theta >0$
	and $\kappa \geq \kappa_{1}/\varepsilon^2$,  
	and any initial state 
	$p_{0}$ in $\textrm{\rm Int}({\mathcal S}_{d})$, the 
	MFCP 
	\eqref{eq:Cost:mfc}--\eqref{dynpot}
	set over 
	${\mathbb F}$-progressively measurable processes $\balpha$ that are bounded by 
	$M=2 (\|g \|_{\infty}+ T \| f \|_{\infty})$
	has a unique solution.
	Moreover, there exists $\gamma'\in (0,\gamma]$, possibly depending on $\varepsilon$ and $\kappa$, such that the corresponding Hamilton-Jacobi-Bellman equation has a unique solution ${\mathcal V}^{\varepsilon,\varphi}$ in $\mathcal{C}^{1,2+\gamma'}_{\rm WF}([0,T]\times\mathcal{S}_d)$.
\end{thm}

The proof of this result is given in Section \ref{sec:3}, see Theorem \ref{thm:3.2}. Generally speaking and using the notation ${\mathfrak D}=({\mathfrak d}_{j})_{j \in \ES}$ and 
$\mathfrak{D}^2=({\mathfrak d}_{j,k})_{j,k \in \ES}$ for the first and second order derivatives on the simplex --the second derivative being defined similarly to the first, {see
Subsection \ref{subse:classical:solutions}
for a short account and} \cite[Subsection 3.2.1]{mfggenetic} for more details--, the Hamilton-Jacobi-Bellman (HJB) equation has the following form:
\begin{equation}
	\label{hjbpotnew:sec:2}
	\left\{
	\begin{array}{l}
		\partial_t \mathcal{V} + {\mathcal{H}}_{{M}}^{\varphi}(p, {\mathfrak D} \mathcal{V})
		+ F(p) + \frac{\varepsilon^2}{2} \sum_{j,k \in \ES}(p_j \delta_{j,k}-p_{j} p_{k}) {\mathfrak d}^2_{j,k} \mathcal{V}=0,
		\\
		\mathcal{V}(T,p)=G(p),
	\end{array}
	\right.  
\end{equation}
for $(t,p) \in [0,T] \times {\mathcal S}_{d}$, 
where ${\mathcal H}_{{M}}^{\varphi}$ is an Hamiltonian term depending explicitly on $\varphi$
{and $M$}, the precise form of 
which is not so relevant at this early stage of the paper and will be just given in the sequel of the text, {see \eqref{eq:mathcal:H:varphi:M}}. In fact,  
we feel 
more useful for the reader to be aware of the key fact that, here, this HJB equation is shown to have a unique classical solution. Obviously, this is a strong result that is true because of the presence of the common noise and, in particular, 
that bypasses any use of convexity on $F$ and $G$ (and hence of monotonicity on $f$ and $g$). 
The proof makes use of the smoothing properties obtained in \cite{epsteinmazzeo}
and \cite{mfggenetic} for so-called \textit{Kimura} diffusions that are second-order diffusion operators on the simplex:
In \eqref{hjbpotnew:sec:2}, the second-order structure manifests through the operator 
$\frac{\varepsilon^2}{2} \sum_{j,k}(p_j \delta_{j,k}-p_{j} p_{k}) {\mathfrak d}^2_{j,k}$, which is --and this is the main difficulty in the analysis-- degenerate at the boundary of the simplex. 
The latter explains why we need the forcing $\varphi$ to be sufficiently strong --see the condition $\kappa \geq \kappa_{1}/\epsilon^2$ in the statement--  in order to guarantee for the existence of a classical solution.  

Our second main result is to prove that there is a uniquely solvable MFG that derives 
from the viscous MFCP. Noticeably, this is a non-trivial fact.
The reason is that, because of the presence {of} stochastic terms in \eqref{dynpot}, the standard computations that permit to pass from inviscid MFCPs to inviscid potential games are no longer true. To wit, the result below says that the shape of the cost of the viscous potential game is not the same as the shape of the cost of the original inviscid one. 

\begin{thm}
	\label{main:thm:2}
	Take $F$, $G$, and $M$ as in the statement of Theorem \ref{main:thm}. Then, there exists $\kappa_{2} \geq \kappa_{1}$, 
{only depending on $\|f \|_{\infty}$, $\|g\|_{\infty}$, $T$ and $d$},
	 such that, for any ${\varepsilon \in (0,1]}$, any $\theta>0$  and any $\kappa \geq \kappa_{2}/\varepsilon^2$, we can find a time-dependent 
	coefficient $\vartheta_{\varepsilon,\varphi} : [0,T] \times {\mathcal S}_{d}
	\rightarrow {\mathbb R}^d$ that is continuous on 
	{$[0,T] \times {\mathcal S}_{d}$}, such that, for any 
	initial condition  $p_{0} \in \textrm{\rm Int}({\mathcal S}_{d})$, 
	the optimal trajectory ${\boldsymbol p}^{\varepsilon,\varphi}$ of 
	the 
	MFCP 
	\eqref{eq:Cost:mfc}--\eqref{dynpot} is also the unique equilibrium of the  MFG  with common noise 
	{driven by the ${\boldsymbol p}$-dependent cost functional} 
	\begin{equation}
		\label{eq:new:J:varepsilon,varphi}
		\tilde{J}^{\varepsilon,\varphi}(\balpha;{\boldsymbol p}) =
		{\mathbb E}
		\biggl[
		\int_0^T \sum_{i \in \ES} q^i_t \left( {\mathfrak L}^i(\alpha_t) +   f^i(p_t) + \vartheta_{\varepsilon,\varphi}^i(t,p_{t})\right)dt +  \sum_{i \in \ES} q^i_T g^i(p_T) \biggr],
	\end{equation}
	defined over 
	{pairs $({\boldsymbol q},{\boldsymbol \alpha})$
	solving 
	\eqref{dynmfg}, for ${\mathbb F}$-progressively measurable processes $\balpha$ that are bounded by $M$,  
	and over 
	${\mathbb F}$-adapted continuous processes ${\boldsymbol p}$ that take values in $\textrm{\rm Int}({\mathcal S}_{d})$}.  
\end{thm}

The statement of Theorem \ref{main:thm:2} deserves some explanations. {First, we feel useful to specify the definition of an equilibrium in our framework:}
\begin{defn}
	\label{def:MFG}
	With the same notation as in Theorem \ref{main:thm:2} (in particular $\kappa$ large enough), an ${\mathbb F}$-adapted continuous process ${\boldsymbol p}$ with values in ${\mathcal S}_{d}$ is said to be an equilibrium if the following two properties are satisfied:
	\vspace{4pt}
	
	(i) There exists an $M$-bounded and ${\mathbb F}$-progressively measurable process $\balpha$ such that ${\boldsymbol p}$ solves 
	the SDE 
	\eqref{dynpot}
	--obtained by equalizing ${\boldsymbol p}$ and ${\boldsymbol q}$ in 
	\eqref{dynmfg}--,
	with $p_{0}$ as initial condition;
	\vspace{4pt}
	
	(ii) For any other $M$-bounded and ${\mathbb F}$-progressively measurable process ${\boldsymbol \beta}$ for which \eqref{dynmfg} is uniquely solvable,
	$\tilde{J}^{\varepsilon,\varphi}({\boldsymbol \alpha},{\boldsymbol p}) \leq 
	\tilde{J}^{\varepsilon,\varphi}({\boldsymbol \beta},{\boldsymbol p})$.
\end{defn}


{In particular, 
from item (i) in the above definition,  
${\boldsymbol p}$ in \eqref{eq:new:J:varepsilon,varphi} is implicitly required to solve 
	\eqref{dynpot}
for some $M$-bounded and ${\mathbb F}$-progressively measurable control process
(even though this control is denoted by ${\balpha}$ in \eqref{dynpot}, we feel better not to use this notation here in order to 
distinguish from the control ${\balpha}$ used in 
\eqref{eq:new:J:varepsilon,varphi}, which stands for the control used in 
\eqref{dynmfg}). Also, as recalled above,
it is proven in \cite[Proposition 2.2 and 2.3]{mfggenetic} that, 
whatever the choice of the control in 
\eqref{dynpot}, the solution 
${\boldsymbol p}$
is uniquely defined,
provided that $\kappa$ in \eqref{eq:varphi:2} is greater than some threshold 
$\kappa_{0} \varepsilon^2$, with $\kappa_{0}$ only depending on the dimension; moreover,
$\int_{0}^T (1/p_{t}^i) dt$ has exponential moments of sufficiently 
high order so that
\eqref{dynmfg}
always has a unique solution that is square-integrable, whatever the choice of ${\boldsymbol \alpha}$ therein
(here, 
${\boldsymbol \alpha}$ fits 
${\boldsymbol \alpha}$  in 
\eqref{eq:new:J:varepsilon,varphi}).  In particular, 
under the assumption of Theorem \ref{main:thm:2} ($\kappa$ large enough), we should not worry for the exponential integrability 
of  $\int_{0}^T (1/p_{t}^i) dt$, for $i \in \ES$, nor for the well-posedness of 
\eqref{dynmfg} when {${\boldsymbol p}$
in  \eqref{eq:new:J:varepsilon,varphi}
(and hence ${\boldsymbol p}^{\varepsilon,\varphi}$ itself)}
is a candidate for solving the MFG}. 

The proof of Theorem \ref{main:thm:2} in given in Section \ref{sec:4}, together with the precise definition of the additional cost $\vartheta$; see Theorem \ref{thm4} {and \eqref{eq:vartheta}}.
 As for the latter, it is certainly fair to say that the definition of $\vartheta$ is implicit, meaning that it depends on ${\mathcal V}$ itself, which might seem a bit disappointing but looks in the end inevitable. As for the proof itself, it relies on a variant of the argument used in \cite{mfggenetic},  the main point being to take benefit of the smoothness of the 
solution to the HJB equation \eqref{hjbpotnew:sec:2} in order to identify the equilibria. In this regard,  
a key step in the proof is to expand (as in a verification argument)
the intrinsic gradient\footnote{Consistently with the notation $\vartheta$ in Theorem \ref{main:thm:2},
we here put the parameters $\varepsilon$ and $\varphi$ in subscripts as we sometimes write 
$V^i_{\varepsilon,\varphi}$ for denoting the coordinates of $V_{\varepsilon,\varphi}$; even though ${\mathcal V}_{\varepsilon,\varphi}$ is scalar-valued, we feel more consistent to use the same convention for it.} {$V_{\varepsilon,\varphi}={\mathfrak D} {\mathcal V}_{\varepsilon,\varphi}$} 
of the value function ${\mathcal V}_{\varepsilon,\varphi}$ -- solving the HJB equation \eqref{hjbpotnew:sec:2}--
along 
any possible equilibrium ${\boldsymbol p}$. This allows us to prove that, whatever the equilibrium ${\boldsymbol p}$, the optimal solution to $\tilde{J}^{\varepsilon,\varphi}(\, \cdot \, ; {\boldsymbol p})$ is in the form  
\begin{equation}
	\label{eq:optimal:control:mfg}
	(\alpha_{t}^{\varepsilon,\varphi,{\boldsymbol p}} )^{i,j}  = a^\star \Bigl( {\mathfrak d}_{i} {\mathcal V}_{\varepsilon,\varphi}(t,p_{t}) - {\mathfrak d}_{j} {\mathcal V}_{\varepsilon,\varphi}(t,p_{t}) \Bigr), \quad t \in [0,T],   \quad i\neq j,
\end{equation}
where
\begin{equation}
	\label{eq:astar}
	a^\star(r) = \left\{ 
	\begin{array}{ll}
		0 \quad &\textrm{\rm if} \quad r < 0,
		\\
		z \quad &\textrm{\rm if} \quad r \in [ 0,M],
		\\
		M \quad &\textrm{\rm if} \quad r > M.
	\end{array}
	\right.
\end{equation}
By plugging \eqref{eq:optimal:control:mfg} into 
\eqref{dynpot}, we get that any equilibrium satisfies 
the same SDE. Thanks to the smoothness properties we have on ${\mathfrak D} {\mathcal V}_{\varepsilon,\varphi}$,
the latter is uniquely solvable, hence the uniqueness property. See Section \ref{sec:4}
for more details.

\subsection{Selection}
The next step in our program is to address the asymptotic behavior of the equilibria 
$({\boldsymbol p}^{\varepsilon,\varphi})_{\varepsilon,\varphi}$ as $\varepsilon$ tends to $0$ and
the support of ${\varphi}$ shrinks to the boundary of the simplex. 
In this regard, 
one difficulty is that, in the statements of both Theorems \ref{main:thm}
and \ref{main:thm:2}, the function $\varphi$ is implicitly required to become larger and larger, as $\varepsilon$ tends to $0$,
on the interval $[0,\theta]$. Equivalently, 
 {the constant $\kappa$ therein blows}
up as $\varepsilon$ tends to $0$.  
Obviously, this looks a serious hindrance for passing to the limit. In Section \ref{sec:5} below, we bypass this difficulty by using the fact that, in the limit, the 
solutions
of the Fokker-Planck equation \eqref{eq:fp} without common noise 
cannot reach the boundary when starting from the interior of the simplex, and in fact the solution stays away from the boundary with an explicit threshold 
 {(this advocates once more for taking $p_{0}$ in $\textrm{\rm Int}({\mathcal S}_{d})$)}.
{Also, for the subsequent analysis,} we introduce a new parameter $\delta$, which is understood below as the  {half length of the} support of $\phi$: {In short, $\theta$ should be understood as the half length of the interval on which $\varphi$ blows up as $\varepsilon$ tends to $0$ (as explained right above) and 
$\delta$ for the half length of the interval on which it is non-zero, see \eqref{newphi} for the details. Obviously $\theta \leq \delta$ (in fact, we even require $2\theta \leq  \delta$); also, $\delta$ is taken small in the sequel. The dependence of the solution ${\boldsymbol p}^{\varepsilon,\varphi}$ on  {$\delta$}, $\theta$ and $\kappa$ is implicitly written as a dependence upon $\phi$}. 
We then get the following result, which holds without any further condition on $\theta$ and $\delta$, so that $\delta$ can be taken as $\theta$:
\begin{thm}
	\label{main:thm:4}
	Let the assumptions of both Theorems \ref{main:thm} and \ref{main:thm:2}
	be in force and, 
	with the same notation as in \eqref{eq:optimal:control:mfg}, let
	\begin{equation*}
		\alpha^{\varepsilon,\varphi}_{t} := 
		\alpha^{\varepsilon,\varphi,{\boldsymbol p}^{\varepsilon,\varphi}}_{t}, \quad t \in [0,T].
	\end{equation*}
	Then, for any initial condition $p_{0} \in \textrm{\rm Int}({\mathcal S}_{d})$, 
	there is a constant $\delta_{0}>0$ such that
	the family of  laws $({\mathbb P}\circ ({\boldsymbol p}^{\varepsilon,\varphi},{\boldsymbol \alpha}^{\varepsilon,\varphi})^{-1})_{{\varepsilon \in (0,1]},\delta \in (0,\delta_{0}),{2\theta \leq \delta}}$ 
	is tight {in} 
	${\mathcal{P}({\mathcal C}([0,T];{\mathbb R}^d) \times L^2([0,T];[-dM,dM]^{d^2}))}$, the first factor being equipped with the topology of uniform convergence and the second one with the weak topology. 
	Moreover, any weak limit ${\mathbb M}$, as $\varepsilon$ and $\delta$ tend to $0$, is supported by pairs $({\boldsymbol p},{\boldsymbol \alpha})$ that minimize the cost functional 
	${\mathcal J}$ in 
	\eqref{eq:cost:functional:mfc},
	with $\boldsymbol{p}= \boldsymbol{q}$ therein, with $p_{0}$ as initial condition. 
\end{thm}

The proof is given in Subsection \ref{subse:selection}, together with the precise definition of the function $\phi$ that we use; see Theorem  \ref{thm10}.
It is worth mentioning that
the inviscid MFCP may have a unique minimizer even though the MFG has several equilibria. 
To wit, we provide 
an example 
in Subsection \ref{subse:example}. In such a case, the probability ${\mathbb M}$ in the above statement reduces to one point and 
the family $({{\mathbb P} \circ} ({\boldsymbol p}^{\varepsilon,\varphi},{\boldsymbol \alpha}^{ \varepsilon,\varphi})^{-1})_{{\varepsilon \in (0,1]},\delta \in (0,\delta_{0}),{2\theta \leq \delta}}$ converges to the unique pair $({\boldsymbol p},\balpha)$ minimizing the cost functional 
${\mathcal J}$ in 
\eqref{eq:cost:functional:mfc}
(with 
${\boldsymbol q}={\boldsymbol p}$
therein). As explained in the next subsection, it happens quite often that the minimizer of the inviscid MFCP is unique: For almost every $(t,p) \in [0,T]\times \textrm{\rm Int}({\mathcal S}_{d})$, the MFCP 
\eqref{eq:Cost:mfc}--\eqref{eq:fp} has a unique solution whenever ${\boldsymbol q}$ in 
\eqref{eq:fp} starts from $p$ at time $t$. These are the points in which the value function of the inviscid MFCP is differentiable, thus they have full measure since the value function 
 {can be shown to be}
 Lipschitz in time and space,  {see Proposition \ref{prop:solution:mfc:zero:epsilon}}.

Back to the statement of 
Theorem \ref{main:thm:2}, 
we get the announced limiting behavior for the equilibria therein. 
Anyhow, the reader may also wonder about the behavior of the cost functional 
$\tilde{J}^{\varepsilon,\varphi}$ in 
\eqref{eq:new:J:varepsilon,varphi}
as $\varepsilon$ tends to $0$ and the support of $\varphi$ shrinks to the boundary  {(and hence $\delta$ vanishes)}. In fact, 
this asks us to revisit 
the shape of the coefficient $\vartheta_{\varepsilon,\varphi}$, 
which is certainly 
the most intriguing term therein,  {see again} Section \ref{sec:4}. Importantly, we learn from its construction that, in order to control the impact of 
$\vartheta_{\varepsilon,\varphi}$ accurately in the cost functional $\tilde{J}^{\varepsilon,\varphi}$,
we cannot play for free with  {$\varepsilon$, $\delta$ and $\theta$}
at the same time --{the three} of them popping up in the definition of $\varphi$--.
The reason is that, even though this may only happen with small probability, 
the process ${\boldsymbol p}^{\varepsilon,\varphi}$
may visit the neighborhood of the boundary of the simplex where the function $\varphi$ is non-zero.  {Even more}, 
$\varphi$ may become very large when $\varepsilon$ tends to $0$. Since the 
geometry of this neighborhood of the boundary of the simplex is determined by  {$\delta$
and $\theta$}, this explains why 
some trade-off between 
 {$\varepsilon$, $\delta$ and $\theta$} is necessary when averaging out the cost functional $\vartheta_{\varepsilon,\varphi}$ with respect to 
all the possible trajectories of  
${\boldsymbol p}^{\varepsilon,\varphi}$. 
In this context, the following result says that we can tune both $\delta$,  {$\theta$} and $\varphi$ in terms of $\varepsilon$ such that, along the equilibrium, the influence of 
$\vartheta_{\varepsilon,\varphi}$ vanishes as $\varepsilon$ tends to $0$:
\begin{prop}
	\label{prop:cv:cost}
	Let the assumptions of both Theorems \ref{main:thm} and \ref{main:thm:2}
	be in force.
	Then,
	for any $\varepsilon \in (0,1]$, we can choose $\delta$
	{as $\delta=\hat{\delta}(\varepsilon)$ 
	and 
	$\theta$ as $\theta=\hat{\theta}(\varepsilon) \leq \hat{\delta}(\varepsilon)/2$,
	for some 
	(strictly) positive-valued functions 
	$\hat{\delta}$
	and $\hat{\theta}$,
	 with $0$ as limit in $0$,}
	and then $\varphi=\hat{\varphi}(\varepsilon)$ 
	in 
	\eqref{eq:varphi:2},
	such that all the assumptions required in the statements of Theorems \ref{main:thm} and \ref{main:thm:2}
	are satisfied {together with the following limit:}
	\begin{equation*}
		\lim_{\varepsilon \rightarrow 0} {\mathbb E} \Bigg[ \bigg|\int_{0}^T \sum_{i \in \ES} q^{i,\varepsilon,\hat{\varphi}(\epsilon)}_t \vartheta^{i}_{\varepsilon,\hat\varphi(\varepsilon)}({\boldsymbol p}^{\varepsilon,\hat\varphi(\varepsilon)}) dt
		\bigg| \Bigg]= 0, 
	\end{equation*}
	for any $p_0\in \mathrm{Int}(\mathcal{S}_d)$
	and $q_0\in  {{\mathcal S}_{d}}$, where $\boldsymbol{q}^{\varepsilon,\hat{\varphi}(\epsilon)}$ solves \eqref{dynmfg} with initial condition $q_0$ and ${\boldsymbol p}={\boldsymbol p}^{\varepsilon,\hat\varphi(\varepsilon)}$ therein.
\end{prop} 

 {In the statement, it is implicitly understood that $\hat \varphi(\varepsilon)$ is parametrized by 
$\hat \theta(\varepsilon)$ and $\hat \delta(\varepsilon)$ (see the discussion above Theorem \ref{main:thm:4}
for the meaning of these two parameters). As for the proof, it is given in 
Subsection 
\ref{subse:distance:MFCP}, see  
Proposition \ref{prop:5:9} and Theorem \ref{thm:5.10}.} For sure, the above result says
that $\sup_{\balpha : \vert \alpha_{t} \vert \leq M}
\vert \tilde{J}^{\varepsilon,\hat\varphi(\varepsilon)}(\balpha,{\boldsymbol p}^{\varepsilon,\hat\varphi(\varepsilon)})
- {\mathbb E}J(\balpha,{\boldsymbol p}^{\varepsilon,\hat\varphi(\varepsilon)})\vert$
tends to $0$ as $\varepsilon$ tends to $0$. 
Since the sequence 
of laws
$({\mathbb P} \circ ({\boldsymbol p}^{\varepsilon,\hat\varphi(\varepsilon)},{\boldsymbol \alpha}^{ \varepsilon,\hat{\varphi}(\varepsilon)})^{-1})_{{\varepsilon \in (0,1]}}$ 
is tight in the same space as in the statement of Theorem \ref{main:thm:4}, we deduce that,
along 
any converging subsequence (still indexing the latter by $\varepsilon$)
with ${\mathbb M}$ as weak limit\footnote{It looks like \eqref{eq:convergence:values} could be recast differently, in a fashion closer to $\Gamma$-convergence, but
this would ask for more materials in the text and 
 we would make little use of it in the end. Instead, our formulation suffices to 
 address the convergence of the solution to the master equation, which is a key point in our paper.} , 
\begin{equation}
	\label{eq:convergence:values}
	\lim_{\varepsilon \rightarrow 0} \tilde{J}^{\varepsilon,\varphi}\bigl( {\boldsymbol \alpha}^{ \varepsilon,\hat{\varphi}(\varepsilon)};{\boldsymbol p}^{\varepsilon;\hat{\varphi}(\varepsilon)}\bigr) = {\mathbb E}^{\mathbb M} \bigl[ J({\boldsymbol \alpha};{\boldsymbol p}) \bigr],
\end{equation}
where ${\mathbb E}^{\mathbb M}$ denotes the expectation under ${\mathbb M}$. 
At this stage, we recall from Theorem \ref{main:thm:4}
that, under the probability ${\mathbb M}$, almost every path $({\boldsymbol p},{\boldsymbol \alpha})$ -- understood as the canonical processes in ${\mathcal C}([0,T];{\mathbb R}^d) \times L^2([0,T];[-dM,dM]^{d^2})$-- forms an equilibrium 
of the original inviscid MFG 
\eqref{eq:cost:functional}--\eqref{eq:fp}. In particular, 
$J({\boldsymbol \alpha};{\boldsymbol p})$ is nothing but $J^\star({\boldsymbol p}) = \inf_{\boldsymbol \beta} J({\boldsymbol \beta};{\boldsymbol p})$, the infimum being here taken over all the
deterministic processes ${\boldsymbol \beta}$, see \eqref{eq:cost:functional}. 
At the end of the day, we may interpret the right-hand side as a mean over the values of the equilibria 
of the inviscid MFG. Obviously, the argument inside the limit symbol in the left-hand side is also the value of the unique equilibrium of the viscous MFG, hence proving that the limit points of the values of the viscous MFGs are means over the values of the inviscid MFG. 
Importantly, the probability ${\mathbb M}$ here just charges the minimizers of the 
inviscid MFCP:  In case when the inviscid MFCP has a unique solution, the expectation ${\mathbb E}^{\mathbb M}[J(\balpha,{\boldsymbol p})]$ then reduces to 
$\inf_{\balpha} J(\balpha;{\boldsymbol p})$, where ${\boldsymbol p}$ is the unique minimal path of the inviscid MFCP.

\subsection{Master equation}
\label{subse:2:master}
It is worth recalling that the value of an MFG --at least when the latter is uniquely solvable-- has a nice interpretation 
in terms of the solution of a partial differential equation set on the space of probability measures. This equation, see for instance 
\cite{gomes2014dual, gomes2014socio}
and \cite[Chapter 7]{CarmonaDelarue_book_I} for finite state MFGs
and
\cite{cardaliaguet,ben-fre-yam2015, CarmonaDelarue_book_II, CardaliaguetDelarueLasryLions, cha-cri-del2019}
for continuous state MFGs, is usually known as the master equation for the underlying MFG and should be understood as the asymptotic version, as the number of players tends to $\infty$, of the Nash system associated with the finite $N$-player game. 
Our first main result in this direction concerns the master equation of the viscous MFG: It is here a system of second-order partial differential equations on the simplex, driven by the same Kimura operator as the HJB equation 
\eqref{hjbpotnew:sec:2}. 
It has the following general form:
\begin{equation}
	\label{eq:master:equation:introl}
	\left\{
	\begin{array}{l}
		\displaystyle 
		\partial_t U^i + H_{{M}}\Bigl( (U^i-U^j)_{j \in \ES} \Bigr) + {\sum_{j \in \ES} \varphi(p_j) ( U^j - U^i) }
		+
		\bigl( f^i  + \vartheta^{i,\varepsilon,\varphi} \bigr) (t,p)  
		\\
		\displaystyle 
		\quad+
		\sum_{j,k \in \ES} p_k\bigl[ \varphi(p_{j}) + (U^k-U^j)_+ \bigr] \bigl( {\mathfrak d}_{j} U^i - {\mathfrak d}_{k} U^i\bigr) 
		\\
		\displaystyle 
		\quad+ \varepsilon^2\sum_{j \in\ES} p_{j} \bigl( {\mathfrak d}_{i} U^i - {\mathfrak d}_{j} U^i\bigr)
		+\tfrac12 {\varepsilon^2} \sum_{j,k \in \ES}\bigl(p_j \delta_{j,k}-p_{j} p_{k}\bigr) {\mathfrak d}_{j,k}^2 U^i =0,\\
		U^i(T,p)= g^i(p),
	\end{array}
	\right.
\end{equation}
for $(t,p) \in [0,T] \times {\mathcal S}_{d}$, 
where $H_{{M}}$ is 
{the Hamiltonian:
\begin{equation}
\label{HM}
		 {H}_{M}(w) = 
		\sum_{j\in \ES} \Bigl\{ - a^\star(w_{j}) w_{j} + \tfrac12  |a^\star(w_{j})|^2  \Bigr\},  \quad w \in {\mathbb R}^d, 
\end{equation} 
with $a^\star$ as in 	\eqref{eq:astar}.}
A key fact 
--which we implicitly use in our text-- is that, under the assumption of Theorem \ref{main:thm:2},  
this master equation has a unique classical solution (with a suitable behaviour at the boundary, see the definition of the so-called Wright-Fischer spaces in Appendix): This result is mostly due to \cite{mfggenetic}. 
Given a classical solution $U_{\varepsilon,\varphi}=(U_{\varepsilon,\varphi}^i)_{i \in \ES}$ 
to \eqref{eq:master:equation:introl}, the value of the viscous MFG, when initialized from 
a state $p \in \textrm{\rm Int}({\mathcal S}_{d})$ at some time $t \in [0,T)$, is given by 
$\sum_{i \in \ES} p_{i} U_{\varepsilon,\varphi}^i(t,p)$. In other words, $U_{\varepsilon,\varphi}^i(t,p)$ is nothing but $\inf_{\balpha} \tilde{J}^{\varepsilon,\varphi}(\balpha;{\boldsymbol p}^{\varepsilon,\varphi})$ whenever 
${\boldsymbol p}^{\varepsilon,\varphi}$ is initialized from $p$ at time $t$ and ${\boldsymbol q}$ in
\eqref{dynmfg} 
is initialized at time $t$ from $(q_{t}^j=\delta_{i,j})_{j \in \ES}$.

Due to the potential structure of the game, there is in fact a strong connection between the HJB equation 
\eqref{hjbpotnew:sec:2} and the 
master equation \eqref{eq:master:equation:introl}. We can not have directly $U_{\varepsilon,\varphi}^i(t,p) = 
{\mathfrak d}_{i} {\mathcal V}_{\varepsilon,\varphi}(t,p)$ for any $i\in\ES$, because the intrinsic gradient sum to zero, while the 
functions $U_{\varepsilon,\varphi}^i$ do not. This is {by the way} part of the difficulty in proving Theorem \ref{main:thm:2}, see Section \ref{sec:4}. 
{What we can show is} that
\begin{equation}
	\label{eq:master:HJB}
	U_{\varepsilon,\varphi}^i(t,p) - U_{\varepsilon,\varphi}^j(t,p) = 
	{\mathfrak d}_{i} {\mathcal V}_{\varepsilon,\varphi}(t,p) - 
	{\mathfrak d}_{j} {\mathcal V}_{\varepsilon,\varphi}(t,p), 
\end{equation}
for 
$t \in [0,T]$, $p \in \textrm{\rm Int}\bigl({\mathcal S}_{d}\bigr)$ and $i,j \in \ES$,  
which is reminiscent of \cite[Theorem 3.7.1]{CardaliaguetDelarueLasryLions} (in the sense that, heuristically, space derivatives in continuous state space are replaced here by differences). 
Notably, \eqref{eq:master:HJB} sufficies to prove that the MFG and the MFCP have the same solution, because the optimal control is given by \eqref{eq:optimal:control:mfg}.
Interestingly, 
Proposition 
\ref{prop:cv:cost}
provides a way to pass to the limit  for $U_{\varepsilon,\varphi}$. In case
when the 
inviscid MFCP \eqref{eq:cost:functional:mfc}--\eqref{eq:fp}
has a unique minimizer initialized from $p$ at time $t$, Proposition 
\ref{prop:cv:cost} implies that the limit of $U_{\varepsilon,\varphi}^i$ (provided that $\delta$ is chosen 
as $\delta=\hat{\delta}(\varepsilon)$, 
 {$\theta$ as $\theta=\hat{\theta}(\varepsilon)$}
 and $\varphi$ as $\varphi=\hat{\varphi}(\varepsilon)$) is 
$U^i(t,p)$, where now $U^i(t,p)$ stands for $\inf_{\balpha} J(\balpha;{\boldsymbol p})$ with 
${\boldsymbol p}$ denoting the unique minimizer of the inviscid MFCP initialized from $p$ at time $t$ and ${\boldsymbol q}$ in
\eqref{eq:fp}
being initialized at time $t$ from $(q_{t}^j=\delta_{i,j})_{j \in \ES}$. 

Obviously, a natural question is to relate such a limit $U$ 
with the value function ${\mathcal V}$ of 
the inviscid MFCP 
\eqref{eq:cost:functional:mfc}, where, for
$(t,p) \in [0,T] \times {\mathcal S}_{d}$, ${\mathcal V}(t,p)$
is defined
as $\inf_{\balpha} {\mathcal J}(\balpha)$ whenever 
${\boldsymbol q}$ in 
\eqref{eq:fp} starts from $p$ at time $t$. 
We manage to prove {(see Theorem \ref{thmcomparison})} that
${\mathcal V}$ is the unique {Lipschitz} viscosity solution of the following HJB equation:
\begin{equation}
	\label{eq:hjb:inviscid}
	\left\{
	\begin{array}{l}
		\partial_{t} {\mathcal V} +  \sum_{k,j \in \ES} p_{k} H  \bigl( {\mathfrak d}_{k} {\mathcal V} - 
		{\mathfrak d}_{j} {\mathcal V} \bigr)  + F(p) = 0,
		\\
		{\mathcal V}(T,p)=G(p),
	\end{array}
	\right.
	\quad (t,p) \in [0,T] \times {\mathcal S}_{d},
\end{equation}
where $H$ is the Hamiltonian associated to ${\mathfrak L}$ in 
\eqref{eq:lagrangian:main}, namely 
\begin{equation}
	\label{eq:hamiltonian:main}
	H(u) = - \tfrac12 \sum_{j \in \ES} ( u_{j} )_{+}^2, \quad u \in {\mathbb R}^d. 
\end{equation}
Pay attention that there is no condition on the boundary of the simplex, see 
Definition \ref{defvisco}
for the details. 
Obviously, 
\eqref{eq:hjb:inviscid}
should be regarded as the 
inviscid version of the equation
\eqref{hjbpotnew:sec:2} (up to the fact that controls are truncated by $M$ in the latter, {but this may be in fact easily handled 
		by using the fact that optimal controls to \eqref{eq:cost:functional} are bounded by $M$, see 
	Proposition \ref{prop:solution:mfc:zero:epsilon}}).
Importantly, ${\mathcal V}$ is shown to be Lipschitz continuous in time and space, see if needed Proposition 
\ref{prop:solution:mfc:zero:epsilon} in the core of the text. Hence, it is almost every differentiable in 
$[0,T]\times \mathcal{S}_d$, which plays a crucial role in our analysis:
We also prove in Proposition \ref{prop:solution:mfc:zero:epsilon} that the inviscid MFCP has a unique solution when it is initialized from $p \in \textrm{\rm Int}({\mathcal S}_{d})$ at time $t$
such that ${\mathcal V}$ is differentiable in $(t,p)$ --and hence for almost every 
$(t,p)\in [0,T]\times \textrm{\rm Int}({\mathcal S}_{d})$--, which permits to 
pass to the limit (as $\varepsilon$ tends to $0$) in 
$\mathfrak{D}\mathcal{V}_{\varepsilon,\phi}$ almost everywhere in time and space {in \eqref{eq:master:HJB}} ({the simplex being equipped with the $(d\!-\! 1)$ Lebesgue measure}).
  To this end, we need to make the slightly stronger assumption that $F\in \mathcal{C}^{1,1}(\mathcal{S}_d)$, meaning that $f$ is Lipschitz continuous ({on ${\mathcal S}_{d}$ and hence up to the boundary}), in order to ensure that $\mathcal{V}$ is semiconcave; see again Proposition \ref{prop:solution:mfc:zero:epsilon}.

We build on this idea to obtain the following:
\begin{thm}[Part I]
	\label{main:thm:6}
	Under the same assumption and notation as in the statement of Proposition 
	\ref{prop:cv:cost}, we have
	\be 
	\lim_{\varepsilon\rightarrow0}
	\mathcal{V}_{\varepsilon,\hat{\varphi}(\varepsilon)}
	= \mathcal{V} \quad \mbox{locally uniformly in } [0,T]\times \textrm{\rm Int}({\mathcal S}_{d}),
	\ee
	and, moreover, if in addition $F\in \mathcal{C}^{1,1}(\mathcal{S}_d)$,
	\begin{align}
		\label{eq:conv:optimal:policy}
		\lim_{\varepsilon\rightarrow0}
		\mathfrak{D}\mathcal{V}_{\varepsilon,\hat{\varphi}(\varepsilon)}
		&= \mathfrak{D}\mathcal{V}
		\quad \mbox{ a.e. in } [0,T]\times \textrm{\rm Int}({\mathcal S}_{d})
		\quad \mbox{ and \quad in } [L^1_{loc}([0,T]\times \textrm{\rm Int}({\mathcal S}_{d}))]^d.
	\end{align}
\end{thm}
This is the most technical and demanding result of the paper, and  is proved is Subsection \ref{subse:distance:MFCP}, see Theorem \ref{thm:5.10}; notice that, in the end,  we can not prove convergence at any points of differentiability {of ${\mathcal V}$}, but just almost everywhere. 
Passing to the limit in 
\eqref{eq:master:HJB}, 
equations \eqref{eq:conv:U} and \eqref{eq:conv:optimal:policy} provides 
a strong form of selection for the value of the inviscid MFG. In the above notations, it says that, for almost every $(t,p) \in [0,T] \times {\mathcal S}_{d}$, the value $U(t,p)$ of the game that is selected is given by the derivative of 
${\mathcal V}$, namely
\begin{equation}
	\label{eq:selection:optimal feedback}
	U^i(t,p) - U^j(t,p) = {\mathfrak d}_{i} {\mathcal V}(t,p) - {\mathfrak d}_{j} {\mathcal V}(t,p),
\end{equation} 
for any $i,j \in \ES$. {At first sight, it looks  like that only finite differences of the 
vector
$(U^1(t,p),\cdots,U^d(t,p))$ are hence selected.
In fact, we can reconstruct \textit{a posteriori} the full-fledge collection 
$(U^i(t,p))_{i \in \ES}$ by observing that each $U^i(t,p)$ should coincide with 
the optimal cost to \eqref{eq:cost:functional}
when ${\boldsymbol p}$ is the minimizer of the inviscid MFCP (which is unique for almost every initial 
point $(t,p)$) and when ${\boldsymbol q}$ in 
	\eqref{eq:fp}
	starts from the Dirac point mass in $i$ at time $t$. Hence we may complement 
	Theorem \ref{main:thm:6} in the following way}, which is also proved in Theorem \ref{thm:5.10}:

{\begin{customthm}{\ref{main:thm:6}}[Part II]
	Under the same assumption and notation as in the statement of Proposition 
	\ref{prop:cv:cost} {and provided that  
	$F\in \mathcal{C}^{1,1}(\mathcal{S}_d)$}, we have
		\begin{align}
		\label{eq:conv:U}
		\lim_{\varepsilon\rightarrow0}
		U_{\varepsilon,\hat{\varphi}(\varepsilon)}
		&= U
		\quad \mbox{ a.e. in } [0,T]\times \textrm{\rm Int}({\mathcal S}_{d})
		\quad \mbox{ and \quad in } [L^1_{loc}([0,T]\times \textrm{\rm Int}({\mathcal S}_{d}))]^d,
	\end{align}
	where, for any initial condition $(t,p) \in [0,T]\times \textrm{\rm Int}({\mathcal S}_{d})$ 
	from which the inviscid 
	  MFCP has a unique optimal trajectory ${\boldsymbol p}$, $U^i(t,p)$ is defined as 
$\inf_{\balpha}
	J(\balpha;{\boldsymbol p})$
in \eqref{eq:cost:functional}, the problem being set over the time interval $[t,T]$ 
and ${\boldsymbol q}$ in 
	\eqref{eq:fp}
starting from $q_{t}=(q_{t}^j=\delta_{i,j})_{j \in \ES}$. 
\end{customthm}
}
\subsection{Weak solution to the master equation}
The last step in our program is hence to provide an intrinsic approach to the relationship \eqref{eq:selection:optimal feedback} by addressing 
directly the master equation of the inviscid MFG. 
The latter writes (see for instance
\cite{GomesMohrSouza_discrete,Gomes2013} and \cite[Chapter 7]{CarmonaDelarue_book_I}):
\begin{equation}
	\label{eq:master:equation}
	\left\{
	\begin{array}{l}
		\partial_t U^i + H\Bigl( (U^i-U^j)_{j \in \ES} \Bigr) +  f^i (p)  
		+
		\sum_{j,k \in \ES} p_k  (U^k-U^j)_+  \left( {\mathfrak d}_{j} U^i - {\mathfrak d}_{k} U^i\right) =0,
		\\
		U^i(T,p)= g^i(p), 
	\end{array}
	\right.
\end{equation}
for $i \in \ES$, 
which is informally obtained by taking $\varphi \equiv 0$ and $\varepsilon=0$ in 
\eqref{eq:master:equation:introl}.  
Recasted in terms of the centered value functions $({\mathring U}^i := U^i - \bar U)_{i \in \ES}$, 
with $\bar U = \frac1d \sum_{j \in \ES} U^j$,
\eqref{eq:master:equation} becomes:
\begin{equation}
	\label{eq:master:equation:centred}
	\left\{
	\begin{array}{l}
		\partial_t \mathring U^i + \mathring H^i \Bigl( (\mathring U^j- \mathring U^k)_{j,k \in \ES} \Bigr) +  \mathring{f}^i (p)  
		+
		\sum_{j,k \in \ES} p_k  (\mathring U^k- \mathring U^j)_+  \bigl( {\mathfrak d}_{j} \mathring U^i - {\mathfrak d}_{k} \mathring U^i\bigr) =0,
		\\
		\mathring U^i(T,p)= \mathring{g}^i(p),
	\end{array}
	\right.
\end{equation}
for $i \in \ES$, 
where we have let 
\begin{align*}
	&\mathring{H}^i\Bigl(\bigl(u^{j,k}\bigr)_{j,k \in \ES}\Bigr) := H\Bigl(\bigl(u^i - u^j\bigr)_{j \in \ES}\Bigr)- \tfrac1d \sum_{j \in \ES} H\Bigl(\bigl(u^j-u^k\bigr)_{k \in \ES}\Bigr),
\end{align*} 
{and similarly 
$\mathring{f}^i(p)= f^i(p)-\frac1d \sum_{j\in\ES} f^j(p)$
and
$\mathring{g}^i(p)= g^i(p)-\frac1d \sum_{j\in\ES} g^j(p)$}.
As we have already explained,
the master equation is typically non-uniquely solvable (see the next subsection for a benchmark example). 
The question for us is thus to rephrase \eqref{eq:selection:optimal feedback} as a uniqueness result for the master equation --or at least for its centered version \eqref{eq:master:equation:centred}-- within a well-chosen class of functions. 
Loosely speaking, we succeed to obtain such a result in Section \ref{sec:uniqueness:master:equation} below but for 
the conservative form of 
\eqref{eq:master:equation:centred}, namely
\begin{equation}
	\label{eq:master:equation:conservative}
	\left\{
	\begin{array}{l}
		\partial_t \mathring U^i + \mathring H^i \Bigl( (\mathring U^j- \mathring U^k)_{j,k \in \ES} \Bigr) +  \mathring{f}^i (p)  
		-\frac12
		\sum_{j,k \in \ES} p_k  {\mathfrak d}_{i} \bigl[ (\mathring U^k- \mathring U^j)_+^2 \bigr]=0,
		\\
		\mathring U^i(T,p)= \mathring{g}^i(p),
	\end{array}
	\right. 
\end{equation}
for $i \in \ES$. 
Clearly, the two equations 
\eqref{eq:master:equation}
and
\eqref{eq:master:equation:conservative}
may be identified 
within the class of differentiable functions $\mathring U$ that satisfy, for any $i,j,k \in \ES$,  
\begin{equation*}
	-\frac12 {\mathfrak d}_{i} \bigl[ \bigl(\mathring U^k- \mathring U^j\bigr)_+^2 \bigr]
	=  \bigl( \mathring U^k- \mathring U^j \bigr)_+  \left( {\mathfrak d}_{j} \mathring U^i - {\mathfrak d}_{k} \mathring U^i \right),
\end{equation*}
which indeed holds true if, for any $i,j \in \ES$,  
\begin{equation}
	\label{eq:schwarz}
	{\mathfrak d}_{i}  \mathring U^j
	=  {\mathfrak d}_{j} \mathring U^i.
\end{equation}
As we clarify in Section \ref{sec:uniqueness:master:equation}, 
identity \eqref{eq:schwarz}
guarantees that $\mathring U$ derives --in space-- from a potential, meaning that 
$\mathring U^i(t,p) = {\mathfrak d}_{i} \mathring{\mathcal {V}}(t,p)$ for some real-valued function 
$\mathring{\mathcal {V}}$ defined on $[0,T] \times {\mathcal S}_{d}$. 
As a byproduct, it 
prompts us to regard the conservative formulation 
\eqref{eq:master:equation:conservative}
of the master equation 
as the derivative system obtained by applying the operator ${\mathfrak d}_{i}$, for 
each $i \in \ES$, to the
HJ equation
\eqref{eq:hjb:inviscid}. In 
words, 
\eqref{eq:master:equation:conservative}
may be rewritten as
\begin{equation}
	\label{eq:master:equation:conservative:full}
	\left\{
	\begin{array}{l}
		\partial_t \mathring U^i + 
		{\mathfrak d}_{i} \Bigl(
		\sum_{k,j \in \ES} p_{k} H  \bigl( 
		\mathring U^k - \mathring U^j \bigr)
		\Bigr) 
		+  {\mathfrak d}_{i} F (p)  
		=0,
		\\
		\mathring U^i(T,p)= {\mathfrak d}_{i} G(p),
	\end{array}
	\right. 
\end{equation}
for $i \in \ES$. Interestingly enough, 
the formulation 
\eqref{eq:master:equation:conservative:full}
makes clear the link between the HJ equation  
\eqref{eq:hjb:inviscid} and the master equation, at least when the latter is understood in its conservative form. 
For scalar conservation laws, the usual notion of admissibility which is used to restore uniqueness of weak solutions is the one of entropy solution. In space dimension 1, which is the case when $d=2$ (see next Subsection), the entropy solution to a scalar conservation law is also shown to be the space derivative of the viscosity solution of the corresponding HJB equation; see e.g.  \cite{KarlsenRisebro,Lions_HJB,cannarsa}. However, for hyperbolic systems of PDEs with multiple space dimension, which is the case  {here} when $d\geq 3$, there might be non-uniqueness of entropy solutions and there are very few results in the literature about such systems. In particular, system \eqref{eq:master:equation:conservative:full} is hyperbolic in the wide sense, but not strictly hyperbolic. Nevertheless, exploiting the connection with the HJB equation \eqref{eq:hjb:inviscid}, borrowing the idea from the paper of Kru\v{z}kov
\cite{kruzkov}, it is possible to establish uniqueness in a suitable set of admissible weak solutions.
The whole is captured by the following statement\footnote{The reader should be aware of the fact that the assumption on $G$ in the statement below is weaker than what we required in the previous statements, see for instance Theorems \ref{main:thm} and \ref{main:thm:2}.}:

\begin{thm}
	\label{thm:uniqueness:master}
	 Assume that $F$ and $G$ are in  $\mathcal{C}^{1,1}(\mathcal{S}_d)$. The conservative form \eqref{eq:master:equation:conservative:full} of the master equation has a unique
	weak solution that is bounded and weakly semi-concave in space. This solution is the almost everywhere space derivative
	of the unique viscosity solution ${\mathcal V}$ of the HJ equation 
	\eqref{eq:hjb:inviscid}. 
\end{thm} 

The proof is given in Section \ref{sec:uniqueness:master:equation}, see Theorem \ref{thm:6.6}. 
The notions of weak solution and weak semi-concavity in space are clarified in Definition 
\ref{def:admissible:solution} below. In a shell, the proof of the above statement holds in three steps: The first one 
is to show that any weak solution to the conservative form of the master equation derives from a potential; 
The second one is to prove that the potential must be an almost everywhere and semi-concave solution of the HJ equation 
\eqref{eq:hjb:inviscid}; The last step is to identify almost everywhere and semi-concave solutions with viscosity solutions of
\eqref{eq:hjb:inviscid}, which are shown to be unique, despite the lackness of boundary conditions, see Corollary \ref{CORUNIQ} below. 
To put it differently, the striking facts that we use here to restore a form of uniqueness to the master equation are, on the one hand, the existence of a potential and, on the other hand, the semi-concavity assumption. 
In this regard, it must be fair pointing out that the existence of a potential is somewhat enclosed in the 
conservative form of 
\eqref{eq:master:equation:conservative:full}. 
In other words, the
conservative form not only permits to address solutions in a weak sense, but it also permits to reduce the space 
of solutions to gradient functions. 
As for the semi-concavity assumption, it plays a crucial role in the selection: The connection between semi-concave solutions of HJ equations and entropy  
solutions of scalar conservation laws has been widely discussed, see for instance the first chapter in the monograph of Cannarsa and Sinestrari \cite{cannarsa} together with the bibliography therein; In the case of of hyperbolic systems with a potential structure --like \eqref{eq:master:equation:conservative:full}--, the role of semi-concavity is exemplified in the earlier paper of Kru\v{z}kov
\cite{kruzkov} from which we borrow part of the proof of Theorem \ref{thm:uniqueness:master}.
For sure, it is also important to say that, in the end, we are not able to define weak solutions for  the non-conservative versions 
\eqref{eq:master:equation}
and
\eqref{eq:master:equation:centred} of the master equation. 
However, we prove in Proposition \ref{prop:classical:weak} below that  classical solutions to \eqref{eq:master:equation:centred} are indeed weak solutions to \eqref{eq:master:equation:conservative:full}, the key point being that  Schwarz identity \eqref{eq:schwarz} holds true for classical solutions to \eqref{eq:master:equation:centred}.


\subsection{Example}
\label{subse:example}

To illustrate our results, we feel useful to revisit the $d=2$ example addressed in 
\cite{cecdaifispel} (the reader may also have a look at \cite{bayzhang2019} which shares 
many features with 
\cite{cecdaifispel}). Therein, a selection result is proven by addressing directly the large $N$ behavior of the 
$N$-player game, both in terms of the value functions for the feedback Nash equilibria and of the optimal trajectories. Although this is certainly a much more satisfactory approach than ours, at least from a modelling point view, making a detour via the finite case remains however much more challenging and difficult. To wit, the selection result established in 
\cite{cecdaifispel} is partial only, as it leaves open the case when the initial point of the equilibrium is precisely a singular point of the master equation --we go back to this point next--. And most of all, it is by no means clear --at least for us-- how the 
potential structure we use here could help for addressing the convergence of the corresponding finite player game in the case $d \geq 3$.
Actually, this informal comparison of the tractability of the two limits over $N$ and $\varepsilon$ should not come as a surprise for the reader. Intuitively, it is indeed more difficult to handle the $N$-player game for a large $N$ than the $\varepsilon$-viscous game
for a small intensity $\varepsilon$ of the common noise: 
Even though their solutions are randomized, MFGs with common noise indeed share with standard MFGs the key property that any unilateral deviation from an equilibrium has no influence on the global state of the population; 
this turns out to be useful when addressing the asymptotic behavior --as $\varepsilon$ vanishes-- and, obviously, 
this is false for finite games. Noticeably, this argument in support of the vanishing viscosity approach is exemplified in the paper \cite{delfog2019}: Therein, the authors prove a selection result for linear quadratic games with a continuous state space by both methods; In this setting, the vanishing viscosity method is clearly the easiest one. 

The case $d=2$ is very special because any MFG becomes potential. {Below, we first provide a general description of} two state mean field games and then we specialize our results to the example analyzed in \cite{cecdaifispel}.  For the same two 
sets of coefficients 
$(f^i)_{i=1,2}$ and $(g^i)_{i=1,2}$ as in \eqref{eq:cost:functional},
we can easily reconstruct two potentials $F$ and $G$ such that
\[
\fd_1 F (p) = \frac{f^1(p) -f^2(p)}{2} = - \fd_2 F (p),  \qquad  p \in {\mathcal S}_{2},
\]
{by letting}
\begin{equation*}
	F(p_1,1-p_1) := \int_{0}^{p_{1}} \Bigl( f^1(q,1-q) - f^2(q,1-q) \Bigr) dq, \qquad p_1 \in [0,1],
\end{equation*}
{and similarly for $G$.}
Interestingly the centered master equation \eqref{eq:master:equation:centred}, which is a system in the general case $d \geq 3$, becomes a mere equation when $d=2$. 
Indeed, we then have
$\widetilde U^1 = -\widetilde U^2$, 
which implies that \eqref{eq:master:equation:centred} can be rewritten in terms of the sole $\widetilde U^1 = (U^1-U^2)/2$. Accordingly, 
the conservative version of the master equation takes the form 
\begin{equation}
	\label{eq:conservative:form:d=2}
	\left\{
	\begin{array}{l}
		\partial_{t} \widetilde U^1 +   {\mathfrak d}_{1} \Bigl(
		{\mathcal H}\bigl( p, \widetilde U^1\bigr)
		\Bigr) + \frac12 \bigl(f^1(p_{1},1-p_{1})-f^2(p_{1},1-p_{1})\bigr)=0, 
		\\
		\widetilde U^1(T,p) = \frac12 \bigl(g^1(p_{1},1-p_{1}) - g^2(p_{1},1-p_{1})\bigr),
	\end{array}
	\right.
\end{equation}
where 
\begin{equation*}
	\begin{split}
		-{\mathcal H}(p,u) = 
		2 p_{1} (u)_{+}^2 + 2 (1-p_{1}) (- u)_{+}^2
		&=
		2 p_{1} \Bigl( \frac{\vert u \vert+u}{2} \Bigr)^2 + 2 (1-p_{1})
		\Bigl( \frac{\vert u \vert - u}2 \Bigr)^2
		\\
		&= u^2 + (2p_{1}-1) u \vert u \vert. 
	\end{split}
\end{equation*}
The latter expression prompts us to change the variable $p_{1}$ into $m=2p_{1}-1$ (which should be thought of as the mean
of $(p_{1},p_{2})$ 
if the state space was $\{1,-1\}$ instead of $\{1,2\}$). Letting $Z(m) :=  {-2\widetilde U^{1}}(\tfrac{1+m}2,\tfrac{1-m}2) = (U^2-U^1)(\tfrac{1+m}2,\tfrac{1-m}2)$, for $m \in [-1,1]$, 
we can rewrite  \eqref{eq:conservative:form:d=2} in the form 
\begin{equation}
	\label{eq:conservative:form:d=2:mean}
	\left\{
	\begin{array}{l}
		-\partial_{t} Z +     \partial_{m} \Bigl(
		\frac{mZ \vert Z \vert}{2} -\frac{Z^2}{2}
		\Bigr) {=} f^2\bigl(\tfrac{1+m}2,\tfrac{1-m}2 \bigr) - f^1(\tfrac{1+m}2,\tfrac{1-m}2 \bigr),
		\\
		Z(T,p) = g^2\bigl( \frac{1+m}2,\frac{1-m}2 \bigr) -g^1(\tfrac{1+m}2,\tfrac{1-m}2 \bigr),
	\end{array}
	\right.
\end{equation}
for $(t,m) \in [0,T] \times [-1,1]$. 

In \cite{cecdaifispel}, 
the cost coefficients are chosen as  
\begin{center}
	$F \equiv 0$ \quad and \quad $g^1(p)=- (2p_{1}-1)$, \quad $g^{-1}(p)= 2p_{1}-1$,
\end{center}
so that 
\[
G(p_1,1-p_1)= \int_0^{p_1}-2(2q-1)dq= -2p_1^2 + 2p_1 =  2 p_1  p_{2}.
\] 
The reduced master equation in  \cite{cecdaifispel}, see (3.11) 
therein, is exactly Equation \eqref{eq:conservative:form:d=2:mean} 
 for  $Z$ (up to a time reversal). 
Note also that the potential $\mathcal{G}(p_1,p_2) = - (2p_1-1)^2/2$ therein differs from $G$ above by a constant (which is $1/2$), but obviously this does not matter. 
Importantly, 
the master equation \eqref{eq:conservative:form:d=2:mean} may have multiple weak solutions when $T$ is large enough, hence the need for a selection argument. The solution selected in \eqref{eq:conservative:form:d=2:mean}, following the theory for scalar conservation laws, is the entropy solution, which can be shown to be unique in this case despite the lackness of boundary conditions. As explained in the previous subsection, the entropy solution is the space derivative of the viscosity solution to the HJB equation, making this selection consistent with Theorem \ref{thm:uniqueness:master}. Moreover, {in 
\cite{cecdaifispel}}, the value functions for the feedback Nash equilibrium of the $N$-player game are shown to converge to {this} entropy solution \cite[{Theorem} 8]
{cecdaifispel}; this says
{that the solutions to the master equation that are selected by taking the limit over $\varepsilon$ or over $N$ are the same}.
So, in a shell, our result is fully
consistant with \cite{cecdaifispel}.

As far as convergence of the optimal trajectories is concerned,  
the equilibria are shown to be non-unique, provided that
the time horizon 
$T$ is chosen large enough: Whatever the initial condition at time $0$, there are three solutions to the MFG if $T>2$, see
\cite[Proposition 2]{cecdaifispel}.
In this regard, the main result in 
\cite{cecdaifispel} states that, 
whenever the initial condition $p_{0}=(p^1_{0},{p^{2}_{0}})$ of the population at time $0$ satisfies $m_{0}:=2p_{0}^1 -1 \not =0$ (\textit{i.e.}, the mean parameter is non-zero), 
there is a unique equilibrium $(p_{t})_{0 \leq t \le T}$ that is selected by letting $N$ tend to $\infty$ in the corresponding 
$N$-player game; it satisfies the equation
\begin{equation}
	\label{eq:m_{t}:d=2}
	\frac{d}{dt} m_{t} = - 2m_{t}  \vert Z(t,m_{t}) \vert
	+ 2Z(t,m_{t}), \quad t \in [0,T],
\end{equation}
with  {$(m_{t}=p_{t}^1- p_{t}^{2})_{0 \le t \le T}$} and $Z$ {being} the unique entropy solution to \eqref{eq:conservative:form:d=2:mean}, see \cite[(23)]{cecdaifispel}. Notably, this equation is shown to admit a unique solution, when $m_0\neq 0$; see \cite[Prop 6]{cecdaifispel} . Again, this is consistent with our results: 
\cite[Theorem 15]{cecdaifispel}
asserts that this equilibrium is also the unique minimizer of the corresponding inviscid MFCP
initialized from $(0,m_{0})$, see 
\eqref{eq:cost:functional:mfc}
plugging $F \equiv 0$ and $G(p)=2p_{1} {p_{2}}$ therein. While the proof 
of Theorem \cite[Theorem 15]{cecdaifispel}
is carried out by explicit computations, 
our Theorem 
\ref{main:thm:4}
applies directly. Interestingly, we may recover 
\eqref{eq:m_{t}:d=2} explicitly. Indeed, 
in \cite{cecdaifispel}, 
the function $m \mapsto Z(0,m)$ is shown to be discontinuous 
at $m=0$ only (provided that $T$ is large enough; if $T$ is small, $m \mapsto Z(0,m)$ is continuous);
Accordingly, the function $m \mapsto {\mathcal V}(0,{\tfrac{1+m}2 })$ in  
\eqref{eq:hjb:inviscid} is continuously differentiable at $m_{0}$ (since $m_{0}$ is assumed to be non-zero) which, as we already explained --see also
Proposition 
\ref{prop:solution:mfc:zero:epsilon}--, implies that there is indeed a unique minimizer to the MFCP
initialized from $(0,{\tfrac{1+m_{0}}2})$. Also, our discussion {(see (viii) in Proposition 
\ref{prop:solution:mfc:zero:epsilon})}
 says that this unique minimizer, say $(p_{t}^\star)_{0 \leq t \leq T}$, solves the equation
\begin{equation*}
	\begin{split}
		&\frac{d}{dt} p_{t}^{\star,1}  =
		(1-p_{t}^{\star,1}) \bigl( \widetilde U^{{2}}(t,p_{t}^\star) - 
		\widetilde U^1(t,p_{t}^\star) \bigr)_{+}
		- p_t^{1,\star} \bigl( \widetilde U^1(t,p_{t}^\star) - 
		\widetilde U^{{2}}(t,p_{t}^\star) \bigr)_{+},
	\end{split}
\end{equation*}
for $t \in [0,T]$. 
Letting $(m_{t}^\star:=
2 p^{1,\star}_{t}-1)_{0 \le t \le T}$, 
we easily derive that $m_t^\star$ solves Equation  \eqref{eq:m_{t}:d=2}, whence we get $m_t = m_t^\star$. 

Last but not least, the case $m_0=0$ is left open in \cite{cecdaifispel}. In that case, the 
inviscid MFCP is shown to have two non-trivial symmetric minimizers, see again \cite[Theorem 15]{cecdaifispel}. It is also claimed in \cite{cecdaifispel}, see 
Section 4 therein, that, numerically, equilibria of the $N$-player game are tending 
to converge in law to those two minimizers, with weight $1/2$ each; In other words there are numerical evidences for ruling out the third equilibrium (recalling that the MFG has exactly three solutions). Obviously, our Theorem 
\ref{main:thm:4} sounds as a confirmation of this latter intuition, as it precisely 
says that the third equilibrium (which is shown to be the constant zero) 
is indeed excluded by the vanishing viscosity method. The fact that the two remaining ones should be charged with probability $1/2$ each comes from an additional symmetry argument, which is similar to the one used in \cite{delfog2019}: 
If ${\boldsymbol p}^{\star,\varepsilon,\varphi}=(p_{t}^{\star,1,\varepsilon,\varphi},
p_{t}^{\star,2,\varepsilon,\varphi})_{0 \le t \le T}$ is an optimal trajectory of the viscous mean field 
control problem, then, thanks to the symmetric form of $G$, 
$(p_{t}^{\star,2,\varepsilon,\varphi},
p_{t}^{\star,1,\varepsilon,\varphi})_{0 \le t \le T}$
is an admissible path with the same cost and hence is also an optimal trajectory but for the common noise $(B^{2,1}_{t},B^{1,2}_{t})_{0 \le t \le T}$ (instead of $(B^{1,2}_{t},B^{2,1}_{t})_{0 \le t \le T}$). By uniqueness in law of the equation characterizing the optimal trajectory, 
this shows that 
$(p_{t}^{\star,1,\varepsilon,\varphi},
p_{t}^{\star,2,\varepsilon,\varphi})_{0 \le t \le T}$
and 
$(p_{t}^{\star,2,\varepsilon,\varphi},
p_{t}^{\star,1,\varepsilon,\varphi})_{0 \le t \le T}$
have the same distribution. Consequently, under any weak limit 
${\mathbb M}$ as in
the statement of Theorem 
\ref{main:thm:4}, the marginal law of the first variable --which must be understood as the law of ${\boldsymbol p}$-- has to be symmetric. Here, we know that the support of ${\mathbb M}$ is necessarily included in a set of two non-trivial trajectories. 
Hence, 
each of them should be charged with probability $1/2$. 

Obviously, the thrust of our approach is that it applies to more general 
coefficients $F$ and $G$ and to any number of states $d\geq 2$; of course, the symmetry argument we have just alluded to only applies under appropriate 
forms of symmetry. 
\section{Mean field control problem}
\label{sec:3}

The main goal of this section is to prove Theorem \ref{main:thm}. 
We feel useful to recall that, for a function 
$\varphi$ as in 
\eqref{eq:varphi:2}, 
we aim at minimizing 
${\mathcal J}^{\varepsilon,\varphi}(\balpha)$
in 
\eqref{eq:Cost:mfc}
where ${\boldsymbol p}=(p_{t})_{0 \le t \le T}$ therein solves the $\balpha$-driven SDE
\eqref{dynpot}.
Importantly, the pair $(\varepsilon,\varphi)$ is kept fixed throughout the section,
which prompts us to drop out the superscript $(\varepsilon,\varphi)$ in  the subsequent notations. 
As explained in the previous section, we restrict ourselves to 
processes $\balpha$ that are bounded by $M=2(
\|g \|_{\infty}
+ T \| f\|_{\infty})$, in the sense that $\vert \alpha_{t}^{i,j} \vert \leq M$, $dt \otimes {\mathbb P}$ almost everywhere, 
for any $(i,j) \in \ES$ with $i \not =j$. 
The bound $M$ has  the following interpretation in terms of the inviscid MFG \eqref{eq:cost:functional}-\eqref{eq:fp}: 
For a given (deterministic) path 
${\boldsymbol p}=(p_{t})_{0 \le t \le T}$
with values in ${\mathcal S}_{d}$, 
optimizers of \eqref{eq:cost:functional} are given in terms of the value function 
$((u_{t}^i)_{0 \le t \le T})_{i \in \ES}$, 
namely $\alpha_{t}^{i,j}=(u_{t}^i - u_{t}^j)_{+}$, for $t \in [0,T]$ and $i,j \in \ES$ with $i \not =j$, see 
\cite[Chapter 7]{CarmonaDelarue_book_I}. 
Here,  $u_{t}^i$ is defined as the optimal 
cost when ${\boldsymbol q}$ starts at time $t$ from the initial condition 
$q_{t}^j = \delta_{i,j}$ and hence satisfies $\vert u_t^i \vert \leq T\|f\|_\infty + \|g\|_{\infty}$: the upper bound holds by choosing the zero control, while the lower bound follows from the sign of ${\mathfrak L}$.
With the same {meaning for $a^\star$} as in 
\eqref{eq:astar}, this allows us to express the corresponding Hamiltonian in the form
\begin{equation}
	\label{eq:widetilde H}
	\widetilde{\mathcal{H}}_{M}(p,w) :=
	\inf_{(\alpha_{i,j})_{i,j \in \ES : i \not =j} : 0 \leq \alpha_{i,j} \leq M}
	\widetilde{\mathbb H}_{M}(p,\alpha,w) 
	=  \sum_{i \in \ES} p_i \widetilde{H}_{M}^i(w),
\end{equation}
for $p \in {\mathcal S}_{d}$ and $w=(w_{i})_{i \in \ES} \in \RR^d$, 
with
\begin{equation}
	\label{eq:hamiltonians}
	\begin{split}
		\widetilde{\mathbb H}_{M}(p,\alpha,w) &:= 
		\sum_{i \in \ES} p_{i} 
		\sum_{j \in \ES : j \not = i} \Bigl( \alpha_{i,j} (w_{j} - w_{i}) + \tfrac12 \vert \alpha_{i,j} \vert^2 \Bigr)
		\\
		\widetilde{H}^i_{M}(w) &:= 
		\inf_{(\alpha_j)_{j \in \ES : j \not = i} : 0 \le \alpha_{j} \le M}
		\sum_{j \in \ES}
		\Bigl( 
		\alpha_{j}  \bigl( w_{j} - w_{i} \bigr)  
		+ \tfrac12 \vert \alpha_{j} \vert^2 \Bigr)
		\\
		&\phantom{:}=
		\sum_{j\neq i} \Bigl\{ a^\star(w_{i}-w_{j}) (w_{j}-w_{i}) + \tfrac12  |a^\star(w_{i}-w_{j})|^2  \Bigr\}. 
	\end{split}
\end{equation}
By boundedness of $a^\star$ (which in turn follows from our choice to restrict ourselves to controls that are bounded by $M$),
$\widetilde{H}_{M}^{i}$ is Lipschitz continuous and continuously differentiable with Lipschitz and bounded derivatives (pay attention that it is not $\mathcal{C}^2$).
The Hamiltonian $\widetilde{H}_{M}^{i}$ is used in the rest of the paper; note however that, in Section \ref{sec:main}, see \eqref{HM}, we preferred to use the slightly different Hamiltonian $H_M$, but the two are clearly related by 
the identity $\tilde{H}^i_M (w)= H_M((w_{i}-w_{j})_{j\in\ES})$, for $w\in\R^d$.
The HJB equation for the value function is nothing but 
\eqref{hjbpotnew:sec:2}, with ${\mathcal H}^{\varphi}_{M}$ therein given by
\begin{equation}
\label{eq:mathcal:H:varphi:M}
	{\mathcal H}^{\varphi}_{M}(p,w) := 
	\widetilde{\mathcal H}_{M}(p,w) + \sum_{i \in \ES}
	\sum_{j \not = i} p_i \varphi(p_j) (w_{j} - w_{i}). 
\end{equation}
{The following is straightforward but useful for us:
\begin{equation}
\label{eq:derivative:H}
\partial_{w_i} \mathcal{H}^{\varphi}_{M}(p,w) = 
\sum_{j\in \ES} p_{j}\Bigl(  \varphi(p_{i})+a^\star(w_{j}-w_{i})   \Bigr) - p_{i}\sum_{j\in \ES}
\Bigl(
\varphi(p_{j})
+
 a^\star(w_{i}-w_{j})
\Bigr).
\end{equation}}
\subsection{Classical solutions}
\label{subse:classical:solutions}
The well-known verification argument may be easily adapted to the simplex: 
If there exists a classical solution $\mathcal{V}$ to the HJB equation, then the optimal control is unique (clearly bounded), if the initial condition is in the interior of the simplex, and given in feedback form trough the feedback function $\tilde{\alpha}^{\star,i,j}:= a^\star({\mathfrak d}_{i}\mathcal{V}-{\mathfrak d}_{j}\mathcal{V})$.
The proof proceeds in the same way, by expanding the trajectories along ${\mathcal V}$, and {by} using the fact that 
{solutions} to \eqref{dynpot} remain in $\mathrm{Int}(\mathcal{S}_d)$ {(which makes it possible to use interior smoothness of ${\mathcal V}$ and coercivity of the Hamiltonian on $\mathrm{Int}(\mathcal{S}_d)$)}.

Although intrinsic derivatives are the most {canonical ones}, and will hence be used in the next sections, a key tool to prove the well-posedness of the HJB equation \eqref{hjbpotnew:sec:2}  is to work 
with local charts. In this respect, {it is worth recalling that}
any function {$h$} defined in the simplex $\mathcal{S}_d$ may be easily regarded as a function defined on the set $\widehat{\mathcal S}_{d}$. It suffices to identify $h$ with $\widehat{h}$ defined by 
\begin{equation*}
	\widehat{h}(x) := h (t,\check x), \quad \check{x}:=\bigl( x_{1},\cdots,x_{d-1},1- (x_{1}+\cdots+x_{d-1}) \bigr), \quad x\in \widehat{\mathcal{S}}_d
\end{equation*}
As explained in \cite{mfggenetic}, $h$ is then  once or twice differentiable on the (interior of) the simplex if 
$\widehat{h}$ is once or twice  differentiable in the usual sense as a function defined on an open subset of $\RR^{d-1}$, in which case we have a dictionary to pass from ${\mathfrak D} h$ and ${\mathfrak D}^2 {h}$ to 
$D_{x} \widehat{h}$ and $D_{x}^2 \widehat{h}$ and conversely.  
In short, $\partial_{x_{i}} \widehat{h}(t,x)= {\mathfrak d}_{i} h(t,\check{x}) - {\mathfrak d}_{d} h(t,\check{x})=
{\mathfrak d}_{i} h(t,\check{x}) + \sum_{j \in \ESd}
{\mathfrak d}_{j} h(t,\check{x})$, for $i \in \ESd$ and $x\in \mathrm{Int}(\widehat {\mathcal{S}}_d)$, {and conversely ${\mathfrak d}_{i} h(t,p)=(\partial_{x_{i}} \hat{h} - \tfrac1d \sum_{j \in \ESd} \partial_{x_{j}} \hat{h})(t,p_{1},\cdots,p_{d-1})$, for $i \in \ESd$, and 
${\mathfrak d}_{d} h(t,p)=- \tfrac1d \sum_{j \in \ESd} \partial_{x_{j}} \hat{h}(t,p_{1},\cdots,p_{d-1})$, for 
$p \in \textrm{\rm Int}({\mathcal S}_{d})$}. 
As for the second order derivatives, 
$\partial_{x_i x_{j}}\widehat{h}(t,x) =  \fd_{ij} h(t,\check{x})- \fd_{id} h(t,\check{x})- \fd_{{jd}} h(t,\check{x})+\fd_{dd}h(t,\check{x})$, 
for $i,j = 1,\cdots,d-1$. 
As a byproduct, 
the HJB has the following writing in local charts (sums being taken over $\ESd -1$):
\be
\label{hjbchart}
\left\{
\begin{array}{l}
	\partial_t \widehat{\mathcal{V}} + \widehat{\mathcal{H}}_{M}(x, D_x \widehat{\mathcal{V}}) 
	+\widehat F(x) + \frac{1}{2} \varepsilon^2 \sum_{j,k} (x_j \delta_{j,k}-x_{j} x_{k}) \partial_{x_{j} x_{k}}^2 \widehat{\mathcal{V}} \\
	\quad+\sum_{i} x_{i} \bigl[\sum_{j} \varphi(x_{j})(\partial_{x_{j}} \widehat{\mathcal{V}}-\partial_{x_{i}}\widehat{\mathcal{V}}) - \phi (x^{-d}) 
	\partial_{x_i} \widehat{\mathcal{V}} \bigr] + x^{-d} \sum_{j} \varphi(x_{j})\partial_{x_{j}} \widehat{\mathcal{V}}
	=0,\phantom{\Bigr]}
	\\
	\widehat{\mathcal{V}}(T,x)=\widehat{G}(x),
\end{array}
\right.  
\ee
for $t \in [0,T]$ and $x \in \mathrm{Int}(\widehat {\mathcal{S}}_d)$, where $\widehat{\mathcal{H}}_{M}(x,z) = {\mathcal H}_{M}(\check x,z) = \sum_{j \in \ESd} x_j \widehat{H}_{M}^j(z) +  x^{-d}  \widehat{H}_{M}^d(z)$,   
$\widehat{H}^i_{M}$, for $i=1,\dots,d$, is the Hamiltonian
\be
\label{eq:theta} 
\widehat H_{M}^i(z) := \tilde{H}_{M}^i
\bigl( \Theta(z)\bigr), \quad \Theta(z)=
\bigg(z_{1}-\frac1d \sum_{j=1}^{d-1} z_{j}, \cdots,
z_{d-1} -\frac1d \sum_{j=1}^{d-1} z_{j}, -\frac1d \sum_{j=1}^{d-1} z_{j}\bigg), 
\ee
for 
$z = (z_{1},\cdots,z_{d-1}) \in \RR^{d-1}$ and we denote $x^{-d}= 1- \sum_{j=1}^{d-1}x^j$; we refer to \cite{mfggenetic} for the derivation of the second order term, see Eq. (2.26) therein. 
Interestingly enough, the optimal feedback then writes ({in local chart}) in the form (provided that the HJB equation has a classical solution) 
$(\widehat{a}_{i,j}^\star(D_{x} \widehat{\mathcal V}))_{i,j \in \ES :i \not =j}$
with (recall the definition of $a^\star$ in \eqref{eq:astar})
\begin{equation}
	\label{eq:widehat:a}
	\widehat{a}_{i,j}^\star(z) = 
	\left\{
	\begin{array}{ll}
		a^\star(z_{i} - z_{j}), &\quad i,j \in \ESd,
		\\
		a^\star(z_{i}), &\quad j =d,
		\\
		a^\star(-z_{j}), &\quad i=d.
	\end{array}
	\right.
\end{equation}

We remark that, if the value function is in the Wright-Fischer space $\mathcal{C}^{1,2+\gamma}_{\rm WF}([0,T]\times\mathcal{S}_{d})$ {(to which we already alluded and which is defined in more detail in the Appendix)}, then $\mathcal{V}$ solves \eqref{hjbpotnew:sec:2} if and only in $\hat{\mathcal V}$ solves \eqref{hjbchart}. We choose to express the last coordinate in terms of the first $d-1$ for convenience only, and in fact the choice of the local chart is arbitrary. This is one reason why we expressed the main results in terms of intrinsic derivatives. {Anyhow},  the local chart {is more adapted to} the proof of Theorem \ref{thmder} below.
{Indeed, it is worth emphasizing that, in order to prove the 
well-posedness of \eqref{hjbpotnew:sec:2}, it is enough to check that, provided that 
it belongs to the right space,
 $\widehat{\mathcal{V}}$
 solves \eqref{hjbchart} in the interior of the simplex
for the fixed chart we have chosen.
In this regard, the precise choice of the local chart is not of a great importance 
and expressing any other coordinate than $x_{d}$ in terms of the other ones would work as well; 
to wit, by the same arguments as in \cite[Subsection 3.2.1]{mfggenetic}, Equation \eqref{hjbchart} can be equivalently written in terms of another local chart. In fact, 
the choice of the local chart really matters in the definition of the Wright-Fisher space carrying the solution, 
in order to describe finely the behaviour of the solution at the boundary. 
Fortunately, in the sequel, there is no need for returning to the details of the 
Wright-Fisher space and it is absolutely fine for us to work with the same local chart throughout the analysis. 
 This claim holds also for the derivative systems  \eqref{derhjbsimplex} 
 and
 \eqref{derhjbchart} 
 that we introduce below.}

\subsection{Derivative system}
\label{subse:3:2}

In order to address the HJB equation 
\eqref{hjbpotnew:sec:2}, we first study the derivative system. 
The rationale to do so is that, obviously, the nonlinear term in the derivative system is of order zero only while it is of order one in the HJB equation. As a byproduct, it makes it possible to apply \textit{a priori} estimates proven in \cite{mfggenetic}. 
As explained above, we can use both intrinsic derivatives and local charts.
Deriving \eqref{hjbpotnew:sec:2} ({by means of \eqref{eq:derivative:H}}), we formally get\footnote{{The computations in the
derivation of  
\eqref{derhjbsimplex}
and
\eqref{derhjbchart}
are rather tedious; anyhow, there is nothing difficult. We feel it is sufficient to just provide the final results.}
} the following expression for $V= \fD \mathcal{V}$, applying the Schwarz identity 
$\fd_i V^j = \fd_j V^i$
(the indices in the sums {below} belonging to $\ES$),
%
\be
\label{derhjbsimplex}
\begin{split}
	&\partial_t V^i + \widetilde{H}_{M}^i(V)  - \tfrac1d \sum\nolimits_j \widetilde{H}_{M}^j(V)
	+\sum\nolimits_j \bigl( \varphi(p_{j}) - p_{j} \varphi'(p_{i}) \bigr) (V^j-V^i)  \\
	&\qquad-\tfrac1d \sum\nolimits_l \sum\nolimits_j \bigl( \varphi(p_j) - 
	p_{j}\varphi'(p_l)
	\bigr)
	(V^j-V^l)  
	+f^i(p)
	-\tfrac1d\sum\nolimits_j  f^j(p)\\
	&\qquad 
	+ 	\sum\nolimits_{j,k} p_{k} \bigl( \varphi(p_{j})+  a^\star(V^k-V^j) \bigr) \left( \fd_{j} V^i - \fd_{k} V^i\right) 
	+\tfrac{1}{2} \varepsilon^2 \sum\nolimits_{j,k}(p_j \delta_{j,k}-p^j p^k) \fd_{j k}^2 V^i\\
	&\qquad + \tfrac{1}{2}
	\varepsilon^2 
	\Bigl( \fd_i V^i - 2 \sum\nolimits_{j} p_j\fd_j V^i - \tfrac1d \sum\nolimits_{j} \fd_j V^j \Bigr)
	=0,\\
	&V^i(T,p)= g^i(p) - \tfrac1d \sum\nolimits_j g^j(p),
\end{split}
\ee
where $\widetilde{H}^i_M(V)$ is defined by \eqref{eq:theta}. 
Instead, differentiating \eqref{hjbchart} with respect to 
$x$ (using in the sequel the generic notation $Z$ for $D_{x} \hat{\mathcal V}$)   and applying the Schwarz identity 
$\partial_{x_i} Z^j(t,x)=  \partial_{x_j} Z^i(t,x)$, for $i,j \in \ESd$, we then get, at least formally, 
the following system of equations (all the sums below are taken over $\ESd$): 
\begin{equation}
	\label{derhjbchart}
	\left\{
	\begin{array}{l} 
		\partial_t Z^i + \widehat{H}^i_{M}(Z) -\widehat{H}^d_{M}(Z)+{\hat{f}^i(x)} -{\hat{f}^d(x)}+ 
		\sum_{j}  \Bigl( \hat{b}^j(x,Z) + \tfrac12{\varepsilon^2}  \delta_{i,j} - \varepsilon^2 x_{j} \Bigr) \partial_{x_{j}} Z^i
		\\
		\quad
		+ \sum_{j  } \hat{c}^{i,j}(x) Z^j
		+\tfrac12 {\varepsilon^2}  \sum_{j,k }(x_j \delta_{j,k}-x_{j} x_k) \partial^2_{x_j x_{k}} Z^i =0,
		\phantom{\Bigr]}
		\\
		Z^i(T,x)= {\hat{g}^i(x)} -{\hat{g}^d(x)},
	\end{array}
	\right.
	\ee
	on $[0,T] \times \mathrm{Int}(\widehat{\mathcal{S}}_d)$, for 
	$i \in \ESd$, where, for $j \in \ESd$
	and $z=(z_{k})_{k \in \ESd} \in {\mathbb R}^{d}$, 
	\be
	\label{notation:a:b}
	\begin{split}
		\hat c^{i,j}(x) &= 
		\Bigl(\varphi'(x_{i}) - \varphi(x^{-d}) - \sum_{k \in \ESd} \varphi(x_{k}) \Bigr) \delta_{i,j} 
		+ \bigl( \varphi'(x^{-d}) - \varphi'(x_{i})\bigr) x_{j}
		\\
		\hat{b}^j(x,z) &=
		\sum_{k \in \ESd} \Bigl\{ x_{k} \big[ \phi(x_{j})+ a^\star(z_{k}-z_{j})\big]- x_j \big[\phi(x_{k})+a^\star(z_{j}-z_{k})\big] \Bigr\}
		\\
		&\hspace{15pt} + x^{-d} \big[\phi(x_{j})+a^\star(-z_{j})\big]
		-x_{j} \big[\phi\big( x^{-d}\big)+a^\star(z_{j})\big].
	\end{split}
\end{equation}
The two equations are equivalent, by using the identities
$Z^i = {\widehat V^i - \widehat V^d}$, 
${\widehat V^i = Z^i -\frac1d \sum_{j=1}^{d-1} Z^j}$ and  
${\widehat V^d = -\frac1d \sum_{j=1}^{d-1} Z^j}$, 
given by the {aforementioned dictionary to pass from one derivative to another}.

Here, we prove well-posedness of \eqref{derhjbchart}, because it is needed for solving the HJB equation \eqref{hjbchart}. {Recalling the shape of $\varphi$ from 
\eqref{eq:varphi:2}, our main solvability result is:}

\begin{thm}
	\label{thmder}
	If $f\in[\mathcal{C}^{{0},\gamma}_{\rm WF}({{\mathcal S}_d})]^{d}$ and $g\in[\mathcal{C}^{{0},2+\gamma}_{\rm WF}({{\mathcal S}_d })]^{d}$ for a given $\gamma\in(0,1)$, then
	there exists a constant $\kappa_{1}>0$ only depending on $M$, $T$ and $d$, such that for any $\varepsilon \in (0,1]$, ${\theta} >0$
	and $\kappa \geq \kappa_{1}/\varepsilon^2$, 
	there exists $\gamma'\in (0,\gamma]$, possibly depending on $\varepsilon$ and $\kappa$,  such that  Equation \eqref{derhjbchart} admits a unique solution in 
	$[\mathcal{C}^{{0},2+\gamma'}_{\rm WF}([0,T] \times {{\mathcal S}_d})]^{{d-1}}$.
\end{thm}

\begin{proof}
	The proof of existence is done via Leray-Schauder fixed point theorem. Let $\gamma'\in (0,\gamma]$ to be chosen later.
	Letting\footnote{Our notation for the Wright-Fisher space here is a bit abusive since it is regarded as a space of functions on $[0,T] \times \widehat{{\mathcal S}_{d}}$; as we already explained, there is no difficulty in passing from functions defined on $[0,T] \times \widehat{{\mathcal S}_{d}}$
	 	to functions defined on 
		$[0,T] \times {{\mathcal S}_{d}}$, and conversely. And in fact, the construction of the Wright-Fisher spaces, as outlined in Appendix, is based itself on a local description of the functions (that it contains) through a convenient choice of local charts.} $\mathcal{X}= \mathcal{C}^{{0},\gamma'}_{\rm WF}([0,T] \times \widehat{\mathcal S}_d)$, we consider the map 
	$\Phi:\mathcal{X}^{{d-1}} \rightarrow \mathcal{X}^{{d-1}}$, defined by $\Phi^i(Z)={Y^i}$, 
	where $Y^i$ is the solution to the linear equation 
	{obtained by freezing the zero order terms in \eqref{derhjbchart}}
	(all the sums being taken over $\ESd$):
	\be
	\label{fixpointeq}
	\left\{\begin{array}{l}
		\partial_t Y^i +   
		\sum_{j} \Bigl( \hat{b}^j(x,Z) + \tfrac12 {\varepsilon^2} \delta_{i,j} - 
		\varepsilon^2 x_{j}
		\Bigr)	\partial_{x_{j}} Y^i
		+\tfrac12 {\varepsilon^2}  \sum_{j,k }(x_j \delta_{j,k}-x_{j} x_{k}) \partial^2_{x_{j} x_{k}} Y^i
		\\ 
		\qquad=-\Big[\widehat{H}_{M}^i(Z) -\widehat{H}_{M}^d(Z)+\hat f^i(x) - \hat f^d(x) + \sum_{j} \hat c^{i,j}(x) Z^j
		\Big]
		\\
		Y^i(T,x)= \hat g^i(x) - \hat g^d(x).
	\end{array}
	\right.
	\ee
	The key remark is that, once $Z$ is given, this is a scalar equation for each $Y^i$, in the sense that there is no $Y^j$, $j\neq i$, in the equation. Therefore we are allowed to invoke Theorem 10.0.2 of \cite{epsteinmazzeo}, which states that there exists a unique solution  $Y^i\in\mathcal{C}^{{0},2+\gamma'}_{\rm WF}([0,T] \times \widehat{\mathcal S}_d)$ to \eqref{fixpointeq}, for any $i$, if the right hand side and the drift belong to $\mathcal{C}^{\gamma'}_{\rm WF}([0,T] \times \widehat{\mathcal S}_d)$ and the terminal condition is in  $\mathcal{C}^{{0},2+\gamma'}_{\rm WF}(\widehat{\mathcal S}_d)$. Such assumptions are satisfied in the present situation because $\widehat{H}_{M}^i-\widehat{H}_{M}^d$ and {$a^\star$
	(which shows up in $\hat b$, see \eqref{notation:a:b})}
	 are Lipschitz continuous and $\phi$ and $\phi'$ are bounded and Lipschitz; thus the map $\Phi$ is well-defined. The claim hence follows if $\Phi$ admits a fixed point. In order to apply Leray-Shauder fixed point theorem we must show that $\Phi$ is continuous and compact and that the set
	\[
	\mathfrak{X}=\left\{ Z\in\mathcal{X}^{{d-1}} : Z =\lambda \Phi(Z) \mbox{ for some } \lambda \in (0,1] \right\}
	\]
	is bounded in $\mathcal{X}^{{d-1}}$. 
	\vspace{5pt}
	
	\emph{Step 1.} 
	We first show that $\Phi$ is continuous and compact. {To do so, we may restrict ourselves to inputs $Z$ such that 	
$\max_{j \in \ESd} \|Z^j\|_{\textrm{WF},{0},\gamma'}$
is less than some arbitrarily fixed real $R>0$. Then, Theorem 10.0.2 of \cite{epsteinmazzeo} gives, for any $i \in \ESd$},  
	\be
	\label{epsm}
	\| Y^i \|_{\textrm{WF},{0},2+\gamma'} \leq C_{{R}} 
	\Bigl( \max_{j \in \ESd} \|Z^j\|_{\textrm{WF},{0},\gamma'} + \|f^i - f^d\|_{\textrm{WF},{0},\gamma'} + \|g^i-g^d\|_{\textrm{WF},{0},2+\gamma'} \Bigr),
	\ee
	{for some constant $C_{{R}} \geq 0$ depending on $R$ through the drift
	$\hat{b}(x,Z)$ in 
		\eqref{fixpointeq}}, 
	which yields {(up to a new value of $C_{R}$)}
\begin{equation}
\label{eq:compactness:phi}
{\max_{i \in \ESd} \|Y^i\|_{\textrm{WF},{0},2+\gamma'} \leq C_{R}}.
\end{equation}
{The above inequality} implies that the map $\Phi$ is compact, as $\mathcal{C}^{{0},2+\gamma'}_{\rm WF}({[0,T] \times {\mathcal S}_{d}})$ is compactly embedded in $\mathcal{C}^{{0},\gamma'}_{\rm WF}({[0,T] \times {\mathcal S}_{d}})$, {see the Appendix}. To prove continuity, we consider 
	the analogue of \eqref{epsm}, but  applied to $Y-Y'$ with $(Y,Y') =(\Phi(Z),\Phi(Z'))$, for $(Z,Z') \in ({\mathcal X}^{{d-1}})^2$. 
	{Again, we assume that 
	$\max_{j \in \ESd} \|Z^j\|_{\textrm{WF},{0},\gamma'}$
	and
	$\max_{j \in \ESd} \|(Z')^j\|_{\textrm{WF},{0},\gamma'}$ are less than $R$}. 
	So, 
	{using 
	\eqref{eq:compactness:phi}
	together with the fact that the derivatives of 
	$\widehat{H}_{M}^i-\widehat{H}_{M}^d$ 
are Lipschitz}, 
	we have
	\[
	\|\Phi(Z')-\Phi(Z)\|_{\mathcal{X}^{{d-1}}} \leq \max_{i \in \ESd} \| (Y')^i- Y^i \|_{\textrm{WF},{0},2+\gamma'} 
	\leq C_{{R}} \| Z'-Z\|_{\mathcal{X}^{{d-1}}},
	\]
	which proves continuity.
	\vspace{5pt}
	
	\emph{Step 2.}
	We now prove an $L^\infty$ bound of ${\mathfrak X}$. For $Z \in {\mathfrak X}$, we have, for some $\lambda \in (0,1]$, 
	\be
	\label{eqlambda}
	\left\{\begin{array}{l}
		\partial_t Z^i +   
		\sum_{j} \Bigl( \hat{b}^j(x,Z) + \tfrac12 {\varepsilon^2} \delta_{i,j} - 
		\varepsilon^2 x_{j}
		\Bigr)	\partial_{x_{j}} Z^i
		+\tfrac12 {\varepsilon^2}  \sum_{j,k }(x_j \delta_{j,k}-x_{j} x_{k}) \partial^2_{x_{j} x_{k}} Z^i
		\\ 
		\qquad=-\lambda \Big[\hat{H}^i_{M}(Z) -\hat{H}^d_{M}(Z)+\hat f^i(x) -\hat f^d(x) + \sum_{j} \hat c^{i,j}(x) Z^j
		\Big]
		\\
		Z^i(T,x)= \lambda \bigl( \hat g^i(x) - \hat g^d(x)\bigr).
	\end{array}
	\right.
	\ee
	The proof follows from a standard representation of $Z$ along the solution of the SDE that is driven by the second-order differential operator appearing in \eqref{eqlambda}. To make it clear, we have, for any 
	$i \in \ESd$ and $(t,x) \in [0,T] \times \textrm{\rm Int}(\widehat{S}_{d})$,
	\be
	\label{repre}
	\begin{split}
		&Z^i(t,x)= \lambda \E \biggl[ \int_{t}^T \overline f^i\bigl(X_{s}^{i,\cdot},
		Z(s,X^{i,\cdot}_{s}) \bigr) ds + \overline g^i(X_{T}^{i,\cdot})\biggr],
	\end{split}
	\ee
	where, for convenience, we have let
	$Z(s,X_{s}^{i,\cdot}) := 
	\bigl(Z^j(s,{X^{i,\cdot}_{s}})\bigr)_{j \in \ESd}$ together with 
	\be
	\label{notation:overline:f:g}
	\begin{split}
		&\overline g^i(x) = \hat g^i(x) - \hat g^i(d),
		\\
		&\overline f^i(x,z) = \widehat{H}_{M}^i(z) -\widehat{H}_{M}^d(z)+\hat f^i(x) - \hat f^d(x) + \sum_{i,j \in \ESd} \hat c^{i,j}(x) z_{j},
	\end{split}
	\ee
	for $x \in {\widehat{\mathcal S}_{d}}$ and $z$ in $\RR^{{d-1}}$. In \eqref{repre},  
	${\boldsymbol X}^{i,\cdot}=({\boldsymbol X}^{i,j}=(X^{i,j}_{t})_{t \le s \le T})_{j \in \ESd}$
	denotes a $({d\!-\!1})$-dimensional process 
	solving the SDE
	\be 
	\label{Xij}
	\begin{split}
		dX_{s}^{i,j} &= \Bigl( \hat{b}^j \bigl( X_{s}^{i,j}, Z(s,X_{s}^{i,\cdot}) \bigr) + \frac{\varepsilon^2}2 \delta_{i,j} - \varepsilon^2 X_{s}^{i,j} 
		\Bigr) ds
		\\
		&+ \frac{\varepsilon}{\sqrt{2}} \biggl\{\sum_{k \in \ESd} \sqrt{X^{i,j}_{s} X_{s}^{i,k}} d(W^{j,k}-W^{k,j}) 
		+ \sqrt{X_s^{i,j} X_{s}^{i,-d} } d(W^{j,d}-W^{d,j})\biggr\}
	\end{split}
	\ee
	for $t\leq s\leq T$, with initial condition $X_{t}^{i,\cdot}=x$, where we have denoted  $X^{i,-d}= 1- \sum_{j=1}^{{d-1}} X^{i,j}$. 
	
	Representation \eqref{repre} follows from 
	the fact that $Z^i \in \mathcal{C}^{1,2}({[0,T] \times} \Int)$
	{(which is here the usual space of functions that are once continuously differentiable in time and twice in space)}
	and hence 
	from
	It\^o's formula applied to $(Z^i(s,X_{s}^{i,\cdot}))_{t \leq s \leq T}$, provided that the solution to \eqref{Xij} remains in $\Int$.
	Assume for a while that the latter holds true. Then, having \eqref{repre} (together with the notations \eqref{notation:a:b} and \eqref{notation:overline:f:g}), we exploit the Lipschitz continuity of $(\widehat{H}^i_M)_{i\in\ES}$, the boundedness of $\phi$ and $\phi'$, the fact $\lambda\leq 1$, and the uniform bounds on $f$ and $g$  to obtain
	\begin{equation}
		\label{eq:linfty}
		|Z^i(t,x)| \leq \|{\hat g^i - \hat{g}^d}\|_\infty + T\|{\hat{f}^i-\hat{f}^d}\|_\infty + C\int_t^T\max_{j \in \ESd} \sup_{x' \in\Int}|Z^j_\lambda(s,x')|ds.
	\end{equation}
	Taking the supremum over $x \in \Int$ and the maximum over $i\in\ESd$ in the left-hand side and applying  
	Gronwall's lemma, we get a bound for $\max_{i \in \ESd} \sup_{(t,x) \in [0,T] \times \Int}|Z^i(s,x)|$. 
	By continuity of $Z$,  the $L^\infty$ bound also holds for $x$ in the boundary of $\widehat{\mathcal{S}}_d$.
	
	It remains to address the solvability of 
	\eqref{Xij}. We mostly borrow arguments from \cite[Proposition 2.1]{mfggenetic}. In order to apply the latter, we notice that 
	${\boldsymbol X}^{i,-d}$ solves (noticing that the sum over $j$ in the first line in the definition 
		{\eqref{notation:a:b}}
	of $\hat{b}^j$ 
	is null and similarly for the first term in the second line of \eqref{Xij})
	\be 
	\begin{split}
		dX_{s}^{i,-d} &= \Bigl\{ \sum_{j \in \ESd} X^{i,j}_{s} \bigl[ \varphi(X_{s}^{i,-d}) + 
		a^\star\bigl(Z^j(s,X_{s}^{i,\cdot})\bigr) \bigr] 
		\\
		&\hspace{30pt} - X^{i,-d}_{s}
		\sum_{j \in \ESd}  \bigl[ \varphi(X_{s}^{i,j}) + 
		a^\star\bigl(-Z^j(s,X_{s}^{i,\cdot})\bigr) \bigr] + \varepsilon^2 \bigl( 
		\tfrac12 - X_{s}^{i,-d}
		\bigr) \Bigr\} ds
		\\
		&\hspace{15pt} - \frac{\varepsilon}{\sqrt{2}} 
		\sum_{{j \in \ESd}}\sqrt{X^{i,j}_{s} X_{s}^{i,-d}} \bigl(dW_{s}^{j,d}-dW^{d,j}_{s}\bigr). 
	\end{split}
	\ee 
	The key fact is then to observe that, whenever $X^{i,j}_{s}$ is close to zero, 
	$\varphi(X^{i,j}_{s})$ 
	({which shows up in 
	the definition of the drift, 
	compare 
	\eqref{notation:a:b}
with \eqref{Xij}})	
	is greater than $\kappa$, and thus helps for pushing the particle towards the interior of the simplex. 
	This guarantees that, provided that $\kappa \geq \varepsilon^2/2$, the equation is well-posed and that the unique solution stays in $\Int$, see \cite[Proposition 2.1]{mfggenetic} for the details. 
	%
	\vspace{5pt}
	
	\emph{Step 3.} We now provide a (uniform) H\"older estimate for the elements of ${\mathfrak X}$. 
	Again we borrow the result from \cite{mfggenetic}. Indeed, 
	\eqref{eqlambda} can be rewritten as a system of $d-1$ equations on $[0,T] \times \mathcal{S}_{d}$, using the dictionary to pass from intrinsic derivatives to derivatives in the local chart.
	Thus we can apply\footnote{\label{foo:3}{In fact, this requires a modicum of care, since the function $\varphi$ in \cite{mfggenetic} is assumed to  vanish outside $[0,2\theta]$, see (2.16) therein, with $\delta$ replaced by $\theta$, and $\delta$ itself (with the same notation as therein) is required to be small enough. 
The key point is that we can always modify the function $\varphi$ so that it fits the assumption
of \cite[Theorem 3.6]{mfggenetic}: Going back to 
\cite[(3.20)]{mfggenetic}, it is indeed easy to see that the values of $\varphi$ 
taken at points that are away from the boundary can be inserted in the function $b^\circ$ therein. Since 
$b^\circ$ does not enter the definition of the threshold $\kappa_{0}$ in 
\cite[Theorem 3.6]{mfggenetic}, this leaves the conclusion 
of 
\cite[Theorem 3.6]{mfggenetic} unchanged. 
}} 
Theorem 3.6 of \cite{mfggenetic}, which states that 
	there exist $\kappa_1$ and $\gamma'$ as in the statement, and 
	a constant $C'$, depending on $\varepsilon,{\theta},M,d,T$ and the  $L^\infty$ norm of the r.h.s. of \eqref{fixpointeq} (hence on $f$, $g$, $\phi$, $\phi'$, and $Z$, which is uniformly bounded by step 2) such that $\|Z \|_{\textrm{WF},\gamma'}\leq C'$ {if $\kappa \geq \kappa_{1}/\varepsilon^2$}.
	Therefore $\mathfrak{X}$ is bounded, choosing such $\gamma'$, and the proof is completed.
	\vskip 4pt
	
\emph{Step 4}. {Uniqueness of classical solutions can be proved 
by using the so-called four step-scheme, see \cite{Delarue02,MaProtterYong}. Any classical solution $Z$
can be indeed represented in the form of a multi-dimensional forward-backward SDE (which is nothing 
but a system of stochastic characteristics). In turn, the fact that 
	\eqref{derhjbchart} has a classical solution forces the former 
	forward-backward SDE to be uniquely solvable, and hence 
	\eqref{derhjbchart}  itself to be also uniquely solvable. 
	This argument is in fact explained in detail in 
	 \cite[Theorem 3.3 and Corollary 3.4]{mfggenetic}. 
	 The specific subtlety (which is common to \cite{mfggenetic} and to our 
case) is that, due to the fact that the Kimura operator 
driving \eqref{derhjbchart} degenerates near the boundary, some exponential integrability is needed for 
the inverse of the forward component in the forward-backward system of characteristics. In fact, this integrability 
property is very similar to 
the integrability property discussed after 
	Definition \ref{def:MFG}.
	In short, it holds true provided that $\kappa$ is bigger than (up to a multiplicative constant) $\varepsilon^2$, which is
	obviously the case in our setting since $\kappa$ scales here (at least) like $\varepsilon^{-2}$.  
	 This point is discussed with care in
the paper \cite{mfggenetic}.}
\end{proof}

\subsection{Solving for the HJB equation}
We now turn to the well posedness of \eqref{hjbpotnew:sec:2}, or equivalently of \eqref{hjbchart}, and prove the following refined version of Theorem \ref{main:thm}.

\begin{thm}
	\label{thm:3.2}
	If $F \in \mathcal{C}^{1,\gamma}_{\rm WF}({{\mathcal S}_{d}})$ and $G\in \mathcal{C}^{1,2+\gamma}_{\rm WF}({{\mathcal S}_{d}})$, for a given $\gamma\in(0,1)$, then
	{there exists a constant $\kappa_{1}>0$ only depending on $M$, $T$ and $d$, such that for any $\varepsilon \in (0,1]$, $\theta  >0$
	and $\kappa \geq \kappa_{1}/\varepsilon^2$}, 
	Equation \eqref{hjbpotnew:sec:2} admits a unique solution $\mathcal{V}\in\mathcal{C}^{1,2+\gamma'}_{\rm WF}([0,T]\times {\mathcal S}_d)$. The solution $\mathcal{V}$ 
	is the value function of the viscous MFCP and the optimal feedback function is given by
	\be 
	\label{optconmfcp}
	\tilde{\alpha}^{\star,i,j}(t,p)= a^\star\bigl(\fd_i \mathcal{V}(t,p) - \fd_j \mathcal{V}(t,p)\bigr).
	\ee
	The latter gives the unique optimal control in the sense that, for any 
	initial state $p_{0} \in \mathrm{Int}({\mathcal S}_{d})$
	and any 
	pair of optimal trajectory  ${\boldsymbol p}$ and optimal control  ${\boldsymbol \alpha}$ (which is an ${\mathbb F}$-progressively measurable {process}  bounded by 
	$M$), it holds  $\alpha_t = \tilde{\alpha}^{\star}(t,p_t)$ for  $dt \otimes {{\mathbb P}}$ a.e. $(t,\omega)$.  
	Moreover, the derivative $\fD \mathcal{V}$ is the unique solution to \eqref{derhjbsimplex} in $\mathcal{C}^{0,2+\gamma'}_{\rm WF}([0,T]\times {{\mathcal S}_{d}})$. 
	
	Equivalently, with the same assumptions and in the same space ({up to a change of coordinate}), Equation \eqref{hjbchart} admits a unique solution $\hat{\mathcal V}$ and its derivative $D_x \hat{\mathcal V}$ is the unique solution to \eqref{derhjbchart} {(denoted by $Z$ in the statement of Theorem \ref{thmder})}.

\end{thm}




\begin{proof}
	As announced before, we prove well posedness of \eqref{hjbchart}. The candidate for being the optimal feedback
	is 
	(see \eqref{eq:widehat:a}) 
	$\alpha^\star_{i,j}(t,x) = \widehat{a}_{i,j}^\star(Z(t,x))$, $i,j \in \ES$, $i \not = j$, 
	for $Z$  given by Theorem \ref{thmder}. 
	Using the same notation as in 
	\eqref{eq:hamiltonians} and 
	\eqref{eq:widehat:a}, 
	we thus consider, on $[0,T] \times \widehat{\mathcal S}_d$, the PDE (sums being taken over $\ESd$):
	\be 
	\label{eq:HJB:Z}
	\left\{
	\begin{array}{l}
		\partial_t {\mathcal{Z}} + \widehat{\mathcal{H}}_{M}(x, Z) 
		+\widehat F(x) + \frac{1}{2} \varepsilon^2 \sum_{j,k} (x_j \delta_{j,k}-x_{j} x_{k}) \partial_{x_{j} x_{k}}^2  {\mathcal{Z}} \\
		\quad+\sum_{k} x_{k} \bigl[\sum_{j} \varphi(x_{j})(\partial_{x_{j}}  {\mathcal{Z}}-\partial_{x_{k}} {\mathcal{Z}}) - \phi (x^{-d}) 
		\partial_{x_k}  {\mathcal{Z}} \bigr] + x^{-d} \sum_{j} \varphi(x_{j})\partial_{x_{j}}  {\mathcal{Z}}
		=0, \phantom{\Bigr]}
		\\
		\mathcal{Z}(T,x)=\widehat G(x),
	\end{array}
	\right.
	\ee
	In particular, we can regard 
	\eqref{eq:HJB:Z}
	as a linear Kimura PDE (the drift coefficient driving the first order term is nothing but $\hat{b}_{j}(x,0)$ 
	and hence points inward the simplex). Since $Z\in [\mathcal{C}^{0,2+\gamma'}_{\rm WF}([0,T] \times \widehat{\mathcal S}_d)]^{{d-1}}$, we know from Theorem 10.0.2 of \cite{epsteinmazzeo} that \eqref{eq:HJB:Z} admits a unique solution $\mathcal{Z}\in\mathcal{C}^{0,2+\gamma'}_{\rm WF}([0,T] \times \widehat{\mathcal S}_d)$.
	
	The key fact is to show that ${\zeta}= Z$ where 
	$\zeta=
	D_x\mathcal{Z}$. Since the second order operator driving \eqref{eq:HJB:Z} is elliptic in the interior of the simplex (and non-degenerate in any ball, see for instance \cite[(2.27)]{mfggenetic}) {and the source term is differentiable in space, with time-space H\"older continuous derivatives}, we know from interior estimates for parabolic PDEs (see Theorem 8.12.1 in \cite{krylov}) that $\zeta$  is once continuously differentiable in time and twice in space on $[0,T) \times \Int$
	--even though we have no guarantee on the behavior at the boundary--. This suffices to differentiate \eqref{eq:HJB:Z}.
	We then get the following variant of \eqref{derhjbchart} at any point $(t,x)$ of ${[0,T)} \times \Int$   (the sums below being taken over $j \in \ESd$):
	\begin{equation*}
		\left\{\begin{array}{l}
			\partial_t \zeta^i +   
			\sum_{j} \Bigl( \hat{b}^j_{2}(x) + \tfrac12 {\varepsilon^2} \delta_{i,j} - 
			\varepsilon^2 x_{j}
			\Bigr)	\partial_{x_{j}} \zeta^i
			+\tfrac12 {\varepsilon^2}  \sum_{j,k }(x_j \delta_{j,k}-x_{j} x_{k}) \partial^2_{x_{j} x_{k}} \zeta^i+ \sum_{j} \hat c^{i,j}(x) \zeta^j
			\\ 
			\qquad=-\Big[\widehat{H}^i_{M}(Z) -\widehat{H}^d_{M}(Z)+\hat f^i(x) -\hat f^d(x) 
			+
			\sum_{j} \Bigl( \hat{b}^j_{1}(x,Z) + \tfrac12 {\varepsilon^2} \delta_{i,j} - 
			\varepsilon^2 x_{j}
			\Bigr)	\partial_{x_{j}} Z^i
			\Big],
			\\
			\zeta^i(T,x)= \hat g^i(x) -\hat g^d(x),\phantom{\Bigr]}
		\end{array}
		\right.
	\end{equation*}
	where, for $j \in \ESd$, 
	$x$ and 
	$z$ (as usual sums below are over $k \in \ESd$),
	\begin{equation*}
		\begin{split}
			\hat{b}_{1}^j(x,z) &=
			\sum\nolimits_{k} \Bigl\{ x_{k} a^\star(z_{k}-z_{j}) - x_j  a^\star(z_{j}-z_{k})  \Bigr\}
			+ x^{-d}  a^\star(-z_{j})
			-x_{j}  a^\star(z_{j}).
			\\
			\hat{b}_{2}^j(x) &=
			\sum\nolimits_{k} \Bigl\{ x_{k}   \phi(x_{j}) - x_j  \phi(x_{k})  \Bigr\}
			+ x^{-d}  \phi(x_{j}) 
			-x_{j}  \phi\big( x^{-d}\big) .
		\end{split}
	\end{equation*}
	Obviously, $\hat{b}_{1}^j$ and $\hat{b}_{2}^j$ should be compared with 
	$\hat{b}^j$ in \eqref{notation:a:b}. In particular, $\hat{b}^j(x,z)$ is nothing 
	but $\hat{b}_{1}^j(x,z)+ \hat{b}_{2}^j(x)$. 
	This prompts us to make the difference with \eqref{derhjbchart}, from which we get 
	\begin{equation*}
		\left\{\begin{array}{l}
			\partial_t \bigl( \zeta^i - Z^i \bigr) +   
			\sum_{j} \Bigl( \hat{b}^j_{2}(x) + \tfrac12 {\varepsilon^2} \delta_{i,j} - 
			\varepsilon^2 x_{j}
			\Bigr)	\partial_{x_{j}} \bigl( \zeta^i - Z^i \bigr) 
			\\
			\qquad	 +\tfrac12 {\varepsilon^2}  \sum_{j,k }(x_j  {\delta_{j,k}}
			-x_{j} x_{k}) \partial^2_{x_{j} x_{k}} \bigl( \zeta^i - Z^i \bigr)
			+ 
			\sum_{j} \hat{c}^{i,j}(x) \bigl( 
			\zeta^j  
			- Z^j  \bigr) = 0,
			\\
			(\zeta^i-Z^i)(T,x)= 0. \phantom{\Bigr]}
		\end{array}
		\right.
	\end{equation*}
	
	In order to prove that $\zeta=Z$, we can use It\^o's formula
	as we did in the proof of 
	Theorem \ref{thmder}. Indeed, the interior smoothness of $\zeta^i$, for each $i =1,\cdots,{d-1}$, suffices to apply 
	It\^o's formula to {$(\sum_{j \in \ESd} R_{s}^{i,j} (\zeta^j-Z^j)(s,X^{i,\cdot}_{s}))_{t  \le s \le T}$} 
	for any given $t$ where 
	${\boldsymbol X}^{i,\cdot}$
	solves 
	\eqref{Xij}, but for $\hat{b}^j(x,z)$ therein replaced by $\hat{b}^j_{2}$, 
	with some $x \in \Int$ as initial condition at time $t$, {and 
	$(({R}^{i,j}_{s})_{i,j \in \ESd})_{t \leq s \leq T}$ solves the SDE 
	$dR_{s}^{i,j}= \sum_{\ell \in \ESd} R_{s}^{i,\ell} \hat{c}^{\ell,j}(X_{s}^{i,\cdot}) $, for 
	$s \in [t,T]$ with $(R_{t}^{i,j}=\delta_{i,j})_{i,j \in \ESd}$}. 
	Following 
	the standard proof of Feynman-Kac formula, we get that 
	$\zeta^i(t,x)=Z^i(t,x)$. Hence, $\zeta^i(t,\cdot)$ and $Z^i(t,\cdot)$ coincide on $\Int$ and
	then, by continuity, on the entire $\widehat{\mathcal S}_{d}$. In particular, this implies that $\mathcal{Z}\in\mathcal{C}^{1,2+\gamma'}_{\rm WF}({[0,T] \times \widehat{\mathcal S}_d})$, see 
the definition of the hybrid spaces in Appendix.
	
	By replacing $Z$ by $D_{x} {\mathcal Z}$ in 
	\eqref{eq:HJB:Z}, we deduce that 
	${\mathcal Z}$ solves \eqref{hjbchart}. 
	By a straightforward adaptation of the verification theorem, we deduce that ${\mathcal Z}$ must be the value 
	function of the MFCP and, as by-product, it must be the unique solution 
	of \eqref{hjbchart} in the space $\mathcal{C}^{1,2+\gamma'}_{\rm WF}({[0,T] \times \widehat{\mathcal S}_d})$.
	Also, since $\widetilde{\mathbb H}_{M}(p,\alpha,w)$ (see \eqref{eq:hamiltonians}) is strictly convex with respect to $\alpha$ as long as $p$ is in $\textrm{\rm Int}({\mathcal S}_{d})$ and since any controlled trajectory ${\boldsymbol p}$ in 
	\eqref{dynpot}
	stays in $\textrm{\rm Int}({\mathcal S}_{d})$ 
	(see \cite[Proposition 2.1]{mfggenetic}), 
	we deduce that the optimal control is unique and is in a feedback form. In local coordinates, the optimal feedback function
	writes 
	\[  
	\begin{split}
	&\hat{\alpha}^{i,j}(t,x)= a^\star(\partial_{x_{i}}{\mathcal{Z}}(t,x)-\partial_{x_{j}}\mathcal{Z}(t,x)) \quad \mbox{ if } i,j\in {\ESd},
	\\
	&\hat{\alpha}^{i,d} =a^\star(\partial_{{x_{i}}}\mathcal{Z}(t,x)), \qquad  \hat{\alpha}^{d,i} =a^\star(-\partial_{{x_{i} }}\mathcal{Z}(t,x)),
	\end{split}
	\]
	and this is equivalent to \eqref{optconmfcp} in intrinsic derivatives.
	Relabelling  {${\mathcal Z}$ into $\widehat{\mathcal V}$}, this completes the proof. 
\end{proof}

\section{Potential game with a common noise}
\label{sec:4}

The main purpose of this section is to prove 
Theorem \ref{main:thm:2}.

\subsection{New MFG}
Our first step is to introduce an MFG that derives from the MFCP studied in the previous section. 
Equivalently, we would like the corresponding MFG system to represent the necessary condition for optimality of the MFCP. As we already explained in Section \ref{sec:main}, the problem is that, 
if we use the same dynamics as in 
\eqref{dynmfg} (which are the basis of the results of \cite{mfggenetic}, on which our paper is built), we can no longer use 
the cost functional $J^{\epsilon,\phi}$ (see \eqref{eq:cost:functional:noise}) 
to get a potential structure. To wit, the master equation associated with 
\eqref{eq:cost:functional:noise} (which may be computed along the same lines as in \cite{mfggenetic}, see (3.14) therein)
does not identify with 
the derivative system \eqref{derhjbchart}. 
In particular, the master equation associated with \eqref{dynmfg}-\eqref{eq:cost:functional:noise} 
{(which is an equation for the value of the game)} 
can not be closed as an equation for 
{the centered value of the game (centered here means that the sum over the states $i \in \ES$ is null)}, due to the additional drift in the equation given by the common noise: This means that the master equation can not be the intrinsic derivative of a HJB equation. Instead, this holds true for the MFG without common noise and will be exploited in the next sections.

In order to define the new cost of the MFG (see \eqref{eq:new:J:varepsilon,varphi}), let $V:= \fD \mathcal{V}$ be the classical solution to \eqref{derhjbsimplex} --which is the derivative of the MFCP value function by Theorem \ref{thm:3.2}-- and consider $(v^i_{t}=V^i(t,p_{t}^\star))_{0 \leq t \leq T}$ where ${\boldsymbol p}^\star=(p_{t}^\star)_{0 \leq t \leq T}$ solves 
the SDE driven by the optimal feedback, namely ({sums below are over $\ES$})
\begin{equation}
	\label{eq:optim:path}
	\begin{split}
		dp_{t}^{\star,i} &=
		\Bigl( \sum\nolimits_{j} p_{t}^{\star,j} \bigl( \varphi (p_{t}^{\star,i}) + a^\star(v_{t}^j - v_{t}^i) \bigr) 
		-
		\sum\nolimits_{j} p_{t}^{\star,i} \bigl( \varphi (p_{t}^{\star,j}) + a^\star(v_{t}^i - v_{t}^j) \bigr)
		\Bigr) dt 
		\\
		&\hspace{15pt}+ \frac1{\sqrt{2}} \varepsilon \sum\nolimits_{j} \sqrt{p_{t}^{\star,i} p_{t}^{\star,j} } \bigl( dB_{t}^{i,j} - dB_{t}^{j,i} \bigr),
	\end{split}
\end{equation}
{see \cite[Proposition 2.1]{mfggenetic} for the unique solvability, the unique solution remaining inside the interior of ${\mathcal S}_{d}$}. By It\^o's formula (the fact that we can apply It\^o's formula with intrinsic derivatives can be justified by using the local chart, at least in the interior of the simplex), we get 
(sums below being over indices in $\ES$)
%
\be
\label{ponsim}
\begin{split} 
	&dv^i_t  = - \Bigl( \widetilde H^i_{M}(v_{t}) -\tfrac1d \sum\nolimits_l \widetilde H^l_{M}(v_t)+ f^i(p_{t}^\star) -\tfrac1d \sum\nolimits_l f^l(p_t^\star)\Bigr) dt 
	+ \sum\nolimits_{j,k} w^{i,j,k}_t dB_{t}^{j,k} \\
	&-\biggl( \sum\nolimits_j \Bigl( \varphi(p^{\star,j}_t) - p_{t}^{\star,j} \varphi'(p_{t}^{\star,i}) \Bigr) (v^j_t-v^i_t)  - \tfrac1d  \sum\nolimits_{j,l} \Bigl( \varphi(p^{\star,j}_t) - p_{t}^j \varphi'(p_{t}^{\star,l}) \Bigr) (v^j_t-v^l_t) 
	\biggr)dt\\
	&- \tfrac{1}{\sqrt{2}}  \varepsilon \biggl( \sum\nolimits_{j} 
	\sqrt{p_{t}^{\star,j}(p_{t}^{\star,i})^{-1}} 
	\Bigl( w_{t}^{i,i,j} +w_t^{j,j,i}\Bigr) 
	-  \tfrac1{d}\sum\nolimits_{j,l} \sqrt{p_{t}^{\star,j} (p_{t}^{\star,l})^{-1}}\Bigl(w_{t}^{l,l,j}  + w_t^{j,j,l} \Bigr)\biggr)dt
	\\
	&v^i_T=g^i(p_T^\star) -\tfrac1d\sum\nolimits_l g^l(p_T^\star),
\end{split}
\ee
where
\begin{equation}
	\label{eq:w:i,j,k}
	w^{i,j,k}_{t}
	= W^{i,j,k}(t,p_{t}^\star), \quad \textrm{\rm with} \quad
	W^{i,j,k}(t,p)=
	\tfrac1{\sqrt{2}} \varepsilon \sqrt{p_{j} p_{k}} \bigl( {\mathfrak d}_{j} V^i - {\mathfrak d}_{k} V^i\bigr)(t,p). 
\end{equation}
Notice in particular that 
\begin{equation}
	\label{eq:nice:w:i,i,j}
	\begin{split}
		&\tfrac{1}{ \sqrt{2}} \varepsilon \sum\nolimits_{j}\sqrt{p_{t}^{\star,j}(p_{t}^{\star,i})^{-1}} 
		w_{t}^{i,i,j}  
		= 
		\tfrac12 \varepsilon^2 \sum\nolimits_{j} p_{t}^{\star,j} \bigl( {\mathfrak d}_{i} V^i - {\mathfrak d}_{j} V^i\bigr)(t,p_{t}^\star) ,
		\\
		&\tfrac1{\sqrt{2}} \varepsilon  \sum\nolimits_{j}\sqrt{p_{t}^{\star,j}(p_{t}^{\star,i})^{-1}}
		w_{t}^{j,j,i}  
		= \tfrac12 \varepsilon^2
		\sum\nolimits_{j} p_{t}^{\star,j} \bigl( {\mathfrak d}_{j} V^j - {\mathfrak d}_{i} V^j\bigr)(t,p_{t}^\star),
	\end{split}
\end{equation}
which permits to recover the penultimate line in 
\eqref{derhjbsimplex}, since
${\mathfrak d}_{i} V^j=
{\mathfrak d}_{j} V^i$
by Schwarz' Theorem for intrinsic derivatives.

Ideally, we would like to {see}
\eqref{ponsim}
as the stochastic HJB equation associated with our new MFG with common noise (see \cite[Lemma 3.1]{mfggenetic} for its derivation). 
However, 
we cannot do so directly because 
the pair $({\boldsymbol p}^\star,{\boldsymbol v})$
in 
\eqref{ponsim} takes values in the tangent bundle to the simplex, namely 
$\sum_{i \in \ES} v^i_{t}=0$ for any $t \in [0,T]$. 
Obviously, the latter is not consistent with our original MFG, whether there is a common noise or not. Indeed, if this were consistent, then, discarding for a while the common noise, we would have to think of $v^i_0$ as the minimum of $J(\cdot;{\boldsymbol p}^\star)$
in 
\eqref{eq:cost:functional}
whenever ${\boldsymbol q}$ therein starts from the Dirac mass at point $i$, but, then, there is no reason 
why the sum of all these costs over $i \in \ES$ should be null. 
In fact, we here recover the point raised in \eqref{eq:master:HJB}:
Therein, we can identify the two vectors $(U^{\varepsilon,\varphi,i}(t,p))_{i \in \ES}$ and 
$({\mathfrak d}_{i} {\mathcal V}^{\varepsilon,\varphi}(t,p))_{i \in \ES}$ up to a constant only. 
The idea below is thus to reconstruct from scratch the sum of the value functions. To do so, we notice
from 
\cite[Theorem 10.0.2]{epsteinmazzeo} again that we can solve
the PDE in the simplex (sums below are over indices in $\ES$)
\begin{equation}
	\label{eq:PDE:mathcal Y}
	\left\{
	\begin{array}{l}
		\partial_t \mathcal{Y} +
		\sum_{j,k} p_{j} \bigl( \varphi(p_{k}) + a^\star(V^j - V^k) \bigr)
		\bigl( 
		{\mathfrak d}_{j} {\mathcal Y}
		-{\mathfrak d}_{k} {\mathcal Y}
		\bigr) + \frac{\varepsilon^2}{2} \sum_{j,k }(p_j \delta_{j,k}-p_{j} p_{k}) {\mathfrak d}^2_{j,k} \mathcal{Y}
		\vspace{3pt}
		\\
		\hspace{0pt}+
		\tfrac12 \sum_{j,k} p_{j} \vert a^\star(V^j - V^k) \vert^2
		+
		\sum_{j,k} p_j  p_{k} \varphi'(p_{j})  
		\bigl( V^j - V^k \bigr) + \langle p ,f^{\bullet} (p)\rangle = 0, 
		\vspace{3pt}
		\\
		\mathcal{Y}(T,p)=\langle p,g^\bullet(p) \rangle,
	\end{array}
	\right.  
\end{equation}
where we recall that 
$\langle p,f^{\bullet}(p) \rangle$ (and similarly with $f$ replaced by $g$) here denotes the inner product 
$\sum_{i} p_{i}f^i(p)$. 

We are now in the position  to elucidate the shape of 
$\vartheta_{\varepsilon,\varphi}^i$ in 
\eqref{eq:new:J:varepsilon,varphi}, by letting
(we remove the superscripts $\varepsilon$ and $\varphi$ for simplicity)
\begin{equation}
	\label{eq:vartheta}
	\vartheta^{i}(t,p):= \sum\nolimits_{j }
	\Bigl[
	p_{j} \varphi'(p_{i})\bigl(V^i - V^j\bigr)(t,p) 
	+  \tfrac1{\sqrt2}\varepsilon
	\sqrt{p_{j}p_{i}^{-1}}
	\bigl( \widetilde W^{j,j,i} - \widetilde W^{i,i,j} - 2 \Upsilon^{i,j}\bigr)(t,p)\Bigr],
\end{equation}
where
\begin{equation}
	\label{eq:Upsilon:i,j}
	\begin{split}
		&\widetilde W^{i,j,k}(t,p) = W^{i,j,k}(t,p)
		- \langle p, W^{\bullet,j,k}(t,p) \rangle
		\\
		&\Upsilon^{i,j}(t,p) = \tfrac1{\sqrt2} \varepsilon \sqrt{p_{i}p_{j}} \Bigl( {\mathfrak d}_{i} {\mathcal Y}(t,p)
		-
		{\mathfrak d}_{j} {\mathcal Y}(t,p) - (V^i-V^j)(t,p)
		\Bigr).
	\end{split}
\end{equation}
Observe in particular that, despite the factor $\sqrt{p_{i}^{-1}}$ in 
\eqref{eq:vartheta}, the function $\vartheta$ is bounded and continuous on the entire $[0,T] \times {\mathcal S}_{d}$.
Using 
\eqref{eq:nice:w:i,i,j}, we indeed have
\begin{equation}
	\label{eq:expression:vartheta:i}
	\begin{split}
		\vartheta^{i}(t,p)&= \sum\nolimits_{j }
		p_{j} \varphi'(p_{i})\bigl(V^i - V^j\bigr)(t,p) 
		\\
		&\hspace{15pt}+ 
		\tfrac12 \varepsilon^2
		\sum\nolimits_{j} p_{j} \bigl( {\mathfrak d}_{j} V^j - {\mathfrak d}_{i} V^i\bigr)(t,p)
		+ \varepsilon^2 
		\sum\nolimits_{j} p_{j} \bigl( {\mathfrak d}_{j} {\mathcal Y} - {\mathfrak d}_{i} {\mathcal Y} - (V^j- V^i)\bigr)(t,p)
		\\
		&\hspace{15pt}+ 
		\varepsilon^2 \sum\nolimits_{j,k} p_{j} p_{k} \bigl( {\mathfrak d}_{i} V^k - {\mathfrak d}_{j} V^k \bigr)(t,p).    
	\end{split}
\end{equation}
Now, we recall \eqref{eq:new:J:varepsilon,varphi} {together with 
Definition \ref{def:MFG}}: 
For an adapted continuous process ${\boldsymbol p}$ with   
values in ${\mathcal S}_{d}$, such that $\int_0^T (1/p^i_t) dt$ has exponential moments of sufficiently high order (which we recall holds true if 
{${\boldsymbol p}$
solves an equation of the same type as 
	\eqref{dynpot}
and
$\kappa$ is large enough independently of ${\boldsymbol p}$})
for a progressively-measurable process $\balpha=((\alpha_{t}^{i,j})_{i,j \in \ES : i \not = j})_{0 \leq t \leq T}$ 
such that $0 \leq \alpha^{i,j}_t\leq M$
and for ${\boldsymbol q}$ solving 
\eqref{dynmfg}, we  let
\begin{equation}
	\label{newcostmfg}
	\begin{split}
		\tilde{J}^{\epsilon,\phi}\bigl({\boldsymbol \alpha};{\boldsymbol p}\bigr) &:=  \E \biggl[\int_0^T \sum_{i \in \ES} q^i_t \Bigl[{\mathfrak L}^i(\alpha_t) + f^i(p_t)  
		+ \vartheta^i(t,p_{t})  
		\Bigr]
		dt + \sum_{i \in \ES} q^i_T g^i(p_T) \biggr].
	\end{split}
\end{equation} 
By following \cite[Subsection 3.1.1]{mfggenetic}, the Stochastic HJB (SHJB) equation associated with this minimization problem here writes down
(sums being taken over $\ES$)
\begin{equation}
	\label{newshjb}
	\begin{split}
		&du_{t}^i = - \Bigl( \widetilde H^i_{M}(u_{t}) +\sum\nolimits_j\varphi(p^j_t)(u^j_t-u^i_t) 
		+f^i(p_{t}) + \vartheta^i(t,p_{t}) \Bigr) dt 
		\\
		&\qquad - \tfrac{1}{\sqrt{2}}\varepsilon \sum\nolimits_{j\neq i} 
		\sqrt{ p_{t}^j (p_{t}^i)^{-1}} \bigl( \nu_{t}^{i,i,j} - \nu_{t}^{i,j,i} \bigr) dt
		+ \sum\nolimits_{j\neq k} \nu_{t}^{i,j,k} dB_{t}^{j,k}, \\
		&u^i_T=g^i(p_T).
	\end{split}
\end{equation}
Hence, our new MFG (in the sense of 
Definition \ref{def:MFG})
is characterized by 
the forward-backward system made of the SHJB equation \eqref{newshjb}
and of the Stochastic FP (SFP) equation \eqref{dynpot}, see again 
\cite[Subsection 3.1.1]{mfggenetic} for the proof. 

Of course, the core of our construction is to show that 
the {optimal trajectory} ${\boldsymbol p}^\star$ of the MFCP is the unique possible equilibrium of this new MFG. 
In this regard, our choice for $\vartheta$ is especially designed {so} that 
$(v_{t}^i - \langle p_{t}^\star,v_{t}^{\bullet} \rangle + {\mathcal Y}(t,p_{t}^\star))_{0 \leq t \leq T}$ solves 
\eqref{newshjb} whenever ${\boldsymbol p}$ is taken as ${\boldsymbol p}^\star$. In such a case, 
{by equalizing the martingale terms in the expansions of 
$(u_{t}^i)_{0 \le t \le T}$
and
$(v_{t}^i - \langle p_{t}^\star,v_{t}^{\bullet} \rangle + {\mathcal Y}(t,p_{t}^\star))_{0 \leq t \leq T}$}, 
we get from
\eqref{ponsim} and \eqref{eq:w:i,j,k} 
\begin{equation}
	\label{eq:nu:i,j,k:star}
	\begin{split}
		\nu_{t}^{i,j,k} &= w^{i,j,k}_{t} - \langle p_{t}^\star,w^{\bullet,j,k}_{t} \rangle 
		+ \tfrac{1}{\sqrt{2}} \varepsilon \sqrt{p_{t}^{\star,j}p_{t}^{\star,k}}
		\Bigl( {\mathfrak d}_{j} {\mathcal Y}(t,p_{t}^\star) - {\mathfrak d}_{k} {\mathcal Y}(t,p_{t}^\star) - \bigl(v_{t}^j - v_{t}^k \bigr) \Bigr)
		\\
		& = \tilde w^{i,j,k}_{t}  +  \Upsilon^{j,k}(t,p_{t}^\star),
	\end{split}
\end{equation}
where $\tilde w^{i,j,k}_{t}= w^{i,j,k}_t - \langle p , w^{\bullet,j,k}_r\rangle$, which explains why $\Upsilon$ appears in 
\eqref{eq:vartheta}. The details are given in the proof of Theorem \ref{thm4} below. 

\subsection{Solvability}

This is the refined version  of Theorem \ref{main:thm:2}. Recall that $\phi$ {satisfies} \eqref{eq:varphi:2}.

\begin{thm}
	\label{thm4}
	If $F\in\mathcal{C}^{1,\gamma}_{\rm WF}({{\mathcal S}_{d}})$ and $G\in\mathcal{C}^{1,2+\gamma}_{\rm WF}({{\mathcal S}_{d}})$ for a given $\gamma\in(0,1)$, then
	there exists a constant $\kappa_{2}\geq \kappa_1$ ($\kappa_1$ and $\gamma'$ being given by Theorem \ref{thmder}) only depending on $M$, $T$ and $d$, such that for any {$\varepsilon \in (0,1]$}, ${\theta} >0$
	and $\kappa \geq \kappa_{2}/\varepsilon^2$, 
	there exists $\gamma''\in (0,\gamma']$, possibly depending on $\varepsilon$ and $\kappa$,
	such that the new MFG, associated with the dynamics \eqref{dynmfg} and with the cost \eqref{newcostmfg},  admits a unique solution 
	$(\boldsymbol{p}, \boldsymbol{\alpha})$ for any $p_0\in \mathrm{Int}(\mathcal{S}_d)$.  It is equal to the unique optimizer of the MFCP  
	\eqref{eq:Cost:mfc}-\eqref{dynpot}. Moreover, the master equation \eqref{eq:master:equation:introl} associated with the modified MFG admits a unique solution 
	$U\in [\mathcal{C}^{0,2+\gamma''}_{\rm WF}({[0,T] \times {\mathcal S}_d})]^{d}$ and 
	\eqref{eq:master:HJB} holds.
\end{thm}

\begin{proof}
	We first prove existence of a MFG solution, by using the solution of the MFC problem, and then show 
	uniqueness by invoking the results from \cite{mfggenetic}.
	\vskip 4pt
	
	\textit{Existence.}
	As announced in the previous subsection, we choose ${\boldsymbol p}={\boldsymbol p}^\star$ 
	{with ${\boldsymbol p}^\star$ as in 
		\eqref{eq:optim:path}}
	(here, we drop the superscript $\star$ to alleviate the notation) and then let, as a candidate for solving the SHJB equation
	\eqref{newshjb}:
	\begin{equation}
		\label{eq:candidate:solving:MFG}
		u_{t}^i := v_{t}^i - \langle p_{t},v_{t}^{\bullet} \rangle + {\mathcal Y}(t,p_{t}), \quad t \in [0,T], \quad i \in \ES. 
	\end{equation}
	Importantly, we notice that $u_{t}^i - u_{t}^j = v_{t}^i - v_{t}^j$ for any $i,j \in \ES$ with $i \not = j$. 
	With the same notation as in \eqref{eq:nu:i,j,k:star}, we then get ({some explanations are given after the formula; moreover, the sums below are over $\ES$})
	\begin{equation*}
		\begin{split}
			du_{t}^i &= - \Bigl( \widetilde H^i_{M}(u_{t}) - \langle p_{t}, \widetilde H_{M}^{\bullet}(u_{t}) \rangle
			+f^i(p_{t}) - \langle p_{t},f^{\bullet}(p_{t}) \rangle   \Bigr) dt
			\\
			&\hspace{15pt}-
			\Bigl( \sum\nolimits_j\varphi(p^j_t)(u^j_t-u^i_t) - \sum\nolimits_{j,k}
			p^k_{t}
			\varphi(p^j_t)(u^j_t-u^k_t)\Bigr) dt
			\\
			&\hspace{15pt}+
			\Bigl( \sum\nolimits_j p_{t}^j \varphi'(p^i_t)(u^j_t-u^i_t) - \sum\nolimits_{j,k}
			p^j_{t}p_{t}^k
			\varphi'(p^k_t)(u^j_t-u^k_t)\Bigr) dt
			\\
			&\hspace{15pt}- \tfrac{1}{\sqrt{2}}  \varepsilon \biggl( \sum\nolimits_{j} 
			\sqrt{p_{t}^{j}(p_{t}^{i})^{-1}} 
			\Bigl( w_{t}^{i,i,j} +w_t^{j,j,i}\Bigr) 
			- \sum\nolimits_{j,k} \sqrt{p_{t}^{ j} p_{t}^{k} }\Bigl(w_{t}^{k,k,j}  + w_t^{j,j,k} \Bigr)\biggr)dt
			\\
			&\hspace{15pt} -
			\Bigl(
			\sum\nolimits_{j,k} p_{t}^{j} \bigl( \varphi (p_{t}^{k}) + a^\star(u_{t}^j - u_{t}^k) \bigr) \bigl(u_{t}^k - u_{t}^j \bigr)
			\Bigr) dt
			- \tfrac{1}{\sqrt{2}} \varepsilon\sum\nolimits_{j,k} \sqrt{p_{t}^j p_{t}^k} \bigl( 
			w^{j,j,k}_{t} -w^{j,k,j}_{t} \bigr) dt  
			\\
			&\hspace{15pt} - \Bigl( \tfrac12 \sum_{j,k} p^{j}_{t} \vert a^\star(u^j_t - u^k_{t}) \vert^2 
			+ \sum_{j,k} p^j_{t} p^k_{t} \varphi'(p_{t}^j)\bigl( u^j_{t} - u^k_{t} \bigr) +\langle p_{t},f^{\bullet}(p_{t}) \rangle \Bigr) dt + 
			\sum_{j,k} \nu_{t}^{i,j,k} dB_{t}^{j,k}. 
		\end{split}
	\end{equation*}
	In short, the term on the first line come from the expansion of $dv_{t}^i - d \langle p_{t},v_{t}^\bullet \rangle$, see the first line in 
	\eqref{ponsim}. 
	Similarly, the terms on the second and third lines come from the second line in 
	\eqref{ponsim}. And the fourth line derives from the third line 
	in \eqref{ponsim}. The first term on the fifth line comes from $\langle v^{\bullet}_{t},dp_{t} \rangle$ 
	and the second term on the same line is the bracket in the expansion of the inner product $d\langle v^{\bullet}_{t},p_{t}\rangle$. 
	The first term on the last line comes from the expansion of $({\mathcal Y}(t,p_{t}))_{0 \le t \le T}$ by means of It\^o's formula. 
	The last term is given by \eqref{eq:nu:i,j,k:star}.
	
	We first treat terms that cancel in the above expansion. Obviously, the inner products $\langle p_{t},f^{\bullet}(p_{t})\rangle$
	on the top and bottom lines cancel. Similarly, 
	the second term on the second line cancel out with 
	half of the first term on the penultimate line, and
	the second term on the third line cancel out with the second term on the last line. 
	As for the inner product $\langle p_{t},\widetilde H_{M}^{\bullet}(u_{t}) \rangle$ on 
	the first line, it cancels with the second half of the first term on the fifth line and 
	with the first on term on the last line. Now, using the fact that $w^{j,k,j}_{t}=-w^{j,j,k}_{t}$, we have
	\begin{equation*}
		\begin{split}
			&  \tfrac{1}{\sqrt{2}}  \varepsilon  \sum\nolimits_{j,k} \sqrt{p_{t}^{ j} p_{t}^{k} }\Bigl(w_{t}^{k,k,j}  + w_t^{j,j,k} \Bigr)
			-
			\tfrac{1}{\sqrt{2}} \varepsilon\sum\nolimits_{j,k} \sqrt{p_{t}^j p_{t}^k} \bigl( 
			w^{j,j,k}_{t} -w^{j,k,j}_{t} \bigr) =0,
		\end{split}
	\end{equation*}
	so that the last terms on the fourth and fifth lines also cancel out. 
	Moreover, adding $\vartheta^{i}(t,p_{t})$
	(using 
	\eqref{eq:vartheta}) to the first term on the third line and 
	the first term on the fourth line, we get 
	\begin{equation*}
		\begin{split}
			& \vartheta^{i}(t,p_{t}) +
			\sum\nolimits_j p_{t}^j \varphi'(p^i_t)(u^j_t-u^i_t)
			-
			\tfrac1{\sqrt{2}} \varepsilon \sum\nolimits_{j} 
			\sqrt{p_{t}^j (p_{t}^i)^{-1}} \bigl(w^{i,i,j}_{t} + w^{j,j,i}_{t} \bigr)
			\\
			&= 
			\sum\nolimits_{j }
			\Bigl[ \tfrac1{\sqrt2}\varepsilon
			\sqrt{p_t^{j} (p_{t}^{i})^{-1}}
			\bigl( \widetilde w_{t}^{j,j,i} - \widetilde w_{t}^{i,i,j} - 2 \Upsilon^{i,j}\bigr)(t,p_{t})\Bigr]
			-
			\tfrac1{\sqrt{2}} \varepsilon \sum\nolimits_{j} 
			\sqrt{p_{t}^j (p_{t}^i)^{-1}} \bigl(w^{i,i,j}_{t} + w^{j,j,i}_{t} \bigr)
			\\
			&= 
			\sum\nolimits_{j }
			\Bigl[ \tfrac1{\sqrt2}\varepsilon
			\sqrt{p_t^{j} (p_{t}^{i})^{-1}}
			\bigl(- 2 w_{t}^{i,i,j} + 2 \langle p_{t},w^{\bullet,i,j}_t \rangle - 2 \Upsilon^{i,j}\bigr)(t,p_{t})\Bigr].
		\end{split}
	\end{equation*}
	where, in the second and third lines, we used the two equalities 
	$\langle p_{t},w^{\bullet,i,j}_{t} \rangle = -
	\langle p_{t},w^{\bullet,j,i}_{t} \rangle$
	and $  \Upsilon^{i,j}(t,p_{t})
	= -  \Upsilon^{j,i} (t,p_{t})$. 
	
	It remains to see 
	from 
	\eqref{eq:nu:i,j,k:star}
	that 
	\begin{equation*}
		\begin{split}
			\nu_{t}^{i,i,j} - \nu_{t}^{i,j,i} &= w^{i,i,j}_{t} - w_{t}^{i,j,i} - \langle p_{t},w^{\bullet,i,j}_t \rangle
			+   \langle p_{t},w^{\bullet,j,i}_t \rangle +  \Upsilon^{i,j}_{t}(t,p_{t}) - \Upsilon_{t}^{j,i}(t,p_{t})
			\\
			&= 2w^{i,i,j}_{t}  -
			2   \langle p_{t},w^{\bullet,i,j}_t \rangle   + 2 \Upsilon_{t}^{i,j}(t,p_{t}).
		\end{split}
	\end{equation*}
	We hence get that the pair $(p_{t},u_{t})_{0 \le t \le T}$ solves 
	\eqref{eq:optim:path} (with $v_{t}$ replaced by $u_{t}$ therein) 
	and 
	\eqref{newshjb}. 
	\vskip 4pt
	
	\textit{Uniqueness.}
	For $\kappa_2$  as in the statement, uniqueness follows from\footnote{{As we already explained in 
	footnote \ref{foo:3}, some care is needed to apply the results of \cite{mfggenetic}, since the framework therein
	is not exactly the same. In footnote \ref{foo:3}, we already commented on the 
	shape of the function $\varphi$. This observation is still relevant here. 
	Also, it must be stressed that, in 
	\cite{mfggenetic}, the constant $\kappa_{0}$ (with the same notation as therein, $\kappa_{0}$ denoting the threshold for $\kappa$ in 
	\cite{mfggenetic} and hence being the analogue of $\kappa_{2}/\varepsilon^2$ here)
	 is allowed to depend on $\| f \|_{\infty}$
	and $\|g\|_{\infty}$. In our analysis, $\kappa_{2}$ is allowed to depend $M$, which in turn depends on 
	$\| f \|_{\infty}$
	and $\|g\|_{\infty}$. So, the latter is consistent with the results of 
	\cite{mfggenetic}. In fact, our framework is easier since the controls are already required to be bounded by 
	$M$, which is not the case in \cite{mfggenetic}. This explains why $M$ directly shows up  in our statement; in short, it provides an upper bound for the drift in the dynamics of ${\boldsymbol p}$.}} \cite[Theorem 2.9]{mfggenetic}, using the fact that ${\vartheta}$ in 
	\eqref{eq:expression:vartheta:i} is H\"older continuous, which is turn follows from the fact that 
	the solution to the linear equation 
	\eqref{eq:PDE:mathcal Y}
	belongs to $\mathcal{C}^{0,2+\gamma'}_{\rm WF}([0,T] \times  {\mathcal S}_d)$, by \cite[Theorem 10.0.2]{epsteinmazzeo}. The new exponent $\gamma''$, as well as existence and uniqueness of a classical solution to the master equation \eqref{eq:master:equation:introl}, 
	then follow from \cite[Theorem 3.8]{mfggenetic}. Finally, \eqref{eq:master:HJB} follows from \eqref{eq:candidate:solving:MFG}.
\end{proof}

\section{Selection by vanishing viscosity}
\label{sec:5}

\subsection{Selection of equilibria}
\label{subse:selection}
The purpose of this subsection is to prove Theorem 
\ref{main:thm:4}. To do so, 
we choose $\varphi=\varphi_{\theta,\delta,\varepsilon}$ satisfying (in addition to the aforementioned monotonicity and regularity properties)   
\be
\label{newphi}
\varphi_{\theta,\delta,\varepsilon}(r) = 
\begin{cases}
	\kappa_\varepsilon \qquad & r\leq \theta\\
	\kappa_0 & 2\theta \leq r\leq \delta \\
	0 &r\geq 2\delta,
\end{cases}
\ee 
for $0<\delta\leq 1/2$ and $0< {2\theta}<\delta$ and $\kappa_0>0$. {Above, 
we choose 
$\kappa_{\varepsilon}$ of the form $\kappa_{\varepsilon}:=\varepsilon^{-2} \kappa_{2}$, 
for $\kappa_{2}$ as in 
the statement of Theorem \ref{main:thm:2}; in particular, $\kappa_{2}$ is fixed once for all and 
only depends on $\| f \|_{\infty}$, $\|g \|_{\infty}$, $T$ and $d$, and is thus independent of
the four remaining parameters  
$\theta$, $\delta$, $\varepsilon$ and $\kappa_{0}$
in \eqref{newphi}. As for $\kappa_{0}$, it is a constant whose value is fixed later on; 
say that, in the end, it must be above some threshold $\bar \kappa_{0}$, only depending on 
$d$ and $\kappa_{2}$, see for instance 
Theorem 	\ref{thm:5.10}.} In order for  $\phi_{\theta,\delta,\varepsilon}$ to be non-increasing, we will impose a smallness condition on $\varepsilon$, {namely} $\varepsilon^2\leq {\varepsilon_{0}^2:=\min( \kappa_2 / \kappa_0,1)}$; {again, we stress the fact that, in this condition,} the constant $\kappa_0$ will be chosen later on, while $\kappa_2$ is fixed. We also assume that 
\begin{equation}
\label{eq:bound:phiprime}
\forall r \geq 0, \quad \bigl\vert \varphi_{\theta,\delta,\varepsilon}'(r) \bigr\vert \leq 
\frac{2 \kappa_{\varepsilon}}{\theta} {\mathbbm 1}_{[0,2\theta]}(r) + 
\frac{2 \kappa_{0}}{\delta} {\mathbbm 1}_{[0,2\delta]}(r).
\end{equation}
The rationale for introducing an additional parameter $\theta$ in the decomposition 
\eqref{newphi} is the following: Whilst $\kappa_{\varepsilon}$ blows up with $\varepsilon$, 
$\kappa_{0}$ does not; Here, $\kappa_{\varepsilon}$ is used to force uniqueness of the MFG equilibrium 
(as in \cite{mfggenetic}), hence the need to have it large when the intensity of the common noise is small;
Differently, $\kappa_{0}$ is used below to force the equilibrium to stay sufficiently
far 
away from the boundary, see 
Proposition \ref{expfin}
below. The new decomposition 
\eqref{newphi}
is thus a way to disentangle {the} two issues. 

Accordingly, 
for any initial condition $(t_{0},p_{0}) \in [0,T] \times \textrm{\rm Int}({\mathcal S}_{d})$, 
we write 
$({\boldsymbol p}_{[t_{0},p_{0}]}^{\theta,\delta,\varepsilon},{\boldsymbol \alpha}_{[t_{0},p_{0}]}^{\theta,\delta,\varepsilon}):= (p_{[t_{0},p_{0}],t}^{\theta,\delta,\varepsilon},\alpha_{[t_{0},p_{0}],t}^{\theta,\delta,\varepsilon})_{t_{0} \leq t\leq T}$ 
for the minimizer of 
${\mathcal J}^{\varepsilon,\varphi}$ defined by \eqref{eq:Cost:mfc}
with $\varphi$ being given by \eqref{newphi}. {When there is no ambiguity on the choice of the initial condition, 
we merely write $({\boldsymbol p},{\boldsymbol \alpha})$.}
Importantly, in this notation, 
the $\mathbb{F}$-progressively measurable process ${\boldsymbol \alpha}_{[t_{0},p_{0}]}^{\theta,\delta,\varepsilon}$
is given by \eqref{eq:optimal:control:mfg} through a feedback function; we feel useful to recall that its off-diagonal entries are 
bounded by $M$. 
 In order to state our main result\footnote{The attentive reader will observe that the $L^2$ space we here call 
${\mathcal E}$ is slightly different from the $L^2$ space used in the statement of 
Theorem \ref{main:thm:4}; obviously, ${\mathcal E}$ is here a smaller (closed) 
subset and the result proven below suffices to derive 
Theorem \ref{main:thm:4}. In fact, we felt better not to introduce the space ${\mathcal E}$ earlier in the text, which explains why Theorem \ref{main:thm:4} and \ref{thm10} are slightly different.} here, we also let $\mathcal{E}:= L^2 ([0,T];{\mathscr A} )$
where ${\mathscr A}=\{(a_{i,j})_{i,j=1,\cdots,d} \in \RR^{d \times d} : a_{i,j} \in [ 0,M], \ i \not =j \ ; \ a_{i,i}
= - \sum_{j \not =i} a_{i,j}\}$. 	
 We endow $\mathcal{E}$ with the weak topology, which makes it a compact metric (and hence Polish) space and for which the convergence is denoted by $\rightharpoonup$.
We also denote by $\mathcal{V}^{\theta, \delta,\varepsilon}$ the value function of the viscous MFCP and by $\mathcal{V}$ the value function of the inviscid MFCP. {The following result then subsumes Theorem \ref{main:thm:4}}:

\begin{thm}
\label{thm10}
{Assume that $F\in\mathcal{C}^{1,\gamma}_{\rm WF}({{\mathcal S}_{d}})$ and $G\in\mathcal{C}^{1,2+\gamma}_{\rm WF}({{\mathcal S}_{d}})$ for a given $\gamma\in(0,1)$.
Moreover, fix the value of $\kappa_{0}$ in Definition 
\ref{newphi} and} let $p_0 \in \mathrm{Int}(\mathcal{S}_{d})$ stand for the initial condition 
of 
\eqref{dynpot}
at time $0$. Then there exists $\delta_0$, depending only on $p_0$, $M$, $T$ and $d$, such that
the family
 $({\boldsymbol p}^{\theta,\delta,\varepsilon},{\boldsymbol \alpha}^{\theta,\delta,\varepsilon})_{0<{2}\theta<\delta\leq \delta_0,0<\varepsilon\leq {\varepsilon_{0}}}$ is tight in $ \mathcal{C}([0,T];\mathcal{S}_d) \times {\mathcal E}$.
  The
limit in law of any converging subsequence $({\boldsymbol p}^{\theta_n,\delta_n,\varepsilon_n},{\boldsymbol \alpha}^{\theta_n,\delta_n,\varepsilon_n})$,
with $\lim_{n \rightarrow   \infty} (\theta_n, \delta_n,\varepsilon_n)=(0,0,0)$, 
 is a 
 probability ${\mathbb M}$ that satisfies the conclusion of Theorem 
 \ref{main:thm:4}. Moreover, {for any such converging subsequence}, 
 \begin{equation}
 \label{limcosts}
 \lim_{n \rightarrow \infty}
 {\mathcal J}^{\varepsilon_{n},\varphi_{\theta_{n},\delta_{n},\varepsilon_{n}}}\Bigl(
 {\boldsymbol \alpha}^{\theta_n,\delta_n,\varepsilon_n}
 \Bigr) 
 = \min_{\boldsymbol \beta \in \mathcal{E}} {\mathcal J}({\boldsymbol \beta}).
 \end{equation}
In particular, for any $t_0\in[0,T]$ and any $p_0\in \mathrm{Int}(\mathcal{S}_{d})$, 
 \begin{equation}
  \label{limcosts:bis}
 \lim_{(\theta, \delta,\varepsilon) \rightarrow (0,0,0)} \mathcal{V}_{\theta, \delta,\varepsilon}(t_0,p_0) = \mathcal{V}(t_0,p_0).
\end{equation}
\end{thm}

\begin{proof}
Throughout the proof, we use the following notation. For ${\boldsymbol \alpha}=(\alpha_{t})_{0 \le t \le T}$
a {bounded} deterministic path in ${\mathcal E}$, we call $ {\boldsymbol p}^{\circ}({\boldsymbol \alpha})$ the 
solution of the equation
{(obviously, the solution exists and is unique as the equation is linear)}
\be
	p^i_t =  p^i_{0} + \int_0^t \sum_{j\in \ES} \left(p^{j}_s\alpha^{j,i}_s 
	- p^{i}_s \alpha^{i,j}_s \right)ds, \quad t \in [0,T], \quad i \in \ES. 
	\label{61}
	\ee

\emph{Step 1.} 	The distributions of the random variables 
$({\boldsymbol \alpha}^{\theta,\delta,\varepsilon})_{\theta,\delta,\varepsilon}$ (regarded as taking values within 
${\mathcal E}$) is tight as $\mathcal{E}$ is compact. 
	For any ${\boldsymbol \alpha}$ in ${\mathcal E}$, the corresponding solution ${\boldsymbol p}^\circ:=
	{\boldsymbol p}^\circ({\boldsymbol \alpha})$ to \eqref{61} is such that 
	$p^{\circ,i}_t\geq p^i_0 -M({d-1})\int_0^t {p}^{\circ,i}_s ds$ and thus, by Gronwall's lemma,
	${p}^{\circ,i}_t \geq p_{0}^{i} e^{-t M({d-1})}$. This prompts us to define
	\[
	\delta_0 := \tfrac 14 e^{-TM({d-1})} \min_{i \in \ES}  p^i_0.
	\]
	and let
	\[
	\tau_{\varepsilon,\theta,\delta}:= \inf \biggl\{ 0\leq t\leq T :
	\min_{i \in \ES} p^{i,\theta,\delta,\varepsilon}_t \leq 3\delta_0	
	\biggr\} \wedge T,
	\]
	with the convention $\inf \varnothing=+\infty$,
	and 
	$\hat{\boldsymbol p}^{\theta,\delta,\varepsilon}
:= \big(p_{\tau_{\theta,\delta,\varepsilon}\wedge t}^{\theta,\delta,\varepsilon} \big)_{0\leq t\leq T}$.

Notice in particular that, for $\delta < \delta_{0}$,  	
	$\hat{\boldsymbol p}^{\theta,\delta,\varepsilon}$ does not see the function $\varphi$ in its drift since the support of latter is restricted to $[0,2\delta]$. By Kolomogorov's criterion, since $\boldsymbol{\alpha}$ is bounded by $M$ and $\boldsymbol{p}$ is bounded by 1, it is then standard to show the tightness of the distributions of the processes $(\hat{\boldsymbol p}^{\theta,\delta,\varepsilon})_{0<{2}\theta<\delta\leq \delta_0,0<\varepsilon\leq {\varepsilon_{0}}}$ in $\mathcal{C}([0,T];\mathcal{S}_{d})$. 
\vskip 4pt

	\emph{Step 2.} We hence consider a weakly convergent subsequence $\big(\hat{\boldsymbol p}^{\theta_n,\delta_n,\varepsilon_n}, 
	{\boldsymbol \alpha}^{\theta_n,\delta_n,\varepsilon_n}\big)_{n\geq 0}$, with some $({\boldsymbol p},{\boldsymbol \alpha})$ as weak limit, where 
	$\lim_{n \rightarrow \infty}(\theta_n,\delta_n,\varepsilon_n) =0$.
	To simplify the notations, we let $(\hat{\boldsymbol p}^{(n)},{\boldsymbol \alpha}^{(n)}):=(\hat{\boldsymbol p}^{\theta_{n},\delta_{n},\varepsilon_{n}},
	{\boldsymbol \alpha}^{\theta_n,\delta_n,\varepsilon_n})$. 
	Applying Skorokhod's representation Theorem, we can assume 
	without any loss generality that the convergence holds almost surely, 
	provided that we allow the Brownian motions $({\boldsymbol B}^{i,j})_{i,j \in \ES : i \not =j}$ to depend on $n$.
	We hence write the latter in the form ${\boldsymbol B}^{(n)}=({\boldsymbol B}^{(n),i,j})_{i,j \in \ES : i \not =j}$. So, we can assume that there exists a full event $\Omega_{0}$ on which 
	$\sup_{0\leq t\leq T} |\hat p^{(n)}_t- p_t| \rightarrow 0$ and ${\boldsymbol \alpha}^{(n)} \rightharpoonup {\boldsymbol \alpha}$. 
	
	We then write ({pay attention that, although we don't mention it explicitly, the last four terms in the right-hand side below depend on 
	$i$})
	\[
	\hat p^{(n),i}_t := p_{0}^{i} + R^{(n),1}_{t} + R^{(n),2}_{t} +R^{(n),3}_{t} +R^{(n),4}_{t}, \quad t \in [0,T],
	\] 
	where (sums being over indices in $\ES$)
	\begin{align*}
	R^{(n),1}_t &= \int_0^t \sum\nolimits_j \left(\hat p^{(n),j}_s\alpha^{(n),j,i}_s 
	- \hat p_{s}^{(n),i} \alpha_{s}^{(n),i,j} \right)ds
	-\int_0^t \sum\nolimits_j \left(p^{j}_s\alpha^{(n),j,i}_s 
	- p^{i}_s \alpha^{(n),i,j}_s \right)ds
	,\\
	R^{(n),2}_t &= \int_0^t \sum\nolimits_j \left(p^{j}_s\alpha^{(n),j,i}_s - p^{i}_s\alpha^{(n),i,j}_s\right)ds
	,\\
	R^{(n),3}_t&= \int_0^t \Big[(1-\hat p^{(n),i}_s)\varphi_{\theta_n,\delta_n,\epsilon_n}\bigl(\hat p^{(n),i}_s\bigr) 
	-\hat p^{(n),i}_s\sum\nolimits_{j\neq i} \varphi_{\theta_n,\delta_n,\epsilon_n} \bigl(\hat p^{(n),j}_s\bigr)\Big]ds
	,\\
	R^{(n),4}_t &= \tfrac{1}{\sqrt{2}} \varepsilon_{n} \int_0^t \sum\nolimits_{j\neq i} \sqrt{\hat p^{(n),i}_s \hat p^{(n),j}_s} d\bigl(B_s^{(n),i,j}-B^{(n),j,i}_s\bigr).
	\end{align*}
	We then work on $\Omega_{0}$ in order to handle the almost sure convergence of the first three terms. By uniform convergence of $\hat {\boldsymbol p}^{(n)}$ to ${\boldsymbol p}$ and by uniform boundedness of ${\boldsymbol \alpha}^{(n)}$, $(R^{(n),1}_{t})_{0 \le t \le T}$ tends to $0$, uniformly in $t \in [0,T]$. 
	By weak convergence 
	of
	${\boldsymbol \alpha}^{(n)}$ to ${\boldsymbol \alpha}$, we deduce that, for any $t \in [0,T]$, 
	 $R_t^{(n),2}$ tends to 
	$\int_0^t \sum_j \left(p^{j}_s\alpha^{j,i}_s 
	- p^{i}_s \alpha^{i,j}_s \right)ds$; by Arzel\`a-Ascoli Theorem, the convergence is uniform in $t \in [0,T]$. 
{Since it is implicitly required that $\delta_{n} < \delta_{0}$
and hence
$\varphi_{\theta_n,\delta_n,\varepsilon_{n}}(\hat p^{(n),j}_t)$ is 0 for all $j\in\ES$ and $t \in [0,T]$
	(recall that 
	$\hat {\boldsymbol p}^{(n),j}$ is stopped before entering the support of $\varphi$), 
	 the term 
	$(R_{t}^{(n),3})_{0 \le t \le T}$ is constantly 0}.
	We hence derive the almost sure limit of the first three terms (the initial condition being excluded) in the expansion of 
	$\hat {\boldsymbol p}^{(n)}$. 
	As for $(R^{(n),4}_{t})_{0 \le t \le T}$, we observe by Doob's inequality that, since the second moment of the stochastic integral
	is uniformly bounded with respect to $n$, $\sup_{t \in [0,T]} \vert R^{(n),4}_{t} \vert$ tends to $0$  in probability.  

	Thus we can conclude that, {with probability 1}, 
	the limit process $({\boldsymbol p},{\boldsymbol \alpha})$ solves equation \eqref{61}. 
	\vskip 4pt

	\emph{Step 3}. We keep the same notation as in the second step ({working in particular with the same Skorokhod representation sequence}). 
	Since 
\[
\PP\Bigl(
\inf_{0\le t\le T} \min_{i \in \ES} \hat p_{t}^{i,\theta_n,\delta_n,\varepsilon_n} \leq 3\delta_0
\Bigr)
\geq \PP \left(  \tau_{\theta_n,\delta_n,\varepsilon_n} <T \right)
\]	
		and the limit process satisfies  
	\[\PP\Bigl(
\inf_{0\le t \le T} \min_{i \in \ES} p^i_t \leq 3\delta_0
\Bigr)=0,
	\]
	 Portmanteau Theorem gives
	\be
	\label{62}
	\lim_{n \rightarrow \infty}  \PP \left(  \tau_{\theta_n,\delta_n,\varepsilon_n} <T \right) =0.
	\ee
	Now, 
	\begin{align*}
	\E\Bigl[ \sup_{t \in [0,T]} \bigl\vert {p}_{t}^{\theta_n,\delta_n,\varepsilon_n}
	-p_{t} \bigr\vert 
	\Bigr] 
	& \leq
	\E\Bigl[ \sup_{t \in [0,T]} \bigl\vert {p}_{t}^{\theta_n,\delta_n,\varepsilon_n}
	-p_{t} \bigr\vert 
	\mathbbm{1}_{\left\{\tau_{\theta_n,\delta_n,\varepsilon_n}<T\right\}}
	\Bigr] 
	+
	\E\Bigl[ \sup_{t \in [0,T]} \bigl\vert \hat{p}_{t}^{\theta_n,\delta_n,\varepsilon_n}
	-p_{t} \bigr\vert 
	\Bigr] 
	\\
	&\qquad\leq2  \PP \left(  \tau_{\theta_n,\delta_n,\varepsilon_n} <T \right) 
	+\E\Bigl[ \sup_{t \in [0,T]} \bigl\vert \hat{p}_{t}^{\theta_n,\delta_n,\varepsilon_n}
	-p_{t} \bigr\vert 
	\Bigr] .
	\end{align*}
	The first term in the right-hand side tends to $0$ by \eqref{62} and the second term by the convergence result proved in the second step ({the almost sure convergence also holding true in $L^1$ since the underlying processes take values in the simplex}). Therefore we obtain $\lim_{n\rightarrow \infty} {\mathbb E}[\sup_{t \in [0,T]} \vert {p}_{t}^{\theta_n,\delta_n,\varepsilon_n} -{p}_{t}\vert]=0$. This in particular implies that the collection of the laws of the random variables $({\boldsymbol p}^{\theta,\delta,\varepsilon},{\boldsymbol \alpha}^{\theta,\delta,\varepsilon})_{0<{2}\theta<\delta\leq \delta_0,0<\varepsilon\leq {\varepsilon_{0}}}$ is tight (on the same space as before).
	\vskip 4pt

	\emph{Step 4.} 
	We pass to the limit in the cost. To this end, we use the convenient notations ${\boldsymbol p}^{(n)}:= p^{\theta_n,\delta_n,\varepsilon_n}$ and ${\mathcal J}^{(n)}(\cdot):=\mathcal{J}^{\varepsilon_{n},\varphi_{\theta_{n},\delta_{n},\varepsilon_{n}}}(\cdot)$, see 
	\eqref{eq:Cost:mfc}. By convexity (splitting $\alpha^{i,j}_{t}$ into $\alpha^{(n),i,j}_{t}+(\alpha^{i,j}_{t} - \alpha^{(n),i,j}_{t})$), we have 
	\begin{align}
	&\mathcal{J}^{(n)}({\boldsymbol \alpha}^{(n)})-\E\bigl[\mathcal{J}\bigl({\boldsymbol \alpha}\bigr)\bigr] \nonumber
	\\
	&= 
	\E \biggl[\int_0^T \biggl(\tfrac12 \sum\nolimits_{i \not = j} p^{(n),i}_t  |\alpha^{(n),i,j}_t|^2 + F(p^{(n)}_t) -
	\tfrac12  \sum\nolimits_{i \not = j} p^i_t |\alpha^{i,j}_t|^2  - F(p_t)\biggr)dt \nonumber
	\\
	&\hspace{15pt}+ G(p^{(n)}_T)-G(p_T)\biggr] \label{eq:thm:5:1:1}
	\\
	&\geq \E \biggl[\int_0^T\tfrac12  \sum\nolimits_{i} (p^{(n),i}_t -p^i_t)\sum\nolimits_{j \not =i}|\alpha^{(n),i,j}_t|^2 dt \biggr]
+ \E\biggl[\int_0^T\sum\nolimits_i p^i_t \sum\nolimits_{j \not = i} \alpha^{i,j}_t(\alpha^{(n),i,j}_t - \alpha^{i,j})\biggr]
\nonumber
\\
	&\hspace{15pt} +\E \biggl[\int_0^T \bigl(F(p^{(n)}_t)-F(p_t)\bigr)dt + G(p^{(n)}_T)-G(p_T)\biggr].\nonumber
	\end{align}
	Since 
	$\lim_{n\rightarrow \infty} {\mathbb E}[\sup_{t \in [0,T]} \vert {p}_{t}^{(n)} -{p}_{t}\vert]=0$,
	the first and third term in the lower bound go to 0  (using the boundedness of ${\boldsymbol \alpha}^{(n)}$
	and the regularity of $F$ and $G$). As for the second term, it can be proved to tend to 0 by combining the convergence ${\boldsymbol \alpha}^{(n)} \rightharpoonup {\boldsymbol \alpha}$ with Lebesgue dominated convergence theorem. 
	Thus we obtain $\E[\mathcal{J}({\boldsymbol \alpha})] \leq \liminf_{n \rightarrow \infty} \mathcal{J}^{(n)}({\boldsymbol \alpha}^{(n)})$.

	In order to complete the proof, consider any deterministic control ${\boldsymbol \beta} \in\mathcal{E}$ (in particular, the off-diagonal components are bounded by $M$). Then, 
	denote by {${\boldsymbol p}^{(n),{\boldsymbol \beta}}$} the {simplex-valued} solution to $\eqref{dynpot}$	
	 with $\varepsilon=\varepsilon_n$ therein
	under the {same initial condition $(0,p_{0})$ as before but under the} deterministic control  ${\boldsymbol \beta}$
	 (see \cite[Proposition 2.1]{mfggenetic} for a solvability result). 
	 {Differently from the analysis 
	 performed for ${\boldsymbol p}^{(n)}$, the choice of the noise does not really matter here, meaning that we can work with the original Brownian motions $({\boldsymbol B}^{i,j})_{i \not =j}$. Indeed, by the same localization argument as in Steps 2 and 3, it can be proved by a standard stability argument (without any further need of weak compactness) that ${\mathbb E}[\sup_{t \in [0,T]} \vert p^{(n),{\boldsymbol \beta}}_{t} - p^\circ_{t}({\boldsymbol \beta})\vert ]$ tends to $0$ 
	as $n$ tends to $\infty$}. We deduce $\lim_{n \rightarrow \infty} \mathcal{J}^{(n)}({\boldsymbol \beta}) = \mathcal{J}({\boldsymbol \beta})$. Therefore we obtain
	\be
	\label{eq:thm:5:1:2}
	\E [\mathcal{J}({\boldsymbol \alpha})] \leq \liminf_{n\rightarrow \infty} \mathcal{J}^{(n)}\bigl({\boldsymbol \alpha}^{(n)}\bigr) 
	\leq {\limsup_{n\rightarrow \infty} \mathcal{J}^{(n)}\bigl({\boldsymbol \alpha}^{(n)}\bigr)} 
	\leq \lim_{n \rightarrow \infty} \mathcal{J}^{(n)}({\boldsymbol \beta}) = \mathcal{J}({\boldsymbol \beta})
	\ee
	for any ${\boldsymbol \beta} \in\mathcal{E}$. 
	Provided that all the minimizers of ${\mathcal J}$ belong to ${\mathcal E}$, this implies that ${\boldsymbol \alpha}$ belongs with probability 1 to the set of minimizers of $\mathcal{J}$ over $\mathcal{E}$ and further that \eqref{limcosts} holds (in particular the limit exists).
The fact that optimizers of ${\mathcal J}$ --over $L^\infty$ controls-- belong to $\mathcal{E}$ is proved in the next Proposition \ref{prop:solution:mfc:zero:epsilon}, together with other properties of the inviscid MFCP.	
	\vskip 4pt
	
{\emph{Step 5.} As for the proof of   \eqref{limcosts:bis}, 
we can assume without any loss of generality that $t_{0}=0$. 
Observing that the family $({\mathcal V}_{\theta,\delta,\varepsilon}(0,p_{0}))_{\theta,\delta,\varepsilon}$
is bounded (since $F$ and $G$ themselves are bounded and the control process in 
\eqref{eq:Cost:mfc}--\eqref{dynpot}
is bounded by $M$), 
  \eqref{limcosts:bis}
 follows from 
  \eqref{limcosts}
 together with a standard compactness argument.}	
\end{proof}

\subsection{Properties of the inviscid MFCP}
{Before we turn to the proof of Theorem \ref{main:thm:6}}, 
we address various properties of the value function of the inviscid MFCP. In this respect,
it is useful to work with the system of local 
coordinates $(x_{1},\cdots,x_{d-1})$ introduced in Subsection \ref{subse:classical:solutions}. 
The dynamics over which the MFCP 
\eqref{eq:cost:functional:mfc}
is defined then have the form (sums are here over $\llbracket d-1 \rrbracket$)
\be 
\label{dynch}
\dot{x}^i_t= \sum\nolimits_{j \not =i} \left(x_{t}^j \alpha_{t}^{j,i} - x_t^i\alpha_t^{i,j}\right) 
+x_{t}^{-d} \alpha_t^{d,i} 
-x^i_t \alpha_t^{i,d},
\ee
for $i \in \ESd$, with the useful notation that 
$x_{t}^{-d}=1-\sum\nolimits_{{l \in \ESd}} x^l_{t}$. Above, the {(deterministic)} control ${\boldsymbol \alpha}=((\alpha^{i,j}_{t})_{i,j \in \ES})_{0 \le t \le T}$ is 
as in \eqref{eq:fp:control}{;  as we already explained, we assume\footnote{\label{foo:linfty} 
The unbounded case looks more difficult. One issue is that the Lagrangian $\frac12 \sum_i p_i \sum_{j\neq i}|\alpha^{i,j}|^2$ is not Lipschitz continuous in $p$, uniformly in $\alpha$, if $\alpha$ is not in a compact set. 
Another issue is that the Lagrangian 
is not uniformly coercive on the simplex: 
As a result, we can easily cook up instances of unbounded controls that drive the trajectory to the boundary but that remains of a bounded energy; and, in turn, those trajectories precisely fall within the region where the Lagrangian is degenerate, which makes their analysis more difficult.}
 it to be bounded (but not uniformly bounded by $M$)}. 
Also, the initial condition is taken in the interior of $\widehat{\mathcal S}_{d}$, which implies in particular that the whole 
path ${\boldsymbol x}$ remains in the interior of the simplex. Lastly, 
following 
\eqref{eq:hamiltonians} and \eqref{eq:theta} (but paying attention that $M$ is formally taken as $+\infty$), 
the Hamiltonian of the problem is given, for $z\in\mathbb{R}^{d-1}$, by (sums are here over $\ESd$)
\be 
\label{eq:mathcal H:5:9}
\begin{split}
	&\widehat{\mathcal{H}}(x,z)= 
	\sum\nolimits_{i} x_{i} \hat{H}^i(z) + 
	x^{-d} \hat{H}^d(z), 
	\\
	&\text{with} \quad \hat{H}^i(z) = -\tfrac12 \Bigl(\sum\nolimits_{j \not =i}
	(z_{i} - z_{j})_{+}^2+(z_{i})_{+}^2) \Bigr), \quad \hat{H}^d(z) = 
	-\tfrac12 \sum\nolimits_{j}
	(-z_{j})_{+}^2.
\end{split}
\end{equation}
It is important to observe that this Hamiltonian is strictly concave in $z$, for any $x\in \Int$.
Indeed, $\widehat{\mathcal{H}}$ is the sum of a concave function
and of $-\frac12 \min(\min_{i \in \ESd}(x_{i}),x^{-d}) \sum_{i} z_{i}^2$. 
Moreover, we may also write down 
the corresponding Pontryagin principle: 
\begin{equation}
\label{eq:fb:inviscid:local}
\left\{
\begin{array}{l}
	\dot{x}_{t}^i = \sum_{j \not = i} \bigl( x_{t}^j  (z_{t}^j-z_{t}^i)_{+} - x_{t}^i (z_{t}^i -z_{t}^j)_{+} \bigr)
	+ x_{t}^{-d}(-z_{t}^i)_{+} - x_{t}^i (z_{t}^i)_{+},
	\\
	\dot{z}_{t}^i = - \Bigl( \widehat H^i ( z_{t}) - \widehat H^d (z_{t}) + \hat f^i(x_{t}) - \hat f^d(x_{t}) \Bigr), \quad z_{T}^i = \hat g^i(x_T) - \hat g^d(x_{T}),
\end{array}
\right.
\end{equation}
for $i \in \ESd$ and for a given initial condition in $[0,T] \times \Int$. It is worth noticing that the Pontryagin principle is here stated in local coordinates, or equivalently in dimension $d-1$. For sure, we can also state it in dimension $d$, in which case the forward-backward system coincides with the standard MFG system {(the sum below is over $\ES$)}
\begin{equation}
\label{eq:fb:inviscid:intrinsic}
\left\{
\begin{array}{l}
	\dot{p}_{t}^i = \sum_{j \not = i} \bigl( p_{t}^j  (u_{t}^j-u_{t}^i)_{+} - p_{t}^i (u_{t}^i -u_{t}^j)_{+} \bigr),
	\\
	\dot{u}_{t}^i = - \Bigl( H \bigl( (u_{t}^i - u_{t}^j)_{j \in \ES} \bigr) +  f^i(p_{t})  \Bigr), \quad u_{T}^i = g^i(p_T),
\end{array}
\right.
\end{equation}
with $H$ as in \eqref{eq:hamiltonian:main}. It is pretty easy to see that the two systems 
\eqref{eq:fb:inviscid:local}
and
\eqref{eq:fb:inviscid:intrinsic} are equivalent: Given a solution ${\boldsymbol u}$ to \eqref{eq:fb:inviscid:intrinsic}, it suffices to 
let ${\boldsymbol z}=((z_{t}^i:=(u_{t}^i - u_{t}^d))_{i \in \ESd})_{t_{0} \le t \le T}$, where $t_{0}$ is the initial time. 
Conversely, given ${\boldsymbol z}$ a solution to \eqref{eq:fb:inviscid:intrinsic}, it suffices to solve 
\eqref{eq:fb:inviscid:intrinsic} where all the occurrences of $u_{t}^i - u_{t}^j$ have been replaced by $z_{t}^i - z_{t}^j$ if $j \in \ESd$ and by $z_{t}^i$ if $j=d$.

The fact that $\widehat{\mathcal H}$ is strictly concave permits to apply to our situation
several results from \cite[Chapter 7, Section 4]{cannarsa}, which 
we collect in the form of a single proposition, although part of the notions 
are introduced in detail or explicitly used in Section 
\ref{sec:uniqueness:master:equation}. 
It is worth mentioning that the results of \cite{cannarsa} are stated for a dynamics in $\R^d$, but is it straightforward to see that they apply also to our situation because, when working in local coordinates, any trajectory remains in $\Int$, if starting {from} $\Int$.

\begin{prop}
	\label{prop:solution:mfc:zero:epsilon}
	Assume that $F$ and $G$ are in $\mathcal{C}^1(\mathcal{S}_d)$. {Recall that} $\mathcal{V}:[0,T]\times {\mathcal S}_d \rightarrow \R$  is the value function of the MFCP 
	\eqref{eq:cost:functional:mfc}, and {call} $\widehat{\mathcal{V}}:[0,T]\times \widehat{\mathcal S}_d \rightarrow \R$ its formulation in local coordinates, i.e. 
	$\widehat{\mathcal{V}}(t,x)= {\mathcal{V}}(t,\check{x})$ .
If the initial condition $p_0\in \mathrm{Int}(\mathcal{S}_d)$, then	
	\begin{enumerate}
			\item[(i)] An optimal ({bounded}) control exists  and is bounded by $M= 2(\|g \|_{\infty} + T \| f \|_{\infty})$;
		\item[(ii)] If $\boldsymbol{\alpha}$ is an optimal control and  $\boldsymbol{p}$ the related optimal trajectory, then there exist $\boldsymbol{u}$ solving \eqref{eq:fb:inviscid:intrinsic} and $\boldsymbol{z}$ solving \eqref{eq:fb:inviscid:local}, and {${\boldsymbol \alpha}$ is given by 
		$(\alpha^{i,j}_t= (u^i_t-u_t^j)_+)_{t_{0} \le t \le T}$};
		\item[(iii)] $\widehat{\mathcal{V}}$ is a viscosity solution of 
		\eqref{eq:hjb:inviscid}
		on {$[0,T] \times \Int$}, 
		{at least when 
		\eqref{eq:hjb:inviscid}
	is formulated in local coordinates, see
		Definition 
		\ref{defvisco}};
		\item[(iv)] $\widehat{\mathcal{V}}$ is {(time-space)} Lipschitz-continuous on $[0,T]\times\Int$ and thus also on $[0,T]\times \widehat{\mathcal S}_d$.
	\end{enumerate}
	If $F$ and $G$ are in $\mathcal{C}^{1,1}(\mathcal{S}_d)$, then 
	\begin{enumerate}
		\item[(v)]  $\widehat{\mathcal{V}}$ is semiconcave on $[0,T]\times\Int$ and thus also in $[0,T]\times \widehat{\mathcal S}_d$;
		\item[(vi)] $\mathcal{V}$  is differentiable at $(t,p_t)$, for any $t>t_0$ and any optimal trajectory ${\boldsymbol p}$ starting from $(t_0,p_0)$, with $p_0\in \mathrm{Int}(\mathcal{S}_{d})$;
		\item[(vii)] If the optimal control for the MFCP starting in $(t_0,p_0)$ is unique (in particular this holds true when \eqref{eq:fb:inviscid:intrinsic} is uniquely solvable), then $\mathcal{V}$ is differentiable in $(t_0,p_0)$.
			\end{enumerate}
	If we assume in addition that  $\mathcal{V}$ is differentiable at $(t_0,p_0)$, with $p_0\in \mathrm{Int}(\mathcal{S}_{d})$, then
	\begin{enumerate}
		\item[(viii)] There exists a unique optimal control process ${\boldsymbol \alpha}$ and an optimal trajectory ${\boldsymbol p}$ for the MFCP starting from $(t_0,p_0)$, and the optimal control is given in feedback form by
		\[
		\alpha^{i,j}_t =  \left(\mathfrak{d}_i \mathcal{V}(t,p_t)-\mathfrak{d}_j\mathcal{V}(t,p_t)\right)_+,
		\qquad t\in [t_0,T];
		\]
		\item[(ix)] The adjoint equation in \eqref{eq:fb:inviscid:local} is such that $z^i_t =\partial_{x_i}\widehat{\mathcal{V}}(t,x_t)$ for any $t\in [t_0,T]$ and $i \in \ESd$. 
		
	\end{enumerate}
\end{prop}

In point (iii), 
we refer to Definition \ref{defvisco} for a reminder on the notion of viscosity solution; 
In points (v)--(ix),
we refer to the 
 assumption of 
	Theorem 	\ref{main:thm:6} for the meaning of the notation
	$\mathcal{C}^{1,1}$;
In point (v), by time-space semiconcavity, we mean that 
there exists a constant $c$ such that, for any $t \in [0,T]$, $x\in\Int$, $s$ with $t\pm s \in [0,T]$ and $\xi$ with $x\pm\xi\in \Int$,
\be
\label{semitimespace} 
\frac{v(t+s,x+\xi) - 2v(t,x)+ v(t-s
	,x-\xi)}{|s|^2 + |\xi |^2}\leq c.
\ee
Notice that we also exploit the notion of semiconcavity, but in space only, in the next section, see \eqref{semicon}.
Also, 
not only in the statement but also 
 throughout the rest of the text, 
differentiability of ${\mathcal V}$ is understood 
as time-space differentiability (unless it is  stated differently, in which case 
differentiability is explicitly referred to as space differentiability).
Last, we stress that (vii) follows from Theorem 7.4.20 in \cite{cannarsa}, but the statement therein assumes that the Hamiltonian is strictly convex; in fact, it is clear from the proof that the authors mean strictly convex in $z$ only.
\begin{proof}
	To prove (i), assume first that controls are bounded by $R$, for some $R>M$.
	 Then an optimal  control ${\boldsymbol \alpha}_R$ exists by \cite[{Theorem} 7.4.5]{cannarsa} and, by the Pontryagin principle {\cite[Theorem 7.4.17]{cannarsa}}, point (ii) holds but with the truncated Hamiltonian $(H^i_R)_{i\in\ES}$ defined as in \eqref{eq:hamiltonians}.   
	Thus, 
	{${\boldsymbol \alpha}_R$} induces  an equilibrium
	${{\boldsymbol p}^\circ({{\boldsymbol \alpha}_{R}})}$
	 to the MFG 
	\eqref{eq:cost:functional}--\eqref{eq:fp} and ({using the coercivity of the Lagrangian on the interior of the simplex})
	is of the form $((\alpha^{i,j}_{R,t}=(u_{t}^i-u_{t}^j)_{+})_{i,j \in \ES : i \not = j})_{0 \le t \le T}$, 
	where ${\boldsymbol u}=((u_{t}^i)_{i \in \ES})_{0 \le t \le T}$ is
	the value process {associated with} the optimization problem $J(\, \cdot\, ;{{\boldsymbol p}^\circ({{\boldsymbol \alpha}_{R}}}))$ in 
	\eqref{eq:cost:functional}, set over controls {that are} bounded by $R$.  Choosing $0$ as control in \eqref{eq:cost:functional}, we observe that  ${\boldsymbol u}$ is upper bounded by $M/2$. In order to prove that $-M/2$ is a lower bound, it suffices to lower bound the quadratic cost by zero in the cost functional $J(\, \cdot\, ;{{\boldsymbol p}^\circ({\boldsymbol \alpha}_{R}}))$. Hence, {${\boldsymbol \alpha}_R$} is bounded by $M$, which is independent of $R$, implying that an optimal control exists over the set of {bounded} controls. Therefore, (i) and (ii) are proved and the other points follow now from the results in \cite[Section 7.4]{cannarsa}. 
\end{proof}

Since $\mathcal{V}$ is  almost everywhere differentiable in  $[0,T]\times {\mathcal S}_d$, the above result, together with Theorem \ref{thm10}, implies that the sequence of optimal trajectories 
{$({\boldsymbol p}^{\theta,\delta,\varepsilon})_{0 < 2 \theta \leq \delta \leq \delta_{0},0<\varepsilon \leq \varepsilon_{0}}$}
admits a true limit for almost every initial condition $(t_0,p_0)$ ({the convergence hence holding true in probability}).  Moreover, point (vi) above permits to say more about the convergence also when starting from a point of non-differentiability: The randomness of the limit {trajectory} is enclosed in the initial time only. We summarize in the following:

\begin{cor}
	\label{cor:selection}
{Assume that $F$ is in $\mathcal{C}^{1,1}({\mathcal S}_{d})$ and $G$ in 
${\mathcal C}_{\textrm{\rm WF}}^{1,2+\gamma}({\mathcal S}_{d})$ for a given $\gamma \in (0,1)$.}	
	Then, if $\mathcal{V}$ is differentiable in $(t_0,p_0)$, with $p_0\in \mathrm{Int}(\mathcal{S}_{d})$, then, the following 
	holds true in probability (the first one on ${\mathcal C}([0,T];{\mathcal S}_{d})$ and the second one on 
	${\mathcal E}$), 
	\be 
	\lim_{(\theta,\delta,\varepsilon)\rightarrow (0,0,0)} {\boldsymbol p}^{\theta,\delta,\varepsilon}_{[t_0,p_0]}
	= {\boldsymbol p}_{[t_0,p_0]}
	\qquad
	\lim_{(\theta,\delta,\varepsilon)\rightarrow (0,0,0)} {\boldsymbol \alpha}^{\theta,\delta,\varepsilon}_{[t_0,p_0]}
	= {\boldsymbol \alpha}_{[t_0,p_0]},
	\ee
	 where  ${\boldsymbol p}_{[t_0,p_0]}$ is the unique optimal trajectory and  ${\boldsymbol \alpha}_{[t_0,p_0]}$ the unique optimal control process of the limiting MFCP, see the notation in the introduction of Subsection 
	\ref{subse:selection}.
	
	Moreover, if $\mathcal{V}$ is not differentiable in $(t_0,p_0)$, the limit of any converging subsequence is supported on a set of (optimal) trajectories which do not {branch strictly} after the initial time. 
	
\end{cor}

\subsection{Selection for the master equation}
\label{subse:distance:MFCP}
Although Corollary \ref{cor:selection} provides an interesting information about the 
limiting behavior of the equilibrium
${\boldsymbol p}^{\theta,\delta,\varepsilon}$ as 
the parameters $(\theta,\delta,\varepsilon)$ tend to $0$, it says nothing about the 
asymptotic behavior of the related equilibrium cost. 
{We address this question in this subsection; in particular,
we prove here 
Proposition 
\ref{prop:cv:cost}
and Theorem \ref{main:thm:6}. Throughout, we assume 
that $F\in\mathcal{C}^{1,\gamma}_{\rm WF}({{\mathcal S}_{d}})$ and $G\in\mathcal{C}^{1,2+\gamma}_{\rm WF}({{\mathcal S}_{d}})$ for a given $\gamma\in(0,1)$; 
at some point, we need to strengthen the condition on
$F$ and assume it to belong to ${\mathcal C}^{1,1}({\mathcal S}_{d})$, see 
Proposition 
	\ref{thm:5:9} and 
	Theorem \ref{thm:5.10}}.

Actually, part of the difficulty for passing to the limit in the cost ${\widetilde J}^{\varepsilon,\varphi}$
 defined by \eqref{newcostmfg} 
is to control the distance from the equilibrium to the boundary.
Back to the formula 
\eqref{newphi}, it is indeed plain to see that 
$\varphi$ should become steeper and steeper 
(and hence 
$\vert \varphi' \vert$ larger and larger)
in the neighborhood of $0$
as $(\theta,\delta,\varepsilon)$ tends to $0$, whence the need
for some uniform integrability properties 
on the inverse of the distance from 
${\boldsymbol p}^{\theta,\delta,\varepsilon}$ 
to the boundary. 
We here collect several useful \textit{a priori} bounds 
in this direction. {Proofs of the first three  statements are postponed to the end of the section, see Subsection 
\ref{subse:auxiliary:proofs:integrability}.}

\begin{prop}
	\label{expfin}
	For $(\theta,\delta,\varepsilon)$ and
	$\varphi=\varphi_{\theta,\delta,\varepsilon}$ as in
	\eqref{newphi}, with 
	$\kappa_0 \geq \epsilon^2/2$ and {$\varepsilon_{0}:=\kappa_{2}/\kappa_{0} \geq \varepsilon^2$}, and for any initial condition $(t_{0},p_{0})\in [0,T]\in \mathrm{Int}(\mathcal{S}_d)$ and any $\lambda >0$ and $i \in \ES$, 
	\begin{align}
		\label{eq:expfin1}
		&{\mathbb E}\biggl[ \exp \biggl\{\frac{\lambda}{\varepsilon^2} \Bigl(\kappa_0-\frac{\varepsilon^2+\lambda}{2}  \Bigr) \int_{t_{0}}^T  \frac{1}{p^{i,\theta,\delta,\varepsilon}_t}
		\mathbbm{1}_{[0,\delta]}(p^{i,\theta,\delta,\varepsilon}_t) dt\biggr\}\biggr] \leq 
		\frac{e^{TC(\delta,\varepsilon,\lambda)} }{(p^i_0)^{\lambda/\varepsilon^2} },
		\\
		&{\mathbb E}\biggl[ \exp \biggl\{\lambda
		\Bigl(\kappa_0-\frac{\varepsilon^2(1+\lambda)}{2}  \Bigr)
		\int_{t_{0}}^T  \frac{1}{p^{i,\theta,\delta,\varepsilon}_t}\mathbbm{1}_{[0,\delta]}
		(p^{i,\theta,\delta,\varepsilon}_t) dt\biggr\}
		{\mathbbm 1}_{\{ \inf_{t_0 \le t \le T} p_{t}^{i,\theta,\delta,\varepsilon}
			> 2 \theta \}}
		\biggr] \leq 
		\frac{e^{TC(\delta,\lambda)} }{(p^i_0)^{\lambda} },
		\label{eq:expfin2}
	\end{align}
	with $C(\delta,\varepsilon,\lambda):= \varepsilon^{-2}[\lambda(1+\lambda)/(2\delta)+ \lambda d (\kappa_\varepsilon +{\kappa_{0}}+M)]$
	and $C(\delta,\lambda):=  \lambda(1+\lambda)/(2\delta)+\lambda d (\kappa_0 +M)$, 
	and where ${\boldsymbol p}^{\theta,\delta,\varepsilon}$ is here understood as ${\boldsymbol p}^{\theta,\delta,\varepsilon}_{[t_{0},p_{0}]}$. 
\end{prop}


\begin{prop}
	\label{propexp}
	For any
	 $\lambda\geq 1$, there exists a
	constant $\bar\kappa_{0}$ (depending on $\lambda$ and $\kappa_{2}$)
	such that, for any 
	{$\kappa_{0} \geq \bar \kappa_{0}$}
	and any
	compact subset\footnote{Here, ${\mathcal K}$ is regarded as a compact subset of ${\mathcal S}_{d}$, but, obviously, we could regard it as a $(d-1)$-dimensional compact subset of $\widehat{\mathcal S}_{d}$.} ${\mathcal K} \subset  \mathrm{Int}({\mathcal S}_{d})$, 
	we can find (strictly) positive constants
	$\bar C$, $\bar \delta_{0}$, $\bar \varepsilon_{0}$ and 
	(strictly) positive-valued functions $\hat{\theta}(\delta,\varepsilon)$
	$\hat{\varepsilon}(\delta)$ 
	and 
	$\hat{\delta}(\varepsilon)$ 
	converging to $0$ in $(0,0)$, $0$ and $0$ respectively (all these items  
	 {only depending on $\kappa_{0}$, $\kappa_{2}$, ${\mathcal K}$, $\lambda$, $M$, $T$ and $d$}),
	such that  
	\begin{align}  
	&\label{bigexp}	
	 {
	\forall \delta \in (0,\bar \delta_{0}], 
	\ \forall \varepsilon \in (0,\hat{\varepsilon}(\delta)], \ \forall \theta \in (0,\hat{\theta}(\delta,\varepsilon)]}, \quad 
	\Psi\bigl(\lambda,\theta,\delta,\varepsilon,{\mathcal K}\bigr) \leq \bar C,
\\
	&\label{bigexp2}
 {	\forall \varepsilon \in (0,\bar \varepsilon_{0}], 
	\ \forall \theta \in (0,\hat{\theta}(\hat{\delta}(\varepsilon),\varepsilon)],} 
	\quad
	\Psi\bigl(\lambda,\theta,\hat \delta(\varepsilon),\varepsilon,{\mathcal K}\bigr) \leq \bar C,
	\end{align}
	where 
	\begin{align*}
		\Psi\bigl(\lambda,\theta,\delta,\varepsilon,{\mathcal K}\bigr)
		= \hspace{-2pt}
		\max_{i \in \ES}
		\sup_{(t_{0},p_{0}) \in [0,T] \times {\mathcal K}} 
		\hspace{-2pt}
		\E\bigg[
		\exp\biggl\{\lambda \int_{t_{0}}^T 
		\biggl(
		\Bigl[
		\varphi_{\theta,\delta,\varepsilon}
		-\varphi'_{\theta,\delta,\varepsilon} \Bigr](p_{[t_{0},p_{0}],t}^{i,\theta,\delta,\varepsilon})
		+
		\frac{1}{p_{[t_{0},p_{0}],t}^{i,\theta,\delta,\varepsilon}}
		\biggr)
		dt
		\biggr\}
		\biggr].
	\end{align*}
With the same notations, it also holds that 
\begin{equation}
\begin{split}
&	\forall \varepsilon \in (0,\bar \varepsilon_{0}], 
	\ \forall \theta \in (0,\hat{\theta}(\hat{\delta}(\varepsilon),\varepsilon)], 
\\
&\hspace{15pt}	\min_{i \in \ES}
	\inf_{(t_{0},p_{0}) \in [0,T] \times {\mathcal K}}
{\mathbb P} \Bigl( \inf_{t_{0} \leq t \leq T} 
p_{[t_{0},p_{0}],t}^{i,\theta,\hat{\delta}(\varepsilon),\varepsilon}
\geq \bar C \Bigr) \geq 1 - 2 \exp(-\varepsilon^{-1}),
\label{bigexp3}
\end{split}
\end{equation}
\end{prop}
In what follows, we prefer to state the convergence result as limits {as the viscosity parameter $\varepsilon^2$ tends} to $0$, instead of limits as $\delta$ {tends to $0$}, which explains why, in \eqref{bigexp2}
and \eqref{bigexp3}, we consider $\delta$ as a function of $\epsilon$, and $\theta$ as function of $\epsilon$ and $\delta$. {In order to formulate the next statement properly, we need 
another notation. Similar to ${\boldsymbol p}_{[t_{0},p_{0}]}^{\theta,\delta,\varepsilon}$, 
${\boldsymbol q}_{[t_{0},p_{0},q_{0}]}^{\theta,\delta,\varepsilon}:=(q_{[t_{0},p_{0},q_{0}],t}^{\theta,\delta,\varepsilon})_{t_{0} \le t \le T}$ denotes the solution to 
\eqref{dynmfg}
with $q_{0}$ as initial condition at time $t_{0}$, 
when $({\boldsymbol p},{\boldsymbol \alpha})$ therein is understood as $({\boldsymbol p}_{[t_{0},p_{0}]}^{\theta,\delta,\varepsilon},{\boldsymbol \alpha}_{[t_{0},p_{0}]}^{\theta,\delta,\varepsilon})$.
In particular, it should be clear for the reader that 
${\boldsymbol p}^{\theta,\delta,\varepsilon}_{[t_{0},p_{0}]}$
and
${\boldsymbol q}^{\theta,\delta,\varepsilon}_{[t_{0},p_{0},p_{0}],t}$
are the same. 
When there is no ambiguity on the choice of the initial condition, 
we merely write ${\boldsymbol q}$.}

\begin{lem}
	\label{lem:moments:p,q:l}
	{For  $\ell \geq 1$,
we can find $\lambda:=\bar \lambda(\ell)$,  {only depending on $\ell$ and $d$}
and then
	take
	$\bar \kappa_{0}$ 
	 {accordingly in Proposition \ref{propexp}
	(in terms of $\lambda$ and $\kappa_{2}$ only)}
	such that, 
	for any $\kappa_{0} \geq \bar \kappa_{0}$
and any compact subset 	${\mathcal K}$ included in $\mathrm{Int}({\mathcal S}_{d})$, 
it holds that, for $\bar \varepsilon_{0}$, $\hat{\theta}$ and $\hat{\delta}$ as in Proposition \ref{propexp}}, 
	for any state $i \in \ES$ and  {any initial point $(t_{0},p_{0},q_{0}) \in [0,T] \times {\mathcal K} \times {\mathcal S}_{d}$,
	and for
	any $\varepsilon\in (0,\bar{\varepsilon}_0]$
	and 
	$\theta \in (0, \hat{\theta}(\hat{\delta}(\varepsilon),\varepsilon)]$,}
	\begin{align}
		&\sup_{t_{0}\leq t\leq T} \E \Bigl[ \sum\nolimits_{i \in \ES} \big(p_{[t_{0},p_{0}],t}^{i,\theta,\hat{\delta}(\varepsilon),\varepsilon}\big)^{-\ell}\Bigr]
		\leq C,\label{momentp}
		\\
		& \E \Bigl[\sup_{t_{0} \leq t\leq T}  \sum\nolimits_{i \in \ES} \big(q_{[t_{0},p_{0},q_{0}],t}^{i,\theta,\hat{\delta}(\varepsilon),\varepsilon}\big)^\ell\Bigr] \leq C,
		\label{momentq}
	\end{align} 
where $C$ depends only on  {$\kappa_{0}$, $\kappa_{2}$, ${\mathcal K}$, $\ell$, $M$, $T$ and $d$}.
\end{lem}

\begin{prop}
	\label{thmbounds}
{We can find $\bar{\kappa}_{0} \geq 0$, only depending on $\kappa_{2}$ and $d$, such that, 
for any $\kappa_{0} \geq \bar \kappa_{0}$
and any compact subset 	${\mathcal K}$ included in $\mathrm{Int}({\mathcal S}_{d})$, there exist 
constants $\bar C$ and  {$\bar \varepsilon_{0}$}, only depending on 
$\kappa_{0}$, $\kappa_{2}$, ${\mathcal K}$, $\|f\|_\infty$, $\|g\|_\infty$, 
$T$ and $d$
and functions 
$\hat{\theta}(\delta,\varepsilon)$ and $\hat{\delta}(\varepsilon)$ 
as in the statement of Proposition \ref{propexp} (with $\lambda$ therein 
a fixed constant whose
value 
is made explicit in the proof in terms
of $d$ only and is, in particular, required to be greater than $\bar \lambda(12)$
in Lemma 
\ref{lem:moments:p,q:l})
such that}, for
$V_{\theta,\delta,\varepsilon}=(V_{\theta,\delta,\varepsilon}^i)_{i \in \ES}$ denoting the solution 
 {to \eqref{derhjbsimplex}}
with $\varphi=\varphi_{\theta,\delta,\varepsilon}$ therein, 
 {and for any 
$\varepsilon \in (0,\bar \varepsilon_0]$
and 
$\theta\in (0,\hat{\theta}(\hat{\delta}(\varepsilon),\varepsilon)]$}, 
\be
\label{boundVK}
\sup_{t\in [0,T]} \sup_{p\in {\mathcal K}} \max_{i \in \ES} \bigl|V^i_{\theta, \hat{\delta}(\varepsilon),\varepsilon}(t,p)\bigr| \leq \bar C.
\ee 
Moreover, if $(t_0,p_0)\in[0,T]\times {\mathcal K}$,
 {for the same values of  
$\kappa_0$, 
$\varepsilon$ and 
$\theta$
(indices in the sums below being taken in $\ES$), }
\begin{align}
\label{boundV0}
&\E\biggl[
\sup_{t_0\leq t\leq T}
\max_{i \in \ES}
\Big|V^i_{\theta, \hat{\delta}(\varepsilon),\varepsilon}(t,p_t^{\theta, \hat{\delta}(\varepsilon),\varepsilon} )\Big|^2 \biggr] 
\leq \bar C, 
\\
\label{boundzeta0}
&\E\bigg[\int_{t_0}^T \sum\nolimits_{i,j,k}\Bigl|W_{\theta,\hat{\delta}(\varepsilon),\varepsilon}^{i,j,k}\Bigl(t, p_t^{\theta, \hat{\delta}(\varepsilon),\varepsilon}\Bigr)\Bigr|^2 dt\biggr] \leq \bar C,
\\
& \E\biggl[\int_{t_0}^T  \sum\nolimits_{j,k}
\Bigl\vert
\Upsilon_{\theta,\hat{\delta}(\varepsilon),\varepsilon}^{j,k}\Bigl(t, p_t^{\theta, \hat{\delta}(\varepsilon),\varepsilon}\Bigr)
 \Bigr|^2 dt
\biggr]
\leq \bar C,
\label{boundfrak}
\end{align} 
where $W_{\theta,\hat{\delta}(\varepsilon),\varepsilon}$ is defined by 
\eqref{eq:w:i,j,k}
 and $\Upsilon_{\theta,\hat{\delta}(\varepsilon),\varepsilon}$ 
 by \eqref{eq:Upsilon:i,j}, with $\varphi=\varphi_{\theta,\hat{\delta}(\varepsilon),\varepsilon}$. 
 \end{prop}

\begin{proof}[Proof of Proposition \ref{thmbounds}]
{For a suitable $\lambda>1$ that will be fixed in 
	\eqref{ineq3}
	below in terms of $d$ only, we consider $\bar{\kappa}_{0}$ as in the statement of Proposition \ref{propexp}
	but with $\lambda$ therein replaced by $2\lambda d$ (the need for changing $\lambda$ into $2 \lambda d$ is made clear in the proof, see again the discussion right after \eqref{ineq3}; in short $\lambda$ in the statement should be understood as $2\lambda d$ in the sequel of the proof). 
	Then, for any $\kappa_{0} \geq \bar \kappa_{0}$ and for any compact subset ${\mathcal K}$ included in $\mathrm{Int}({\mathcal S}_{d})$, 
	we consider 
		$\bar \delta_{0}$, $\bar \varepsilon_{0}$, $\hat{\theta}(\delta,\varepsilon)$
	and 
	$\hat{\delta}(\varepsilon)$ also 
as in the statement of Proposition \ref{propexp}. 
We then fix
$\varepsilon \in (0, \bar \varepsilon_{0}]$, 
$\theta \in (0,\hat{\theta}(\hat{\delta}(\varepsilon),\varepsilon)]$}
and
	 $(t_0,p_0)\in[0,T]\times {\mathcal K}$ and we write 
	${\boldsymbol p}$ for the process ${\boldsymbol p}^{\theta, \hat{\delta}(\varepsilon),\varepsilon}=(p_{t}^{\theta, \hat{\delta}(\varepsilon),\varepsilon})_{t_{0} \le t \le T}$, 
	$\varphi$ for  {$\varphi_{\theta,\hat{\delta}(\varepsilon),\varepsilon}$}
	and $(V^i)_{i \in \ES}$ for 
	the corresponding solution to 
	\eqref{derhjbsimplex}, and similarly for 
	$(W^{i,j,k})_{i,j,k \in \ES}$
	and $(\Upsilon^{i,j})_{i,j \in \ES}$. 
	We then let 	
	 $(v_t^i:=V^i(t,p_t))_{t_{0} \le t \le T}$, for $i \in \ES$, and 
	 $(w_{t}^{i,j,k}:=W^{i,j,k}(t,p_t))_{t_{0} \le t \le T}$, for $i,j,k \in \ES$. 
 {We know that $(v_t,w_t)_{t_0\leq t\leq T}$ satisfy \eqref{ponsim}. 
We	  consider then}
	\[
	\mathcal{E}_t := \exp\biggl\{ \lambda
	\int_{t_0}^t \sum\nolimits_{j \in \ES}
\Bigl(
\bigl[\varphi -\varphi'\bigr](p^j_s) + \bigl( p^j_s \bigr)^{-1}
\Bigr)
    ds
	\biggr\}, \quad t_{0} \le t \le T.
	\]
It\^o's formula and \eqref{ponsim} give
(indices being taken in $\ES$)
\begin{align*}
d&\left( \mathcal{E}_t |v_t^i|^2 \right)
=2\mathcal{E}_t  v_t^idv_t^i + \lambda \mathcal{E}_t  |v_t^i|^2  
\sum\nolimits_j \Bigl(
\bigl[ \varphi 
-\varphi' \bigr](p^j_t)   + \bigl( p^j_t \bigr)^{-1}
\Bigr)dt
+\mathcal{E}_t \sum\nolimits_{j,k}|w_t^{i,j,k}|^2dt\\
&= -2 \mathcal{E}_t v_t^i \Bigl(  \widetilde H^i(v_{t}) -\tfrac1d \sum\nolimits_l \widetilde H^l(v_t)+ f^i(p_{t}) -\tfrac1d \sum\nolimits_l f^l(p_t) \Bigr)dt
\\
& - 2 \mathcal{E}_t v_t^i  \biggl( \sum\nolimits_j \Bigl( \varphi(p^{j}_t) - p_{t}^{j} \varphi'(p_{t}^{i}) \Bigr) (v^j_t-v^i_t)  - \tfrac1d  \sum\nolimits_{j,l} \Bigl( \varphi(p^{j}_t) - p_{t}^j \varphi'(p_{t}^{l}) \Bigr) (v^j_t-v^l_t) 
  \biggr)dt
\\
&-  {\sqrt{2}}  \varepsilon \mathcal{E}_t v_t^i \biggl( \sum\nolimits_{j} 
\sqrt{p_{t}^{j}(p_{t}^{i})^{-1}} 
\Bigl( w_{t}^{i,i,j} +w_t^{j,j,i}\Bigr) 
 - \tfrac1d \sum\nolimits_{j,l} \sqrt{p_{t}^{j} (p_{t}^{l})^{-1}}\Bigl(w_{t}^{l,l,j}  + w_t^{j,j,l} \Bigr)\biggr)dt
\\
&+ 2\mathcal{E}_t v_t^i\sum\nolimits_{j,k} w^{i,j,k}_t dB_{t}^{j,k}
+ \lambda \mathcal{E}_t  |v_t^i|^2  
\sum\nolimits_j \Bigl(
\bigl[ \varphi 
-\varphi' \bigr](p^j_t)   + \bigl( p^j_t\bigr)^{-1}
\Bigr)dt
+\mathcal{E}_t \sum\nolimits_{j,k}|w_t^{i,j,k}|^2dt.
\end{align*}
{Integrating from $t\geq t_0$ to $T$ and using the Lipschitz continuity of the Hamiltonian and the boundedness of $f$ and $g$,
we deduce that there exists a constant $C$, which is allowed to vary from line to line as long as it only depends on the 
same parameters as $\bar C$ in the statement}, such that
\begin{align*}
	&\mathcal{E}_t |v_t^i|^2 
	 + \lambda \int_t^T  \mathcal{E}_s
	 |v_s^i|^2
	\sum\nolimits_j \Bigl(
	\bigl[ \varphi -\varphi'\bigr] (p^j_s)   + (p^j_s)^{-1}
	\Bigr)ds
	+ \int_t^T\mathcal{E}_s \sum\nolimits_{j,k}|w_s^{i,j,k}|^2ds 
	\\
	&\leq  \mathcal{E}_T |g^i(p_T)|^2  
	+
	2 \int_{t}^T 
	\mathcal{E}_s v_s^i \sum\nolimits_{j,k} w^{i,j,k}_s dB_{s}^{j,k}
	+  \int_t^T \mathcal{E}_s  \vert v_s^i \vert \bigg\{ 
	C + C |v_{s}| + c_{d} \vert v_{s} \vert 
	  \sum\nolimits_j \bigl[  \varphi -\varphi' \bigr](p^j_s)  
	\\
	&\hspace{200pt} +  c_{d} \varepsilon  
	\sum\nolimits_{j,l} 	\sqrt{ p_{s}^j  (p_{s}^l)^{-1}} \bigl( | w_{s}^{l,l,j}|  + |w_s^{j,j,l}|\bigr) \biggr\} ds,
\end{align*}	
	where $c_{d}$ only depends on $d$. 
	Hence, by 
Young's inequality $ab\leq 2\eta a^2+ b^2/2\eta$,  which holds true for any $\eta>0$,  
\begin{align*}
	& \mathcal{E}_t |v_t^i|^2  
	 + \lambda \int_t^T  \mathcal{E}_s 
	 |v_s^i|^2
	\sum\nolimits_j \Bigl(
	\bigl[ \varphi -\varphi'\bigr] (p^j_s)   + (p^j_s)^{-1}
	\Bigr)ds 
	+ \int_t^T\mathcal{E}_s \sum\nolimits_{j,k}|w_s^{i,j,k}|^2ds 
	\\
&\leq C  \mathcal{E}_T + C  \int_t^T \mathcal{E}_s
\bigl( 1+  |v_s|^2 \bigr) ds 
+c_{d}  \int_t^T \mathcal{E}_s |v_s|^2 \sum\nolimits_j \bigl[ \varphi - 
\varphi' \bigr](p^j_s) 
 ds 
\\
&\hspace{15pt} + 8\eta \int_t^T \mathcal{E}_s
\sum\nolimits_{j,l} | w_{s}^{j,j,l}|^2
ds 
+\frac{c_{d}^2}{2\eta} \varepsilon^2
 \int_t^T \mathcal{E}_s |v^i_s|^2 \sum\nolimits_l  (p^l_s)^{-1} ds 
 +
	2 \int_{t}^T 
	\mathcal{E}_s v_s^i \sum\nolimits_{j,k} w^{i,j,k}_s dB_{s}^{j,k}.
	\end{align*}
	By summing over $i \in \ES$, we get
	\begin{align*}
		&|v_t|^2 \mathcal{E}_t 
	+ \lambda \int_t^T  \mathcal{E}_s
	|v_s|^2
	\sum\nolimits_j \Bigl( \bigl[ \varphi 
	-\varphi' \bigr](p^j_s) +  (p^j_s)^{-1}
	\Bigr)ds 
	+ \int_t^T \mathcal{E}_s \sum\nolimits_{i,j,k}|w_s^{i,j,k}|^2ds 
	\\
	&\leq C \mathcal{E}_T + C \int_t^T \mathcal{E}_s \bigl( 1+ |v_s|^2 \bigr) ds
	+c_{d} d  \int_t^T \mathcal{E}_s |v_s|^2\sum\nolimits_j \bigl[ \varphi  - \varphi' \bigr]
(p^j_s)
	 ds
	\\
	&\hspace{15pt} + 8\eta d \int_t^T \mathcal{E}_s
	\sum_{i,j,k}| w_{s}^{i,j,k}|^2 
	ds 
	+\frac{c_{d}^2}{2\eta}  \varepsilon^2
	 \int_{t_0}^T \mathcal{E}_s |v_s|^2  \sum\nolimits_l  (p^l_s)^{-1} ds  +
	2 \int_{t}^T 
	\mathcal{E}_s \sum\nolimits_{i,j,k} v_s^i  w^{i,j,k}_s dB_{s}^{j,k}.
	\end{align*}
	Choosing $\eta = 1/(16d)$ and $\lambda= {\max(\bar \lambda(12)/(2d), (32 c_{d}^2 +c_{d})d + 1/2)}$, we obtain 
	\begin{align}
	&\mathcal{E}_t |v_t|^2 
	+ \tfrac12 \int_t^T  \mathcal{E}_s
|v_s|^2	\sum\nolimits_j \Bigl( \bigl[ \varphi 
	-\varphi' \bigr](p^j_s) +  (p^j_s)^{-1}
	\Bigr)ds 
	+ \tfrac12  \int_t^T \mathcal{E}_s \sum\nolimits_{i,j,k}|w_s^{i,j,k}|^2ds 
	\nonumber
	\\
	&\leq C  
\mathcal{E}_T + C \int_t^T \mathcal{E}_s \bigl( 1+ |v_s|^2 \bigr) ds +
	2 \int_{t}^T 
	\mathcal{E}_s \sum\nolimits_{i,j,k} v_s^i  w^{i,j,k}_s dB_{s}^{j,k}.\label{ineq3}
	\end{align} 
The stochastic integral is a martingale  since $(v_t^i)_{t_{0} \le t \le T}$ and $(w_t^{i,j,k})_{t_{0} \le t \le T}$ are bounded  (possibly not uniformly in $\varepsilon$ {at this stage of the proof}). Also, 
by \eqref{bigexp2} in Proposition \ref{propexp}, replacing therein $\lambda$ by $2 \lambda d$ ({as we already explained}) and then using H\"older's inequality, 
we 
have 
$\E[\mathcal{E}_T^2]\leq C$  {for 
our choices of $\bar \kappa_{0}$ and 
$\kappa_{0}$ (the latter being greater than ${\bar \kappa_0}$)}.
Therefore, taking expectation in the above inequality and applying Gronwall's lemma, we get 
	\be
	\label{boundVV}
	\sup_{t_{0}\leq t\leq T} \E  \left[ \mathcal{E}_t |v_t|^2 \right]
	+
{\mathbb E}\biggl[	\int_{t_{0}}^T \mathcal{E}_t \sum\nolimits_{i,j,k}|w_t^{i,j,k}|^2dt \biggr] 
	\leq C.
	\ee
In order to pass the supremum inside the expectation in the first term of the left-hand side, we return back to 
	\eqref{ineq3}, take the supremum therein and then apply Burkholder-Davis-Gundy's inequality to handle the martingale, noticing that 
\begin{equation*}
\begin{split}
	{\mathbb E} \biggl[ \biggl( 
	\int_{t_{0}}^T 
	\mathcal{E}_t^2 \sum\nolimits_{i,j,k} \bigl\vert  v_t^i  w^{i,j,k}_t \bigr\vert^2 dt
	\biggr)^{1/2} \biggr]
	&\leq 
	{\mathbb E} \biggl[ \Bigl( \sup_{t_{0} \le t \le T}
	{\mathcal E}_{t}  
	\vert  v_t \vert^2 \Bigr)^{1/2} 
	 \biggl( 
	\int_{t_{0}}^T 
	\mathcal{E}_t \sum\nolimits_{i,j,k} \bigl\vert  w^{i,j,k}_t \bigr\vert^2 dt
	\biggr)^{1/2} \biggr]
	\\
	&\leq C 	{\mathbb E} \Bigl[  \sup_{t_{0} \le t \le T}
	{\mathcal E}_{t}  
	\vert  v_t \vert^2 \Bigr]^{1/2},
\end{split}
\end{equation*}
where we used \eqref{boundVV} together with Cauchy-Schwarz inequality to get the last line. We easily obtain 
\begin{equation*}
	{\mathbb E} \Bigl[ 
	\sup_{t_{0}\leq t\leq T}  {\mathcal E}_{t}  	\vert  v_t \vert^2 \Bigr]
	\leq C + C {\mathbb E} \Bigl[  \sup_{t_{0} \le t \le T}
	{\mathcal E}_{t}  
	\vert  v_t \vert^2 \Bigr]^{1/2},
\end{equation*}
which is enough to derive \eqref{boundV0}, recalling that the left-hand side is already known to be finite.
 Taking $t=t_{0}$ in \eqref{boundV0} and then 
	letting $(t_{0},p_{0})$ vary over the entire $[0,T] \times {\mathcal K}$, 
we obtain	
	 \eqref{boundVK}.
%

	Inequality \eqref{boundzeta0} derives from 
{\eqref{boundVV}}.
	Finally, in order to prove \eqref{boundfrak}, 
	we return back to 
	\eqref{eq:PDE:mathcal Y} and then expand $({\mathcal Y}(t,p_{t}))_{t_{0} \le t \le T}$ by It\^o's formula. We get 
\begin{equation*}
	\begin{split}
	{\mathcal Y}(t,p_{t}) &=\bigl\langle p_{T},g^{\bullet}(p_{T}) \bigr\rangle + 
	\int_{t}^T 
	\biggl[ \sum\nolimits_{j,k} 
	\Bigl( \tfrac12 p_{s}^j \vert
	\alpha_{s}^{j,k} 
	\vert^2 + p_s^j p_{s}^k \varphi'(p_{s}^j) \bigl( v_{s}^j - v_{s}^k \bigr) 
	\Bigr) 
	+ \bigl\langle p_{s},f^{\bullet}(p_{s}) \bigr\rangle \biggr] ds
	\\
	&\hspace{15pt} 
	+ \tfrac{1}{\sqrt{2}} \varepsilon
	\int_{t}^T
	 \sqrt{p_{s}^{j}p_{s}^{k}}
 \Bigl( {\mathfrak d}_{j} {\mathcal Y}(s,p_{s}) - {\mathfrak d}_{k} {\mathcal Y}(s,p_{s})
 \Bigr) dB_{s}^{j,k}.  
	\end{split} 
\end{equation*}
By Proposition \ref{propexp}
and 
{by \eqref{boundV0}
(recall also that
${\boldsymbol \alpha}$ is bounded by $M$)}, we have a bound for the second order moment of the right-hand side in the first line. 
Taking $t=t_{0}$, passing the stochastic integral to the left, squaring the whole equality and then tacking expectation, we 
get the announced result. 
%
%
%
\end{proof}

{We now address the (local) uniform convergence of the value function $\mathcal{V}_{\theta,\delta,\epsilon}$ (of the viscous MFCP) towards the value function $\mathcal{V}$ of the inviscid MFCP. Recall that
the convergence is already known to hold pointwise, see 
 Theorem 
\ref{thm10}. 
Recall also that 
$V_{\theta,\delta,\epsilon} = {\mathfrak D} \mathcal{V}_{\theta,\delta,\epsilon}$ solves \eqref{derhjbsimplex}.}

\begin{prop}
\label{cor:5:8}
{We can find $\bar{\kappa}_{0} \geq 0$, only depending on $\kappa_{2}$ and $d$, such that, 
for any $\kappa_{0} \geq \bar \kappa_{0}$
and any compact subset 	${\mathcal K}$ included in $\mathrm{Int}({\mathcal S}_{d})$, 
for the same two 
 functions 
$\hat{\theta}(\delta,\varepsilon)$ and $\hat{\delta}(\varepsilon)$ 
as in the statement of Proposition 	\ref{thmbounds}
(which only depend on
$\kappa_{0}$, $\kappa_{2}$, ${\mathcal K}$, $\|f\|_\infty$, $\|g\|_\infty$, 
$T$ and $d$), it holds that}
	\be
	\label{eq:mathcal V:varepsilon} 
	\lim_{\varepsilon\rightarrow0}
	{\mathcal V}_{\hat{\theta}(\hat{\delta}(\varepsilon),\varepsilon), \hat{\delta}(\varepsilon), \varepsilon}= \mathcal{V},
	\ee
	uniformly on $[0,T]\times {\mathcal K}$.
\end{prop}

\begin{proof}
{Throughout the proof, we consider 
$\bar \kappa_{0}$
as in the statement of Proposition 
\ref{thmbounds} 
and then, for $\kappa_{0} \geq \bar \kappa_{0}$
and for two compact subsets 
${\mathcal K}$
and ${\mathcal K}'$ included in 
$\textrm{\rm Int}({\mathcal S}_{d})$ such that the interior of ${\mathcal K}'$ contains 
${\mathcal K}$, we consider 
$\bar \varepsilon_{0}$, $\widehat{\theta}(\delta,\varepsilon)$ and 
$\widehat{\delta}(\varepsilon)$
as in the statement of Proposition 
\ref{thmbounds} when the compact subset therein is not ${\mathcal K}$ but ${\mathcal K}'$.
For simplicity, we let 
${\mathcal V}_{\varepsilon}:=	\mathcal{V}_{\hat{\theta}(\hat{\delta}(\varepsilon),\varepsilon), \hat{\delta}(\varepsilon), \varepsilon}$ 
and similarly for $V_{\varepsilon}$, 
for $\varepsilon \in (0,\bar \varepsilon_{0}]$. 
\vspace{5pt}

\emph{Step 1.}
The first step is to prove that the functions 
$({\mathcal V}_{\varepsilon})_{0<\varepsilon \leq \bar \varepsilon_{0}}$ are uniformly continuous on $[0,T] \times {\mathcal K}$.
In fact, recalling that $V_{ \varepsilon}={\mathfrak D} {\mathcal V}_{\varepsilon}$, 
we already know from \eqref{boundVK} that the functions 
$({\mathcal V}_{\varepsilon})_{0<\varepsilon \leq \bar \varepsilon_{0}}$
are uniformly Lipschitz continuous in space on $[0,T] \times {\mathcal K}'$. 
	
In order to prove uniform continuity in time, we fix 
some $\varepsilon \in (0,\bar \varepsilon_{0}]$
together with 
an initial condition $(t_{0},p_{0}) \in [0,T] \times {\mathcal K}$. Writing ${\boldsymbol p}^{\varepsilon}$ for 
${\boldsymbol p}^{\hat{\theta}(\hat{\delta}(\varepsilon),\varepsilon), \hat{\delta}(\varepsilon), \varepsilon}_{[t_{0},p_{0}]}$ and similarly for 
${\boldsymbol \alpha}^{\varepsilon}$, we define the stopping time 
$\sigma_{\varepsilon} := \inf \{ t \geq t_{0} : p_{t}^{\varepsilon} \not \in {\mathcal K}'\} \wedge T$. 
Since $\widehat{\delta}(\varepsilon)$ tends to $0$ with $\varepsilon$, we can change the value of 
$\bar \varepsilon_{0}$ in such a way that
$q_{i} > 2 \widehat{\delta}(\varepsilon)$, 
for any $\varepsilon \in (0,\bar \varepsilon_{0}]$, $i \in \ES$ and $q \in {\mathcal K}'$. Since 
$\varphi_{\varepsilon}
:= 
\varphi_{{\hat{\theta}(\hat{\delta}(\varepsilon),\varepsilon), \hat{\delta}(\varepsilon), \varepsilon}}$
is zero outside $[0,2 \widehat{\delta}(\varepsilon)]$, we deduce that, up 
to the stopping time $\sigma_{\varepsilon}$,  
${\boldsymbol p}^{\varepsilon}$ does not see the function $\varphi_{\varepsilon}$ in its own dynamics 
\eqref{dynpot}. Also, since the off-diagonal entries of the control ${\boldsymbol \alpha}^{\varepsilon}$ in \eqref{dynpot} are bounded by $M$, we easily deduce that there exists a constant $C$, independent of $\varepsilon$ and $(t_{0},p_{0})$, such that, for  any $t \in [t_{0},T]$,   
\begin{equation}
\label{eq:holder:continuity:1}
{\mathbb E} \Bigl[ \sup_{t_{0} \le s \leq t \wedge \sigma_{\varepsilon}} \vert p_{s}^{\varepsilon} - p_{0} \vert^2 \Bigr] \leq C (t-t_{0}). 
\end{equation}
In particular, denoting by $\textrm{\rm dist}({\mathcal K},({\mathcal K'})^{\complement})$ the distance from 
${\mathcal K}$ to the complementary of ${\mathcal K'}$ and then allowing the value of $C$ to vary from line to line (and to depend on both ${\mathcal K}$ and ${\mathcal K}'$ but not on $\varepsilon$), we have
\begin{equation}
\label{eq:holder:continuity:2}
{\mathbb P} \bigl( \sigma_{\varepsilon} < t \bigr) 
\leq 
{\mathbb P} \Bigl(\sup_{t_{0} \le s \leq t \wedge \sigma_{\varepsilon}} \vert p_{s}^{\varepsilon} - p_{0} \vert
\geq \textrm{\rm dist}\bigl({\mathcal K},({\mathcal K'})^{\complement}\bigr)
\Bigr) 
\leq C (t-t_{0}). 
\end{equation}
We now apply It\^o's formula to $({\mathcal V}_{\varepsilon}(t,p_{t}))_{t_{0} \le t \le \sigma_{\varepsilon}}$. 
By the HJB equation 
\eqref{hjbpotnew:sec:2}
(see also \eqref{hjbchart}), 
we obtain, for any $t \in [t_{0},T]$,  
\begin{equation*}
{\mathcal V}_{\varepsilon}(t_{0},p_{0}) = {\mathbb E} \biggl[  \int_{t_{0}}^{t \wedge \sigma_{\varepsilon}} 
\biggl( 
\frac12 \sum\nolimits_{i \in \ES} p_{s}^i \sum\nolimits_{j \in \ES : j \not = i}
\vert \alpha_{s}^{\varepsilon,i,j} \vert^2  +  
F\bigl(p_{s}^{\varepsilon}\bigr) \biggr) ds + {\mathcal V}_{\varepsilon} \bigl( t \wedge \sigma_{\varepsilon}, p_{t \wedge \sigma_{\varepsilon}}^{\varepsilon} \bigr) \biggr]. 
\end{equation*}
Subtracting ${\mathcal V}_{\varepsilon}(t,p_{0})$ to both sides and recalling that the integrand in the 
right-hand side can be bounded independently of $\varepsilon$, we deduce that 
\begin{equation*}
\begin{split}
&\bigl\vert {\mathcal V}_{\varepsilon}(t_{0},p_{0})
-
{\mathcal V}_{\varepsilon}(t,p_{0})
\bigr\vert
\\
&\leq C \bigl( t-t_{0}) + {\mathbb E} \Bigl[ \bigl\vert {\mathcal V}_{\varepsilon} \bigl( t \wedge \sigma_{\varepsilon}, p_{t \wedge \sigma_{\varepsilon}}^{\varepsilon} \bigr) - 
{\mathcal V}_{\varepsilon} \bigl( t, p_{t \wedge \sigma_{\varepsilon}}^{\varepsilon} \bigr)
\bigr\vert \Bigr] + 
{\mathbb E}\Bigl[
  \bigl\vert 
   {\mathcal V}_{\varepsilon} \bigl( t, p_{t \wedge \sigma_{\varepsilon}}^{\varepsilon} \bigr)
   -
   {\mathcal V}_{\varepsilon} \bigl( t, p_0 \bigr)
   \bigr\vert \Bigr]
   \\
   &\leq C \bigl( t-t_{0}) 
   + 2 \| {\mathcal V}_{\varepsilon} \|_{\infty} {\mathbb P} \bigl( \sigma_{\varepsilon} < t \bigr) + 
   C {\mathbb E}\bigl[
 \vert 
p_{t  \wedge \sigma_{\varepsilon}}^{\varepsilon} -  p_0 
 \vert \bigr],
   \end{split}
\end{equation*}
where we used the Lipschitz property of ${\mathcal V}_{\varepsilon}$ in the space variable (at least whenever the latter belongs 
to 
${\mathcal K}'$) to derive the last line. 
Since the value function ${\mathcal V}_{\varepsilon}$
can be bounded independently of $\varepsilon$ (using for instance the fact that controls themselves are required to be bounded), 
we deduce from 
\eqref{eq:holder:continuity:1} and \eqref{eq:holder:continuity:2} that 
\begin{equation*}
\begin{split}
&\bigl\vert {\mathcal V}_{\varepsilon}(t_{0},p_{0})
-
{\mathcal V}_{\varepsilon}(t,p_{0})
\bigr\vert
\leq C \bigl( t-t_{0}\bigr)^{1/2},
   \end{split}
\end{equation*}
which shows that the functions $({\mathcal V}_{\varepsilon})_{0 < \varepsilon \leq \bar \varepsilon_{0}}$
are uniformly continuous in time (and hence in time and space) on $[0,T] \times {\mathcal K}$.
\vskip 4pt

\emph{Step 2.}
Applying Ascoli-Arzel\'a theorem, 
we deduce 
that
there exist a subsequence $
	(\mathcal{V}_{\varepsilon_n})_{n \geq 0}$ and a function $\overline{\mathcal V}_{\mathcal K}$, a priori depending on ${\mathcal K}$, such that 
	$\lim_{n\rightarrow \infty} \mathcal{V}_{\varepsilon_n} = \overline{\mathcal V}_{\mathcal K}$ uniformly in $[0,T]\times {\mathcal K}$. Thanks to \eqref{limcosts}, we have pointwise convergence $\lim_{\varepsilon\rightarrow 0} \mathcal{V}_{\varepsilon}(t_0,p_0)= \mathcal{V}(t_0,p_0)$ for any $t_0\in [0,T]$ and $p_0\in\Int$. Hence any subsequence $(\mathcal{V}_{\varepsilon_n})_{n \geq 0}$ converges uniformly to the same limit which is the value function, and thus we obtain $\lim_{\varepsilon\rightarrow 0} \mathcal{V}_{\varepsilon}= \mathcal{V}$ uniformly in $[0,T]\times {\mathcal K}$. 
}
\end{proof}

{We are now in position to prove a preliminary version Proposition \ref{prop:cv:cost}, but restricted to initial conditions in a compact subset of $[0,T] \times \textrm{\rm Int}({\mathcal S}_{d})$:}
	{
		\begin{prop}
			\label{prop:5:9}
			We can find $\bar{\kappa}_{0} \geq 0$, only depending on $\kappa_{2}$ and $d$, such that, 
for any $\kappa_{0} \geq \bar \kappa_{0}$
and any compact subset 	${\mathcal K}$ included in $\mathrm{Int}({\mathcal S}_{d})$, 
for the same two 
 functions 
$\hat{\theta}(\delta,\varepsilon)$ and $\hat{\delta}(\varepsilon)$ 
as in the statement of Proposition 	\ref{thmbounds}
(which only depend on
$\kappa_{0}$, $\kappa_{2}$, ${\mathcal K}$, $\|f\|_\infty$, $\|g\|_\infty$, 
$T$ and $d$),
the additional cost induced by \eqref{eq:vartheta} tends to $0$ with $\varepsilon$:
			\be	
			\label{conv:vartheta}
			\begin{split}
			&\lim_{\varepsilon \rightarrow 0}
			\Xi_{\hat{\theta}(\hat{\delta}(\varepsilon),\varepsilon),\hat{\delta}(\varepsilon),\varepsilon}\bigl(t_{0},p_{0},q_{0}\bigr)
			=0, 
			\end{split}
			\ee
uniformly in $t_0\in [0,T]$, $p_0 \in {\mathcal K}$ and $q_0\in \mathcal{S}_d$, where
\begin{equation*}		
				\Xi_{\theta,\delta,\varepsilon}\bigl(t_{0},p_{0},q_{0}\bigr)	
				=
			\E\biggl[\bigg|\int_{t_0}^T \sum_{i\in\ES} q^{i,\theta,\delta, \varepsilon}_{[t_0,p_0,q_0],t} \vartheta^{i,\varepsilon,\phi_{\theta,\delta,\varepsilon}}
			(t,p^{\theta,\delta,\varepsilon}_{[t_0,p_0],t}) dt \bigg|\biggr].
			\end{equation*}
		\end{prop}
		}

\begin{proof}
Throughout the proof, we consider 
$\bar \kappa_{0}$
as in the statement of Proposition 
\ref{thmbounds} 
({and implicitly the same 
value of $\lambda$ as in its proof, see
	\eqref{ineq3} and the discussion below 	\eqref{ineq3}})
and then, for $\kappa_{0} \geq \bar \kappa_{0}$
and for a compact subset 
${\mathcal K}$ included in 
$\textrm{\rm Int}({\mathcal S}_{d})$, we consider 
$\bar \varepsilon_{0}$, $\widehat{\theta}(\delta,\varepsilon)$ and 
$\widehat{\delta}(\varepsilon)$, also 
as in the statement of Proposition 
\ref{thmbounds}.
For simplicity, we let 
	$\varphi_{\varepsilon}
	:= 
	\varphi_{{\hat{\theta}(\hat{\delta}(\varepsilon),\varepsilon), \hat{\delta}(\varepsilon), \varepsilon}}$
	and 
	$V_{\varepsilon}:= {V}
	_{\hat{\theta}(\hat{\delta}(\varepsilon),\varepsilon), \hat{\delta}(\varepsilon), \varepsilon}$
	for $\varepsilon \in (0,\bar \varepsilon_{0}]$. 
	{Similarly, we use the abbreviated notations $\boldsymbol{p}^\epsilon$ and $\boldsymbol{q}^{\epsilon}$
	for the two processes appearing in 
	\eqref{conv:vartheta}, the underlying initial condition $(t_{0},p_{0},q_{0})$ 
	being fixed in $[0,T] \times {\mathcal K} \times {\mathcal S}_{d}$ (which is licit provided we prove that
	the convergences below hold uniformly with respect to $(t_{0},p_{0},q_{0})$)}.	
	To prove the claim, we have to show (see \eqref{eq:vartheta}) that
	({uniformly with respect to the initial condition})
	%
	\begin{align}
	&\lim_{\varepsilon\rightarrow 0}
	\E \biggl[ \biggl\vert \int_{t_0}^T \sum\nolimits_{k} q^{\varepsilon,k}_{t}
	\sum\nolimits_j p^{\varepsilon,j}_{t} \varphi_{\varepsilon}'(p^{\varepsilon,k}_{t})\bigl(V^k_\varepsilon(t,p^{\varepsilon}_{t})-V^j_\varepsilon(t,p^{\varepsilon}_{t})\bigr) dt \biggr\vert \biggr] =0,
	\label{lim1}
	\\
	& \lim_{\varepsilon\rightarrow 0} 
	\frac{\varepsilon}{\sqrt2} \E\biggl[ \biggl\vert \int_{t_0}^T \sum\nolimits_k q^{\varepsilon,k}_{t} 
	\sum\nolimits_j \sqrt{p^{\varepsilon,j}_{t} (p^{\varepsilon,k}_{t})^{-1}} 
	\bigl( \widetilde W_{\varepsilon}^{j,j,k} - \widetilde W_{\varepsilon}^{k,k,j}
	- 2 \Upsilon_{\varepsilon}^{k,j}
	\bigr)(t,p_{t}^{\varepsilon}) 
	dt  \biggr\vert \biggr]=0,
	\label{lim3}
	\end{align}
	where $W_{\varepsilon}=W_{\hat{\theta}(\hat{\delta}(\varepsilon),\varepsilon),\hat{\delta}(\varepsilon),\varepsilon}$
	and similarly for $\Upsilon_{\varepsilon}$. 
	
	We begin by proving \eqref{lim1}. 
{We know
from 
\eqref{bigexp3}
(all the results from Proposition \ref{propexp} are applied with $2 \lambda d$, see again the discussion 
below 	\eqref{ineq3})
	that 
	$\lim_{\varepsilon \rightarrow 0}{\mathbb P}(\inf_{t_{0} \le t \le T} \min_{k \in \ES}
	p_{t}^{\varepsilon,k} \leq \eta)=0$ for $\eta>0$ small enough (independently of $\varepsilon$), the convergence being uniform with respect to the initial point $(t_{0},p_{0}) \in [0,T] \times {\mathcal K}$}. 
	Together with the fact that the support of $\varphi_{\varepsilon}'$ shrinks with $\varepsilon$,  
	we deduce that the integrand tends to 0 in probability as $\varepsilon\rightarrow 0$. Hence, to obtain \eqref{lim1}, {we have to prove uniform integrability, namely it is enough to show that
		\[
		\E \bigg[ \bigg| \int_{t_0}^T \sum\nolimits_{k} q^{\varepsilon,k}_{t}
		\sum\nolimits_j p^{\varepsilon,j}_t \varphi_{\varepsilon}'\bigl(p^{\varepsilon,k}_t\bigr)
		\Bigl(V^k_\varepsilon\bigl(t,p_t^{\varepsilon}\bigr)-V^j_\varepsilon\bigl(t,p_t^{\varepsilon}\bigr)\Bigr)d t \biggr|^{3/2}  \biggr] \leq C,
		\]
		for a constant $C$ independent of $\varepsilon$ and of $(t_{0},p_{0},q_{0})$. 
		By 
		\eqref{boundV0}
		and by H\"older inequality with exponents $8$, $8$ and $4/3$, it suffices to prove 
		that 
		\[
		\E \biggl[ 
		\sum\nolimits_{k}
		\sup_{t_{0} \le t \le T}
		\vert q_{t}^{\varepsilon,k} \vert^{12} \biggr]^{1/8}
		{\mathbb E}
		\biggl[
		\bigg| \int_{t_0}^T \sum\nolimits_{k}  \varphi_{\varepsilon}'\bigl(p^{\varepsilon,k}_t\bigr)
		d t \biggr|^{12}  \biggr]^{1/8} \leq C.
		\]}The first term in the left-hand side is easily bounded by means of 
	Lemma 
	\ref{lem:moments:p,q:l}, {recalling that
	$\lambda$ in the statement of Proposition 
	\ref{thmbounds} 
	is required to satisfy
	 $\lambda \geq \bar \lambda(12)$}. 
	As for the second one, it follows from 
	\eqref{bigexp2}.

	{To prove \eqref{lim3}, we have to show that the expectation is bounded (since there is the additional factor $\varepsilon$ in front of it), but this easily follows from 
		Holder's inequality, with $1= 1/3+1/6+1/2$,
		together with}
	\eqref{momentp}, \eqref{momentq},  \eqref{boundzeta0} and \eqref{boundfrak}.
\end{proof}

We now address the convergence of the master equation. 
 {To do so, we}
denote by $U_{\theta, \delta,\epsilon}$ the solution to the viscous master equation \eqref{eq:master:equation:introl}
({as provided by Theorem \ref{thm4}}), with 
$\phi=\phi_{\theta,\delta,\epsilon}$ therein. 
We recall (see Proposition \ref{prop:solution:mfc:zero:epsilon}, (viii)) that there exist a unique optimal control $\boldsymbol{\alpha}$ and optimal trajectory $\boldsymbol{p}$ for the inviscid MFCP starting at points $(t_0,p_0)\in [0,T]\times\mathrm{Int}(\mathcal{S}_d)$ where 
the value function $\mathcal{V}$  is differentiable.   
For such points, let, as in Section \ref{sec:main} ({see Theorem \ref{main:thm:6}, part II}), 
$U^i(t_0,p_0):=\inf_{\balpha} J(\balpha;{\boldsymbol p})$ with 
 ${\boldsymbol q}$
being initialized at time $t_0$ from $(q_{t_0}^j=\delta_{i,j})_{j \in \ES}$.

\begin{prop}
	\label{thm:5:9}
{On top of the assumptions quoted in the beginning of the subsection, assume that $F$ is in $\mathcal{C}^{1,1}({\mathcal S}_{d})$. Then, we can find $\bar{\kappa}_{0} \geq 0$, only depending on $\kappa_{2}$ and $d$, such that, 
for any $\kappa_{0} \geq \bar \kappa_{0}$
and any compact subset 	${\mathcal K}$ included in $\mathrm{Int}({\mathcal S}_{d})$, 
for the same two 
 functions 
$\hat{\theta}(\delta,\varepsilon)$ and $\hat{\delta}(\varepsilon)$ 
as in the statement of Proposition 	\ref{thmbounds}
(which only depend on
$\kappa_{0}$, $\kappa_{2}$, ${\mathcal K}$, $\|f\|_\infty$, $\|g\|_\infty$, 
$T$ and $d$),
and for any $(t_{0},p_{0}) \in [0,T] \times {\mathcal K}$
	at which ${\mathcal V}$
	is differentiable}, 
	\begin{align} 
	\label{eq:D:mathcal V:varepsilon} 
	&\lim_{\varepsilon\rightarrow0}
	\mathfrak{D}\mathcal{V}
	_{\hat{\theta}(\hat{\delta}(\varepsilon),\varepsilon), \hat{\delta}(\varepsilon), \varepsilon}(t_0,p_0)= \mathfrak{D}\mathcal{V}(t_0,p_0),
\\
	&\lim_{\varepsilon\rightarrow0}
	U
	_{\hat{\theta}(\hat{\delta}(\varepsilon),\varepsilon), \hat{\delta}(\varepsilon), \varepsilon}(t_0,p_0)= U(t_0,p_0),
	\label{convergence:U}
	\end{align}
	Moreover, these convergence hold  in $[L^r_{loc}([0,T]\times {\rm Int}({\mathcal S}_{d}))]^d$, for any $r \geq 1$,
	{where $\textrm{\rm Int}({\mathcal S}_{d})$ is equipped with the $(d\!-\!1)$-dimensional Lebesgue measure}. 
\end{prop}

\begin{proof}
	\emph{Step 1.}
As in the previous proof, we consider 
$\bar \kappa_{0}$
as in the statement of Proposition 
\ref{thmbounds} 
({and implicitly the same 
value of $\lambda$ as in its proof})
and then, for $\kappa_{0} \geq \bar \kappa_{0}$
and for a compact subset 
${\mathcal K}$ included in 
$\textrm{\rm Int}({\mathcal S}_{d})$, we consider 
$\bar \varepsilon_{0}$, $\widehat{\theta}(\delta,\varepsilon)$ and 
$\widehat{\delta}(\varepsilon)$, also 
as in the statement of Proposition 
\ref{thmbounds}.	{We also use the same notations}
	$\varphi_{\varepsilon}$, 
	$V_{\varepsilon}$, 	${\boldsymbol p}^\varepsilon$ as in the previous proof, and similarly
 {we write} 
 ${\boldsymbol \alpha}^\varepsilon$
 {for the corresponding optimal control}
 and 
 $U_\epsilon$ {for the solution of the (viscous) master equation.
 Here, the initial condition of 
 ${\boldsymbol p}^\varepsilon$
 is implicitly understood as a point $(t_{0},p_{0}) \in [0,T] \times {\mathcal K}$ at which ${\mathcal V}$ is differentiable.} By Corollary \ref{cor:selection}
	({writing $({\boldsymbol p},{\boldsymbol \alpha})$ for 
	$({\boldsymbol p}_{[t_0,p_0]},{\boldsymbol \alpha}_{[t_0,p_0]})$ 
	therein}), 
	the convergence of 
	$({\boldsymbol p}^{\varepsilon},{\boldsymbol \alpha}^{\varepsilon})_{\varepsilon \in (0,{\varepsilon_{0}}]}$
	to $({\boldsymbol p},{\boldsymbol \alpha})$ holds in probability
	(for the same topology as in the statement of Theorem \ref{thm10}).
	In fact, 
	by combining
	\eqref{eq:thm:5:1:1} 
	and
	\eqref{eq:thm:5:1:2} (with ${\boldsymbol \beta}={\boldsymbol \alpha}$ therein), we have (indices below are in 
		$\ES$)
	\begin{equation*}
	\lim_{\varepsilon \rightarrow 0}
	{\mathbb E} \int_{t_{0}}^T \sum\nolimits_{i \not =j}
	p_{t}^{i}  \vert \alpha_{t}^{\varepsilon,i,j} \vert^2 dt
	= 
	\int_{t_{0}}^T \sum\nolimits_{i \not =j}
	p_{t}^{i}  \vert \alpha_{t}^{i,j} \vert^2 dt, 
	\end{equation*}
	from which we deduce that 
	\begin{equation}
	\label{eq:strong:alpha:eps}
	\lim_{\varepsilon \rightarrow 0}
	{\mathbb E} \int_{t_{0}}^T \sum\nolimits_{i \not =j}
	p_{t}^{i}  \vert \alpha_{t}^{\varepsilon,i,j} - \alpha_{t}^{i,j} \vert^2 dt
	= 0. 
	\end{equation}
	As the limit process ${\boldsymbol p}$ does not touch the boundary of the simplex, 
	the latter shows that 
	${\boldsymbol \alpha}^{\varepsilon}$ converges to ${\boldsymbol \alpha}$ in probability but 
	for the strong (instead of weak) topology on ${\mathcal E}$. We make use of this property later on in the proof.

	{In order to prove \eqref{convergence:U}, 
	it is worth recalling that $U_{\varepsilon}^{i}(t_0,p_0)$ is the value function of the 
	cost functional 
	$\tilde{J}^{\varepsilon,\varphi_{\varepsilon}}( \, \cdot \, ; {\boldsymbol p}^{\varepsilon})$ when 
	the state trajectory ${\boldsymbol q}$ in 
		\eqref{dynmfg} is initialized from 
	$q^k_{t_0}=\delta_{i,k}$, $k \in \ES$, and similarly $U^i(t_0,p_0)$
	 is the value function of the 
	cost functional 
	${J}( \, \cdot \, ; {\boldsymbol p})$ when the state trajectory in 
			\eqref{eq:fp} is also initialized from 
	$q^k_{t_0}=\delta_{i,k}$, $k \in \ES$.}
	Recalling 
	\eqref{eq:new:J:varepsilon,varphi}, we have {(indices in the sum belonging to $\ES$)}
	\begin{equation}
	\label{eq:535}
	\tilde{J}^{\varepsilon,\varphi_{\varepsilon}}\bigl({\boldsymbol \alpha}^{\varepsilon};{\boldsymbol p}^{\varepsilon}\bigr) =
	{\mathbb E}
	\biggl[
	\int_{t_{0}}^T \sum\nolimits_{k} q^{\varepsilon,k}_t \Bigl({\mathfrak L}^k(\alpha_t^{\varepsilon}) +   f^k(p_t^{\varepsilon}) + \vartheta^{\varepsilon,\varphi_{\varepsilon},k}(t,p_{t}^{\varepsilon})\Bigr)dt +  \sum\nolimits_k q^{\varepsilon,k}_T g^k(p_T^{\varepsilon}) \biggr],
	\end{equation}
	%
	%
	where
	$((q^{\varepsilon,k}_{t})_{k \in \ES})_{t_{0} \le t \le T})$ solves
	\begin{equation*}
	\begin{split}
	dq_{t}^{\varepsilon,k} &= \sum\nolimits_{j\neq k} \Bigl(q^{\varepsilon,j}_{t} 
	\bigl(\varphi_\varepsilon(p^{\varepsilon,k}_{t})+\alpha_{t}^{\varepsilon,j,k}\bigr) - 
	q_t^{\varepsilon,k}\bigl(\varphi_\varepsilon(p^{\varepsilon,j}_{t}) + \alpha_{t}^{\varepsilon,k,j}\bigr) \Bigr) dt
	\\ 
	&\hspace{15pt} + \frac{\varepsilon}{\sqrt{2}} \sum\nolimits_{j\neq k} 
	\frac{q_t^{\varepsilon,k}}{p_t^{\varepsilon,k}} 
	\sqrt{p_t^{\varepsilon,k} p_t^{\varepsilon,j}} \bigl( dB_{t}^{{k},j} - dB_{t}^{j,{k}} \bigr),
	\end{split}
	\end{equation*}
	with $q_{t_{0}}^{\varepsilon,k}=\delta_{i,k}$.

%
By 	\eqref{bigexp3}, there exists 
	$\eta >0$ such that 
	$\lim_{\varepsilon \rightarrow 0}{\mathbb P}(\inf_{t_{0} \le t \le T} \min_{k \in \ES}
	p_{t}^{\varepsilon,k} \leq \eta)=0$. 
	This suffices to 
	kill asymptotically the terms $\varphi_{\varepsilon}$ in the drift right above on the model of the proof of Theorem \ref{thm10}. As 
	for the martingale part, we may invoke 
	Lemma 
	\ref{lem:moments:p,q:l} with $\ell=4$ ({recall that 
	$\bar \kappa_{0}$ is chosen in such a way that 	
	Lemma 
	\ref{lem:moments:p,q:l} applies with $\ell=12$})
	to show that its supremum norm converges to $0$ in probability. 
	Altogether with  
	\eqref{eq:strong:alpha:eps}, 
	we easily deduce that 
	\begin{equation*}
	q_{t}^{\varepsilon,k} = \delta_{k,i}+ \int_{t_{0}}^t 
	\sum\nolimits_{j\neq k} \Bigl(q^{\varepsilon,j}_{s}  \alpha_{s}^{j,k}  - 
	q_t^{\varepsilon,k}  \alpha_{s}^{k,j} \Bigr) ds + r_{t}^{\varepsilon,k}, \quad t \in [t_{0},T], 
	\end{equation*}
	where $\sup_{t_{0} \leq t \leq T} \vert r_{t}^{\varepsilon,k} \vert$ tends to $0$ in probability with 
	$\varepsilon$, which prompts us to consider the differential equation 
	\begin{equation*}
	\dot{q}_{t}^k = \sum\nolimits_{j\neq k} \bigl(q_{t}^j \alpha_{t}^{j,k} - 
	q_t^i\alpha_t^{k,j}\bigr), \quad 
	q_{t_0}^k = \delta_{k,i}.
	\end{equation*}
	Forming the differences 
	$((q_{t}^{\varepsilon,k}-q_{t}^{k})_{k \in \ES})_{t_{0} \le t \le T}$, we easily deduce that 
	$\sup_{t_{0} \le t \le T} \vert q_{t}^{\varepsilon}-q_{t} \vert$
	tends to $0$ in probability.  By Lemma 
	\ref{lem:moments:p,q:l} again, the convergence holds in $L^{2}$ (recall that $\lambda \geq \bar {\lambda}(12)$). 
	Using \eqref{eq:new:J:varepsilon,varphi}
	and 	\eqref{eq:strong:alpha:eps}
	together with the fact that all the off-diagonal controls are bounded by $M$, we get  
	\begin{align*}
	&\lim_{\varepsilon\rightarrow 0}
	\E \biggl[
	\int_{t_0}^T \sum\nolimits_{k} q^{\varepsilon,k}_{t} \Bigl( {\mathfrak L}^k(\alpha_{t}^{\varepsilon}) + f^k(p_t^{\varepsilon}) \Bigr) dt +
	\sum\nolimits_k   q^{\varepsilon,k}_{T}g^k(p^\varepsilon_{T}) 
	\biggr] 
	\\
	&\hspace{15pt} = 
	\int_{t_0}^T \sum\nolimits_k q^{k}_{t} \Bigl({\mathfrak L}^k(\alpha_{t}) + f^k(p_{t}) \Bigr) dt 
	+
	\sum\nolimits_k   q^{k}_{T}g^k(p_{T}) 
	= J({\boldsymbol \alpha};{\boldsymbol p})=  {U^i(t_0,p_0)},
	\end{align*} 
	which, together with \eqref{conv:vartheta} that holds for any initial condition, gives \eqref{convergence:U}.
	\vskip 4pt
	
\emph{Step 2}. 	To prove \eqref{eq:D:mathcal V:varepsilon}, we note that, by  item (viii) in 
	Proposition 
	\ref{prop:solution:mfc:zero:epsilon}
	again
	$\alpha^{i,j}_t= \left(\mathfrak{d}_i \mathcal{V}(t,p_t)-\mathfrak{d}_j\mathcal{V}(t,p_t)\right)_+$ for any $t_0\leq t\leq T$
	(recalling that ${\mathcal V}$ is differentiable at any $(t,p_{t})$),
	and by item {(ix)} in 
	Proposition 
	\ref{prop:solution:mfc:zero:epsilon}, 
	the backward equation
	in 
	\eqref{eq:fb:inviscid:local} 
	(in the unknown ${\boldsymbol z}=(z_{t})_{t_{0} \le t \le T}$)
	represents the gradient of $\mathcal{V}$, so in particular at the initial time we have $z^i_{t_0}= \partial_{x_i} \widehat{\mathcal V}(t_0,x_0)$ with $x_{0}=(p_{0}^1,\cdots,p_{0}^{d-1})$. But $z^i_{t_0}$ is also equal to $U^i(t_0,p_0)-U^d(t_0,p_0)$, 
	which is exactly  
	\eqref{eq:fb:inviscid:intrinsic}.   
	Thus we have  $\partial_{x_i}\mathcal{V}(t_0,x_0) = U^i(t_0,p_0)-U^d(t_0,p_0)$. 
	Importantly, 
	we have a similar identity when {$\varepsilon \in (0, \varepsilon_{0}]$}, which is provided by \eqref{eq:master:HJB} proved in Theorem \ref{thm4}. 
	Therefore \eqref{eq:D:mathcal V:varepsilon} now follows from \eqref{convergence:U}.
	
	\vskip 4pt

	\emph{Step 3.}
	The last claim follows from uniform boundedness of 
	$\fD \mathcal{V}_\epsilon$ and $U_\epsilon$. The former is given by 
	\eqref{boundVK},
	together with 
	the fact that ${\mathcal V}$ is almost everywhere ({for the $(d\!-\!1)$-dimensional Lebesgue measure}) differentiable,  while the latter follows easily from the definition \eqref{eq:535} together with the bounds in Proposition \ref{thmbounds}.  
\end{proof}

At this stage of the proof, the reader must understand that 
 {Propositions \ref{prop:5:9}
and
\ref{thm:5:9} do not provide complete proofs of
Proposition \ref{prop:cv:cost}
and
Theorem
\ref{main:thm:6}}. The reason is that the functions $\hat{\theta}$ and $\hat{\delta}$ therein depend on 
the underlying compact set ${\mathcal K}$. In words, we should write 
$\hat{\theta}_{\mathcal K}$ 
and $\hat{\delta}_{\mathcal K}$. Now, we would like to choose ${\mathcal K}={\mathcal K}_\varepsilon$
depending on $\varepsilon$ such that, letting
\be
\label{not:bar:eps}
\overline{\mathcal{V}}_\varepsilon := \mathcal{V}_{\hat{\theta}_{{\mathcal K}_{\varepsilon}}(\hat{\delta}_{{\mathcal K}_\varepsilon}(\varepsilon),\varepsilon),\hat{\delta}_{{\mathcal K}_\varepsilon}(\varepsilon),\varepsilon},
\quad 
{\overline{U}_\varepsilon := U_{\hat{\theta}_{{\mathcal K}_{\varepsilon}}(\hat{\delta}_{{\mathcal K}_\varepsilon}(\varepsilon),\varepsilon),\hat{\delta}_{{\mathcal K}_\varepsilon}(\varepsilon),\varepsilon}},
\quad
{\overline{\Xi}_\varepsilon := \Xi_{\hat{\theta}_{{\mathcal K}_{\varepsilon}}(\hat{\delta}_{{\mathcal K}_\varepsilon}(\varepsilon),\varepsilon),\hat{\delta}_{{\mathcal K}_\varepsilon}(\varepsilon),\varepsilon}},
\ee
\eqref{eq:mathcal V:varepsilon} 
 {and
\eqref{conv:vartheta}}
hold locally uniformly on $[0,T] \times \textrm{\rm Int}({\mathcal S}_{d})$,
 and  \eqref{eq:D:mathcal V:varepsilon}
 and
 	\eqref{convergence:U}
   hold almost everywhere. 
We mostly argue by an inversion argument very similar to the proof of 
Proposition \ref{propexp} ({which is given below}). 
\begin{thm}
	\label{thm:5.10}	
{Under the assumptions quoted in the beginning of the subsection, we can find $\bar{\kappa}_{0} \geq 0$, only depending on $\kappa_{2}$ and $d$, such that, 
for any $\kappa_{0} \geq \bar \kappa_{0}$
and any $e \in (0,\varepsilon_{0}=\sqrt{\kappa_{2}/\kappa_{0}}]$, 
there exist a compact subset ${\mathcal K}_{e}$ included in $\mathrm{Int}({\mathcal S}_{d})$, 
with ${\mathcal K}_{e} \supset {\mathcal K}_{e'}$ if $e<e'$ and 
$\cup_{e \in (0,\varepsilon_{0}]} {\mathcal K}_{e}= \textrm{\rm Int}({\mathcal S}_{d})$},
together with functions 
$\hat{\theta}_{{\mathcal K}_{e}}(\delta,\varepsilon)$ and $\hat{\delta}_{{\mathcal K}_{e}}(\varepsilon)$ 
as in the statement of Proposition \ref{thmbounds}
such that, 
{using the same notations
as in 
\eqref{not:bar:eps} (and in particular letting $e=\varepsilon$)},
\begin{align}
&\lim_{\varepsilon\rightarrow0}
	\overline{\mathcal{V}}
	_{\varepsilon}= \mathcal{V} &&\mbox{locally uniformly in } [0,T] \times \textrm{\rm Int}({\mathcal S}_{d}),
	\label{L1}\\
&{\lim_{\varepsilon\rightarrow0}
	\overline{\Xi}
	_{\varepsilon}= 0} &&{\mbox{locally uniformly in } [0,T] \times \textrm{\rm Int}({\mathcal S}_{d})
	\times {\mathcal S}_{d}}.
	\label{L1b}
	\end{align}
{If, in addition 
$F$ is in $\mathcal{C}^{1,1}({\mathcal S}_{d})$, then} 
\begin{align}
&\lim_{\varepsilon\rightarrow0}
\mathfrak{D}\overline{\mathcal{V}}
_{\varepsilon}= \mathfrak{D}\mathcal{V}
 &&\mbox{a.e. on } [0,T]\times \textrm{\rm Int}({\mathcal S}_{d})
\mbox{ and in } [L^1_{loc}([0,T] \times \textrm{\rm Int}({\mathcal S}_{d})]^d,
\label{L2} 
\\
&{\lim_{\varepsilon\rightarrow0}
\overline{U}
_{\varepsilon} = U}
 &&
{ \mbox{a.e. on } [0,T]\times \textrm{\rm Int}({\mathcal S}_{d})
\mbox{ and in } [L^1_{loc}([0,T] \times \textrm{\rm Int}({\mathcal S}_{d})]^d}, 
\label{L3}
\end{align}
{where $\textrm{\rm Int}({\mathcal S}_{d})$ is equipped with the $(d-1)$-dimensional Lebesgue measure}. 
\end{thm}
\begin{proof}
{Throughout the proof, we consider 
$\bar \kappa_{0}$
as in the statement of Proposition 
\ref{thmbounds}
and then, for $\kappa_{0} \geq \bar \kappa_{0}$
and for a compact subset 
${\mathcal K}$ included in 
$\textrm{\rm Int}({\mathcal S}_{d})$, we consider 
$\widehat{\theta}_{\mathcal K}(\delta,\varepsilon)$ and 
$\widehat{\delta}_{\mathcal K}(\varepsilon)$, also 
as in the statement of Proposition 
\ref{thmbounds}.
Without any loss of generality, 
the functions 
$\widehat{\theta}_{\mathcal K}(\delta,\varepsilon)$ and 
$\widehat{\delta}_{\mathcal K}(\varepsilon)$
may be assumed to be defined for all $\delta \in (0,1/2]$ and $\varepsilon \in (0,\varepsilon_{0}=\sqrt{\kappa_{2}/\kappa_{0}}]$; in fact, only the limits 
in $(0,0)$ and $0$ matter for our purpose. }
For any $n\geq 1$, let ${\mathcal K}_n$ be the compact set $\{x\in {\mathcal S}_{d} : \mathrm{dist}(x, \partial \mathcal{S}_{d}) \geq 1/n\}$. Below, we restrict ourselves to the set ${\mathbb N}_{\partial}$ of large enough integers $n$ such that ${\mathcal K}_{n} \not = \emptyset$. Obviously, ${\mathbb N}_{\partial}$ is of the form ${\mathbb N}_{\partial} = \{n_{\partial},n_{\partial}+1,\cdots\}$ for some 
integer $n_{\partial} \geq 1$. 
For $n \in {\mathbb N}_{\partial}$,  
we let $\mathcal{V}_{n,\varepsilon}(t,p):=\mathcal{V}_{\hat{\theta}_{{\mathcal K}_{n}}(\hat{\delta}_{{\mathcal K}_{n}}(\varepsilon), \varepsilon),\hat{\delta}_{{\mathcal K}_{n}}(\varepsilon),\varepsilon}(t,p)$ for 
$t \in [0,T]$ and 
$p\in {\mathcal S}_d$; {similarly, we introduce 
$U_{n,\varepsilon}$
and $\Xi_{n,\varepsilon}$
(the latter being defined on 
$[0,T] \times {\mathcal S}_{d} \times {\mathcal S}_{d}$)}. By Corollary \ref{cor:5:8}, for any fixed $n \in {\mathbb N}_{\partial}$,  
$\lim_{\varepsilon\rightarrow 0} \mathcal{V}_{n,\varepsilon} = \mathcal{V}$, uniformly on $[0,T] \times {\mathcal K}_n$, 
and, {by Proposition
\ref{prop:5:9}}, 
$\lim_{\varepsilon\rightarrow 0} \Xi_{n,\varepsilon} = 0$, uniformly on $[0,T] \times {\mathcal K}_n \times {\mathcal S}_{d}$.
 Further, by Proposition \ref{thm:5:9}, 
$\lim_{\varepsilon\rightarrow0}
\mathfrak{D}\mathcal{V}
_{n,\varepsilon}(t,p)= \mathfrak{D}\mathcal{V}(t,p)$ for a.e. $(t,p) \in [0,T]\times {\mathcal K}_n$,
and $\lim_{\varepsilon\rightarrow0}
\mathfrak{D}\mathcal{V}
_{n,\varepsilon}= \mathfrak{D}\mathcal{V}$ in $[L^1([0,T] \times {\mathcal K}_n)]^d$, {and similarly with $U_{n,\varepsilon}$ and $U$}.
Applying Egoroff's theorem for any $n \in {\mathbb N}_{\partial}$, there exists $E_n\subset [0,T]\times {\mathcal K}_n$ with 
({$d$-dimensional --since the simplex is equipped with the $(d\!-\!1)$ Lebesgue measure--}) Lebesgue measure $|E_n|\leq 2^{-n}$ such that $\lim_{\varepsilon\rightarrow0}
\mathfrak{D}\mathcal{V}
_{n,\varepsilon}= \mathfrak{D}\mathcal{V}$ 
{and
$\lim_{\varepsilon\rightarrow0}
U_{n,\varepsilon}= U$} 
uniformly in $[0,T]\times {\mathcal K}_n\setminus E_n$. Therefore, for any $n \in {\mathbb N}_{\partial}$, there exists $\varepsilon_n \in (0,\varepsilon_{0}]$ such that, for any $\varepsilon\leq \varepsilon_n$, {
\begin{equation}
\label{E3}
\begin{split}
&\sup_{(t,p) \in [0,T] \times {\mathcal K}_n} \bigl|\bigl( 
\mathcal{V}_{n,\varepsilon} -\mathcal{V}\bigr)(t,p)\bigr|
+
\sup_{(t,p,q) \in [0,T] \times {\mathcal K}_n \times {\mathcal S}_{d}} \bigl|\Xi_{n,\varepsilon}(t,p,q)\bigr|
\\
&\hspace{15pt} + \sup_{(t,p)\in [0,T]\times {\mathcal K}_n\setminus E_n}
\Bigl( 
\bigl| \bigl( \fD\mathcal{V}_{n,\varepsilon} -\fD\mathcal{V}\bigr)(t,p)\bigr|
+
\bigl| \bigl( U_{n,\varepsilon} -U\bigr)(t,p)\bigr|
\Bigr) 
\\
&\hspace{15pt} + \int_0^T \int_{{\mathcal K}_n}\Bigl( |(\fD\mathcal{V}_{n,\varepsilon}-\fD\mathcal{V})(t,p)|+
 |(U_{n,\varepsilon}-U)(t,p)|
\Bigr) d {\varrho(p)}  dt \leq \frac1n,
\end{split}
\end{equation}
where $\varrho$ is the image of the $(d\!-\!1)$-dimensional Lebesgue measure by the map $(x_{1},\cdots,x_{d-1})
\mapsto (x_{1},\cdots,x_{d-1},x^{-d})$.} Moreover, we can assume that $\varepsilon_{n+1}<\varepsilon_n \leq 1/n$, so that $\lim_{n \rightarrow \infty} \varepsilon_n =0$.
We now define $n$, and thus ${\mathcal K}_n$, in terms of $\varepsilon$: for any $\varepsilon \in (0,\varepsilon_{n_{\partial}})$, let $n_\varepsilon$ be the unique $n\in\mathbb{N}$ such that $\varepsilon_{n+1}<\varepsilon\leq \varepsilon_n$. 
Obviously, the function $(0,n_{\partial}) \ni \varepsilon \mapsto n_{\varepsilon}$ is decreasing
and the supremum, say $N:=\sup_{\varepsilon \in (0,n_{\partial})} n_\varepsilon$ cannot be finite as otherwise we would have
$0<\varepsilon_{N+1}\leq \varepsilon_{n_\varepsilon + 1} <\varepsilon$ for any $\varepsilon \in (0,\varepsilon_{0}]$, which is a contradiction.
Hence, choosing $n=n_{\varepsilon}$ in the right-hand side of 
\eqref{E3},
{letting
with a slight abuse of notation
${\mathcal K}_{\varepsilon}:={\mathcal K}_{n_{\varepsilon}}$
(the definition of ${\mathcal K}_{\varepsilon}$, for $\varepsilon \in [\varepsilon_{n_{\partial}},\overline 
\varepsilon_{0}]$, does not really matter)
and then using the same notation as in 
\eqref{not:bar:eps},}
we get, 
for any compact subset ${\mathcal K}$ included in $\textrm{\rm Int}({\mathcal S}_{d})$, 
\[
\sup_{(t,p) \in [0,T] \times {\mathcal K}} |\overline{\mathcal{V}}_{\varepsilon}(t,p)-\mathcal{V}(t,p)|\leq 
\sup_{(t,p) \in [0,T] \times {\mathcal K}_\varepsilon} |\mathcal{V}_{n_\varepsilon,\varepsilon}(t,p)-\mathcal{V}(t,p)|\leq \frac{1}{n_\varepsilon},
\]
for $\varepsilon$ small enough, which gives \eqref{L1}. 
{Obviously, the proof of \eqref{L1b} 
is similar, and in fact the same argument applies for 
proving the $L^1$ convergence 
in \eqref{L2} and \eqref{L3}}.
{To prove the a.e. convergence in \eqref{L2} (and similarly in \eqref{L3}}), consider again a compact set ${\mathcal K}\subset \textrm{\rm Int}({\mathcal S}_{d})$ and $\varepsilon$ small enough such that ${\mathcal K}\subset {\mathcal K}_{{\varepsilon}}$. From \eqref{E3} again, we get that 
\[
|\fD\overline{\mathcal{V}}_{\varepsilon}(t,p)-\fD\mathcal{V}(t,p)| \leq \frac{1}{n_\varepsilon}
\]
if $(t,p)\in[0,T]\times {\mathcal K}\setminus E_{n_\varepsilon}$. Therefore, the set of points $(t,p)\in[0,T]\times \textrm{\rm Int}({\mathcal S}_{d})$ such that $\fD\overline{\mathcal{V}}_{\varepsilon}(t,p)$ does not converge to $\fD\mathcal{V}(t,p)$, as $\varepsilon\rightarrow 0$, is included in the set of points $(t,p)$ such that $(t,p)\in E_{n_\varepsilon}$ for infinitely many $n_\varepsilon$. The latter is nothing but $\limsup_{n\geq n_{\partial}} E_n$, which has Lebesgue measure 0 by Borel-Cantelli lemma, since $\sum_{n=1}^\infty |E_n|\leq \sum_{n=1}^\infty 2^{-n} <\infty$. 
Hence $\lim_{\varepsilon\rightarrow0}
\mathfrak{D}\overline{\mathcal{V}}
_{\varepsilon}(t,p)= \mathfrak{D}\mathcal{V}(t,p)$ for a.e. $(t,p) \in [0,T]\times \textrm{\rm Int}({\mathcal S}_{d})$, from which {the a.e. convergence in \eqref{L2} follows}.
 {The a.e. convergence in \eqref{L3} is treated in the same way}.  
\end{proof}

\subsection{Proofs of auxiliary exponential integrability properties}
\label{subse:auxiliary:proofs:integrability}

\begin{proof}[Proof of Proposition \ref{expfin}]
We prove 
\eqref{eq:expfin1}
and 
\eqref{eq:expfin2} in a single row, mostly following \cite[Proposition 2.2]{mfggenetic}. 
Fix $i\in \ES$ and, for simplicity, write
${\boldsymbol p}^i$ for ${\boldsymbol p}^{i,\theta,\delta,\varepsilon}$ and take $t_{0}=0$.  
As in the second step of the proof of \cite[Proposition 2.1]{mfggenetic}, 
we write the equation for ${\boldsymbol p}^i$  in the form 
\begin{equation}
\label{eq:tilde w:i}
{dp_{t}^i = 
\sum_{j \in \ES} \Bigl[  p^j_{t} \bigl(\varphi_{\theta,\delta,\varepsilon}(p^i_t)  + \alpha^{j,i}_t \bigr)  
- p_{t}^i\bigl(\varphi_{\theta,\delta,\varepsilon}(p^j_t)+ \alpha^{i,j}_t \bigr)
\Bigr]
   dt + 
 \varepsilon \sqrt{  p_{t}^i ( 1-p_{t}^i)}
d \widetilde W_{t}^i,}
\end{equation}
for 
$t \in [0,T]$ 
and for
$\widetilde{\boldsymbol W}^i=(\widetilde W_{t}^i)_{0 \le t \le T}$
a $1d$-Brownian motion.  
Then,
It\^o's formula yields (the left-hand side below is well-defined since 
${\boldsymbol p}^i$ does not vanish)
\begin{equation}
\label{ito:log}
\begin{split}
d \biggl[\frac{\lambda}{\varepsilon^2} \ln p_{t}^i \biggr] &=  
\sum_{j \in \ES} \biggl[\frac{\lambda}{\varepsilon^2} \frac{p^j_{t}}{p^i_{t}}\Bigl(\varphi_{\theta,\delta,\varepsilon}(p^i_t)  + \alpha^{j,i}_t \Bigr)  
-\frac{\lambda}{\varepsilon^2}\Bigl(\varphi_{\theta,\delta,\varepsilon}(p^j_t)+ \alpha^{i,j}_t \Bigr)
\biggr]
dt
 - \frac{\lambda}{2} \frac{1- p_{t}^i}{p_{t}^i} dt
\\
&\hspace{15pt}
+   \frac{\lambda}{\varepsilon} \sqrt{\frac{1-p_{t}^i}{p_{t}^i}} d \widetilde W_{t}^{i}, \quad t \in [0,T].
\end{split}
\end{equation}
We now subtract the quantity $ \lambda^2 ( 1-p^i_t)/(2 \varepsilon^2 p^i_{t})$ to the drift of  \eqref{ito:log}
and then get the following lower bound (using the definition of $\varphi$ in \eqref{newphi} together with the fact that 
$0\leq\alpha^{i,j}\leq M$ if $j \not = i$)
\begin{equation}
\label{ito:log:2}
\begin{split}
&\sum_{j \in \ES} \biggl[\frac{\lambda^2}{\varepsilon^2} \frac{p^j_{t}}{p^i_{t}}\Bigl(\varphi_{\theta,\delta,\varepsilon}(p^i_t)  + \alpha^{j,i}_t \Bigr)  
-\frac{\lambda}{\varepsilon^2}\Bigl(\varphi_{\theta,\delta,\varepsilon}(p^j_t)+ \alpha^{i,j}_t \Bigr)
\biggr] - \frac{\lambda}{2} \frac{1- p_{t}^i}{p_{t}^i}
-\frac{\lambda^2}{2\varepsilon^2}\frac{1-p^i_t}{p^i_t}
 \\
&\geq 
\frac{\lambda}{\varepsilon^2} \frac{1}{p_{t}^i}  \kappa_0  
\mathbbm{1}_{[0,\delta]}(p_{t}^i) -
{\frac{\lambda}{\varepsilon^2}
\sum_{j \in \ES}
 \Bigl(\kappa_\varepsilon
{\mathbbm 1}_{[0,2\theta]}(p_{t}^j) +
\kappa_{0} 
{\mathbbm 1}_{[0,2\delta]}(p_{t}^j)
 +M
\Bigr)} -\frac{\lambda}{2\varepsilon^2}(\varepsilon^2+\lambda)\frac{1}{p_{t}^i}\\
&\geq  \frac{\lambda}{\varepsilon^2}\Bigl(\kappa_0- \frac{\varepsilon^2+\lambda}{2}  \Bigr)\frac{1}{p_{t}^i}\mathbbm{1}_{[0,\delta]}(p_{t}^i)
-\frac{\lambda(1+\lambda)}{2\delta\varepsilon^2} -\frac{\lambda d}{\varepsilon^2} \bigl(\kappa_\varepsilon + {\kappa_{0}} + M\bigr).
\end{split}
\end{equation}
Hence, integrating \eqref{ito:log} from $0$ to $T$,
adding and subtracting the compensator $\varepsilon^{-2} \lambda^2 \int_{0}^{T} (1-p_{t}^i)/(2p_{t}^i) dt$
  and then taking the exponential, we get
\begin{equation*}
\label{eq:E:Pt:eta}
\begin{split}
&(p_{T}^i)^{\lambda/\varepsilon^2} \exp \biggl( - \frac{\lambda}{\varepsilon} \int_{0}^{T}  \sqrt{\frac{1-p_{t}^i}{p_{t}^i}}
 d \widetilde W_{t}^{i} - \frac{\lambda^2}{2 \varepsilon^2} 
\int_{0}^{T} \frac{1-p_{t}^i}{p_{t}^i} dt 
\biggr)
\\
&\geq (p_{0}^i)^{\lambda/\varepsilon^2}
\exp \biggl(  
\frac{\lambda}{\varepsilon^2}\Bigl(\kappa_0- \frac{\varepsilon^2+\lambda}{2}  \Bigr)\int_0^T \frac{1}{p_{t}^i} \mathbbm{1}_{[0,\delta]}(p_{t}^i)dt
\biggr)
e^{-TC(\delta,\varepsilon,\lambda)}.
\end{split}
\end{equation*}
Since the left-hand side has expectation less than 1, claim \eqref{eq:expfin1} follows.
In order to get \eqref{eq:expfin2}, it suffices to 
replace $\lambda$
by $\varepsilon^2 \lambda$, 
{to observe that the indicator function ${\mathbf 1}_{[0,2\theta]}$ 
in 
\eqref{ito:log:2} has zero value 
if $\inf_{0 \leq t \leq T} p_{t}^{i} > 2 \theta$, and to integrate from $0$ to the first time when ${\boldsymbol p}^{i}$ becomes lower than 
$2\theta$}. 
\end{proof}

\begin{proof}[Proof of Proposition \ref{propexp}]
Throughout the proof, the initial condition $(t_{0},p_{0}) \in [0,T] \times {\mathcal K}$ is implicitly understood in the notation 
${\boldsymbol p}^{\theta,\delta,\varepsilon}$. 
Also, we fix the state $i \in \ES$ and the value of $\lambda \geq 1$
and we make explicit the dependence of the various constants upon the two parameters 
	$\delta$ and $\varepsilon$. However, we do not indicate the fact that the constants may depend on 
	${\mathcal K}$. Below, we use the same notation $\widetilde {\boldsymbol W}^i$
	as in \eqref{eq:tilde w:i}.
	\vskip 4pt

\vskip 4pt

	
	\emph{Step 1.}	
	\textit{a.}
	We first claim that, for any $\eta >0$, there exists $a_{\eta}(\delta,\varepsilon) \in (0,1)$, such that, 
	for all $(t_{0},p_{0}) \in [0,T] \times {\mathcal K}$ and $\theta \in [0,\delta/2]$,
\begin{equation}	
\label{eq:p:eta}
	{\mathbb P}\Bigl(\inf_{t_{0} \le t \le T}  p_{t}^{i,\theta,\delta,\varepsilon} >  a_{\eta}(\delta,\varepsilon) \Bigr) \geq 1-\eta.
	\end{equation}
The proof is a consequence of 
	\eqref{ito:log} and of \eqref{eq:expfin1} (with $\lambda=\varepsilon^2$ and $\kappa_{0}> \varepsilon^2$). Indeed, the former, together with Doob's maximal inequality, yield
\begin{equation*}
\begin{split}
&{\mathbb P} \Bigl( \sup_{t_{0} \le t \le T} 
\bigl[ - \ln(p_{t}^{i,\theta,\delta,\varepsilon}) \bigr] \geq -\ln(a_{\eta}) 
\Bigr) 
\leq \frac{c}{\vert \ln(a_{\eta}) \vert}{\mathbb E} \int_{t_{0}}^T \Bigl( 1+ \frac{1}{p_{t}^{i,\theta,\delta,\varepsilon}} \Bigr) dt. 
\end{split}
\end{equation*}	
for some $c$ only depending on $\varepsilon$ and $\delta$, $M$, $\kappa_{0}$ and {$\kappa_{2}$}. 
Then, \eqref{eq:expfin1}
gives a bound (depending on $\delta$ and $\varepsilon$) for the above right-hand side. 

\textit{b.} Our second step is to prove that, provided that
$\kappa_{0}$ satisfies
\begin{equation}
\label{75}
\kappa_{0} - {1 - \frac{\lambda}2} \geq 4 {\kappa_{2}},
\end{equation}
there exists $\theta_{1}(\delta,\varepsilon,{\lambda})>0$ such that, for any $\theta \leq \theta_{1}(\delta,\varepsilon)$ and any $(t_{0},p_{0}) \in [0,T] \times {\mathcal K}$,
\be
\label{76} 
\E\biggl[\exp\biggl\{\lambda \int_{t_{0}}^T-\varphi'_{\theta}(p_t^{i,\theta,\delta,
	\varepsilon})dt\biggr\}\biggr] \leq 2,
\ee 	
where we have let for convenience 
$\varphi_{\theta}'(r) := - (2 {\kappa_{\varepsilon}}/\theta) {\mathbbm 1}_{[0,2\theta]}(r)$.

Obviously, $\int_{t_{0}}^T\varphi'_{\theta}(p_t^{i,\theta,\delta,\varepsilon})dt$ converges to $0$ in probability
as $\theta$ tends to $0$, uniformly in $(t_{0},p_{0}) \in [0,T] \times {\mathcal K}$ (the other two parameters $\delta$ and $\varepsilon$ being kept fixed), as the indicator function appearing in the definition of $\varphi_{\theta}'$ vanishes for $\theta$ small enough {(it hence suffices to choose $2 \theta \leq a_{\eta}(\delta,\varepsilon)$ for
$a_{\eta}(\delta,\varepsilon)$ as in \eqref{eq:p:eta}, for a given $\eta>0$ as small as needed)}.
In order to prove 
\eqref{76}, we then notice that 
	\be 
	\label{phieps}
	-\varphi'_{\theta}(p_t^{i,\theta,\delta,\varepsilon}) 
	= 2\frac{\kappa_\varepsilon}{\theta} \mathbbm{1}_{[0,2\theta]}(p^{i,\theta,\delta,\varepsilon}_t)
	\leq 2\frac{\kappa_\varepsilon}{\theta}\frac{2\theta}{p^{i,\theta,\delta,\varepsilon}_t}
	\mathbbm{1}_{[0,2\theta]}(p^{i,\theta,\delta,\varepsilon}_t)
	\leq 
	 \frac{4 \kappa_\varepsilon}{p^{i,\theta,\delta,\varepsilon}_t}
	\mathbbm{1}_{[0,\delta]}(p^{i,\theta,\delta,\varepsilon}_t).	
	\ee
Recall now that ${\kappa_\varepsilon=\varepsilon^{-2} \kappa_2}$. Hence, if we choose 
another real $\lambda'$ that {satisfies
$\kappa_{0}\geq 4 {\kappa_{2}}+(1+\lambda')/2$}
 (take for instance that $\lambda'=\lambda+1$ and recall
	$\kappa_{0}\geq 4 {\kappa_{2}}+1+\lambda/2$), 
	 then
	\eqref{phieps} and \eqref{eq:expfin1} yield
	\[ 
	\E\biggl[
	\exp\biggl\{-\lambda' \int_{t_{0}}^T  \varphi'_{\theta}(p_t^{i,\theta,\delta,\varepsilon}) dt
	\biggr\}
	\biggr]	
	\leq 
	\E\biggl[
	\exp\biggl\{ \frac{4\kappa_2 \lambda'}{\varepsilon^2}\int_{t_{0}}^T \frac{1}{p^{i,\theta,\delta,\varepsilon}_t}
	\mathbbm{1}_{[0,\delta]}(p^{i,\theta,\delta,\varepsilon}_t) dt
	\biggr\}
	\biggr]	
	\leq C_{1}(\delta,\varepsilon,\lambda'),
	\]
	where $C_{1}(\delta,\varepsilon,\lambda')$ is a constant independent of $\theta$ and depending on $(t_{0},p_{0})$ through ${\mathcal K}$ only\footnote{{We do not keep track of the parameters $\kappa_{0}$, $\kappa_{2}$,
	${\mathcal K}$, $M$, $T$ and $d$ in the constants.}}.
Combining the above upper bound with the fact that $\int_{t_{0}}^T\varphi'_{\theta}(p_t^{i,\theta,\delta,\varepsilon})dt$
tends to $0$ in probability (uniformly in $(t_{0},p_{0}) \in [0,T] \times {\mathcal K}$), we easily derive
\eqref{76}.
\vskip 4pt

\emph{Step 2.}
The goal of this step is to address a similar result to \eqref{76} but with $\varphi_{\theta}'$ 
replaced by $\varphi_{\delta}'$, defined 
as $\varphi_{\delta}'(r) := - (2 \kappa_{0}/\delta)  
 \mathbbm{1}_{[0,2\delta]}(r)$.
\vskip 4pt

\textit{a.}
The first step is to notice that 
\begin{equation}
\label{eq:secondstep:a}
\begin{split}
-\varphi'_{\delta}(p_t^{i,\theta,\delta,\varepsilon}) 
&\leq 2\frac{\kappa_0}{\delta}\frac{2\delta}{p^{i,\theta,\delta,\varepsilon}_t}
\mathbbm{1}_{[0,2\delta]}(p^{i,\theta,\delta,\varepsilon}_t)
\leq 
\frac{4 \kappa_0}{p^{i,\theta,\delta,\varepsilon}_t}
\mathbbm{1}_{[0,\delta]}(p^{i,\theta,\delta,\varepsilon}_t)
+ \frac{4 \kappa_0}{p^{i,\theta,\delta,\varepsilon}_t}
\mathbbm{1}_{[\delta,2 \delta]}(p^{i,\theta,\delta,\varepsilon}_t).	
\end{split}
\end{equation}	
\vskip 4pt

\textit{b.} We address the first term in the right-hand side of 
\eqref{eq:secondstep:a}.
To do so, we need a finer lower bound on 
the coordinates of ${\boldsymbol p}^{\theta,\delta,\varepsilon}$ and hence we must 
revisit 
the proof of Proposition 
\ref{expfin}. As in the first step, we fix some $\eta>0$ and, for {$a_{\eta}:=a_{\eta}(\delta,\varepsilon)$ as therein}, we consider the event 
{$A_{\eta}^{1,i} := \{ \inf_{t_{0} \le t \le T}  p_{t}^{i,\theta,\delta,\varepsilon} > a_{\eta}\}$}. 
Obviously, 
\eqref{eq:p:eta}
says that 
${\mathbb P}(A_{\eta}^{1,i}) \geq 1-\eta$.

Next, we define the event 
\begin{equation*}
{A^{2,i}} = \biggl\{ \forall t \in [t_{0},T], \quad \varepsilon \int_{t_{0}}^t \sqrt{\frac{1-p_{s}^{i,\theta,\delta,\varepsilon}}{p_{s}^{i,\theta,\delta,\varepsilon}}} d
\widetilde W_{s}^i \geq -1 - \frac{\varepsilon}2 \int_{t_{0}}^t \frac{1-p_{s}^{i,\theta,\delta,\varepsilon}}{p_{s}^{i,\theta,\delta,\varepsilon}} ds\biggr\}. 
\end{equation*}
We observe that the complementary reads
\begin{equation*}
\begin{split}
({A^{2,i}})^{\complement} = \biggl\{ \exists t \in [t_{0},T] : 
\exp \biggl( -   
 \int_{t_{0}}^t \sqrt{\frac{1-p_{s}^{i,\theta,\delta,\varepsilon}}{p_{s}^{i,\theta,\delta,\varepsilon}}} d
\widetilde W_{s}^i  - \frac1{2 }
\int_{t_{0}}^t \frac{1-p_{s}^{i,\theta,\delta,\varepsilon}}{p_{s}^{i,\theta,\delta,\varepsilon}} ds\biggr) 
> \exp \bigl( {\varepsilon}^{-1} \bigr) \biggr\},
\end{split}
\end{equation*}
from which we get by Doob's inequality that 
${\mathbb P}(A^{{2,i}}) \geq 1 - \exp(-\varepsilon^{-1})$. 

We now work on ${(\cap_{j \in \ES}A^{1,j}_{\eta})} \cap A^{2,i}$ for $2 \theta \leq a_{\eta}$. By 
combining 
\eqref{ito:log}
and 
\eqref{ito:log:2}, we get (choosing $\lambda=\varepsilon^2$ therein and noticing that, since we work on 
${\cap_{j \in \ES} A^{1,j}_{\eta}}$, we can remove the 
second indicator function in the second line of \eqref{ito:log:2}):
\begin{equation*}
\begin{split}
 \ln \bigl(p_{t}^{i,\theta,\delta,\varepsilon} \bigr)  &\geq 
  \ln \bigl(p_{0}^{i,\theta,\delta,\varepsilon} \bigr) - 1
 \\
&\hspace{15pt} +  \int_{t_{0}}^t 
 \bigl( \kappa_{0}- {\varepsilon} \bigr) \frac1{p_{s}^{i,\theta,\delta,\varepsilon}}
{\mathbbm 1}_{[0,\delta]} \bigl( p_{s}^{i,\theta,\delta,\varepsilon} \bigr) ds - \Bigl( \frac{{\varepsilon}}{\delta} + d \bigl( \kappa_{0}+ M \bigr) \Bigr) T. 
\end{split}
\end{equation*}
Hence, for $\kappa_{0} \geq \varepsilon$, we can find a constant 
$C_{2} \geq 0$ ({only depending on $\kappa_{0}$, $M$, $T$ and 
$d$}) such that, on ${(\cap_{j \in \ES}A^{1,j}_{\eta})} \cap A^{2,i}$, for 
$2\theta \leq a_{\eta} \leq \varepsilon \leq \delta$, 
 $i \in \ES$ and $t \in [0,T]$, we have
 $p_{t}^{i,\theta,\delta,\varepsilon} \geq \exp(-C_{2} )$. 
{(Observe indeed that, in 
 \eqref{eq:p:eta}, we can always assume that $a_{\eta}(\delta,\varepsilon) \leq \min(\varepsilon,\delta)$.)}
\vskip 4pt

\textit{c.}
%
Return back to the first term in the right-hand side of 
\eqref{eq:secondstep:a}.
By  \eqref{eq:expfin2} (applied with $\lambda$ replaced by $8 \lambda'$ for $\lambda'>\lambda$), if 
$\kappa_0 \geq \varepsilon^2 (1+ 8\lambda')$,
which is for instance true 
if 
$\kappa_0 \geq 2$
and
 $8 \lambda'  \varepsilon^2 \leq 1$ (in turn the latter is true if 
 ${8 (\lambda+1)} \varepsilon^2 \leq 1$ 
 and $\lambda'-\lambda=1$), then 
\be
\label{76:b}
{\mathbb E}\biggl[ \exp \biggl\{4 \lambda' \kappa_0 \int_{t_{0}}^T  \frac{1}{p^{i,\theta,\delta,\varepsilon}_t}
\mathbbm{1}_{[0,\delta]}(p^{i,\theta,\delta,\varepsilon}_t) dt\biggr\}
{\mathbbm 1}_{\{ \inf_{t_{0} \le t \le T} p_{t}^{i,\theta,\delta,\varepsilon}
> 2 \theta \}}
\biggr] \leq C_{3}(\delta,\lambda'),
\ee	
where $C_{3}(\delta,\lambda')$ is {non-increasing} with $\delta$. 
Under the same condition $\kappa_{0} \geq 2$ and 
 ${16} \lambda ' \varepsilon^2 \leq 1$,
\eqref{eq:expfin1} (applied with $\lambda$ replaced by $8 \varepsilon^2 \lambda'$)
yields
\[{\mathbb E}\biggl[ \exp \biggl\{4 \lambda' \kappa_0 \int_{t_{0}}^T  \frac{1}{p^{i,\theta,\delta,\varepsilon}_t}
\mathbbm{1}_{[0,\delta]}(p^{i,\theta,\delta,\varepsilon}_t) dt\biggr\}
{\mathbbm 1}_{\{ \inf_{t_{0} \le t \le T} p_{t}^{i,\theta,\delta,\varepsilon}
\leq 2 \theta \}}
\biggr] \leq C_{4}(\delta,\varepsilon,\lambda') {\mathbb P} \bigl( ({A_{\eta}^{1,i}})^{\complement} \bigr)^{1/2}.
\]	
at least whenever $2 \theta \leq a_{\eta}$. Choosing 
$\eta$ such that 
$C_{4}(\delta,\varepsilon,\lambda') \eta^{1/2} \leq C_{3}(\delta,\lambda')$
 {and allowing}
for a new value of $C_{3}(\delta,\lambda')$, we may remove the second indicator function
in \eqref{76:b}. Then,
%
we can easily change the first indicator function in \eqref{76:b}
into ${\mathbbm 1}_{[0,2\delta]}$ by noticing that 
$r^{-1}{\mathbbm 1}_{[0,2\delta]}(r) \leq 
r^{-1}{\mathbbm 1}_{[0,\delta]}(r)
+
\delta^{-1}$. For a new value of $C_{3}(\delta,\lambda')$ (as long as it remains {non-increasing} with $\delta$), we then have
\be 
\label{eq:79}
{\mathbb E}\biggl[ \exp \biggl\{ 4 \lambda' \kappa_0 \int_{t_{0}}^T  \frac{1}{p^{i,\theta,\delta,\varepsilon}_t}
\mathbbm{1}_{[0,2\delta]}(p^{i,\theta,\delta,\varepsilon}_t) dt\biggr\}
\biggr] \leq C_{3}(\delta,\lambda').
\ee	
Recall $C_{2}$ from step \textit{2b} and deduce by H\"older's inequality that, for 
$2\theta \leq a_{\eta} \leq \varepsilon \leq \delta$
and  
$3\delta \leq \exp(-C_{2})$,  
\begin{equation*}
\begin{split}
&{\mathbb E}\biggl[ \exp \biggl\{ - \lambda   \int_{t_{0}}^T   
\varphi'_{\delta}(p_t^{i,\theta,\delta,\varepsilon}) 
dt\biggr\}\biggr]
\\
&\leq 
{\mathbb E}\biggl[ \exp \biggl\{ 4 \lambda \kappa_0 \int_{t_{0}}^T  \frac{1}{p^{i,\theta,\delta,\varepsilon}_t}
\mathbbm{1}_{[0,2\delta]}(p^{i,\theta,\delta,\varepsilon}_t) dt\biggr\}\biggr]
\\
&=
1+
{\mathbb E}\biggl[ \exp \biggl\{ 4 \lambda \kappa_0 \int_{t_{0}}^T  \frac{1}{p^{i,\theta,\delta,\varepsilon}_t}
\mathbbm{1}_{[0,2\delta]}(p^{i,\theta,\delta,\varepsilon}_t) dt\biggr\}
{\mathbbm 1}_{( {(\cap_{j \in \ES}A^{1,j}_{\eta})} \cap A^{2,i})^{\complement}} 
\biggr] 
\\
&\leq 1+ C_{3}(\delta,\lambda')^{\lambda/\lambda'} \bigl( \eta {d} + \exp(-\varepsilon^{-1}) \bigr)^{1-\lambda/\lambda'}. 
\end{split}
\end{equation*}
For $3 \delta \leq \exp(-C_{2})$, we may choose 
$\varepsilon \leq \hat{\varepsilon}(\delta)  \wedge \delta$ with 
$ C_{3}(\delta,\lambda')^{\lambda/\lambda'} [ 2 \exp(-\hat \varepsilon^{-1}(\delta))]^{1-\lambda/\lambda'}
=1$ and then $\eta  {d} \leq \exp(-\varepsilon^{-1})$, with $a_{\eta} \leq \varepsilon$, and $\theta \leq \min(a_{\eta}/2,\theta_{1}(\delta,\varepsilon,\lambda))$ ({with $\theta_{1}(\delta,\varepsilon,\lambda)$ as in Step \textit{1b}}).
We get that the above right-hand side is less than 2, which is 
the analogue of \eqref{76}. 
\vskip 4pt

\textit{d.} By collecting 
\eqref{eq:bound:phiprime}, 
\eqref{76}
and the above conclusion
with $\lambda'=\lambda+1$
and
by applying Cauchy-Schwarz inequality, we obtain
\be 
\label{80}
\E\biggl[
\exp\biggl\{ - \frac{\lambda}{2} \int_{t_{0}}^T  \varphi'_{\theta,\delta,\varepsilon}(p_t^{i,\theta,\delta,\varepsilon}) 
dt
\biggr\}
\biggr]	
\leq 2,
\ee 
for the following choices: $ \delta \leq \bar \delta_{0}:= \exp(-C_{2})/3$, 
$\kappa_0 \geq \bar \kappa_{0} :=\max(2,4  {\kappa_{2}} + 1+\lambda/2)$,
$\varepsilon \leq \min( 
1/ {\sqrt{16 \lambda+16}}, \delta,\sqrt{\kappa_2/\kappa_0}, \hat{\varepsilon}(\delta))$
and 
$\theta \leq \min(\theta_{1}(\delta,\varepsilon,\lambda),a_{\eta})$, 
for $\eta   {d}  \leq \exp(-\varepsilon^{-1})$ such that 
$a_{\eta} \leq \varepsilon$. We recall that the condition 
$\epsilon^2 \leq \sqrt{\kappa_2/\kappa_0}$ is required to ensure that $\phi_{\theta, \delta, \epsilon}$ is non-increasing. This is one
part of the inequality in the definition of the term $\Psi$ showing up in \eqref{bigexp}. 
In fact, the term with $\varphi_{\theta,\delta}(p_t^{i,\theta,\delta,\varepsilon})$ in $\Psi$ is bounded in the same way
 {since $\varphi_{\theta,\delta,\varepsilon} \leq -(\varphi_{\theta}'+\varphi_{\delta}')$ for 
$2\theta \leq \delta \leq 1$}, yielding 
(for the same range of parameters)
\be
\label{81}
\E\left[
\exp\left\{\frac{\lambda}{2} \int_{t_{0}}^T 
\varphi_{\theta,\delta,\varepsilon}(p_t^{i,\theta,\delta,\varepsilon})
dt
\right\}
\right]	
\leq 2.
\ee
\vskip 4pt

\emph{Step 3.}
We now handle the term with 
$1/p_t^{i,\theta,\delta,\varepsilon}$ in 
\eqref{bigexp}.
\vskip 4pt
 
\textit{a.} On the one hand,  {recalling that $\kappa_{0} \geq 2$, using \eqref{eq:79} and 
arguing as above}, we obtain
\[
\E\biggl[
\exp\biggl\{\lambda \int_{{t_{0}}}^T 
\frac{1}{p_t^{i,\theta,\delta,\varepsilon}}
\mathbbm{1}_{[0,\delta]}(p_t^{i,\theta,\delta,\varepsilon})
dt
\biggr\}
\biggr]	
\leq 2,
\]
again for $\kappa_0,\kappa_{1},\theta,\delta,\varepsilon$ as in item $d$ of the second step.
\vskip 4pt

\textit{b.}
On the other hand, by following the second step, 
for $2\theta \leq a_{\eta} \leq \varepsilon \leq \delta$, we get  
\begin{equation*}
\begin{split}
&\E\biggl[
\exp\biggl\{\lambda  \int_{t_{0}}^T 
\frac{1}{p_t^{i,\theta,\delta,\varepsilon}}
\mathbbm{1}_{[\delta,1]}(p_t^{i,\theta,\delta,\varepsilon})
dt
\biggr\}
\biggr]	
\\
&\leq 
\exp\bigl(  \lambda  T e^{C_{2}} \bigr)  
	+
\E\biggl[
\exp\biggl\{\lambda  \int_{t_{0}}^T 
\frac{1}{p_t^{i,\theta,\delta,\varepsilon}}
\mathbbm{1}_{[\delta,1]}(p_t^{i,\theta,\delta,\varepsilon})
dt
\biggr\}
{\mathbbm 1}_{({(\cap_{j \in \ES}A^{1,j}_{\eta})} \cap A^{2,i})^{\complement}
} 
\biggr]	
\\
& \leq 
\exp\bigl( \lambda  T e^{C_{2}} \bigr)  
	+
\exp \bigl( \frac{\lambda T}{\delta} \bigr) 
\bigl( \eta {d} + \exp(- \varepsilon^{-1}) \bigr).
\end{split}	
\end{equation*}
Following item \textit{2c}, we can render the last term in the right-hand side 
less than 1.  
\vskip 4pt
 
\textit{c.}
We then combine items \textit{3a} and \textit{3b} by Cauchy-Schwarz inequality. And, then by H\"older inequality, 
we gather all the three cases addressed in item \textit{2d} and in this third step to get 
\eqref{bigexp}, provided we replace $\lambda$ therein by 
$\lambda/6$ and then fix 
the various parameters as in item \textit{2d} {with the additional constraint that 
$2 \exp(\lambda T/\delta) \exp(- \varepsilon^{-1}) \leq 1$ (which is equivalent to 
$\varepsilon \leq (\ln(2) + \lambda T/\delta)^{-1}$}. 
\vskip 4pt

\emph{Step 4.}
{We now want to prove \eqref{bigexp2} and \eqref{bigexp3}.}
Throughout the step, we fix the value of $\lambda$.
\vskip 4pt

\textit{a.} We recall that, for $0< \delta \leq \bar \delta_{0}$, $\hat{\varepsilon}(\delta)$ is defined as
\begin{equation*}
\hat{\varepsilon}(\delta) = \min\biggl(\frac1{4 \sqrt{\lambda}},\sqrt{\frac{\kappa_2}{\kappa_0}},\delta, \Bigl(\ln\bigl( 2 [C_{3}(\delta,\lambda+1)]^{\lambda} \bigr)  \Bigr)^{-1}, {\bigl(\ln ( 2)+ \lambda T/\delta   \bigr)^{-1}} \biggr),
\end{equation*}
where $C_{3}(\delta,\lambda+1)$ is a {non-increasing} function of $\delta$. (We omit to specify the dependence of 
$\hat{\varepsilon}(\delta)$ upon $\lambda$ and ${\mathcal K}
$.)
Clearly, $\hat{\varepsilon}$ is  {non-decreasing} on $(0,\bar \delta_{0}]$, takes positive values and has $0$ as limit in $0$. We then define
\[
\bar{\varepsilon}(\delta):= \int_0^\delta \hat{\varepsilon}(\delta') d\delta '.
\]
It is straightforward to verify that, for $\delta \in (0,\bar \delta_{0}]$, 
$0< \bar{\varepsilon}(\delta) \leq \delta \hat{\varepsilon}(\delta) < \hat{\varepsilon}(\delta)$ (assume 
without any loss of generality that $\bar \delta_{0} \leq 1$). Moreover, $\bar{\varepsilon}$
extends by continuity to $[0,\bar{\delta}_{0}]$, letting $\bar{\varepsilon}(0)=0$, and the extension, still denoted by $\bar{\varepsilon}$, 
is continuous and strictly increasing.

\vskip 4pt

\textit{b.} We now define $\hat{\delta} : [0,\bar \varepsilon_{0}] \ni \varepsilon \mapsto 
\hat{\delta}(\varepsilon) \in [0,\bar \delta_{0}]$ as the  {converse} of the mapping 
$\bar{\varepsilon} : [0,\bar \delta_{0}] \ni \delta \mapsto \bar{\varepsilon}(\delta) \in [0,\bar \varepsilon_{0}]$, 
where 
$\bar \varepsilon_0= \bar{\varepsilon}(\bar \delta_0)$. Conclusion \eqref{bigexp2} hence follows from \eqref{bigexp}, noticing that, 
for any $\varepsilon \in [0,\bar \varepsilon_{0}]$, $\varepsilon=\bar \varepsilon(\hat{\delta}(\varepsilon)) < \hat \varepsilon(\hat{\delta}(\varepsilon))$, from which we indeed deduce that
{$\Psi(\lambda,\theta,\hat{\delta}(	\varepsilon),\varepsilon,{\mathcal K})
 \leq \bar C$, for $\theta \leq \hat{\theta}(\hat{\delta}(\varepsilon),\varepsilon)$}. 
 {As for 
 \eqref{bigexp3}, it follows from Step 2b, recalling that, under our choice
 for $\eta$, the probability of 
 $(\cap_{j \in \ES}A^{1,j}_{\eta}) \cap A^{2,i}$
 is less than $2 \exp(-\varepsilon^{-1})$.}
\end{proof}

\begin{proof}[Proof of Lemma \ref{lem:moments:p,q:l}]
We first prove \eqref{momentp}. For simplicity, we write ${\boldsymbol p}$ for ${\boldsymbol p}_{[t_{0},p_{0}]}^{\theta,\hat{\delta}(\varepsilon),\varepsilon}$ and $\varphi$ for $\varphi_{\theta,\delta,\varepsilon}$. Using 
the same notation as in 
\eqref{eq:tilde w:i}
and applying It\^o's formula, we expand
\begin{align*}
d \frac{1}{(p^i_t)^\ell}&= \biggl\{-\ell \frac{1}{(p^i_t)^{\ell+1}}
\sum\nolimits_{j \in \ES} \left( p^j_t(\varphi(p^i_t)+\alpha^{j,i}_t) - p^i_t (\varphi(p^j_t)+\alpha^{i,j}_t)\right) 
+\ell(\ell+1)\frac{\varepsilon^2}{2} \frac{1-p^i_t}{(p^i_t)^{\ell+1}}\biggr\}dt\\
&\quad -\ell {\varepsilon}  \frac{\sqrt{p^i_t(1-p^i_t)}}{(p^i_t)^{\ell+1}} d\widetilde{W}^i_t.
\end{align*}	
Letting
\[
\mathcal{E}_t = \exp \biggl\{\int_{t_{0}}^t \Bigl(\ell\sum\nolimits_{j \in \ES}  \varphi(p^j_s)+ 
\ell(\ell+1)\frac{\varepsilon^2}{2} {\bigl(\frac{1}{p^i_s}-1\bigr)}\Bigr)ds
\biggr\}, \quad t \in [t_{0},T], 
\]	
we have
\begin{equation*}
\begin{split}
d\left(\frac{\mathcal{E}_t^{-1}}{(p^i_t)^\ell}\right)
&= - \ell \frac{\mathcal{E}_t^{-1}}{(p^i_t)^{\ell+1}}
\sum\nolimits_{j \not  = i}  p^j_t(\varphi(p^i_t)+\alpha^{j,i}_t) dt
+\ell {\sum\nolimits_{j \not =i}} \frac{\mathcal{E}_t^{-1}}{(p^i_t)^\ell}
   \alpha^{i,j}_t dt + dm_t
\\
&\leq  {(d-1)} {\ell M} \frac{\mathcal{E}_t^{-1}}{(p^i_t)^\ell} dt +dm_t,
\end{split}
\end{equation*}	
where $m_t$ is a local martingale. By a standard localization, {we deduce that there 
exists an increasing sequence of (localizing) stopping times $(\sigma_{n})_{n \geq 1}$ converging to $T$ such that}
\[
\forall t \in [t_{0},T], \quad \E\left[\frac{\mathcal{E}_{{t \wedge \sigma_{n}}}^{-1}}{(p^i_{{t \wedge \sigma_{n}}})^\ell}\right]
\leq C + \ell M \int_{t_{0}}^t \E\left[\frac{\mathcal{E}_{{s \wedge \sigma_{n}}}^{-1}}{(p^i_{{s \wedge \sigma_{n}}})^\ell}  \right]ds,
\]	
for a constant $C$ that is allowed to vary from line to line as long as it only depends on the parameters quoted in 
the statement of the lemma ({in particular, it is independent of $n$}). 
Thus Gronwall's lemma and {then Fatou's lemma (letting $n$ tend to $\infty$)} give
\[
\sup_{{t_{0}}\leq t\leq T} \E\left[\frac{\mathcal{E}_t^{-1}}{(p^i_t)^\ell}\right] \leq C.
\] 
 Applying Cauchy-Schwarz inequality and then the above inequality, with $\ell$ replaced by $2\ell$, we obtain
 \begin{equation*}
 \sup_{{t_{0}}\leq t\leq T} \E\left[\frac{1}{(p^i_t)^\ell}\right]
 \leq C
 \biggl(\E\biggl[\exp \bigg\{\int_{{t_{0}}}^T \bigg(2\ell\sum\nolimits_{j \in \ES}  \varphi(p^j_s)+ 
 \ell(2\ell+1)  \frac{1}{p^i_s}\bigg)ds
 \bigg\}\biggr]\biggl)^{1/2},
  \end{equation*}
 which is bounded by a constant thanks to \eqref{bigexp}, choosing $\lambda$ in terms of $\ell$ and $d$.
  
  The proof of \eqref{momentq} follows from the same argument and then from Doob's maximal inequality (to pass the supremum inside the expectation), see for instance \cite[Proof of Proposition 2.3]{mfggenetic}.
\end{proof}

\section{Uniqueness for the master equation}
\label{sec:uniqueness:master:equation}

Here, our aim is to show Theorem \ref{thm:uniqueness:master}, namely that $\mathcal{V}$, the value function of the potential game,  is the unique viscosity solution of the HJ equation
\eqref{eq:hjb:inviscid}
and that its derivative $V=\mathfrak{D} \mathcal{V}$ is the unique solution, in a suitable class which we will determine,  of the conservative form 
\eqref{eq:master:equation:conservative:full}
of the master equation of the MFG. In this regard, it is worth emphasizing that we work
below with the local coordinates $(x_{1},\dots,x_{d-1},x^{-d})=\left(p_1,\dots,p_{d}\right)$ for 
$p \in {\mathcal S}_{d}$ (and thus $x \in \hat{\mathcal S}_{d}$). We recall that, for any $x\in\widehat{\mathcal{S}}_d$, we have $\widehat{\mathcal{V}}(t,x) = \mathcal{V}(t,\check{x})$, where  $\check{x}= (x_1,\dots,x_{{d-1}}, x^{-d})$, and we denote $x^{-d}=1-\sum_{j\in\ESd} x_j$. Following 
\eqref{hjbchart}, 
\eqref{eq:hjb:inviscid}
may be indeed rewritten
\begin{subnumcases}{}
\label{HJBPOT}
\partial_t \widehat{ \mathcal V} + \widehat{\mathcal{H}}\bigl(x, D_x \widehat {\mathcal V}\bigr) + \widehat F(x) =0, 
\\
\label{HJBPOT:bc}
\widehat{\mathcal{V}}(T,x)=\widehat G(x), 
\end{subnumcases}
for $t \in [0,T]$ and $x \in \mathrm{Int}(\widehat {\mathcal{S}}_d)$, and 
with $\widehat{\mathcal H}$ as in 
\eqref{eq:mathcal H:5:9}.
Its derivative $Z =D_x \widehat{\mathcal V}$
should satisfy
\eqref{eq:master:equation:conservative:full}, at least when the latter is formulated in local coordinates, namely 
\be
\label{CONS}
\left\{
\begin{array}{l}
\partial_t Z^i + \partial_{x_{i}} \bigl[\widehat{\mathcal{H}}(x,Z) +\widehat{ F}(x)\bigr] 
=0,
\\
Z^i(T,x)= \partial_{x_{i}} \widehat{G}(x),
\end{array}
\right.
\ee
the latter reading as a multidimensional {hyperbolic} system of PDEs. 
Let us point out a common difficulty in the study of the above two equations: Both are set in a bounded domain, but there are no boundary conditions in space, which is due to the fact that the dynamics of the forward characteristics of the MFG system do not see the boundary of the simplex when starting from its interior.

Concerning the HJ equation \eqref{HJBPOT}, there are no $\mathcal{C}^1$ solutions in general, which prompts us to consider viscosity solutions. Below, we first handle the HJ equation of the MFCP and then turn to the well-posedness of the conservative form of the master equation. 
The idea for proving uniqueness of the latter is to construct a correspondence between weak solutions in a suitable class and viscosity solutions of the HJ equation.

\subsection{HJ equation for the MFCP}
{In this subsection, we assume that $F$ anf $G$ are just Lipschitz-continuous}. As we have just said, 
the HJ equation \eqref{HJBPOT} is set in a bounded domain but without any boundary conditions in space. 
We hence define viscosity solutions in the interior of the simplex only:
\begin{defn}
\label{defvisco}
A function $v\in \mathcal{C}([0,T)\times \Int) $ (hence defined in local coordinates) is said to be:
\begin{itemize}
\item[(i)] a \emph{viscosity subsolution} of \eqref{HJBPOT} on $[0,T) \times \Int$ if, for any 
$\psi\in \mathcal{C}^1([0,T)\times \Int)$, 
\be 
-\partial_t \psi(\bar{t},\bar{x}) - \widehat{\H}\bigl(\bar{x}, D_x\psi(\bar{t},\bar{x})\bigr)- \widehat F(\bar{x}) \leq 0,
\ee
at every  
$(\bar{t},\bar{x})\in [0,T)\times \Int$ which is a local maximum of $v-\psi$ on 
$[0,T)\times \Int$;
\item[(ii)] a \emph{viscosity supersolution} of \eqref{HJBPOT} on $[0,T) \times \Int$ if, for any 
$\psi\in \mathcal{C}^1([0,T)\times \Int)$, 
\be 
-\partial_t \psi(\bar{t},\bar{x}) - \widehat{\H}\bigl(\bar{x}, D_x\psi(\bar{t},\bar{x})\bigr)-\widehat F(\bar{x}) \geq 0,
\ee
at every  
$(\bar{t},\bar{x})\in [0,T)\times \Int$ which is a local minimum of $v-\psi$ on 
$[0,T)\times \Int$;
\item[(iii)]a \emph{viscosity solution} of \eqref{HJBPOT} on $[0,T) \times \Int$ if it is both  a viscosity subsolution and a viscosity supersolution of \eqref{HJBPOT} in $\Int$.
\end{itemize}
\end{defn}

In order to prove uniqueness of viscosity solutions, in absence of boundary conditions in space, we must use the fact that the forward characteristics, given by 
an equation of the type 
\eqref{eq:fp}
with 
$\balpha$ bounded therein,  
do not leave the interior of the simplex. The result is the following:

\begin{thm}[Comparison Principle]
\label{thmcomparison}
Let $u,v$ be Lipschitz continuous in $[0,T]\times \widehat{\mathcal S}_d$, $u$ be a  viscosity subsolution and $v$ be a  viscosity supersolution, respectively, of \eqref{HJBPOT} in $\Int$. 
If $u(T,x)\leq v(T,x)$ for any $x\in\widehat{\mathcal S}_d$, then $u(t,x)\leq v(t,x)$ for any $t\in[0,T]$ and
$x\in\widehat{\mathcal S}_d$.
\end{thm}

Before giving the proof, we state an immediate consequence.

\begin{cor}
\label{CORUNIQ}
There exists a unique viscosity solution of \eqref{HJBPOT} in $\Int$ 
{that} is Lipschitz continuous in $[0,T]\times \widehat{\mathcal S}_d$
and satisfies the terminal condition  \eqref{HJBPOT:bc}. It is the value function $\mathcal{V}$ of the MFCP.
\end{cor}

\begin{proof}
Uniqueness holds in $\widehat{\mathcal S}_d$ by the above theorem. The fact that the value function is a viscosity solution in $\Int$ is given by Theorem 7.4.14 of \cite{cannarsa}
(as we already accounted in the statement of Proposition \ref{prop:solution:mfc:zero:epsilon}).
\end{proof}

\begin{proof}[Proof of Theorem \ref{thmcomparison}]
We borrow ideas from the proofs of Theorem 3.8 and Proposition 7.3 in \cite{porrettaricciardi}.
The idea is to define a supersolution $v_h$ that dominates $u$ at points near the boundary, for any $h$, and then use the comparison principle and pass to the limit in $h$. The parameter $h$ is needed to force $v_h$ to be infinity at the boundary of the simplex. Since the simplex has corners, the distance {to the boundary} is not a smooth function, so the first step is to construct a nice test function that goes to $0$ as $x$ approaches the boundary. Roughly speaking, we consider the product of the distances to the faces of the simplex, and then take its logarithm. 
\vskip 4pt

\emph{Step 1}. Let $\rho_i(x)$, for $x\in\Int$,  be the distance from $x$ to the hyperplane $\{y \in {\mathbb R}^{{d-1}} : y_i=0\}$, for $i \in \ESd$, and $\rho_d(x)$ be the distance to $\{y \in {\mathbb R}^{{d-1}} : \sum_{l=1}^{{d-1}} y_l=1\}$. Specifically, for  $x\in\Int$, we have
\[
\rho_i(x) = 
\begin{cases}
x_{i} \qquad & i \in \ESd,
\\
{x^{-d}}/
\sqrt{{d-1}}
\qquad& i=d,
\end{cases}
\]
{where we recall that $x^{-d}=1 - \sum_{l \in \ESd} x_{l}$}.
Clearly $\rho_i\in \mathcal{C}^{\infty}(\Int)$. 

Since $u$ and $v$ are Lipschitz-continuous, we may let $R:= \max\{\|D_x u\|_\infty, \|D_x v\|_\infty\}$, 
{which is licit since the}
gradients {are} defined almost everywhere.
Hence it is easy to show\footnote{{In short, the argument is as follows: If 
$\psi$ is a continuously differentiable function such that $u-\psi$ has a minimum 
at some point $(\bar t ,\bar x) \in [0,T) \times \Int$, then necessarily $\|D_{x} \psi(\bar t,\bar x)\|_{\infty} \leq R$
and similarly when $(\bar t,\bar x)$ is a maximum of $v-\psi$.}} that  $u$ and $v$ are viscosity subsolution and supersolution, respectively, in $\Int$, of the modified HJB equation
\be
\label{HJBR}
\partial_t \widehat{ \mathcal V} + \widehat{\mathcal{H}}_{2R}\bigl(x, D_x \widehat {\mathcal V}\bigr) + \widehat F(x) =0,
\ee
{with
$\widehat{\mathcal{H}}_{2R}(x,z)  = \sum_{k \in \ESd} x_k \hat{H}^k_{2R}(z) +  x^{-d} \hat{H}^d_{2R}(z)$, 
where
$\widehat{H}^i_{2R}$, for $i \in \ES$, is 
given by 
\eqref{eq:theta}}, with $M$ replaced by $2R$ therein and also in the definition 
{\eqref{eq:astar}}
of $a^\star$, which we denote here by $a^\star_{2R}$ {(see also 
\eqref{eq:hamiltonians}
for the way the latter shows up in the Hamiltonian)}. This modified Hamiltonian has the property that $\widehat{\mathcal{H}}_{2R}(x,z)
=\widehat{\mathcal{H}}(x,z)$ for any $z\in\R^{{d-1}}$ such that $|z|\leq R$, and is further globally Lipschitz continuous in $(x,z)$ and  concave in $z\in \R^{{d-1}}$, even if not strictly. 
Now, we show that there exists a constant $C_R$, depending on $R$, such that 
\be 
\label{ineqDD}
\bigl\langle D_z \widehat{\mathcal H}_{{2R}}(x,z) , D\rho_i(x) \bigr\rangle \geq - C_R \rho_i(x),
\ee
for any $i \in \ES$, $x\in\Int$ and $z\in\mathbb{R}^{{d-1}}$. 
Indeed, 
we have
\[
\partial_{x_j}\rho_i(x) = 
\begin{cases}
\delta_{i,j} \qquad & i \in \ESd,
\\
 -1/\sqrt{{d-1}}
 \qquad& i=d.
\end{cases}
\]
{Similar to 
\eqref{eq:derivative:H}},
we also have,  for $j \in \ESd$, 
\[
\partial_{z_j} {\widehat{\mathcal{H}}_{2R}}(x,z) = 
\sum_{k \in \ESd} \bigl(x_{k} a^\star_{2R}(z_{k}-z_{j}) -x_{j} a^\star_{2R}(z_{j}-z_{k}) \bigr) +
{x^{-d}} a^\star_{2R}(-z_j) - x_{j} a^\star_{2R}(z_{j}). 
\]
Hence, for $i \in \ESd$,
\[
\langle D_z 
{\widehat{\mathcal{H}}_{2R}}(x,z) , D\rho_i(x)\rangle= 
\partial_{z_i} {\widehat{\mathcal{H}}_{2R}}(x,z)
\geq -2R({d-1})x_i = -2R({d-1}) \rho_i(x),
\]
while {(noticing that the contribution of the first sum in the expansion of $D_{z} \widehat{\mathcal{H}}_{2R}$ is null in the computation below)}
\begin{align*}
\langle D_z {\widehat{\mathcal{H}}_{2R}}(x,z) , D\rho_d(x)\rangle 
&= \frac{1}{\sqrt{{d-1}}}\sum_{j \in \ESd} \biggl(x_{j} a^\star_{2R}(z_{j}) - 
{x^{-d}}a^\star_{2R}(-z_{j})\bigg)
\\
&\geq -2R \frac{{d-1}}{\sqrt{{d-1}}} 
{x^{-d}}
= -2R({d-1}) \rho_d(x),
\end{align*}
and thus \eqref{ineqDD} holds with $C_R=2R({d-1})$. 
\vskip 4pt

\emph{Step 2}. For any $h>0$, let
\[
v_h(t,x) :=v(t,x) - h^2 \sum_{i \in \ES} \ln \bigl(\rho_i(x)\bigr) + h (T-t), 
\quad {(t,x) \in [0,T] \times \Int}.
\]
We claim that $v_h$ 
is a viscosity supersolution of \eqref{HJBR} on $[0,T) \times \Int$. 
Let then 
$\psi\in \mathcal{C}^1([0,T)\times \Int)$, 
and  
$(\bar{t},\bar{x})\in [0,T)\times \Int$ be  a local minimum of $v_h-\psi$ on 
$[0,T)\times \Int$. Since $v$ is a viscosity supersolution of \eqref{HJBR} in $[0,T) \times \Int$, considering the test function 
$\psi_h \in\mathcal{C}^1([0,T)\times \Int)$ given by $\psi_{h}(t,x)= \psi(t,x)+ h^2 \sum_{i \in \ES} \ln (\rho_i(x)) - h (T-t)$, we get
\[
-\partial_t \psi_h(\bar{t},\bar{x}) - \widehat{\mathcal{H}}_{{2R}}\bigl(\bar{x}, D_x\psi_h(\bar{t},\bar{x})\bigr)-\widehat F(\bar{x}) 
\geq 0.
\]
Using the concavity of $\widehat{\mathcal{H}}_{{2R}}$ in the second argument, see 
{\eqref{eq:hamiltonians}}, and \eqref{ineqDD}, we obtain
\begin{align*}
0&\leq -\partial_t \psi(\bar{t},\bar{x}) -h  
- \widehat{\mathcal{H}}_{{2R}}\biggl(\bar{x}, D_x\psi(\bar{t},\bar{x})
+h^2 \sum\nolimits_{i \in \ES} \frac{D\rho_i(\bar{x})}{\rho_i(\bar{x})}
\biggr)-\widehat F(\bar{x})\\
&\leq -\partial_t \psi(\bar{t},\bar{x}) -h 
- \widehat{\mathcal{H}}_{{2R}}\bigl(\bar{x}, D_x\psi(\bar{t},\bar{x})\bigr)
-\widehat F(\bar{x}) 
\\
&\hspace{15pt}-\biggl\langle D_z \widehat{\mathcal H}_{{2R}}\biggl(\bar{x}, D_x\psi(\bar{t},\bar{x})
+h^2 \sum\nolimits_{i \in \ES} \frac{D\rho_i(\bar{x})}{\rho_i(\bar{x})}
\biggr) , h^2 \sum\nolimits_{i \in \ES} \frac{D\rho_i(\bar{x})}{\rho_i(\bar{x})} \biggr\rangle 
\\
&\leq -\partial_t \psi(\bar{t},\bar{x}) -h 
- \widehat{\mathcal H}_{{2R}}\bigl(\bar{x}, D_x\psi(\bar{t},\bar{x})\bigr) -\widehat F(\bar{x}) 
+ h^2d C_R,
\end{align*}
giving
\[
-\partial_t \psi(\bar{t},\bar{x})  
- \widehat{\mathcal{H}}_{{2R}}\bigl(\bar{x}, D_x\psi(\bar{t},\bar{x})\bigr) -\widehat F(\bar{x})
\geq h- h^2d C_R \geq 0 \qquad \mbox{ if } h\leq \frac{1}{dC_R},
\]
which implies that $v_h$ is a viscosity supersolution of \eqref{HJBR} on $[0,T) \times \Int$.
\vskip 4pt
\emph{Step 3}.
As $\rho_i\leq 1$, we have $v_h(t,x)\geq v(t,x)$ for any $(t,x) \in [0,T] \times \Int$.  In particular, $v_h(T,x)\geq v(T,x)\geq u(T,x)$ for any $(t,x) \in [0,T] \times \Int$. 
We denote $\rho(x)=\prod_{i=1}^d \rho_i(x)$. Since $u$ and $v$ are bounded, we find that for any $h>0$ there exists $\eta>0$ (which may depend on $h$) such that 
$-h^2 \ln \rho(x) \geq \|u\|_\infty + \|v\|_{\infty}$ if $\rho(x)\leq\eta$. We denote by 
$\Gamma^\eta= \{x\in\widehat{\mathcal S}_d : \rho(x) =\eta\}$, 
$\mathcal{O}^\eta= \{x\in\widehat{\mathcal S}_d : \rho(x) \geq\eta\}$, and 
$\mathcal{O}^\eta_c= \{x\in\widehat{\mathcal S}_d : \rho(x) \leq\eta\}$; note that $\mathcal{O}^\eta$ is a smooth domain.   Thus 
$v_h(t,x)\geq u(t,x)$ for any $t\in[0,T]$ and $x\in \mathcal{O}^\eta_c$, in particular for any $x\in\Gamma^\eta$. Therefore we can apply the comparison principle 
(Theorem 9.1 page 90 in \cite{flemingsoner}) in $[0,T] \times \mathcal{O}^\eta$, because 
$u,v_h \in\mathcal{C}([0,T] \times \mathcal{O}^\eta)$: we obtain $u\leq v_h$ on $[0,T] \times \mathcal{O}^\eta$ and hence $u \leq v_{h}$  on the entire $[0,T] \times \widehat{\mathcal S}_{d}$, since we already have 
$u \leq v_{h}$ on $[0,T] \times 
\mathcal{O}^\eta_{c}$. 
%
Finally, the conclusion follows by sending $h$ to 0, as $\lim_{h\rightarrow 0}v_h(t,x) = v(t,x)$ for any $(t,x) \in [0,T] 
\times 
{\Int}$.
\end{proof}

\subsection{Uniqueness of the MFG master equation}
We now turn to the analysis of \eqref{CONS}. Clearly, it has to be understood in the sense of distributions.  {We assume in this subsection that $F$ and $G$ are in $\mathcal{C}^{1,1}(\mathcal{S}_d)$}. The {multidimensional hyperbolic} system  \eqref{CONS} is known to be ill-posed in general; nevertheless, in this specific potential case, it is possible  to prove uniqueness of solutions in a suitable class, thanks to a result of Kruzkov \cite{kruzkov}.
We remark that the system is hyperbolic in the wide sense, but not strictly hyperbolic.
We denote ${\mathcal Q}_T=(0,T) \times \Int$, $\overline{\mathcal Q}_T=[0,T]\times \widehat{\mathcal S}_d$, 
$\mathfrak{f}(x,z)=\widehat{\mathcal{H}}(x,z)+\widehat F(x)$ and $\mathfrak{g}(x)= D_x \widehat G (x)$.

The set of weak solutions in which we prove uniqueness is the following: 
\begin{defn}
\label{def:admissible:solution}
A function $Z\in [{\mathcal C}([0,T];(L^\infty(\widehat{\mathcal S}_{d}),*))]^{{d-1}}$ 
(where $*$ denotes the weak star topology $\sigma^*(L^\infty(\widehat{\mathcal S}_{d}),
L^1(\widehat{\mathcal S}_{d}))$)
is said to be an \emph{admissible solution} to the Cauchy problem \eqref{CONS} if the following three properties hold true:
\begin{enumerate}
\item For any $\phi= (\phi^1,\dots,\phi^{{d-1}})\in 
\mathcal{C}^1_C({\mathcal Q}_T;\mathbb{R}^{{d-1}})$,
\be 
\label{weaksol}
\int_{{\mathcal Q}_T} \left[ Z^i \partial_t\phi^i +  \mathfrak{f}(x,Z) \partial_{x_{i}}\phi^i\right] dx dt =0 ;
\ee
\item At time $t=T$, $Z(T,\cdot) = {\mathfrak g}$ a.e.; 
in particular, by time continuity of $Z$ with respect to the weak star topology, 
\be 
\label{ini}
Z(t,\cdot) \overset{*}{\rightharpoonup}  \mathfrak{g}
\quad \textrm{as} \quad t \rightarrow T;
\ee
\item
There exists a universal constant $c$ such that, for any 
$\psi\in \mathcal{C}^1_C(\Int;\mathbb{R}_+)$ (where the index $C$ means that $\psi$ is compactly supported) and any nonnegative matrix $A=(A_{i,j})_{i,j \in \ESd}$ with 
{${\rm Trace}(A) \leq 1$},
\be 
\label{dersemi}
\int_{\Int} \left[ \langle D_x\psi, AZ \rangle + c\psi \right] dx \geq 0.
\ee
\end{enumerate}
\end{defn}
By Banach-Steinhaus theorem, note that 
$Z\in [{\mathcal C}([0,T];(L^\infty(\widehat{\mathcal S}_{d}),\sigma^*(L^\infty(\widehat{\mathcal S}_{d}),L^1(\widehat{\mathcal S}_{d}))))]^{{d-1}}$ implies $Z \in L^\infty({\mathcal Q}_{T};
{{\mathbb R}^{d-1}})$. 

Before we say more about the solvability of \eqref{CONS}, we feel useful to elucidate 
the connection between \eqref{CONS}
and the original form 
\eqref{eq:master:equation}
of the master equation. 
For sure, the main difference between the two is that the former is in conservative form while the latter is not, but also the reader must pay attention to the fact that \eqref{CONS} is in local coordinates $(x_{1},\cdots,{x_{d-1}})$ 
while \eqref{eq:master:equation}
is written in intrinsic coordinates $(p_{1},\cdots,p_{d})$. Obviously, \eqref{eq:master:equation} can be easily written in local coordinates, which makes it easier to compare with \eqref{CONS}. Similar to 
\eqref{derhjbchart}, but with $\phi=\epsilon=0$, the version in local coordinates writes
(indices in the sums belonging to $\ESd$):
\begin{equation}
\label{eq:master:eq:inviscid:local}
\left\{
\begin{array}{l}
\partial_{t} \widehat U^i+ H \bigl( (\widehat U^i - \widehat U^j)_{j \in \ES} \bigr) + \sum_{j,k} \Bigl(  x_{k} (\widehat U^k- \widehat U^{j})_{+}
- x_{j} (\widehat U^j - \widehat U^k)_{+} \Bigr) \partial_{x_{j}} \widehat U^i
\\
 \hspace{15pt}
 + \sum_{j} \Bigl( x^{-d} (\widehat U^d - \widehat U^{j})_{+} - x_{j} (\widehat U^{j}- \widehat U^{d})_{+} \Bigr)
 \partial_{x_{j}} \widehat U^i  + \hat f^i(x)  =0,
 \\
\widehat  U^i(T,x) = \hat g^i(x),
 \end{array}
\right.
\end{equation}
 for $(t,x) \in [0,T] \times \Int$ and $i \in \ES$. 
As we explained in Subsections 
\ref{subse:2:master} and \ref{subse:3:2}, 
the key step to pass from one formulation to another is 
Schwarz identity.
 The following statement clarifies this fact.  
%
%

\begin{prop}
\label{prop:classical:weak}
 We have
\begin{enumerate}
\item if $U\in[\mathcal{C}^1([0,T]\times {\mathcal S}_d)]^d$ is a classical solution of the master equation \eqref{eq:master:equation} and $\widehat U$ denotes its version in local coordinates, then $Z$ defined by $Z^i=\widehat U^i-\widehat U^d$, for $i \in \ESd$, is a weak admissible solution to \eqref{CONS}; it satisfies 
$\partial_{x_{j}} Z^i = \partial_{x_{i}} Z^j$ for any $i,j \in \ESd^2$;
\item if $Z$ is a weak solution to \eqref{CONS}, in the sense that it satisfies     \eqref{weaksol}, and $Z$ is in $[\mathcal{C}^1([0,T]\times\widehat{\mathcal S}_d)]^{{d-1}}$, then 
the master equation 
\eqref{eq:master:equation}
has a (unique) classical solution 
$U\in[\mathcal{C}^1([0,T]\times {\mathcal S}_d)]^d$; denoting $\widehat U$ its version in local coordinates, 
the latter satisfies
$Z^i=\widehat U^i-\widehat U^d$, for $i \in \ES$.
\end{enumerate}
\end{prop}

The proof of Proposition 
\ref{prop:classical:weak}
is postponed to the end of the section, as we feel 
better to focus now on the following statement, which is the refined version of Theorem \ref{thm:uniqueness:master}. 
Indeed, 
the next theorem establishes uniqueness of admissible solutions to \eqref{CONS}, by determining a correspondence with viscosity solutions to (\ref{HJBPOT}--\ref{HJBPOT:bc}). The proof is to establish first a connection between admissible solutions to \eqref{CONS} and \emph{semiconcave} solutions to (\ref{HJBPOT}--\ref{HJBPOT:bc}) and then to show that viscosity 
and semiconcave  solutions to (\ref{HJBPOT}--\ref{HJBPOT:bc}) are equivalent.
We recall that, in our case, a function $v\in \mathcal{C}([0,T]\times \widehat{\mathcal S}_d)$ is called semiconcave (in space) if 
there exists a constant $c$ such that, for any $t \in [0,T]$, $x\in\Int$ and $\xi$ with $x\pm\xi\in \Int$,
	\be
	\label{semicon} 
	\frac{v(t,x+\xi) - 2v(t,x)+ v(t,x-\xi)}{|\xi |^2}\leq c.
	\ee
	We stress that only semiconcavity in space is needed in the analysis below (for simplicity, we just call it semiconcavity), although the value function $\mathcal{V}$ is shown to be semiconcave in time and space in Proposition \ref{prop:solution:mfc:zero:epsilon} (v), see \eqref{semitimespace}.
	In this framework, condition \eqref{dersemi} can be referred to as a weak semiconcavity condition, since it reads as the derivative of the above condition. Indeed, assuming for a moment that $v$ is $\mathcal{C}^2$, \eqref{semicon} can be equivalently formulated by saying that $\sum_{i,j \in \ESd} A_{i,j} \partial^2_{x_{i},x_{j}} v \leq c$
		for any  {non-negative} matrix $A$ with 
		{$\textrm{\rm Trace}(A) \leq 1$ (write for instance $A$ as the square of a symmetric matrix)}. Hence, denoting $z=D_xv$ and integrating by parts, we obtain \eqref{dersemi}.
We say that $v$ is a semiconcave solution if \eqref{semicon} holds, $v$ is Lipschitz-continuous in $[0,T]\times \widehat{\mathcal S}_d$, Equation \eqref{HJBPOT} holds almost everywhere and the terminal condition
 \eqref{HJBPOT:bc}
 is satisfied (everywhere).

The proof of the following theorem is mostly due to Kruzkov \cite{kruzkov}, see Theorem 8 therein; for the sake of completeness, we write its adaptation to our framework (as the state variable here belongs to the simplex).
\begin{thm}
	\label{thm:6.6}
There exists a unique admissible solution to \eqref{CONS}. It is given by $D_x\widehat{\mathcal{V}}$, {where $\mathcal{V}$ is the value function of the inviscid MFCP and $\widehat{\mathcal V}$ is its version in local chart}.
\end{thm}

\begin{proof}
As we have just explained, we first establish a connection between admissible solutions to \eqref{CONS} and semiconcave solutions  to (\ref{HJBPOT}--\ref{HJBPOT:bc}) and, then, we show equivalence between semiconcave and viscosity solutions to 
(\ref{HJBPOT}--\ref{HJBPOT:bc}). 
\vskip 4pt

\emph{Step 1}. 
Let $Z\in
 [{\mathcal C}([0,T];(L^\infty(\widehat{\mathcal S}_{d}),*))]^{d-1}$ be an admissible solution to \eqref{CONS}. Let 
$w\in\mathcal{C}^2_C(\Int)$ and $\zeta\in \mathcal{C}^\infty_C((0,T))$, and for fixed $i\neq j$ choose as test functions 
$\phi^i(t,x)= \zeta(t) \partial_{x_{j}}w(x)$, $\phi^j(t,x)= \zeta(t) \partial_{x_{i}}w(x)$. Then \eqref{weaksol} provides
\[
\int_0^T \partial_t \zeta(t) \int_{\widehat{\mathcal S}_d}  \left[ Z^i(t,x) \partial_{x_{j}}w(x) -  
Z^j(t,x) \partial_{x_{i}}w(x)\right] dx dt =0,
\] 
which, by the fundamental lemma of the calculus of variations, implies that the quantity
$ 
\int_{\widehat{\mathcal S}_d}   [ Z^i(t,x) \partial_{x_{j}}w(x) -  
Z^j(t,x) \partial_{x_{i}}w(x) ] dx 
$ is a constant for almost every $t\in [0,T]$. Hence \eqref{ini} and the fact that the final condition is a gradient yield
\be
\label{curl}
\int_{\widehat{\mathcal S}_d}  \left[ Z^i(t,x) \partial_{x_{j}}w(x) -  
Z^j(t,x) \partial_{x_{i}}w(x)\right] dx  =0, \quad \text{for all} \  w\in\mathcal{C}^2_C(\Int),
\ee
which means that $Z(t,\cdot)$ admits a potential, in the weak sense,  for almost every $t$.
\vskip 4pt

\emph{Step 2}. Fix $(s,y)\in {\mathcal Q}_T$ and choose as test function $\varphi$ the mollification kernel 
$\rho_{h}(t,x)=h^{-d}\rho\left(  (s-t)/{h}, (y-x)/{h}\right)$. Then \eqref{weaksol} gives
\be 
\partial_t Z^i_h + \partial_{x_{i}} \mathfrak{f}_h = 0 \qquad \mbox{ in } {\mathcal Q}_T^h,
\ee
where $Z^i_h = \rho_h * Z^i$, $\mathfrak{f}_h = \rho_h * (\mathfrak{f}(\cdot,Z))$ and 
${\mathcal Q}_T^h$ is the set of $(s,y)$ in ${\mathcal Q}_T$ with a distance to the (time-space) boundary that is greater than or equal to $h$. Thanks to \eqref{curl}, $Z_h$ derives from a potential for fixed $t$, and the equation above implies that $(-\mathfrak{f}_h, Z^1_h, \dots, Z^{d}_h)$ also derives from a potential (but in time and space) for $(t,x)\in {\mathcal Q}_T^h$. 
Thus there exists a function $v_h$ defined in ${\mathcal Q}_T^h$ such that 
$\partial_t v_h= -\mathfrak{f}_h$ and $\partial_{x_{i}}v_h= Z_h^i$; since $v_h$ is defined up to a constant, we fix $v^h(T-h,x^M) = \widehat G(x^M)$, where $x^M=(1/{d},\dots,1/{d})\in\mathbb{R}^{{d-1}}$ is the point in the middle of the simplex.

By condition \eqref{dersemi}, substituting again the mollification kernel and integrating over $t$, we obtain, on ${\mathcal Q}_{T}^h$,  
\be 
\label{93}
\sum_{i,j \in \ESd} A_{i,j} \partial^2_{x_{i},x_{j}} v_h \leq c,
\ee
for any  {nonnegative} matrix $A$ with  {$\textrm{\rm Trace}(A) \leq 1$}, which implies in particular that 
for any vector $\nu$ with $|\nu|=1$ we have, also on 
${\mathcal Q}_{T}^h$,
\be 
\label{94}
\frac{\partial^2 v_{h}}{\partial \nu^2} \leq c.
\ee 
\vskip 4pt

\emph{Step 3}. Let $h\rightarrow 0$. We have $\lim_{h\rightarrow 0} \partial_t v_h = -\mathfrak{f}(x,Z)$ and
$\lim_{h\rightarrow 0} D_x v_{h} = Z$ almost everywhere in $\mathcal{Q}_T$. By Ascoli-Arzel\`a theorem 
and by boundedness of ${\mathfrak f}_{h}$ and $Z_{h}$, uniformly in $h >0$, 
the sequence $(v_h)_{h >0}$ is precompact in $\mathcal{C}(\overline{\mathcal Q}_T)$ endowed with the topology of uniform convergence, it being understood that we extend $v_h$ outside ${\mathcal Q}_T^h$ as a Lipschitz function. Let $v$ be any limit point. We have necessarily that $v$ is Lipschitz continuous (with a fixed Lipschitz constant) on $\overline{\mathcal Q}_T$ and has weak derivatives $\partial_t v = -\mathfrak{f}(x,Z)$ and
$D_x v = Z$ a.e. in ${\mathcal Q}_T$, proving that 
$\partial_t v + \mathfrak{f}(x,D_x v) =0$ a.e. in $\mathcal{Q}_T$. 

Since $v_{h}(T-h,x^M)=\widehat{G}(x^M)$, we get $v(T,x^M)=\widehat G(x^M)$. 
Moreover, for any test function 
$w\in\mathcal{C}^2_C(\Int)$ and any $h>0$ that is less than the distance $\textrm{\rm dist}(\textrm{\rm Supp}(w),
\partial \widehat{\mathcal S}_{d})$ 
from the support of $w$ to the boundary of the simplex, 
we have 
\begin{align}
&\int_{\widehat{\mathcal S}_{d}} v_{h}(T-h,x) D_{x} w(x) dx \nonumber
\\
&= - 
\int_{\widehat{\mathcal S}_{d}} Z_{h}(T-h,x) w(x) dx
\label{eq:limit:h:zero}
\\
&= 
- 
\int_{\widehat{\mathcal S}_{d}} (\rho_{h}*Z)(T-h,x) w(x) dx
=
- \int_{{\mathbb R}^{d}}
\rho_{h}(s,y)
\biggl[
\int_{\widehat{\mathcal S}_{d}}  Z(T-h-s,x) w(x+y) dx
\biggr] ds dy. \nonumber
\end{align}
By \eqref{ini}, 
we know that, for any $\vert y \vert \leq
\textrm{\rm dist}(\textrm{\rm Supp}(w),
\partial \widehat{\mathcal S}_{d})/2$, $\lim_{h \rightarrow 0}
[\int_{\widehat{\mathcal S}_{d}}  Z(T-h,x) w(x+y) dx]
= 
\int_{\widehat{\mathcal S}_{d}}  {\mathfrak g}(x) w(x+y) dx$. Since the function in argument of the limit is uniformly continuous with respect to 
$y$, the convergence holds uniformly with respect to $y$.
Hence, the right-hand side in 
\eqref{eq:limit:h:zero} converges to 
$- \int_{\widehat{\mathcal S}_{d}}  {\mathfrak g}(x) w(x) dx$. 
 Since the left-hand side in \eqref{eq:limit:h:zero} converges to $\int_{\widehat{\mathcal S}_{d}} v(T,x) D_{x} w(x) dx$, 
we deduce that   $D_{x} v(T,\cdot) = D_{x} \widehat{G}={\mathfrak g}$ a.e. and then 
$v(T,\cdot)=\widehat{G}$ on ${\widehat{\mathcal S}_{d}}$ since both are continuous and coincide in $x^M$. 

Lastly, by 
inequality \eqref{94} {(writing first the inequality below for $v_{h}$ and then taking the limit as $h$ tends to $0$)}
 \[
 v(t,x+\xi) - 2v(t,x)+ v(t,x-\xi)\leq c |\xi|^2,
 \]
for any $t \in (0,T)$, $x\in\Int$ and $\xi$ such that $x\pm\xi\in \Int$, thus \eqref{semicon} holds. 
%
%
%
Hence, we have proved that $v$ is a semiconcave solution to the Cauchy problem (\ref{HJBPOT}--\ref{HJBPOT:bc}) and $z=D_x v$ a.e. in $\mathcal{Q}_T$.
\vskip 4pt

\emph{Step 4}. On the converse, if $v$ is a semiconcave solution to \eqref{HJBPOT} then, for
any $t \in [0,T]$, $v(t,\cdot)$ is a.e. differentiable in $x$. By integration by parts, it is clear 
that, for any $w \in {\mathcal C}^1_{C}(\widehat{\mathcal S}_{d})$, the function $[0,T] \ni t \mapsto 
\int_{\widehat{\mathcal S}_{d}}
D_{x} v(t,x) w(x) dx$ is continuous. Since $D_{x} v \in L^\infty({\mathcal Q}_{T})$, the result easily extends to any 
$w \in L^1({\mathcal S}_{d})$, hence proving that $D_{x} v \in {\mathcal C}([0,T];(L^\infty({\mathcal S}_{d}),*))$.
Also, $v$ is a.e. differentiable in $(t,x)$ and the $(t,x)$-derivative clearly satisfies \eqref{weaksol}. Obviously, \eqref{ini} holds true. So we have just to  check \eqref{dersemi}. For any $h>0$, let $v_{h}':= \rho_{h}*v$ (it being understood that $v$ can be extended in a Lipschitz fashion outside ${\mathcal Q}_{T}$). From \eqref{semicon} we derive again inequality \eqref{94}, but with 
$v_{h}$ replaced by $v_{h}'$, and then \eqref{93} follows. 
Multiplying \eqref{93} by $\psi\in \mathcal{C}^1_C(\Int;\mathbb{R}_+)$ (provided that $h$ is smaller than 
$\mathrm{dist}(\textrm{\rm Supp}(\psi),
\partial \widehat{\mathcal S}_{d})$)
and integrating by parts we get 
{(for any 
nonnegative matrix $A$
with a trace lower than or equal to 1)}
\[
\int_{\Int} \left[ \langle D_x\psi, A D_x v_h'\rangle + c\psi \right] dx \geq 0,
\]
and letting $h\rightarrow 0$ we obtain \eqref{dersemi}.
\vskip 4pt

\emph{Step 5.} It remains to show that there is a correspondence between semiconcave and viscosity solutions to 
\eqref{HJBPOT}--\eqref{HJBPOT:bc}. 
By Corollary  \ref{CORUNIQ}, 
any viscosity solution $\mathcal{V}$ is in fact the value function of the MFCP. 
By items (iv) and (v) in Proposition 
\ref{prop:solution:mfc:zero:epsilon}, the value function is
Lipschitz continuous and
 semiconcave. 
 By  Proposition 3.1.7 in \cite{cannarsa}, ${\mathcal V}$ solves \eqref{HJBPOT} almost everywhere. On the converse, if $v$ is a semiconcave solution then it is also a viscosity solution on $[0,T) \times \Int$ by Theorem 10.2 in \cite{Lions_HJB}. 
By Corollary \ref{CORUNIQ}, it hence coincides with the value function.
\end{proof}

We now turn:

\begin{proof}[Proof of Proposition
\ref{prop:classical:weak}]
  { \ } 

\emph{Step 1.}
We first assume that $U$ is a classical solution of the master
equation \eqref{eq:master:equation} or equivalently that $\hat{U} \in[\mathcal{C}^1([0,T]\times\widehat{\mathcal S}_d)]^{d}$ is a classical solution to 
\eqref{eq:master:eq:inviscid:local}. 
Then, $(Z^i = \widehat U^i -\widehat U^d)_{i \in \ESd}$
is a (classical) solution of  
\begin{equation}
\label{eq:master:eq:inviscid:local:Z}
\left\{
\begin{array}{l}
\partial_{t} Z^i+ \hat{H}^i ( Z \bigr) - \hat{H}^d (Z) + \sum_{j,k} \Bigl(  x_{k} (Z^k- Z^{j})_{+}
- x_{j} (Z^j - Z^k)_{+} \Bigr) \partial_{x_{j}} Z^i
\\
 \hspace{15pt}
 + \sum_{j} \Bigl( x^{-d} (-Z^{j})_{+} - x_{j} (Z^{j})_{+} \Bigr)
 \partial_{x_{j}} Z^i  + \hat f^i(x) - \hat f^d(x) =0,
 \\
Z^i(T,x) = \hat g^i(x) - \hat g^d(x),
 \end{array}
\right.
\end{equation}
for $(t,x) \in [0,T] \times \Int$, $i \in \ESd$.
Obviously, the system of characteristics of 
\eqref{eq:master:eq:inviscid:local:Z}
is nothing but the Pontryagin system (in local coordinates) 
\eqref{eq:fb:inviscid:local}, {see (ii) in Proposition \ref{prop:solution:mfc:zero:epsilon}}. Hence, the fact that $Z$ is a classical solution 
of 
\eqref{eq:master:eq:inviscid:local:Z}
implies that
\eqref{eq:fb:inviscid:local}
has a unique solution,
for 
any initial condition $(t_{0},x_{0})$ of the forward equation in \eqref{eq:fb:inviscid:local}. The argument is pretty standard:
By expanding $(Z^i(t,x_{t}))_{t_{0} \le t \le T}$ and then comparing with the backward equation, 
we prove that any solution $({\boldsymbol x},{\boldsymbol z})$ of 
\eqref{eq:fb:inviscid:local}
must be of the form $z_{t}^i = Z^i(t,x_{t})$, for $i \in \ESd$ and $t_{0} \leq t \leq T$; 
Conversely, solving the forward equation with $z_{t}^i=Z^i(t,x_{t})$, for 
$i \in \ESd$ and $t_{0} \leq t \leq T$, we can indeed easily construct a solution. In turn, we deduce that the inviscid 
MFCP admits a unique optimizer:   By (i) in Proposition \ref{prop:solution:mfc:zero:epsilon},
there exists a minimizer; uniqueness follows from 
the fact 
the Pontryagin system \eqref{eq:fb:inviscid:local} has a unique solution. 
By Proposition 
\ref{prop:solution:mfc:zero:epsilon} (vii), the value function $\widehat{\mathcal V}$ of the MFCP is differentiable in any $(t_0,x_0)$ and, by point  (ix) in the same Proposition,  
$z^i_{t_0}=\partial_{x_{i}}\widehat {\mathcal V}(t_0,x_0)$, 
whenever 
the forward equation in \eqref{eq:fb:inviscid:local}
starts from $x_{0}$ at time $t_{0}$, 
but in turn 
$z^i_{t_{0}}=Z^i(t_{0},x_{0})$ hence showing that 
$Z^i(t,x)= \partial_{x_{i}}{\mathcal V}(t,x)$ for any $t\in[0,T]$ and $x\in\Int$, which implies that, on $[0,T] \times \Int$, $\widehat {\mathcal V}$ is $\mathcal{C}^2$ and thus $\partial_{x_{j}}Z^i= \partial_{x_i}Z^j$ for any $i,j \in \ESd$. 
Recalling \eqref{eq:mathcal H:5:9}, it is plain to see that 
$\partial_{x_{i}}
\widehat{\mathcal{H}}(x,Z)$ 
coincides with 
the nonlinear terms in 
\eqref{eq:master:eq:inviscid:local:Z}, which shows that 
$Z$ is a solution to \eqref{CONS} on $[0,T] \times \Int$. It is straightforward to see that it satisfies \eqref{weaksol} and \eqref{dersemi}, because $\mathcal{V}$ is $\mathcal{C}^2$ and (obviously) semiconcave.

\emph{Step 2}. Let $Z\in [\mathcal{C}^1([0,T]\times\widehat{\mathcal S}_d)]^{ {d-1}}$ satisfy \eqref{weaksol}. From \eqref{curl} and the fact that $Z\in [\mathcal{C}^1([0,T]\times\widehat{\mathcal S}_d)]^{ {d-1}}$, we obtain that $\partial_{x^j}Z^i= \partial_{x^i}Z^j$ on $[0,T] \times \Int$. Thus $Z$ solves \eqref{eq:master:eq:inviscid:local:Z} on $[0,T] \times \Int$, but then it solves the equation also at the boundary, because it is differentiable up to the boundary.
It remains to construct a classical solution to the master equation
\eqref{eq:master:eq:inviscid:local}. To do so, it suffices to solve 
\eqref{eq:master:eq:inviscid:local} with 
all the occurrences of 
$\widehat U^k - \widehat U^j$ replaced
by $Z^k - Z^j$ and all the occurrences of 
$\widehat U^j-\widehat{U}^d$ replaced by 
$Z^j$. By doing so, we hence solve a linear system of transport equations with 
a vector field that is ${\mathcal C}^1$. Despite the fact that the linear system is
set on the simplex, there is no real difficulty 
for proving that the solution is also 
${\mathcal C}^1$. 
\end{proof}

\appendix

\section{Wright-Fisher Spaces}

We describe the so-called Wright-Fisher spaces used in the paper, as recently introduced in the monograph of
Epstein and Mazzeo \cite{epsteinmazzeo}. We here follow {the exposition given} in \cite{mfggenetic}. 
In short, 
these Wright-Fischer spaces are  H\"older spaces, tailored made to the study of second order operators of the form
 \begin{equation}
 \label{eq:generator:kimura}
 \begin{split}
 {\mathcal L}_{t}  h(p)  &= 
 \sum_{i \in \ES} a_{i}(t,p) \partial_{p_{i}} h(p)
 + 
 \frac{\varepsilon^2}{2} \sum_{i,j \in \ES}  \bigl( p_{i} \delta_{i,j} - {p_{i} p_{j}} \bigr)
 \partial^2_{p_{i} p_{j}} h(p),
 \end{split} 
 \end{equation}
where $p \in \mathcal{S}_{d}$ and $a_{i}(p) \geq 0$ if $p_{i} =0$. As we already alluded to, such operators are called Kimura operators; we refer to \cite{ChenStroock,Kimura,Shimakura} for earlier analyses. 
Clearly, the second order term in \eqref{eq:generator:kimura} is degenerate, which is somehow the price to pay 
for forcing the corresponding SDE
to stay in the simplex; in fact, the latter SDE is nothing but a Wright-Fisher SDE of the same type as \eqref{dynpot}, at least for a relevant choice of $a$.
The
key feature is that, under the identification of 
${\mathcal S}_{d}$
with $\hat{\mathcal S}_{d}$  (see the introduction for the notation),
we may regard the simplex 
as a $(d\!-\!1)$-dimensional \textit{manifold with corners}, the corners being obtained by intersecting at most $d$ of the hyperplanes 
$\{x \in \RR^{d-1} : x_{1}=0\}$, $\dots$, 
$\{x \in \RR^{d-1} : x_{d-1}=0\}$,
$\{x \in \RR^{d-1} : x_{1}+\cdots+x_{d-1}=1\}$ with $\hat{\mathcal S}_{d}$ (we then call the codimension of the corner the number of hyperplanes showing up in the intersection). 
Accordingly, we can rewrite 
\eqref{eq:generator:kimura}
as an
operator acting on functions from $\hat{\mathcal S}_{d}$ to ${\mathbb R}$,  
by reformulating 
\eqref{eq:generator:kimura}
in terms of the sole $d\!-\!1$ first coordinates $(p_{1},\cdots,p_{d-1})$
or, more generally, in terms of $(p_{i})_{i \in \ES \setminus \{ l\}}$ for any 
given coordinate $l \in \ES$. 
Somehow, choosing the coordinate $l$ amounts to choosing a system of local coordinates and 
the choice of $l$ is mostly dictated by the position of $(p_{1},\cdots,p_{d})$ inside the simplex. 
Whenever all the entries of $p=(p_{1},\cdots,p_{d})$ are positive, meaning that 
$(p_{1},\cdots,p_{d})$ belongs to the interior 
of $\hat{\mathcal S}_{d}$, the choice of $l$ does not really matter and we work, for convenience, 
with $l=d$ ({which is, in fact, what we have done throughout the paper)}.  

In \cite[Subsection 2.3.1]{mfggenetic}, {it is shown} that the operator \eqref{eq:generator:kimura} fits the decomposition of \cite[Definition 2.2.1]{epsteinmazzeo}, which allows to use the Schauder-like theory developed in the latter reference. We do not repeat the computations here, but we recall the following two facts: Firstly, the operator \eqref{eq:generator:kimura} is elliptic non-degenerate in the interior of the simplex, when written 
in local coordinates in $\hat{\mathcal{S}_d}$ in the form  
\begin{equation}
\label{eq:reduced:L:d}
\hat {\mathcal L}_{t}  \hat h(x)  
=   \sum_{i \in \ESd} \hat a_{i}(t,x) \partial_{x_{i}} \hat h(x)
+ \frac{\varepsilon^2}{2} \sum_{i,j \in \ESd}  \bigl(x_{i} \delta_{i,j} - {x_{i} x_{j}} \bigr)
\partial^2_{x_{i} x_{j}} \hat h(x),
\end{equation}
where now $x \in \hat{\mathcal S}_{d}$, $\hat{h}$ is a smooth function 
on $\hat{\mathcal S}_{d}$ (which must be thought of $\hat{h}(x)=h(\check x)$)
 and $\hat{a}_{i}(t,x) = a_{i}(t,\check{x})$; Secondly, for a point in the relative interior of a corner of codimension $\ell$, there exist local coordinates, of the form  $(p_{i})_{i \in \ES \setminus \{ l\}}$ for a given $l$ depending on the shape of the corner, such that, in the new coordinates, the operator satisfies the normal form required in \cite[Definition 2.2.1]{epsteinmazzeo} ({the details of which are however useless here}).

Hence, for a point $x^{0} \in \hat{\mathcal S}_{d}$ in the relative interior of a corner ${\mathscr C}$ of $\hat{\mathcal S}_{d}$ of codimension $\ell \in \{0,\cdots,d\}$ (if 
$\ell=0$, then $x^{0}$ is in the interior of $\hat{\mathcal S}_{d}$), we may consider a 
new system of coordinates $(y_{1},\cdots,y_{d-1})$ (obtained as in the second point above) such that ${\mathscr C} = \{y \in \hat{\mathcal S}_{d} : y_{i_{1}}= \cdots = 
y_{i_{\ell}}=0\}$, for $1 \leq i_{1} < \cdots < i_{\ell}$. 
Letting 
$I:=\{i_{1},\cdots,i_{\ell}\}$ and 
denoting by 
$(y^0_{1},\cdots,y^0_{d-1})$ the coordinates of $x^0$ in the new system (for sure $y^0_{i_{j}}=0$ for 
$j=1,\cdots,\ell$), we may find a $\delta^0 >0$ such that:
\begin{enumerate}
	\item 
	the closure 
	$\overline{\mathcal U}(\delta^0,x^0)$ of ${\mathcal U}(\delta^0,x^0):= \{ y \in (\RR_{+})^{d-1} : \sup_{i \in \ESd} \vert y_{i} - y^0_{i} \vert < 
	\delta^0 \}$ is included in $\hat{\mathcal S}_{d}$,
	\item for $y$ in $\overline{\mathcal U}(\delta^0,x^0)$, for $j \not \in I$, $y_{j}>0$,
	\item for $y$ in $\overline{\mathcal U}(\delta^0,x^0)$, for $y_{1}+\cdots+y_{d-1}<1-\delta^0$. 
\end{enumerate}
A function $\hat{h}$ defined on $\overline{\mathcal U}(\delta^0,x^0)$ 
is then said to belong to ${\mathcal C}^{\gamma}_{\textrm{\rm WF}}(\overline{\mathcal U}(\delta^0,x^0))$, for some $\gamma \in (0,1)$, if, in the new system of coordinates, $\hat{h}$ is H\"older continuous on $\overline{\mathcal U}(\delta^0,x^0)$ with respect to the distance
\begin{equation}
\label{eq:distance:d}
d(y,y') := \sum_{i \in \ESd}  \bigl\vert \sqrt{ y_{i}} -\sqrt{ y_{i}'} \bigr\vert. 
\end{equation}
We then let
\begin{equation*}
\bigl\| \hat h \bigr\|_{\gamma;{\mathcal U}(\delta^0,x^0)}
:= \sup_{y \in \overline{\mathcal U}(\delta^0,x^0)}
\bigl\vert \hat h(y) \bigr\vert
+ \sup_{y,y' \in \overline{\mathcal U}(\delta^0,x^0)}
\frac{\vert \hat h(y) - \hat h(y') \vert}{d(y,y')^{\gamma}}.
\end{equation*}
Following \cite[Lemma 5.2.5 and Definition 10.1.1]{epsteinmazzeo}, 
we say that a function $\hat h$ defined on ${\mathcal U}(\delta^0,x^0)$
belongs to ${\mathcal C}^{2+\gamma}_{\textrm{\rm WF}}({\mathcal U}(\delta^0,x^0))$ if, in the new system of coordinates,
\begin{enumerate}
	\item 	
	$\hat{h}$ is continuously differentiable on
	${\mathcal U}(\delta^0,x^0)$ and $\hat{h}$ and its derivatives extend continuously to 
	$\overline{\mathcal U}(\delta^0,x^0)$ and the resulting extensions 
	belong to 
	${\mathcal C}^{\gamma}_{\textrm{\rm WF}}(\overline{\mathcal U}(\delta^0,x^0))$;
	\item $\hat{h}$ is twice continuously differentiable on 
	${\mathcal U}_{+}(\delta^0,x^0) = {\mathcal U}(\delta^0,x^0) \cap \{ (y_{1},\cdots,y_{d-1}) \in (\RR_{+})^{d} : \forall i \in I, y_{i} >0
	\}$. Moreover
	\begin{equation}
	\label{eq:derivatives:boundary:WF}
	\begin{split}
	&\lim_{\min(y_{i},y_{j}) \rightarrow 0_{+}}
	\sqrt{y_{i} y_{j}} \partial^2_{y_{i}y_{j}}
	\hat{h}(y) = 0,
	\quad \lim_{y_{i} \rightarrow 0_{+}}
	\sqrt{y_{i}} \partial^2_{y_{i}y_{k}}
	\hat{h}(y) = 0,
	\end{split}
	\end{equation}
	and the functions
	$y \mapsto  
	\sqrt{y_{i} y_{j}} \partial^2_{y_{i}y_{j}}
	\hat{h}(y)$, 
	$y \mapsto  
	\sqrt{y_{i}} \partial^2_{y_{i}y_{k}}
	\hat{h}(y)$
	and
	$y \mapsto  
	\partial^2_{y_{k}y_{l}}
	\hat{h}(y)$
	belong to ${\mathcal C}^{\gamma}_{\textrm{\rm WF}}(\overline{\mathcal U}(\delta^0,x^0))$ (meaning in particular that they can be extended by continuity to 
	$\overline{\mathcal U}(\delta^0,x^0)$).
\end{enumerate} 
We then let
\begin{equation*}
\begin{split}
\| \hat{h} \|_{2+\gamma;{\mathcal U}(\delta^0,x^0)}
&:= 
\| \hat{h} \|_{\gamma;{\mathcal U}(\delta^0,x^0)}
+
\sum_{i \in \ESd} 
\| \partial_{y_{i}} \hat{h} \|_{\gamma;{\mathcal U}(\delta^0,x^0)}
+
\sum_{i,j \in I}
\|
\sqrt{y_{i} y_{j}} \partial^2_{y_{i}y_{j}}
\hat{h}
\|_{\gamma;{\mathcal U}(\delta^0,x^0)}
\\
&\hspace{15pt} +
\sum_{k,l \not \in I}
\|
\partial^2_{y_{k}y_{l}}
\hat{h}
\|_{\gamma;{\mathcal U}(\delta^0,x^0)}
+
\sum_{i \in I}
\sum_{k  \not \in I}
\|
\sqrt{y_{i}} \partial^2_{y_{i}y_{k}}
\hat{h}
\|_{\gamma;{\mathcal U}(\delta^0,x^0)}, 
\end{split}
\end{equation*}
where 
$\sqrt{y_{i} y_{j}} \partial^2_{y_{i}y_{j}}
\hat{h}$ is a shorten notation for 
$y \mapsto 
\sqrt{y_{i} y_{j}} \partial^2_{y_{i}y_{j}}
\hat{h}(y)$  (and similarly for the others). 
For a given finite covering $\cup_{i=1}^K {\mathcal U}(\delta^0,x^{0,i})$ of $\hat {\mathcal S}_{d}$, which is then fixed in  {the rest of the discussion}, a function 
$\hat{h}$ is said to be in ${\mathcal C}^{\gamma}_{\textrm{\rm WF}}(\hat{\mathcal S}_{d})$, respectively in
${\mathcal C}^{2+\gamma}_{\textrm{\rm WF}}(\hat{\mathcal S}_{d})$ if 
$\hat{h}$ belongs to each ${\mathcal C}^{\gamma}_{\textrm{\rm WF}}({\mathcal U}(\delta^0,x^{0,i}))$, respectively each
${\mathcal C}^{2+\gamma}_{\textrm{\rm WF}}({\mathcal U}(\delta^0,x^{0,i}))$. Equivalently, we write
 {$h\in {\mathcal C}^{\gamma}_{\textrm{\rm WF}}({\mathcal S}_{d})$
(respectively 
$h\in {\mathcal C}^{2+\gamma}_{\textrm{\rm WF}}({\mathcal S}_{d})$)}, for a function $h$ defined on ${\mathcal S}_{d}$, if the associated function $\hat{h}$ defined on $\hat{\mathcal S}_{d}$ belongs to 
 {${\mathcal C}^{\gamma}_{\textrm{\rm WF}}(\hat{\mathcal S}_{d})$
(respectively 
${\mathcal C}^{2+\gamma}_{\textrm{\rm WF}}(\hat{\mathcal S}_{d})$)}.
We then let
\begin{align*}
\| \hat{h} \|_{\rm{WF},\gamma} := \sum_{i=1}^K \| \hat{h} \|_{\gamma;{\mathcal U}(\delta^0,x^{0,i})}, \quad
\| \hat{h} \|_{\rm{WF},2+\gamma} &:= \sum_{i=1}^K \| \hat{h} \|_{2+\gamma;{\mathcal U}(\delta^0,x^{0,i})}. 
\end{align*} 
We refer to 
\cite[Chapter 10]{epsteinmazzeo} and to \cite[Subsection 2.3]{mfggenetic} for more details. 
{Also, we feel useful to notice that, in Subsection \ref{subse:potential}, 
the spaces ${\mathcal C}^{\gamma}_{\textrm{\rm WF}}(\hat{\mathcal S}_{d})$
and 
${\mathcal C}^{2+\gamma}_{\textrm{\rm WF}}(\hat{\mathcal S}_{d})$
are denoted
${\mathcal C}^{0,\gamma}_{\textrm{\rm WF}}({\mathcal S}_{d})$
and 
${\mathcal C}^{0,2+\gamma}_{\textrm{\rm WF}}({\mathcal S}_{d})$, 
with a `0' in superscript and without a `hat' on ${\mathcal S}_{d}$, and similarly for the two norms 
$\| \hat{h} \|_{\rm{WF},\gamma}$
and 
$\| \hat{h} \|_{\rm{WF},2+\gamma}$, which are written 
$\| h \|_{\rm{WF},0,\gamma}$
and 
$\| h \|_{\rm{WF},0,2+\gamma}$ where $h : {\mathcal S}_{d} \rightarrow {\mathbb R}$ is canonically associated with 
$\hat h : \hat{\mathcal S}_{d} \rightarrow {\mathbb R}$. Our choice for inserting the additional index `0' is 
made clear below.}

\subsubsection*{Parabolic Wright-Fisher spaces}
 {Similar definitions hold for the spaces 
${\mathcal C}_{\textrm{\rm WF}}^{\gamma}([0,T] \times {\mathcal S}_{d})$
and
${\mathcal C}_{\textrm{\rm WF}}^{2+\gamma}([0,T] \times {\mathcal S}_{d})$. They are respectively spaces of
time-space functions that are 
$\gamma$-H\"older continuous functions and 
spaces of 
time-space functions that are 
continuously differentiable in time and twice continuously differentiable in space, with derivatives that are locally $\gamma$-H\"older continuous, 
H\"older continuity being understood in both cases
with respect to the time-space distance (in the local system of coordinates)}
\begin{equation}
\label{eq:distance:D}
D\bigl((t,y),(t',y')\bigr) := \vert t-t'\vert^{1/2} + d(y,y').
\end{equation}
To make it clear, a function $\hat{h}$ defined on $[0,T] \times \overline{\mathcal U}(\delta^0,x^0)$ 
is said to belong to ${\mathcal C}_{\textrm{\rm WF}}^{\gamma}([0,T] \times \overline{\mathcal U}(\delta^0,x^0))$, for some $\gamma \in (0,1)$, if, in the new system of coordinates, $\hat{h}$ is H\"older continuous on $[0,T] \times \overline{\mathcal U}(\delta^0,x^0)$ with respect to the distance
$D$.
We then let
\begin{equation*}
\bigl\| \hat h \bigr\|_{\gamma;[0,T] \times {\mathcal U}(\delta^0,x^0)}
:= \sup_{(t,y)  \in [0,T] \times \overline{\mathcal U}(\delta^0,x^0)}
\bigl\vert \hat h(t,y) \bigr\vert
+ \sup_{t,t' \in [0,T],\ y,y' \in \overline{\mathcal U}(\delta^0,x^0)}
\frac{\vert \hat h(t,y) - \hat h(t',y') \vert}{D((t,y),(t',y'))^{\gamma}}.
\end{equation*}
Following \cite[Lemma 5.2.7]{epsteinmazzeo}, we
say that  
a function $\hat h$ defined on $[0,T] \times {\mathcal U}(\delta^0,x^0)$
belongs to the space ${\mathcal C}_{\textrm{\rm WF}}^{2+\gamma}([0,T] \times {\mathcal U}(\delta^0,x^0))$ if, in the new system of coordinates,
\begin{enumerate}
	\item  $\hat{h}$ is continuously differentiable on
	$(0,T) \times {\mathcal U}(\delta^0,x^0)$ and $\hat{h}$ and its time and space derivatives extend continuously to 
	$[0,T] \times \overline{\mathcal U}(\delta^0,x^0)$ and the resulting extensions 
	belong to 
	${\mathcal C}^{\gamma}_{\textrm{\rm WF}}([0,T] \times \overline{\mathcal U}(\delta^0,x^0))$;
	\item $\hat{h}$ is twice continuously differentiable in space on 
	$[0,T] \times {\mathcal U}_{+}(\delta^0,x^0)$. Moreover, for any $i,j \in I$ and any $k,l \not \in I$,
	\begin{equation}
	\label{eq:rate:2nd:order}
	\begin{split}
	&\lim_{\min(y_{i},y_{j}) \rightarrow 0_{+}}
	\sqrt{y_{i} y_{j}} \partial^2_{y_{i}y_{j}}
	\hat{h}(t,y) = 0,
	\quad \lim_{y_{i} \rightarrow 0_{+}}
	\sqrt{y_{i}} \partial^2_{y_{i}y_{k}}
	\hat{h}(t,y) = 0,
	\end{split}
	\end{equation}
	and the functions
	$(t,y) \mapsto  
	\sqrt{y_{i} y_{j}} \partial^2_{y_{i}y_{j}}
	\hat{h}(t,y)$, 
	$(t,y) \mapsto  
	\sqrt{y_{i}} \partial^2_{y_{i}y_{k}}
	\hat{h}(y)$
	and
	$(t,y) \mapsto  
	\partial^2_{y_{k}y_{l}}
	\hat{h}(t,y)$
	belong to ${\mathcal C}_{\textrm{\rm WF}}^{\gamma}([0,T] \times \overline{\mathcal U}(\delta^0,x^0))$.
\end{enumerate} 
We then let
\begin{equation*}
\begin{split}
\| \hat{h} \|_{2+\gamma; [0,T] \times {\mathcal U}(\delta^0,x^0)}
&:= 
\| \hat{h} \|_{\gamma; [0,T] \times {\mathcal U}(\delta^0,x^0)}
+
\| \partial_{t} \hat{h} \|_{  \gamma; [0,T] \times {\mathcal U}(\delta^0,x^0)}
+
\sum_{i=1}^{d} 
\| \partial_{y_{i}} \hat{h} \|_{  \gamma;[0,T] \times  {\mathcal U}(\delta^0,x^0)}
\\
&\hspace{15pt} +
\sum_{i,j \in I}
\|
\sqrt{y_{i} y_{j}} \partial^2_{y_{i}y_{j}}
\hat{h} 
\|_{ \gamma; [0,T] \times {\mathcal U}(\delta^0,x^0)}
+
\sum_{k,l \not \in I}
\|
\partial^2_{y_{k}y_{l}}
\hat{h}
\|_{  \gamma; [0,T] \times {\mathcal U}(\delta^0,x^0)}
\\
&\hspace{15pt}+
\sum_{i \in I}
\sum_{k  \not \in I}
\|
\sqrt{y_{i}} \partial^2_{y_{i}y_{k}}
\hat{h}
\|_{ \gamma; [0,T] \times {\mathcal U}(\delta^0,x^0)}. 
\end{split}
\end{equation*}
For the fixed covering $\cup_{i=1}^K {\mathcal U}(\delta^0,x^{0,i})$ of $\hat{\mathcal S}_{d}$, a function 
$\hat{h}$  is said to be in ${\mathcal C}_{\textrm{\rm WF}}^{\gamma}([0,T] \times \hat{\mathcal S}_{d})$, respectively in 
${\mathcal C}_{\textrm{\rm WF}}^{2+\gamma}([0,T] \times \hat{\mathcal S}_{d})$ (as before, the definition extends  equivalently to the associated function $h$ defined on $[0,T] \times {\mathcal S}_{d}$)), if 
$\hat{h}$ belongs to each ${\mathcal C}_{\textrm{\rm WF}}^{\gamma}([0,T] \times {\mathcal U}(\delta^0,x^{0,i}))$, respectively each ${\mathcal C}_{\textrm{\rm WF}}^{2+\gamma}([0,T] \times {\mathcal U}(\delta^0,x^{0,i}))$.
We then let
\begin{align*}
\| \hat{h} \|_{\rm{WF},\gamma} := \sum_{i=1}^K \| \hat{h} \|_{\gamma;[0,T]\times{\mathcal U}(\delta^0,x^{0,i})},
\quad
\| \hat{h} \|_{\rm{WF},2+\gamma} := \sum_{i=1}^K \| \hat{h} \|_{2+\gamma;[0,T]\times{\mathcal U}(\delta^0,x^{0,i})}. 
\end{align*} 
{As before, we stress the fact that, in the core of the text,  we put an additional index `0' and we removed the `hat' in the notations 
${\mathcal C}_{\textrm{\rm WF}}^{\gamma}([0,T] \times \hat{\mathcal S}_{d})$,
${\mathcal C}_{\textrm{\rm WF}}^{2+\gamma}([0,T] \times \hat{\mathcal S}_{d})$, 
$\| \hat h \|_{\rm{WF},\gamma}$
and 
$\| \hat h \|_{\rm{WF},2+\gamma}$, hence writing
${\mathcal C}_{\textrm{\rm WF}}^{0,\gamma}([0,T] \times {\mathcal S}_{d})$,
${\mathcal C}_{\textrm{\rm WF}}^{0,2+\gamma}([0,T] \times {\mathcal S}_{d})$, 
$\| h \|_{\rm{WF},0,\gamma}$
and 
$\| h \|_{\rm{WF},0,2+\gamma}$.}

\subsubsection*{Hybrid spaces}
\label{subse:hybrid}
{We now introduce hybrid spaces of functions with mixed classical and Wright-Fisher regularity.
Again, this notion is directly borrowed from \cite[Chapter 5]{epsteinmazzeo}.
More precisely, a function 
$h$, defined on ${\mathcal S}_{d}$, belongs to 
${\mathcal C}_{\textrm{\rm WF}}^{1,\gamma}({\mathcal S}_{d})$
(respectively 
${\mathcal C}_{\textrm{\rm WF}}^{1,2+\gamma}({\mathcal S}_{d})$), for some $\gamma \in (0,1)$, if 
it is continuously differentiable on ${\mathcal S}_{d}$ (meaning that it is continuously differentiable 
on the interior and the derivatives extend by continuity up to the boundary)  
and each ${\mathfrak d}_{i}h$, for $i \in \ES$, belongs to  
${\mathcal C}_{\textrm{\rm WF}}^{0,\gamma}({\mathcal S}_{d})$
(respectively 
${\mathcal C}_{\textrm{\rm WF}}^{0,2+\gamma}({\mathcal S}_{d})$). 
For $h \in {\mathcal C}_{\textrm{\rm WF}}^{1,\gamma}({\mathcal S}_{d})$, we then let
\begin{equation*}
\| h \|_{\rm{WF},1,\gamma}
:= \| h \|_{\infty} + 
\sum_{i \in \ES} \bigl\| {\mathfrak d}_{i} h \bigr\|_{\rm{WF},0,\gamma},
\end{equation*}
and, for $h \in {\mathcal C}_{\textrm{\rm WF}}^{1,2+\gamma}({\mathcal S}_{d})$, we let
\begin{equation*}
\| h \|_{\rm{WF},1,2+\gamma}
:= \| h \|_{\infty} + 
\sum_{i \in \ES} \bigl\| {\mathfrak d}_{i} h \bigr\|_{\rm{WF},0,2+\gamma}.
\end{equation*}

The parabolic version of ${\mathcal C}_{\textrm{\rm WF}}^{1,2+\gamma}({\mathcal S}_{d})$
(which is the only one we need in the text) 
is defined in a similar way. 
A function 
$h$, defined on $[0,T] \times {\mathcal S}_{d}$, belongs to 
${\mathcal C}_{\textrm{\rm WF}}^{1,2+\gamma}([0,T] \times {\mathcal S}_{d})$, for some $\gamma \in (0,1)$, if 
it 
belongs to 
 ${\mathcal C}_{\textrm{\rm WF}}^{0,2+\gamma}([0,T] \times {\mathcal S}_{d})$
 (and is hence differentiable in space) 
 and each ${\mathfrak d}_{i}h$, for $i \in \ES$, belongs to  
${\mathcal C}_{\textrm{\rm WF}}^{0,2+\gamma}([0,T] \times {\mathcal S}_{d})$. 
For $h \in {\mathcal C}_{\textrm{\rm WF}}^{1,2+\gamma}([0,T] \times {\mathcal S}_{d})$, we then let
\begin{equation*}
\| h \|_{\rm{WF},1,2+\gamma}
:= \| h \|_{\rm{WF},0,\gamma} +
 \| \partial_{t} h \|_{\rm{WF},0,\gamma}
 + 
\sum_{i \in \ES} \bigl\| {\mathfrak d}_{i} h \bigr\|_{\rm{WF},0,2+\gamma}.
\end{equation*}}

\bibliographystyle{abbrv}
\bibliography{references}
%
%
%
%
%
%
%
%
%
%
%
%
%
%

\end{document}